\documentclass[reqno]{amsart}
\usepackage[lite]{amsrefs}
\usepackage{mathtools,amsmath,amssymb}
\usepackage[pdfencoding=auto, psdextra]{hyperref}
\hypersetup{colorlinks=true,allcolors=[rgb]{0,0,0.6}}
\usepackage[hmargin= 1.3 in,vmargin=1 in]{geometry}
\usepackage{bbm,mathrsfs}
\usepackage{xcolor}
\usepackage[all,cmtip]{xy}
\usepackage{appendix}
\usepackage{mathrsfs}
\usepackage{soul}

\newcommand{\R}{\mathbb{R}}
\newcommand{\Z}{\mathbb{Z}}
\newcommand{\N}{\mathbb{N}}

\newcommand{\smoothcyl}{\mathscr{C}^{\infty}_{\rm{c},{\rm cyl}}(H^{-s})}
\newcommand{\Dspace}{\mathbb{D}^{1,2}(\mu_{0})}
\newcommand{\Drspace}{\mathbb{D}^{1,2}_R(\mu_{0})}
\newcommand{\free}{\mu_{0}}

\newcommand{\ii}{\mathrm{i}}

\newcommand{\T}{\mathbb{T}}
\newcommand{\Real}{\mathrm{Re}}

\newcommand{\supp}{\mathrm{supp}}

\newcommand{\mcH}{\mathcal{H}}

\numberwithin{equation}{section}

\theoremstyle{plain} 
\newtheorem{theorem}{Theorem}[section]
\newtheorem*{theorem*}{Theorem}
\newtheorem{lemma}[theorem]{Lemma}
\newtheorem*{lemma*}{Lemma}
\newtheorem{corollary}[theorem]{Corollary}
\newtheorem*{corollary*}{Corollary}
\newtheorem{proposition}[theorem]{Proposition}
\newtheorem*{proposition*}{Proposition}

\newtheorem*{conjecture*}{Conjecture}

\theoremstyle{definition} 
\newtheorem{definition}[theorem]{Definition}
\newtheorem*{definition*}{Definition}
\newtheorem{example}[theorem]{Example}
\newtheorem*{example*}{Example}
\newtheorem{remark}[theorem]{Remark}
\newtheorem*{remark*}{Remark}
\newtheorem{assumption}[theorem]{Assumption}
\newtheorem*{assumption*}{Assumption}

\newcommand{\ee}{\mathrm{e}}

\newcommand{\bb}{C_0}
\newcommand{\bbb}{C_1}
\newcommand{\bbbb}{C_2}
\newcommand{\bbbbb}{C_3}
\newcommand{\bbbbbb}{C_4}
\newcommand{\bbo}{c_0}
\newcommand{\bboo}{c_1}
\newcommand{\bbooo}{c_2}
\newcommand{\BB}{R}

\title[Gibbs measures and local  KMS states for the focusing NLS]{Gibbs measures as local equilibrium KMS states for focusing nonlinear Schr\"odinger equations}
\author{Zied Ammari}
\address{Univ Rennes, [UR1], CNRS, IRMAR - UMR 6625, F-35000 Rennes, France.}
\email{zied.ammari@univ-rennes.fr}
\author{Andrew Rout}
\address{Univ Rennes, [UR1], CNRS, IRMAR - UMR 6625, F-35000 Rennes, France.}
\email{andrew.rout@univ-rennes.fr}
\author{Vedran Sohinger}
\address{University of Warwick, Mathematics Institute, Zeeman Building, Coventry CV4 7AL, United Kingdom.}
\email{V.Sohinger@warwick.ac.uk.}
\date{December 3, 2024.}

\begin{document}


\begin{abstract}
In this paper, we are concerned with the study of statistical equilibria for focusing  nonlinear Schr\"{o}dinger and Hartree equations on the $d$-dimensional torus $\T^d$ when $d=1,2,3$. Due to the focusing nature of the nonlinearity in these PDEs,  Gibbs measures have to be appropriately localized.   First, we show that these local Gibbs measures are stationary solutions for the Liouville probability density equation and that they satisfy a local equilibrium Kubo-Martin-Schwinger (KMS) condition. Secondly, under some natural assumptions, we characterize all possible local KMS equilibrium  states for these PDEs as local
Gibbs measures.  
Our methods are based on Malliavin calculus in Gross-Stroock Sobolev spaces and on a suitable Gaussian integration by parts formula.  To handle the technical problems due to localization, we rely on the works of Aida and Kusuoka on irreducibility of Dirichlet forms over infinite-dimensional domains. This leads us to the study of  sublevel sets of the renormalized mass and their connectedness properties. 
In this paper, we also revisit Bourgain's proof of the normalizability of the local Gibbs measure for the focusing Hartree equation on $\T^d$ with $d=2,3$
by using concentration inequalities.
\end{abstract}
\maketitle

\section{Introduction}

In the last several decades, there has been a huge interest in studying stochastic and deterministic non-linear parabolic and dispersive PDEs by means of probabilistic methods \cites{FrHai20,GP18,Bourgain_Textbook}. The main motivation is the construction of stochastic or random flows and the existence of invariant measures for  singular or irregular equations.  A significant portion of these PDEs describe  Hamiltonian physical systems which are naturally equipped with formal Gibbs measures. In many cases, these Gibbs measures have been rigorously constructed and random flows that keep them invariant have been built over low regularity spaces.  
The existence of random global flows for dispersive nonlinear Schr\"odinger (NLS) equations goes back to the pioneering work of Bourgain \cites{Bou94,Bou96}. 
Here, the randomness refers to an almost surely defined flow with respect to the Gibbs or Gaussian measure.
In this framework, the Gibbs measure serves as a substitute for a conservation law at low regularity and can thus be used to show  almost sure well-posedness of the corresponding dispersive PDE at low-regularity.  There is by now a vast literature on this problem. We refer the reader to the expository works \cites{BTT18,NS19,OT18} and the references therein for further details. On the other hand, one of the widely open questions in this topic is the description of the large time  statistical behaviour of these flows. To address such a problem, we propose a two-step strategy. The first step is to understand which measures are statistical equilibria  for the non-linear PDE and whether these equilibria are unique or not. To describe such measures one has to  study stationary solutions of the Liouville probability density equation related to the PDE at hand and to find a convenient way to characterize equilibrium states. The second step is to build an ergodic theory around these equilibrium states and to study the asymptotic behaviour of the flows. Of course, ergodic and mixing properties for non-dissipative infinite dimensional 
 systems are generally a challenging  question. Although there is 
only limited
 literature on this subject for dispersive PDEs,  we  believe that this is an  interesting and tractable problem. The ergodicity  issue for NLS and Hartree equations is one of the main  motivations of the present article. In this work, we focus on the first step while the second will be studied in a future work.  Our work here can also be regarded as an improvement of the series of articles \cites{AS21,AFS22,AFS23}, where global solutions were built for general non-linear initial value problems and KMS equilibrium states were generally studied  but only applied for defocusing dispersive PDEs. 

\medskip

In this paper, we consider  the focusing NLS and Hartree equations on the torus $\T^d$, $d=1,2,3$.  On the one hand, we aim to build a statistical equilibrium theory for these focusing dispersive PDEs, and on the other, we aim to construct global solutions almost surely by the method introduced in \cite{AFS23}. The main issue  to understand whether the Gibbs measures are stationary solutions for the Liouville probability density equation and to what extent they represent statistical thermal equilibria. To study this, we  use the fundamental concept of Kubo-Martin-Schwinger condition introduced   by Galavotti and Verboven in \cite{GV75} for classical statistical mechanical systems, which allows to characterize thermal equilibria for extended systems. Our results provide the first characterization of local thermal equilibria for focusing dispersive PDEs on the torus although there is no unique global Gibbs measure in these cases. In the following, we describe  Gibbs measures and KMS states in more details.

\subsection{Gibbs measures}

Let $\mathcal{H}$ be a complex Hilbert space and let $h :\mathcal{H}\to \R$ be a given Hamiltonian functional.
The \emph{Gibbs measure} associated with  $h$ is formally defined as
\begin{equation}
\label{Gibbs_measure_formal}
d \mu_{\mathrm{Gibbs}} := \frac{1}{z}\, \mathrm{e}^{-h} \, d u\,,
\end{equation}
where $z:= \int \mathrm{e}^{-h} \, d u$ is a normalization constant making $\mu_{\mathrm{Gibbs}}$ a probability measure on the space of fields $u \in \mathcal{H}$. We refer to $z$ as the \emph{partition function}.
Above,  $d u$ is the formal Lebesgue measure on the space $\mathcal{H}$ (which is ill-defined when $\mathcal{H}$ is infinite-dimensional).
The rigorous construction of such measures was first given in the framework of constructive quantum field theory. We refer the reader to \cites{GJ87,Nel73a,Nel73b,Sim74} for a detailed discussion and general overview. For more recent works, see also \cites{aizenman1982geometric,aizenman2021marginal, BFS83,BG20,BS96,CFL16,frohlich1982triviality,GH21,LRS88,MV97a,MV97b,OST22,Zhi91} and the references within.  Given a suitable Poisson structure  $\{\cdot,\cdot\}$ on the space of fields, one formally obtains that \eqref{Gibbs_measure_formal} is invariant under the Hamiltonian flow associated with the Hamiltonian $h$. This follows from a (formal) application of Liouville's theorem\footnote{for invariance of Lebesgue measure under the flow of a divergence free vector field.} and conservation of $h$ under the aforementioned flow. However, this invariance property  as well as the construction of the flow itself 
are
far from obvious for dispersive PDEs.     

\medskip
The main example that we will study in this paper is that of the \emph{nonlinear Schr\"{o}dinger equation (NLS)} on the spatial domain $\Lambda=\T^d$ with $d=1,2,3$, which is given by
\begin{equation}
\label{NLS_introduction}
\mathrm{i}\partial_t u + (\Delta -1)u=\mathcal{N}(u)\,.
\end{equation}
In \eqref{NLS_introduction}, $u :\Lambda \times \R \rightarrow \mathbb{C}$ and $\mathcal{N}(u)$ is a nonlinearity of the form 
\begin{equation}
\label{nonlinearity_1}
\mathcal{N}(u)=\pm |u|^{r-1}u 
\end{equation}
for some $r>1$ or 
\begin{equation}
\label{nonlinearity_2}
\mathcal{N}(u)=-(V*|u|^2)\,u
\end{equation}
for some even, integrable function $V: \Lambda \rightarrow \R$. Here, $*$ denotes convolution. The nonlinearity with sign $``+"$ in \eqref{nonlinearity_1} is referred to as \emph{defocusing} with the sign $``-"$ as \emph{focusing}. The choice in \eqref{nonlinearity_2} is often referred to as the \emph{Hartree nonlinearity}. 
One can  realize \eqref{NLS_introduction} as an infinite-dimensional Hamiltonian equation, where the Hamiltonian is
\begin{equation}
\label{h_1}
h(u)=\frac{1}{2} \int_{\Lambda} (|\nabla u|^2+|u|^2)\, d x \pm \frac{1}{r+1} \int_{\Lambda} |u|^{r+1}\, dx 
\end{equation}
when $\mathcal{N}(u)$ is given by \eqref{nonlinearity_1}, and 
\begin{equation}
\label{h_2}
h(u)=\frac{1}{2} \int_{\Lambda} (|\nabla u|^2+|u|^2)\, d x - \frac{1}{4} \int_{\Lambda} (V*|u|^2)\,|u|^2 \, dx
\end{equation}
when $\mathcal{N}(u)$ is given by \eqref{nonlinearity_2}. When studying the Hartree nonlinearity \eqref{nonlinearity_2}, we refer to the problem as defocusing if $\int_{\Lambda} (V*|u|^2)\,|u|^2 \, dx \leq 0$. Otherwise, we refer to the problem as focusing (which we can also interpret as the two-body interaction being indefinite). We note that this is consistent with the earlier terminology if we formally set $V=c\delta$ and consider $r=3$. In practice, when working in dimensions $d=2,3$, we renormalize the interaction in \eqref{h_1}--\eqref{h_2} by Wick ordering; see \eqref{nonlinear_Hamiltonian} and Assumption \ref{Assumption_on_V} below.
 In the context of the NLS on the torus $\Lambda=\mathbb{T}^d$ with $d=1,2,3$, the aforementioned invariance was first rigorously proved by Bourgain \cites{Bou94,Bou96,Bou97}.

\medskip

If $h$ is not bounded from below, in particular in the focusing regime, it is not always possible to normalize the Gibbs measure as defined in \eqref{Gibbs_measure_formal}. In other words, it is possible for the partition function $z$ to be infinite. In this case, a solution is to modify the problem and consider a \emph{truncated Gibbs measure}, which is formally defined as
\begin{equation}
\label{focusing_Gibbs}
d \mu_{\mathrm{Gibbs},\BB} := \frac{1}{z_\BB}\, \mathrm{e}^{-h} \,\Xi_{\BB}(\mathcal{M}(u))\, d u\,, \qquad z_\BB:=\int\mathrm{e}^{-h} \,\Xi_{\BB}(\mathcal{M}(u))\, d u\,,
\end{equation}
for fixed $\BB>0$.
Here, $\Xi_{\BB} : \R \rightarrow [0,\infty)$ is a non-zero function which vanishes outside of $(-\BB,\BB)$. For example, one could consider $\Xi_{\BB}=\mathbbm{1}_{(-\BB,\BB)}$, i.e.\ the indicator function on $(-\BB,\BB)$.
The quantity $\mathcal{M}(u)$ is usually either the mass of $u$, i.e. $\|u\|_{L^2}^2$, or a suitably renormalized quantity (which is not necessarily nonnegative). For the concrete example of the nonlinear Schr\"{o}dinger or Hartree equation on $\T^d$ with $d=1,2,3$, the precise definition of $\mathcal{M}(u)$ is given in \eqref{renormalized_mass}. 
It can be shown that this construction yields a well-defined probability measure in many relevant examples. This problem was first studied in \cite{LRS88} and \cite{Bou94} and later in \cites{AS21,Bou97,BS96,BTT13,CFL16,Deng12, DR23,DRTW23,Liang_Wang,LLW23,OST22,OOT20,RSTW23,RS22,RS23,TW23,Xian}.
Assuming that $\mu_{\mathrm{Gibbs},R}$ in \eqref{focusing_Gibbs} is well-defined, we henceforth refer to it as a \emph{local or truncated Gibbs measure}. The precise definition is given in \eqref{local_Gibbs_measure_rigorous} below.

\medskip
It is  first worth noting that it is possible to consider Gibbs measures for focusing Hartree equations without using a cut-off as in \eqref{focusing_Gibbs}. When $d=3$ (with interaction potential as in Assumption \ref{Assumption_on_V} (i) below), this is done in \cite{OOT20}. In particular, in \cite[Equation (1.7)]{OOT20}, one tames the weight in the Gibbs measure with a suitable power of the Wick-ordered mass, which makes it normalizable. This framework is necessary when applying methods from stochastic quantization. 
Secondly, the invariance of Gibbs measures has also been studied in infinite volume, i.e. with $\Lambda=\R^d$. The first result in infinite volume is that of Bourgain \cite{Bou2000}. Subsequent results have been obtained in \cites{BL22,CdS1,CdS2,DR23,OTWZ22,Rider}.  Thirdly,  Gibbs measures are also of interest in the study of mean-field limits of Bose gases.  Namely, one can obtain Gibbs measures \eqref{Gibbs_measure_formal} and \eqref{focusing_Gibbs} as a suitable mean-field limit of quantum Gibbs states arising from many-body quantum mechanics; see \cites{AR21,AFP24,FKSS17,FKSS18,FKSS20,FKSS20_2,FKSS20_3,FKSS22,LNR15,LNR18,LNR18b,LNR19,LNR21,RS22,RS23,Soh19}. 
We do not discuss further the latter possible extensions and connections in this paper.

\subsection{The Kubo--Martin--Schwinger (KMS) condition}

The \emph{Kubo--Martin--Schwinger (KMS) condition} is a criterion in quantum statistical mechanics used to characterize  equilibrium states for infinite quantum systems; see \cites{BR97,HHW67}. In the 1970s, a classical analogue of the quantum KMS condition was suggested by Galavotti and Verboven, and its relationship to the Dobrushin-Lanford-Ruelle (DLR) condition was studied; see \cites{AGGLM77,GV75}. The classical KMS condition was also studied in the setting of infinite-dimensional classical systems in \cites{FPV77,PR77}, and in the setting of nonlinear PDEs in \cites{Ars83,Chu86,Pes85} and more recently for singular PDEs in \cites{AS21,FKV24}. The many-body derivation of Gibbs states mentioned earlier was obtained for the Bose-Hubbard model on a finite graph using a limiting correspondence between quantum and classical KMS states in \cite{AR21} (see also \cite{AFP24}). In what follows, we always work with the classical KMS condition and henceforth we just refer to it as the KMS condition.

\medskip
We now give a formal discussion of the global KMS condition. Suppose that the   complex Hilbert space $\mathcal{H}$ is regarded  as a phase space equipped with a convenient {\it Poisson bracket} $\{\cdot,\cdot\}:\mathscr{C}^{\infty}(\mathcal{H}) \times \mathscr{C}^{\infty}(\mathcal{H}) \to \mathscr{C}^{\infty}(\mathcal{H})$. Let $X: \mathcal{H} \to \mathcal{H}$ be a (smooth) vector field. We consider the initial value problem given by
\begin{equation*}
\dot{u}(t) = X(u(t)),
\end{equation*}
where $u:\mathbb{R} \to \mathcal{H}$ is a solution with initial condition $u(0) \in \mathcal{H}$. Let $\langle \cdot,\cdot \rangle$ be the inner product on the phase space  $\mathcal{H}$ and suppose $F,G: \mathcal{H} \to \R$ are smooth functions.  Then a measure $\mu$ is called a \emph{(global) KMS state} if
\begin{equation}
\label{KMS_condition_1}
\int_{\mathcal{H}} \{F,G\}(u) \,  d\mu = \int_{\mathcal{H}} \langle \nabla F(u), X(u) \rangle\, G(u) \, d \mu\,,
\end{equation}
for all test functions $F,G$ in an appropriate subclass of $\mathscr{C}^{\infty}(\mathcal{H})$. Here $\nabla F$ denotes the Fr\'echet derivatives of $F$. 
We refer the reader to \cite{AS21} for a detailed analysis of the Kubo--Martin--Schwinger boundary condition. In \cite{AS21}, the KMS condition was studied with a general inverse temperature parameter $\beta>0$. Throughout this paper we set $\beta=1$ for simplicity of notation. The inverse temperature parameter can easily be reintroduced in the definitions and proofs with minor modifications. 
Suppose $(\mu_t)_{t\in \R}$ is a solution to the {\it Liouville equation} with given initial condition $\mu_0$, i.e.
\begin{equation}
\label{Liouville_equation}
\frac{d}{d t} \int_{\mathcal{H}} F(u) \, d \mu_t = \int_{\mathcal{H}} \langle \nabla F (u), X(u)\rangle \, d \mu_t\,,
\end{equation}
for all suitable test functions $F$. Then the KMS condition is an accepted criterion for characterising  thermal equilibrium states (probability measures) among all possible stationary solutions of the Liouville equation \eqref{Liouville_equation}. Indeed, by formally\footnote{The function $G=1$ strictly speaking does note satisfy the condition of a test function in general; see Definition \ref{KMS_definition_rigorous} below.} setting $G=1$ in \eqref{KMS_condition_1}, one sees that any KMS state is a stationary solution of \eqref{Liouville_equation}, in the sense that 
\begin{equation}
\label{stationary_solution_Liouville_equation}
\frac{d}{d t} \int_{\mathcal{H}} F(u) \,d \mu_t=0 \quad \text{and} \quad  \int_{\mathcal{H}} \langle \nabla F (u), X(u)\rangle \, d \mu_0=0,
\end{equation} 
for all test functions $F$. The characterization of  (global) KMS states \eqref{KMS_condition_1}, as Gibbs measures \eqref{Gibbs_measure_formal}, was obtained in the context of dispersive defocusing PDEs by the first and third author in \cite[Theorems 4.11 and 4.14]{AS21}. The framework of the KMS condition was made rigorous by means of the Malliavin derivative and Gross-Sobolev spaces. Recently, the relationship between (global) KMS states, Gibbs measures, and the  Dobrushin–Lanford–Ruelle (DLR) condition was studied in the context of the Korteweg-de Vries (KdV) and modified Korteweg-de Vries (mKdV) equations in \cite{FKV24}. This setup requires a different symplectic structure from the one we use; see \cite[Proposition 2.5 and Theorem 2.6]{FKV24} for precise results. 

\medskip
For the focusing nonlinear  Schr\"{o}dinger equation in one dimension, a local KMS condition was proposed in \cite{AS21} and it was shown  \cite[Proposition 5.7]{AS21} that truncated  Gibbs measures \eqref{focusing_Gibbs} are \emph{local KMS states}. The latter are obtained by modifying  \eqref{KMS_condition_1} to consider one of the test functions to be localized on a ball where $|\mathcal{M}(u)| <\BB'$ for suitable $\BB' \in (0,\BB)$; see Definition \ref{local_KMS_definition} below for a precise statement.  
In this paper, we firstly show that truncated Gibbs measures \eqref{focusing_Gibbs} are local KMS states in dimensions $d=1,2,3$ for the focusing NLS and Hartree equations. Secondly and more importantly, we characterize all local KMS states of these focusing equations as local Gibbs measures. 
Let us note that our convention of local KMS states slightly differs from that in \cite{AS21}, as it is more natural and allows the free Gibbs measure to satisfy the local KMS condition; see Proposition \ref{local_KMS_approximation_lemma} below for a detailed explanation. 

\medskip
Recently the first and third author together with S.~Farhat \cite{AFS23} showed that the KMS condition is closely related to the existence of almost sure global solutions for the aforementioned Hamiltonian systems. The result proved in \cite{AFS23} was obtained for a more general class of initial problems through a new method based on the study of stationary solutions of the Liouville equation \eqref{Liouville_equation}. However, in the applications, the almost sure existence of global solutions was only proved for defocusing dispersive PDEs. Here, we extend this method to focusing NLS and Hartree PDEs on the torus. We refer the reader to the introduction of \cites{AS21,AFS22,AFS23} for further discussion of the relation of the KMS condition and the Liouville equation; see also \cite{Ammari_Liard} for a detailed discussion on the Liouville equation and the background for the techniques used in \cite{AFS23}. 

\medskip
Our main results are stated in Theorem \ref{homogeneous_KMS_thm}, Corollary \ref{almost_sure_WP_corollary}, Theorems  \ref{Gibbs_implies_local_KMS_theorem}  and 
\ref{kms_implies_gibbs_theorem}. 
Theorem \ref{homogeneous_KMS_thm} proves that local Gibbs measures \eqref{focusing_Gibbs} are stationary solutions of the Liouville equation \eqref{Liouville_equation}, in the sense of \eqref{stationary_solution_Liouville_equation}. In particular, by the general theory in \cite{AFS23}, this gives an alternative proof of global existence of solutions of focusing nonlinear Schr\"{o}dinger and Hartree equations for almost all initial data with respect to the local Gibbs measure \cites{Bou94,Bou97}; see Corollary \ref{almost_sure_WP_corollary}. Theorem \ref{Gibbs_implies_local_KMS_theorem} shows one direction of the equivalence between local Gibbs measures  and local KMS states. In particular, it proves that local Gibbs measures are local KMS states. The converse direction of this equivalence, i.e.\ the claim that local KMS states are local Gibbs measures, is proved in  Theorem 
\ref{kms_implies_gibbs_theorem}.

\bigskip
\emph{Acknowledgements:}  This research has been funded by the ANR-DFG project  (ANR-22-CE92-0013, DFG PE 3245/3-1 and BA 1477/15-1).  The authors are very grateful to Fumio Hiroshima for pointing us to the article of S. Aida \cite{Ai98}.  Z.A. \& V.S. would like to thank Nicolas Camps and Nata\v{s}a Pavlovi\'{c}   for useful discussions. V.S. would like to thank J\"{u}rg Fr\"{o}hlich, Antti Knowles and Benjamin Schlein for helpful discussions concerning the reference \cite{Bou97}.

\section{Setting and results}
\subsection{General Setting}
Before precisely stating our main results, we first give our notational conventions and recall the general setup. Since we are dealing with several PDEs, it is convenient to use one single general framework which unifies these examples.  
\subsection*{Notation}
Throughout the paper, we use $C>0$ to denote a constant that might change line to line. If this constant depends on a finite set of parameters $a_1,\ldots,a_n$, we write $C=C(a_1,\ldots,a_n)$.  We sometimes abbreviate $a \leq Cb$ by $a\lesssim b$, and write $a \lesssim_{a_1,\ldots,a_n} b$ to indicate the dependence of the implied constant on the quantities $a_1,\ldots,a_n$. We use $\mathbbm{1}_J$ to denote the indicator on a set $J$. Namely
\begin{equation}
\label{indicator_function}
\mathbbm{1}_{J}(x) := \begin{cases}
1 \quad \textrm{if } x \in J, \\
0 \quad \textrm{if } x \notin J.
\end{cases}
\end{equation}
We denote by $1$ the identity operator on a Hilbert space.
Throughout, we use the convention $\N=\{1,2,\ldots\}$. Given $f \in L^1(\T^d)$ and $k \in \Z^d$, we write  
\begin{equation}
\label{Fourier_transform}
\widehat{f}(k):=\int_{\T^d} f(x)\,\ee^{-2\pi \ii x\cdot k}\,dx
\end{equation}
for the Fourier coefficient of $f$ at $k$.
Given a Hilbert space $\mathcal{H}$, we denote by $\mathscr{P}(\mathcal{H})$ the set of all Borel probability measures on $\mathcal{H}$.

\subsection*{Linear Hamiltonian structure}
Let $\mathcal{H}$ be a complex Hilbert space with inner product $\langle \cdot, \cdot \rangle \equiv \langle \cdot, \cdot \rangle _{\mathcal{H}}$. In our convention, this inner product is linear in the first component and conjugate linear in the second component. In our analysis, we also consider the corresponding real Hilbert space obtained by setting the real inner product to be
\begin{equation}
\label{real_inner_product}
\langle \cdot,\cdot\rangle_{\mcH,\R} := \mathrm{Re}\,\langle\cdot,\cdot\rangle\,.
\end{equation}
We denote this real Hilbert space by $\mcH_\R$.

Throughout the paper, we fix the Hilbert space $\mathcal{H}$ and the linear operator $A$ as follows.

\begin{assumption}
\label{A_choice}
Let $\mathcal{H}=L^2(\T^d)$ for $d=1,2,3$ with the inner product $\langle f,g \rangle=\int_{\T^d} f(x)\overline{g(x)}\,dx$ and let $A=-\Delta+1$ which acts on $\mathcal{H}$ as a densely-defined operator. Here $-\Delta$ denotes the Laplacian operator on $L^2(\T^d)$. 
\end{assumption}
Let us note some properties of $A$ as in Assumption \ref{A_choice}.
\begin{lemma}
\label{A_choice_lemma}
Let $A$ be as in Assumption \ref{A_choice}. Then $A$ satisfies the following properties.
\begin{itemize}
\item[(i)] $A$ is strictly positive, i.e.  there exists $c>0$ such that 
\begin{equation}
\label{positive_operator}
A \geq c 1\,.
\end{equation}
\item[(ii)] $A$ has a compact resolvent.
\item[(iii)] We have that $\mathrm{Tr}({A^{-1-s}})<\infty$ for 
\begin{equation}
\label{choice_of_s}
s>\frac{d}{2}-1\,.
\end{equation}
\end{itemize}
\end{lemma}
\begin{proof}
The result is a consequence of 
\begin{equation}
\label{choice_of_s_consequence}
\sum_{k \in \Z^d} \frac{1}{(4\pi^2 |k|^2+1)^{s+1}}<\infty\,,
\end{equation}
for $s$ as in \eqref{choice_of_s}.
\end{proof}
\begin{assumption}
\label{choice_of_s_assumption}
For the remainder of the paper, we consider $d=1,2,3$ and  
\begin{equation}\label{eq.cond.spower}
    s \in (\frac{d}{2}-1,1].
\end{equation}
In particular, $s$ satisfies \eqref{choice_of_s}.
\end{assumption}

From Lemma \ref{A_choice_lemma}, it follows that we can find an orthonormal basis
\begin{equation}
\label{Hilbert_space_ONB}
\mathbf{B}=\{e_j\}_{j \in \N}
\end{equation}
of $\mathcal{H}$ consisting of eigenvectors of $A$.
Let $\{\lambda_j\}_{j \in \N}$ be the associated sequence of eigenvalues. 
We henceforth fix the orthonormal basis in \eqref{Hilbert_space_ONB} to be such that for all $j \in \N$, we have
\begin{equation}
\label{e_j_choice}
A e_j=\lambda_j e_j\,.
\end{equation}
Moreover, we can rewrite Lemma \ref{A_choice_lemma} (iii) as 
\begin{equation}
\label{trace_assumption}
\mathrm{Tr}({A^{-1-s}}) = \sum_{j=1}^\infty \frac{1}{\lambda_j^{1+s}} < \infty\,.
\end{equation}
Given $\alpha \in \R$, we define the Sobolev space $H^\alpha$ using $A$. More precisely, we define the inner product
\begin{equation*}
\langle x,y \rangle_{H^\alpha} := \langle A^{\alpha/2}x,A^{\alpha/2}y\rangle, \quad \forall x,y \in \mathcal{D}(A^{\alpha/2})\,.
\end{equation*}
Let $\alpha \geq 0$ be given. We denote by $H^{\alpha}$ the Hilbert space $(\mathcal{D}((A^{\alpha/2}), \langle \cdot,\cdot \rangle_{H^\alpha})$. The corresponding space $H^{-\alpha}$ is the completion of the Hilbert space $(\mathcal{D}((A^{-\alpha/2}), \langle \cdot,\cdot \rangle_{H^{-\alpha}})$. Furthermore, we have the embeddings $H^{\alpha} \subset \mcH \subset H^{-\alpha}$. We also note that $H^{-\alpha}$ is the dual of $H^{\alpha}$ with respect to the pairing induced by $\mcH$. Finally, we note that $H^0=\mathcal{H}$.

\subsection*{Gaussian measures on Sobolev spaces}

Let us summarize a few notions from Gaussian measures on Hilbert spaces that we will use in the discussion that follows. For proofs and a more detailed discussion, we refer the reader to \cite[Chapter 2]{Bog98}.
We recall that $m \in \mathcal{H}$ is said to be the \emph{mean-vector} of $\mu \in \mathscr{P}(\mathcal{H})$ if for any $g \in \mathcal{H}$ the function $u \mapsto \langle g, u \rangle_{\mathcal{H},\R}$ is integrable with respect to $\mu$ and 
\begin{equation*}
\langle g, m \rangle_{\mathcal{H},\R}=\int_{\mathcal{H}}  \langle g, u \rangle_{\mathcal{H},\R} \, d\mu\,.
\end{equation*}
When $m=0$, we say that $\mu$ is a mean-zero or centred measure. The covariance of $\mu \in \mathfrak{\mathcal{H}}$ is an $\R$-linear operator $Q:\mathcal{H} \rightarrow \mathcal{H}$ such that for any $g_1, g_2 \in \mathcal{H}$, we have
\begin{equation*}
\langle g_1,Qg_2 \rangle_{\mathcal{H},\R}= \int_{\mathcal{H}} \langle g_1,u-m \rangle_{\mathcal{H},\R}\,
\langle u-m,g_2 \rangle_{\mathcal{H},\R}\,d \mu\,.
\end{equation*}

The following result holds as a result of Lemma \ref{A_choice_lemma}; see \cite[Chapter 2]{Bog98}.
\begin{proposition}
\label{free_Gibbs_prop}
There exists  a unique $\mu_0 \in \mathscr{P}(H^{-s})$ with the property that for all $g_1,g_2 \in H^{-s}$, we have
\begin{equation}
\label{Gibbs_covariance}
\langle g_1,A^{-1-s}g_2 \rangle_{H^{-s},\R} = \int_{H^{-s}} \langle g_1,u\rangle_{H^{-s},\R}\, \langle u,g_2\rangle_{H^{-s},\R} \, d \mu_0\,,
\end{equation}
or equivalently for all $g_{1},g_{2} \in \mathcal{H}$, we have 
\begin{equation}
\label{Gibbs_covariance_2}
\langle g_1,A^{-1}g_2 \rangle_{\mathcal{H},\R} = \int_{H^{-s}} \langle g_1,u \rangle_{\mathcal{H},\R} \,\langle u,g_2 \rangle_{\mathcal{H},\R} \, d \mu_0\,.
\end{equation}
$\mu_0$ is called the Gaussian measure on $H^{-s}$ with mean zero and covariance $A^{-1-s}$.
\end{proposition}

\subsection*{Renormalized mass}

With $A$ given by Assumption \ref{A_choice}, we define the quantity $\mathcal{M}(u)$ in which we truncate to define the local Gibbs measure \eqref{focusing_Gibbs}.
\begin{equation}
\label{renormalized_mass}
\mathcal{M}(u) := 
\begin{cases}
\|u\|_{L^2}^2 \quad &\textrm{if } d=1\,, \\
:\|u\|_{L^2}^2: \quad &\textrm{if } d=2,3\,.
\end{cases}
\end{equation}
Here, $:\|u\|_{L^2}^2:$ denotes the \emph{renormalized} or \emph{Wick-ordered mass}. In order to define it, we treat the mass $\|u\|_{L^2}^2$ as a random variable with respect to the Gaussian measure $\mu_0$ with covariance $A^{-1-s}$ from Proposition \ref{free_Gibbs_prop}, and we then perform Wick ordering with respect to $\mu_0$. The details of the latter construction are summarized in Lemma \ref{Wick_ordered_mass_lemma} of Appendix \ref{Malliavin derivative calculations and normalizability of the Gibbs measure for the focusing Hartree equation}.

\begin{remark}
We know that $\int \|u\|_{L^2}^2\, d \mu_0<\infty$ when $d=1$ (see \eqref{classical_free_field} below). However, one can also show that this expectation is infinite when $d>1$.
Lemma \ref{Wick_ordered_mass_lemma} guarantees that $\int |\mathcal{M}(u)|\,d \mu_0<\infty$ for $d=2,3$. Hence, one considers the two different cases depending on the dimension in \eqref{renormalized_mass}.
\end{remark}

\subsection*{Malliavin calculus}
In order to rigorously interpret the global KMS condition formally given by \eqref{KMS_condition_1} above, and in order to rigorously define the local KMS condition, we work in the framework of Malliavin calculus, analogously to \cite[Section 4.3]{AS21}. We summarize the setup of Malliavin calculus and recall the main results that we will use in the subsequent analysis. For a more detailed discussion, we refer the reader to \cite[Section 4.3 and Appendix A]{AS21}, as well as \cite{Nua06}.

\subsubsection*{Cylindrical test functions}

Let us consider the orthonormal basis \eqref{Hilbert_space_ONB} of $\mathcal{H}$ satisfying \eqref{e_j_choice}. We write 
\begin{equation}
\label{ONB_convention}
e_j^{(1)}:=e_j\,,\qquad e_j^{(2)}:=i e_j\,.
\end{equation}
Therefore, 
\begin{equation*}
\{e^{(1)}_j,e^{(2)}_j\}_{j \in \N} 
\end{equation*}
is an orthonormal basis of $\mathcal{H}_{\R}$ with inner product given by \eqref{real_inner_product}. With this notational convention, we now define the space of cylindrical test functions. For a fixed $n \in \N$, we define the mapping $\pi_n: H^{-s} \to \R^{2n}$ by
\begin{equation}
\label{projection_map}
\pi_n(u) := \left( \langle x,e^{(1)}_1 \rangle_{\mcH,\R} , \ldots, \langle x,e^{(1)}_n \rangle_{\mcH,\R}; \langle x,e^{(2)}_1 \rangle_{\mcH,\R} , \ldots , \langle x,e^{(2)}_n \rangle_{\mcH,\R} \right).
\end{equation}
\begin{definition}[Cylindrical test functions]
\label{cylindrical_test_functions}
We define $\mathscr{C}^{\infty}_{\mathrm{c},\mathrm{cyl}}(H^{-s})$ to be the set of all functions $F: H^{-s} \to \R$ such that there exists $n \in \N$ and $\varphi \in \mathscr{C}^{\infty}_{\text{c}}(\R^{2n})$ such that
\begin{equation}
\label{cylindrical_test_functions_1}
F = \varphi \circ \pi_n\,.
\end{equation}
We define the space 
$\mathscr{C}^{\infty}_{\mathrm{b},\mathrm{cyl}}(H^{-s})$ analogously, where in \eqref{cylindrical_test_functions_1} we take $\varphi \in \mathscr{C}^{\infty}_{\mathrm{b}}(\R^{2n})$. Here $\mathscr{C}^{\infty}_{\mathrm{b}}(\R^{2n})$ denotes the space of smooth functions with bounded derivatives of all order. We refer to $\mathscr{C}^{\infty}_{\mathrm{c},\mathrm{cyl}}(H^{-s})$ and $\mathscr{C}^{\infty}_{\mathrm{b},\mathrm{cyl}}(H^{-s})$ as classes of \emph{cylindrical test functions}.
\end{definition}
Given $F \in \mathscr{C}^{\infty}_{\mathrm{b},\mathrm{cyl}}(H^{-s})$, we define its gradient $\nabla F: H^{-s} \rightarrow H^{-s}$ by 
\begin{equation}
\label{gradient_cylindrical_F}
\nabla F (u) := \sum_{j=1}^n \partial_j^{(1)} \varphi(\pi_n(u)) e^{(1)}_j + \partial_j^{(2)} \varphi(\pi_n(u)) e^{(2)}_j\,.
\end{equation}
In \eqref{gradient_cylindrical_F}, $\partial_j^{(1)}$ and $\partial_{j}^{(2)}$ denote respectively the $j^{th}$ and $(n+j)^{th}$ partial derivatives of $\varphi$. 
 
\begin{definition}[Malliavin Derivative]
\label{Malliavin_definition}
For $p \in [1,\infty)$, we define the Malliavin derivative as the closure of the 
linear operator $\nabla : \mathcal{D} \subset L^p(\mu_{0}) \to L^p(\mu_0;H^{-s})$ acting according to  \eqref{gradient_cylindrical_F}
on the domain $ \mathcal{D} =\mathscr{C}^{\infty}_{\mathrm{c},\mathrm{cyl}}(H^{-s})$. In the following, we also denote the closure by $\nabla$. 
\end{definition}
The linear operator $\nabla : \mathcal{D} \subset L^p(\mu_{0}) \to L^p(\mu_0;H^{-s})$ is closable; see \cite[Lemma 4.9]{AS21}.
In light of Definition \ref{Malliavin_definition}, we define the \emph{Gross-Sobolev space}.
\begin{definition}[Gross-Sobolev Space]
\label{Gross-Sobolev_space}
For $p \in [1,\infty)$ we  set $\mathbb{D}^{1,p}(\mu_0)$ to be the domain of the Malliavin derivative $\nabla$ and endow it with the (graph) norm
\begin{equation}
\label{Dspace_norm}
\|F\|_{\mathbb{D}^{1,p}(\mu_0)}^p := \|F\|_{L^p(\mu_0)}^p + \|\nabla F\|_{L^p(\mu_0;H^{-s})}^p\,.
\end{equation}
\end{definition}
The space $\mathbb{D}^{1,p}(\mu_0)$ for $p \in [1,\infty)$ is a Banach space by \cite[Lemma 4.9]{AS21}.
If $p=2$, we can even endow the Gross-Sobolev space with the inner product
\begin{equation*}
\langle F_1,F_2\rangle_{\mathbb{D}^{1,2}(\mu_0)} := \langle F_1, F_2 \rangle_{L^2(\mu_{0})} + \langle \nabla F_1, \nabla F_2\rangle_{L^2(\mu_0;H^{-s})}\,,
\end{equation*}
making $\mathbb{D}^{1,2}(\free)$ into a Hilbert space. 

\begin{lemma}
\label{Malliavin_derivative_M_lemma}
The quantity $\mathcal{M}$ defined in \eqref{renormalized_mass} belongs to $\mathbb{D}^{1,p}(\mu_{0})$ for all $p \in [1,\infty)$ and we have 
\begin{equation}
\label{Wick_ordered_mass_Malliavin_derivative}
\nabla \mathcal{M}(u)=2u\,.
\end{equation}
\end{lemma}
We give the proof of Lemma \ref{Malliavin_derivative_M_lemma} in Section \ref{Proof_of_Lemma_2.10} of  Appendix \ref{Malliavin derivative calculations and normalizability of the Gibbs measure for the focusing Hartree equation}.

\subsubsection*{Poisson structure}
Let $F,G \in \mathscr{C}^{\infty}_{\mathrm{b},\mathrm{cyl}}(H^{-s})$ be given.  We define their Poisson bracket $\{F,G\}:H^{-s} \rightarrow \R$ as follows. By Definition \ref{cylindrical_test_functions}, we can write
\begin{equation*}
F = \varphi \circ \pi_m\,, \quad G=\psi \circ \pi_n\,,
\end{equation*}
for some $m,n \in \N$ and $\varphi \in \mathscr{C}_{\mathrm{b}}^\infty(\R^{2m}), \psi \in \mathscr{C}_{\mathrm{b}}^\infty(\R^{2n})$.
We let
\begin{equation}
\label{poisson_cylindrical}
\{F,G\}(u) := \sum_{j=1}^{\min(m,n)} \partial^{(1)}_j \varphi(\pi_m(u)) \partial^{(2)}_j \psi(\pi_n(u)) - \partial^{(1)}_j \psi(\pi_n(u)) \partial^{(2)}_j \varphi(\pi_m(u))\,.
\end{equation}

Later, we extend the Poisson bracket to a larger class of functions (See \eqref{poisson_general} for a more general definition).

\subsection*{Nonlinear Hamiltonian structure and the local Gibbs measure}
We now define the nonlinear Hamiltonian $h:H^{-s} \rightarrow \R$. It has the form
\begin{equation}
\label{nonlinear_Hamiltonian}
h=h_0-h^I\,.
\end{equation} 
The non-interacting part of the Hamiltonian $h_0$ in \eqref{nonlinear_Hamiltonian} is given by
\begin{equation}
\label{nonlinear_Hamiltonian_B}
h_0:\mathcal{D}(A^{1/2}) \rightarrow \R\,,\qquad h_0(u):=\frac{1}{2} \langle u,Au \rangle\,,
\end{equation}
for $A$ as in Assumption \ref{A_choice}.

For the nonlinear (interacting) part $h^I$ of the Hamiltonian in \eqref{nonlinear_Hamiltonian}, we make the following assumptions.

\begin{assumption}
\label{Assumption_on_V} 
We consider two cases depending on whether $d=1$ or $d=2,3$.
\begin{itemize}
\item[(i)] When $d=1$, we consider either the local nonlinearity 
\begin{equation}
\label{local_nonlinearity_1D}
h^I(u):=\frac{1}{r+1} \int_{\T} |u|^{r+1}\, dx \,,
\end{equation}
for $r \in \{3,5\}$ or the nonlocal Hartree nonlinearity
\begin{equation}
\label{nonlocal_nonlinearity_1D}
h^I(u):=\frac{1}{4} \int_{\T} \int_{\T} |u(x)|^2\,V(x-y)\,|u(y)|^2\,dx\,dy\,, 
\end{equation}
where $V \in L^1(\T^d)$ is even and real-valued. The local nonlinearity \eqref{local_nonlinearity_1D} is called cubic or quintic when $r=3$ or $r=5$ respectively.
\item[(ii)] When $d=2,3$, we fix $V \in L^1(\T^d)$ even and real-valued such that there exist $\varepsilon \in (0,1)$ and  $C>0$ with the property that for all $k \in \Z^d$, we have
\begin{equation}
\label{V_assumption}
|\widehat{V}(k)| \leq 
\begin{cases}
\frac{C}{(1+|k|)^{\varepsilon}} &\mbox{if } d=2
\\
\frac{C}{(1+|k|)^{2+\varepsilon}}  &\mbox{if } d=3\,.
\end{cases}
\end{equation}
With $V$ as in \eqref{V_assumption}, we let
\begin{equation}
\label{h^I}
h^I(u):=\frac{1}{4}\,\int_{\mathbb{T}^d} \,\int_{\mathbb{T}^d} \,:|u(x)|^2:\,V(x-y)\,:|u(y)|^2:\,d x\,dy\,.
\end{equation}
\end{itemize}
\end{assumption}
\begin{remark} Let us make a few comments on Assumption \ref{Assumption_on_V}.
\begin{itemize}
\item[(i)] 

By our sign conventions, the nonlinearity \eqref{nonlocal_nonlinearity_1D} gives a focusing contribution to the Hamiltonian \eqref{nonlinear_Hamiltonian} in the standard PDE sense, meaning that $-h^I \leq 0$. The same is true for \eqref{nonlocal_nonlinearity_1D} when $V$ is pointwise nonnegative or of positive type and for \eqref{h^I} when $V$ is of positive type. Here, we say that $V$ is of positive type if $\widehat{V}$ is pointwise nonnegative. The latter two observations follow since for any $f : \T^d \rightarrow \R$ square integrable, we have 
\begin{equation}
\label{positive_type}
\frac{1}{4} \int_{\T^d} \int_{\T^d} f(x)\,V(x-y)\,f(y)\,dx \,dy= \frac{1}{4} \sum_{k \in \Z^d} \widehat{V}(k) |\widehat{f}(k)|^2 \geq 0\,.
\end{equation}
One formally takes $f=|u|^2$ and $f=:|u|^2:$ in \eqref{positive_type} when $d=1$ and $d=2,3$ respectively. This step is rigorously justified by applying a frequency truncation as in \eqref{P_n}--\eqref{u_n} and \eqref{:|u_n|^2:} below, and by taking a limit in $L^p(\mu_0)$ using Wick's theorem. We refer the reader to \cite[Section 3]{FKSS17} for a self-contained explanation of this procedure when $d=1$ and to \cite[Section 3]{Soh19} when $d=2,3$. We emphasize that none of these arguments rely on the sign of the nonlinearity.
\item[(ii)] If $-V$ is pointwise nonnegative or of positive type in \eqref{nonlocal_nonlinearity_1D} or if $-V$ is of positive type in \eqref{h^I}, then we are in the defocusing regime, meaning that $h^I \leq 0$. Here, one can consider Gibbs measures without truncation and their equivalence with global KMS states. This analysis was done in \cite{AS21}.
\item[(iii)]When $d=2,3$, we can slightly generalize condition \eqref{V_assumption} as follows. Let us note that we can write 
\begin{equation}
\label{V_hat_splitting}
V=V_{p}+V_{n}\,,\,\mbox{where}\,\,\widehat{V}_{p}\geq 0 \,\,\mbox{and}\,\,\widehat{V}_n \leq 0\,\, \mbox{pointwise}.
\end{equation}
Our analysis carries over if instead of \eqref{V_assumption}, we assume that for all $k \in \Z^d$
\begin{equation}
\label{V_p_assumption}
\widehat{V}_{p}(k) \leq 
\begin{cases}
\frac{C}{(1+|k|)^{\varepsilon}} &\mbox{if } d=2
\\
\frac{C}{(1+|k|)^{2+\varepsilon}}  &\mbox{if } d=3\,.
\end{cases}
\end{equation}
In order to see this, we use \eqref{V_hat_splitting} and \eqref{h^I_rewritten} below to write
\begin{align*}
h^I(u)&=\frac{1}{4} \,\sum_{k \in \Z^d \setminus \{0\}} \bigl|\widehat{|u|^2}(k)\bigr|^2\,\widehat{V}(k)+\frac{1}{4} \,\widehat{V}(0)\,\Bigl(:\|u\|_{L^2}^2:\Bigr)^2
\\
&\leq
\frac{1}{4} \,\sum_{k \in \Z^d \setminus \{0\}} \bigl|\widehat{|u|^2}(k)\bigr|^2\,\widehat{V}_{p}(k)+\frac{1}{4} \,\widehat{V}_{p}(0)\,\Bigl(:\|u\|_{L^2}^2:\Bigr)^2 \,.
\end{align*}
The argument given in Appendix \ref{Proof of Bourgain's large deviation estimate} follows under the assumption \eqref{V_p_assumption}.

\item[(iv)] In \eqref{h^I}, let us note that we  Wick order each factor of $|u|^2$ separately. In particular, we do not Wick-order the full quartic nonlinearity as in \cite{Bou96}. A self-contained rigorous construction of \eqref{h^I} is given in Section \ref{Proof_of_Proposition_2.13_(i)} below. 
\item[(v)] When $d=2$, the condition \eqref{V_assumption} corresponds to a slightly better decay than that of $V=\pm \delta$, which satisfies $\widehat{V}=\pm 1$. Here $\delta$ is the Dirac delta function. When $V=\delta$, it is shown in \cite{BS96} that (with a suitable Wick-ordering of the quartic term), local Gibbs measures of the form \eqref{focusing_Gibbs} (and \eqref{local_Gibbs_measure_rigorous} below) are not well-defined. When $d=3$, the condition \eqref{V_assumption} corresponds to a slightly better decay than that of the Coulomb potential, which satisfies $|\widehat{V}(k)| \sim |k|^{-2}$ for $k \neq 0$.
\end{itemize}
\end{remark}

We now introduce the cut-off function, which we use to truncate in \eqref{renormalized_mass}.
\begin{definition}
\label{cut-off_chi}
Let $\delta \in (0,1)$ be given and let $\chi^{(\delta)} \in \mathscr{C}^{\infty}_0(\R)$ be a nonnegative function with the property that $0 \leq \chi^{(\delta)} \leq 1$ pointwise and 
\begin{equation}
\label{chi}
\chi^{(\delta)}(x) =
\begin{cases}
1 \text{ for } |x| \leq \delta\,, \\
0 \text{ for } |x| \geq 1\,.
\end{cases}
\end{equation}
Recalling \eqref{indicator_function}, we let
\begin{equation}
\label{chi^{(1)}}
\chi^{(1)}:=\mathbbm{1}_{(-1,1)}\,.
\end{equation}
Given $\delta \in (0,1]$ and $\BB>0$, we define 
\begin{equation*}
\label{chi_R}
\chi^{(\delta)}_R (\cdot):=\chi^{(\delta)}(\cdot/\BB)\,.
\end{equation*}
\end{definition}

Let us note the following proposition, which is crucial for our analysis and allows us to rigorously define the local Gibbs measure.

\begin{proposition}[Malliavin differentiabililty of $h^I$ and normalizability of the local Gibbs measure]
\label{Hartree_equation_Malliavin_derivative}
The following claims hold for $h^I$ satisfying Assumption \ref{Assumption_on_V}.
\begin{itemize}
\item[(i)] Let $p \in [1,\infty)$ be given. We have that $h^I \in \mathbb{D}^{1,p}(\mu_{0})$.
\item[(ii)] Let us recall Definition \ref{cut-off_chi}. For all $p \in [1,\infty)$ and $\delta \in (0,1]$, we have that
\begin{equation}
\label{integrability_of_weight}
\ee^{h^I}\chi_{\BB}^{(\delta)} (\mathcal{M}) \in L^p(\mu_{0})\,,
\end{equation}
for all $\BB>0$ sufficiently small when $h^I$ is given by \eqref{local_nonlinearity_1D} with $r=5$ and for all $\BB>0$ in all other cases.
In particular, we have that the partition function satisfies
\begin{equation}
\label{z_R_definition}
z^{(\delta)} \equiv z^{(\delta)}_{\BB}:=\int_{H^{-s}} \ee^{h^I}\chi_{\BB}^{(\delta)} (\mathcal{M})\,d\mu_{0}  \,\in (0,\infty)
\end{equation}
and that the local Gibbs measure 
\begin{equation}
\label{local_Gibbs_measure_rigorous}
d \mu^{(\delta)} \equiv d \mu^{(\delta)}_{\BB}:=\frac{1}{z^{(\delta)}}\,\ee^{h^I}\chi_{\BB}^{(\delta)} \bigl(\mathcal{M}\bigr)\,d\mu_{0}
\end{equation}
is a well-defined probability measure on $H^{-s}$ with $\mu^{(\delta)} \ll \mu_{0}$.
\end{itemize}
\end{proposition}

\begin{proof}
The proof of (i) is given in Section \ref{Proof_of_Proposition_2.13_(i)}. We now prove (ii).
The claim when $d=1$ follows from \cite[Lemma 3.10]{Bou94}. We refer the reader to \cite[Lemma 2.1 and Appendix A]{RS22} for a self-contained summary of the analysis.

Let us now consider the case when $d=2,3$. We note that this claim is proved in the work of Bourgain \cite[Section 1]{Bou97}. As this is a very important result for our analysis, we revisit its proof and give a self-contained account (most of which is contained in Appendix \ref{Proof of Bourgain's large deviation estimate} below).

By Definition \ref{cut-off_chi}, we know that $\mathrm{supp}\, \chi_{\BB}^{(\delta)} \subset [-\BB,\BB]$. 
The claim that $z^{(\delta)}<\infty$ follows if we show that given $C_{\mathrm{Big}}>0$ arbitrarily large depending on $\BB$, one has that for all $\lambda>1$
\begin{equation}
\label{Bourgain_large_deviation_estimate}
\mu_{0} \bigl(h^I>\lambda \,\cap \, \bigl| :\|u\|_{L^2}^2: \bigr| \leq \BB \bigr)\lesssim_{C_{\mathrm{Big}}} \ee^{-C_{\mathrm{Big}} \lambda}\,.
\end{equation}
More precisely, by Definition \ref{cut-off_chi}, we know that 
\begin{multline}
\label{Integral_1}
\bigl\|\ee^{h^I}\chi_{\BB}^{(\delta)}\bigl(:\|u\|_{L^2}^2:\bigr)\bigr\|^p_{L^p(\mu_{0})} \leq
\bigl\|\ee^{h^I} \mathbbm{1}_{[-\BB,\BB]}(:\|u\|_{L^2}^2:\bigr)\bigr\|^p_{L^p(\mu_{0})} 
\\
=\int_{0}^{\infty} \mu_{0} \Bigl[\ee^{p h^I}\, \mathbbm{1}_{[-\BB,\BB]}(:\|u\|_{L^2}^2:\bigr)>
\lambda\Bigr] \, d \lambda\,.
\end{multline}
Let us note that
\begin{equation}
\label{Integral_1.5}
\int_{0}^{\ee^p} \mu_{0} \Bigl[\ee^{p h^I}\, \mathbbm{1}_{[-\BB,\BB]}(:\|u\|_{L^2}^2:\bigr)>
\lambda\Bigr] \, d \lambda \in [0,\ee^p]\,.
\end{equation}
Assuming \eqref{Bourgain_large_deviation_estimate}, the contribution to \eqref{Integral_1} from $\ee^p$ to $\infty$ is
\begin{multline}
\label{Integral_2}
=\int_{\ee^p}^{\infty} \mu_{0} \biggl[h^I>\frac{\log \lambda}{p} \,\cap \, \bigl| :\|u\|_{L^2}^2: \bigr|^2 \leq \BB \biggr] \, d \lambda 
\\
\lesssim_{C_{\mathrm{Big}}} \int_{\ee^p}^{\infty} \ee^{-C_{\mathrm{Big}} \frac{\log \lambda}{p}}\,d\lambda=\int_{\ee^p}^{\infty} 
\lambda^{-\frac{C_{\mathrm{Big}}}{p}}\,d \lambda<\infty\,,
\end{multline}
provided that $C_{\mathrm{Big}}>p$. Combining \eqref{Integral_1}--\eqref{Integral_2}, we deduce that \eqref{Bourgain_large_deviation_estimate} implies (ii). The proof of \eqref{Bourgain_large_deviation_estimate} is given in full detail in Appendix \ref{Proof of Bourgain's large deviation estimate}. 
The claim that $z^{(\delta)}>0$ follows from Lemma \ref{app.lem.partition_funct} in Appendix \ref{Appendix_B_Partition_function}.
\end{proof}

\begin{remark}
\label{Positivity_remark}
Let us make a few observations about the nature of the cut-off in \eqref{renormalized_mass} that is present in the definition of the local Gibbs measure \eqref{local_Gibbs_measure_rigorous}.
\begin{itemize}
\item[(i)] When $d=1$, we note that \eqref{renormalized_mass} is always non-negative. For $\delta \in (1/2,1)$, we can modify Definition \ref{cut-off_chi} and replace \eqref{chi}--\eqref{chi^{(1)}} by
\begin{equation*}
\label{chi^{(delta)}}
\chi^{(\delta)}(x) =
\begin{cases}
0 \text{ for } x < 0\,, \\
1 \text{ for } 1-\delta \leq x  \leq \delta\,, \\
0 \text{ for } x > 1\,.
\end{cases}
\end{equation*}
and $\chi^{(1)}=\mathbbm{1}_{(0,1)}$ respectively\footnote{In practice, our ultimate goal is to analyse the limit $\delta \rightarrow 1$, so it suffices to consider $\delta>1/2$.}. The analysis would then carry over in the same way.
\item[(ii)] When $d=2,3$, the quantity \eqref{renormalized_mass} is not necessarily nonnegative, so the modification from part (i) does not immediately apply. However if in Assumption \ref{Assumption_on_V} (ii) we also assume that 
\begin{equation}
\label{V_hat_zero_assumption}
\widehat{V}(0) =\int_{\T^d} V(x)\, dx \leq 0\,,
\end{equation} 
then 
we can modify Definition \ref{cut-off_chi} as follows. 
For $\delta \in (0,1)$, instead of \eqref{chi}, we consider $\chi^{(\delta)} \in \mathscr{C}^{\infty} (\R)$ nonnegative with the property that $0 \leq \chi^{(\delta)} \leq 1$ pointwise and
\begin{equation*}
\chi^{(\delta)}(x) =
\begin{cases}
1 \text{ for }  x  \leq \delta\,, \\
0 \text{ for } x > 1\,.
\end{cases}
\end{equation*}
Moreover, instead of \eqref{chi^{(1)}}, we consider $\chi^{(1)}=\mathbbm{1}_{(-\infty,1)}$.
The reason for this is that, under the assumption \eqref{V_hat_zero_assumption}, we can drop the absolute values in the constraint on the Wick-ordered mass in \eqref{local_Gibbs_measure_rigorous}. More precisely, we can replace \eqref{Bourgain_large_deviation_estimate} above by
\begin{equation}
\label{Bourgain_large_deviation_estimate_V_extra} 
\mu_{0} \bigl(h^I>\lambda \,\,\cap \, :\|u\|_{L^2}^2:  \,\leq \BB \bigr)\lesssim_{C_{\mathrm{Big}}} \ee^{-C_{\mathrm{Big}} \lambda}\,.
\end{equation}
The fact that the additional assumption \eqref{V_hat_zero_assumption} implies \eqref{Bourgain_large_deviation_estimate_V_extra} follows from the arguments given in Appendix \ref{Proof of Bourgain's large deviation estimate}. More precisely, it follows from \eqref{h^I_rewritten} and the fact that in the proof of \eqref{Bourgain_large_deviation_estimate_rewritten} given below, we only use the upper bound 
\begin{equation*}
:\|u\|_{L^2}^2: \,\leq\, \BB\,;
\end{equation*} 
see \eqref{30} when $d=3$ and \eqref{68c*} when $d=2$ for the details. 
\end{itemize}
\end{remark}

Let us note the following result, which we will use throughout the paper.

\begin{lemma}
\label{Lemma_2.16_5}
For all $p \in [1,\infty)$, we have that the function  $u\in H^{-s} \mapsto \| u\|_{ H^{-s}}$ belongs to $L^p(\mu_0)$.
\end{lemma}
\begin{proof}
By using hypercontractivity bounds (see \cite[Lemma 5.2]{AS21} or \cite[Theorem I.22]{Sim74}), we have that for $p \in [1,\infty)$
\begin{equation}
\label{Lemma_2.16_5*}
\left\|\|\cdot\|_{H^{-s}}\right\|_{L^p(\mu_0)} \lesssim_p \|\|\cdot\|_{H^{-s}}\|_{L^2(\mu_0)} \sim \Biggl(\sum_{k \in \Z^d} \frac{1}{(|k|^2+1)^{s+1}}\Biggr)^{1/2} <\infty\,.
\end{equation}
In \eqref{Lemma_2.16_5*}, we recalled Assumption \ref{choice_of_s_assumption} and \eqref{choice_of_s_consequence}; see also \eqref{P_n(u)_convergence_proof} in Appendix \ref{Malliavin derivative calculations and normalizability of the Gibbs measure for the focusing Hartree equation}.
\end{proof}

\subsection*{Rigorous formulation of the  global and local KMS condition}

Let us now give a rigorous formulation of the global and local KMS conditions. The (global) KMS condition was previously formally stated in \eqref{KMS_condition_1}. We now explain how all of the objects appearing there are rigorously defined.

Let $A$ be as in Assumption \ref{A_choice} and let $h^I$ be as in Assumption \ref{Assumption_on_V}. 
Let us define the nonlinear vector field $X$ acting on $H^{-s}$ as
\begin{equation}
\label{X_definition}
X := -\ii A + \ii \nabla h^I\,.
\end{equation}
Recalling \eqref{nonlinear_Hamiltonian}--\eqref{nonlinear_Hamiltonian_B}, we can rewrite \eqref{X_definition}
as 
\begin{equation}
\label{X_definition_B}
X=-\ii \nabla h\,.
\end{equation}
Let us also write
\begin{equation}
\label{X^I_definition}
X_0:=-\ii A \,,\qquad X^I := \ii \nabla h^I\,.
\end{equation}
We consider the vector field equation 
\begin{equation}
\label{vector_field_equation}
\dot{u}(t)=X(u)=-\ii A u + \ii \nabla h^I(u)\,.
\end{equation}

Let us recall the following definition, also found in \cite[Equation (4.38)]{AS21}.
\begin{definition}[Global KMS condition]
\label{KMS_definition_rigorous}
We say that $\mu \in \mathscr{P}(H^{-s})$ satisfies the \emph{KMS condition} for the dynamical system \eqref{vector_field_equation} induced by the vector field \eqref{X_definition} if for all $F,G \in \mathscr{C}^{\infty}_{\text{c},\text{cyl}} (H^{-s})$, we have that
\begin{equation}
\label{KMS_definition_rigorous_1}
\int_{H^{-s}} \{F,G\}(u)\,d \mu= \int_{H^{-s}} \langle \nabla F(u),-\ii A u +\ii \nabla h^I(u) \rangle_{\mathcal{H},\R}\, G(u)\, d \mu\,.
\end{equation}
Alternatively, we say that $\mu$ is a \emph{global KMS state} of \eqref{vector_field_equation} induced by the vector field \eqref{X_definition}. In \eqref{KMS_definition_rigorous_1}, we recall Definition \ref{cylindrical_test_functions}, \eqref{gradient_cylindrical_F} and \eqref{poisson_cylindrical}. 
\end{definition}

Let us now state a rigorous local alternative to Definition \ref{KMS_definition_rigorous} that will be our main object of study in the rest of the paper.  We introduce some notation. For all $R>0$, we define 
\begin{equation}
\label{B_R'}
\mathbb{B}_{R}:=\{u \in H^{-s}\,,\,\, |\mathcal{M}(u)|< R\}\,,\qquad \mathbb{B}_{R}^c \equiv H^{-s}\setminus \mathbb{B}_{R}\,.
\end{equation}
Here, we recall that $\mathcal{M}(\cdot)$ is given by \eqref{renormalized_mass}.  From \eqref{B_R'}, we then define
\begin{equation}
\label{D^{1,2}_R}
\mathbb{D}^{1,2}_{\BB}(\mu_0):=\bigl\{G \in \mathbb{D}^{1,2}(\mu_0)\,,\,\, \exists \,\BB' \in (0,\BB)\,\,\mbox{s.t.}\,\,G(u) = 0\,\,\mbox{$\mu_0$-a.s. on} \,\, \mathbb{B}_{\BB'}^{c}\bigr\}\,.
\end{equation}

\begin{definition}[Local KMS condition]
\label{local_KMS_definition}
We say that $\mu \in \mathscr{P}(H^{-s})$ satisfies the \emph{local KMS condition} for the dynamical system \eqref{vector_field_equation} induced by the vector field \eqref{X_definition} if there exists $\BB>0$ such that we have
\begin{equation}
\label{local_KMS_definition_rigorous_2}
\int_{H^{-s}} \{F,G\}(u)\,d \mu= \int_{H^{-s}} \langle \nabla F(u),-\ii A u +\ii \nabla h^I(u) \rangle_{\mathcal{H},\R}\, G(u)\, d \mu\,.
\end{equation}
for all $G \in \mathbb{D}^{1,2}_{\BB}(\mu_0)$ and $F\in \mathscr{C}^{\infty}_{\text{c},\text{cyl}} (H^{-s})$. 
In \eqref{local_KMS_definition_rigorous_2}, we take 
\begin{equation}
\label{poisson_general}
\{F,G\}(u) \equiv \langle \nabla F(u), -\ii \nabla G(u)\rangle_{\mathcal{H},\R}\,,
\end{equation}
which is consistent with \eqref{poisson_cylindrical} when $G \in \mathscr{C}^{\infty}_{\text{c},\text{cyl}} (H^{-s})$.
Alternatively, if \eqref{local_KMS_definition_rigorous_2} holds, we say that $\mu$ is a \emph{local KMS state} of \eqref{vector_field_equation} induced by \eqref{X_definition}. 
\end{definition}

Let us note that the expressions on the left and right-hand side in \eqref{local_KMS_definition_rigorous_2} (as well as \eqref{KMS_definition_rigorous_1}) are finite whenever $\mu$ is taken to be a local Gibbs measure of the form \eqref{local_Gibbs_measure_rigorous}  (with $\BB>0$ as in Proposition \ref{Hartree_equation_Malliavin_derivative} (ii)). Actually, it suffices to consider general $G \in \mathbb{D}^{1,2}(\mu_0)$ instead of $G \in  \mathbb{D}_{\BB}^{1,2}(\mu_0)$.

\begin{lemma}
\label{well-definedness_remark}
Let $\delta \in (0,1]$ and let $\mu^{(\delta)}$ as in \eqref{local_Gibbs_measure_rigorous} be given.
The following properties hold for all $F\in \mathscr{C}^{\infty}_{\mathrm{c},\mathrm{cyl}} (H^{-s})$ and $G \in \mathbb{D}^{1,2}(\mu_0)$.
\begin{itemize}
\item[(i)] $\int_{H^{-s}} |\{F,G\}(u)|\,d \mu^{(\delta)}<\infty$.
\item[(ii)] $\int_{H^{-s}} \bigl|\langle \nabla F(u),-\ii A u +\ii \nabla h^I(u) \rangle_{\mathcal{H},\R}\, G(u)\bigr|\, d \mu^{(\delta)}<\infty$.
\end{itemize}
\end{lemma}

\begin{proof}

We first show (i).
From \eqref{poisson_general}, duality, and since $F\in \mathscr{C}^{\infty}_{\text{c},\text{cyl}} (H^{-s})$, we have that for all $u \in H^{-s}$
\begin{equation}
\label{well-definedness_remark_2_a}
|\{F,G\}(u)| \leq \|\nabla F(u)\|_{H^{s}}\,\|\nabla G(u)\|_{H^{-s}} \lesssim_{F,s} \|\nabla G(u)\|_{H^{-s}}\,.
\end{equation}
From \eqref{z_R_definition}--\eqref{local_Gibbs_measure_rigorous}, the Cauchy-Schwarz inequality, \eqref{well-definedness_remark_2_a}, the assumption that $G \in \mathbb{D}^{1,2}(\mu_0)$, and Proposition \ref{Hartree_equation_Malliavin_derivative} (ii), we obtain
\begin{equation}
\label{well-definedness_remark_2_b}
\int_{H^{-s}} |\{F,G\}(u)|\,d \mu^{(\delta)} \lesssim_{F,s} \frac{1}{z^{(\delta)}}\, \|\nabla G(u)\|_{L^2(\mu_0;H^{-s})}\,\|\ee^{h^I}\chi_{\BB}^{(\delta)}(\mathcal{M})\|_{L^2(\mu_0)}<\infty\,,
\end{equation}
and (i) follows.

For the proof of (ii), we first observe that
\begin{equation}
\label{homogeneous_KMS_I_2}
\bigl\|\langle \nabla F(u) , -\ii Au \rangle_{\mathcal{H},\R}\bigr\|_{L^{\infty}(\mu_0)}\lesssim_F 1 \,.
\end{equation}
In order to obtain \eqref{homogeneous_KMS_I_2}, we write $F$ as in \eqref{cylindrical_test_functions_1}, use the self-adjointness of $A$, followed by \eqref{e_j_choice}, \eqref{ONB_convention}, \eqref{gradient_cylindrical_F}, and the triangle inequality to write
\begin{multline}
\label{homogeneous_KMS_I_2A*}
\bigl\|\langle \nabla F(u) , -\ii Au \rangle_{\mathcal{H},\R}\bigr\|_{L^{\infty}(\mu_0)}=\bigl\|\langle A\nabla F(u) , -\ii u \rangle_{\mathcal{H},\R}\bigr\|_{L^{\infty}(\mu_0)}
\\
=\Biggl\|\biggl\langle\sum_{j=1}^n \partial_j^{(1)} \varphi(\pi_n(u)) \lambda_j e^{(1)}_j + \partial_j^{(2)} \varphi(\pi_n(u)) \lambda_j e^{(2)}_j,-\ii u\biggr\rangle_{\mathcal{H},\R}\Biggr\|_{L^{\infty}(\mu_0)} 
\\
\leq \sum_{j=1}^{n} \lambda_j \Bigl\|\bigl|\partial_j^{(1)} \varphi(\pi_n(u))\bigr|\,\bigl|\langle u, e_j^{(2)} \rangle_{\mathcal{H},\R}\bigr| +\bigl|\partial_j^{(2)} \varphi(\pi_n(u))\bigr|\,\bigl|\langle u, e_j^{(1)} \rangle_{\mathcal{H},\R}\bigr|
\Bigr\|_{L^{\infty}(\mu_0)} \lesssim_F 1\,.
\end{multline}
For the last step in \eqref{homogeneous_KMS_I_2A*}, we recalled \eqref{projection_map} and used the assumption that $\varphi \in \mathscr{C}^{\infty}_{\mathrm{c}}(\R^{2n})$. We hence deduce \eqref{homogeneous_KMS_I_2} from \eqref{homogeneous_KMS_I_2A*}.
Let us now fix $p,q \in (2,\infty)$ such that
\begin{equation}
\label{p,q_choice}
\frac{1}{p}+\frac{1}{q}=\frac{1}{2}\,.
\end{equation}
We use duality and the fact that $F \in \smoothcyl$ to obtain 
\begin{multline}
\label{homogeneous_KMS_I_2*}
\bigl\|\langle \nabla F(u) , \ii \nabla h^I(u) \rangle_{\mathcal{H},\R}\bigr\|_{L^q(\mu_0)} \leq \bigl\|\|\nabla F(u)\|_{H^s} \|\nabla h^I(u)\|_{H^{-s}}\bigr\|_{L^q(\mu_0)} 
\\
\lesssim_{F,s} \|\nabla h^I\|_{L^q(\mu_0;H^{-s})}<\infty\,.
\end{multline}
The last inequality holds by Proposition \ref{Hartree_equation_Malliavin_derivative} (i).
By using H\"{o}lder's inequality, \eqref{z_R_definition}--\eqref{local_Gibbs_measure_rigorous}, \eqref{homogeneous_KMS_I_2}, \eqref{p,q_choice}--\eqref{homogeneous_KMS_I_2*}, followed by \eqref{integrability_of_weight} and the fact that $\|G\|_{L^2(\mu_0)}<\infty$, we have
\begin{multline}
\label{homogeneous_KMS_I_2**}
\int_{H^{-s}} \bigl|\langle \nabla F(u),-\ii A u +\ii \nabla h^I(u) \rangle_{\mathcal{H},\R}\, G(u)\bigr|\, d \mu^{(\delta)}
\\
\lesssim_{F,s} \frac{1}{z_R^{(\delta)}} \bigl(1+\|\nabla h^I\|_{L^q(\mu_0;H^{-s})}\bigr)\,\|G\|_{L^2(\mu_0)}\, \|\ee^{h^I} \chi_{\BB}^{(\delta)}(\mathcal{M})\|_{L^p(\mu_0)}<\infty\,.
\end{multline}

\end{proof}

\subsection{Main results}
We can now state our main results. The first result says that, in our setting, local Gibbs measures \eqref{local_Gibbs_measure_rigorous} are stationary solutions to the Liouville probability density equation \eqref{homogeneous_KMS} (see also the formal discussion in \eqref{Liouville_equation}--\eqref{stationary_solution_Liouville_equation}). 
\begin{theorem}[Local Gibbs measures are stationary solutions to the Liouville equation]
\label{homogeneous_KMS_thm}
Consider $A$ as in Assumption \ref{A_choice} and $h^I$ as in Assumption \ref{Assumption_on_V}. 
Let $\delta \in (0,1]$ and let $\mu^{(\delta)}$ be the local Gibbs measure defined in \eqref{local_Gibbs_measure_rigorous}. Then for all $F \in \smoothcyl$, we have
\begin{equation}
\label{homogeneous_KMS}
\int_{H^{-s}} \langle \nabla F ,  -\ii A u + \ii \nabla h^I(u) \rangle_{\mathcal{H},\R} \, d\mu^{(\delta)}= 0\,.
\end{equation}
\end{theorem}
Applying \cite[Corollary 1.17]{AFS23}, we obtain the almost sure existence of global solutions for the focusing  Hartree and NLS equations on the torus $\T^{d}, d=1,2,3$, as a corollary of Theorem \ref{homogeneous_KMS_thm}.
\begin{corollary}[Almost sure existence of global solutions]
\label{almost_sure_WP_corollary} 
Let $A,h^I$ be as in the assumptions of Theorem \ref{homogeneous_KMS_thm}. Let $\delta \in (0,1]$ and let $\mu^{(\delta)}$ be given by \eqref{local_Gibbs_measure_rigorous}. Then the initial value problem  \eqref{vector_field_equation} admits a global solution for $\mu^{(\delta)}$-almost every initial condition $u_0 \in H^{-s}$.
\end{corollary}

\begin{remark}
When deducing Corollary \ref{almost_sure_WP_corollary} from Theorem \ref{homogeneous_KMS_thm}, one notes
that \eqref{vector_field_equation} can be more conveniently written in the interaction representation
\begin{equation}
\label{vector_field_equation_2}
\dot{v}(t)=\ee^{\ii t A} X^I (\ee^{-\ii tA} v(t))\,,
\end{equation}
where $v(t)=\ee^{\ii t A}u(t)$ and where $X^I$ is given as in \eqref{X^I_definition}. The reason why \eqref{vector_field_equation_2} is usually more convenient is that by Proposition \ref{Hartree_equation_Malliavin_derivative} (i) and unitarity, the vector field $\ee^{\ii tA} X^I \ee^{-\ii tA}$ maps $H^{-s}$ to itself. The vector field \eqref{X_definition} does not satisfy this property due to the presence of the $-\ii A$ term which maps $H^{-s}$ to $H^{-s-2}$. This is the approach taken in \cite[Section 1]{AFS23}; we refer the reader to this work for more details.
\end{remark}

The main aim of our paper is to explore the relationship between the local KMS condition from Definition \ref{local_KMS_definition}  and the local Gibbs measure \eqref{local_Gibbs_measure_rigorous}.
We prove that local Gibbs measures are local KMS states.

\begin{theorem}[Local Gibbs measures are local KMS states]
\label{Gibbs_implies_local_KMS_theorem}
Consider $A$ as in Assumption \ref{A_choice} and $h^I$ as in Assumption \ref{Assumption_on_V}. 
Then, the measure $\mu^{(1)}$ given by \eqref{local_Gibbs_measure_rigorous} satisfies the local KMS condition in the sense of Definition \ref{local_KMS_definition} above,  for $\BB>0$ as in Proposition \ref{Hartree_equation_Malliavin_derivative} (ii).
\end{theorem}

By Lemma \ref{well-definedness_remark} above, we recall that both sides of \eqref{local_KMS_definition_rigorous_2} in Definition \ref{local_KMS_definition} are well-defined when $\mu=\mu^{(1)}$.  Let us note that a variant of Theorem \ref{Gibbs_implies_local_KMS_theorem} above with a smooth cut-off was shown for one dimensional NLS equation in \cite[Section 5.3]{AS21}. In the current paper, the analysis is given in dimensions $d=1,2,3$ for Hartree and NLS and with a sharp cut-off. 

\medskip

Recall the definition of $\mathbb{B}_{R}$  and $\Drspace$ from \eqref{B_R'} and \eqref{D^{1,2}_R}. 
The result below shows the opposite implication, namely any local KMS equilibrium state should coincide (up to a multiplicative constant) with the truncated Gibbs measure on the set $\mathbb{B}_{R}$ for some $R>0$.

\begin{theorem}[local KMS states are locally Gibbs measures]
\label{kms_implies_gibbs_theorem}
Let $\mu$ be a Borel probability measure in $\mathscr{P}(H^{-s})$. Suppose that $d \mu = \rho \, d\mu_{0}$ with $\rho \in \mathbb{D}^{1,2}(\mu_{0})\cap L^4(\free)$ and $\mu$ satisfies  the  local KMS condition, i.e., there exists $R>0$ such that
\begin{equation*}
    \int_{H^{-s}} \{F,G\} \, d\mu = \int_{H^{-s}} {\rm Re}\langle \nabla F(u), X(u) \rangle \,G(u) \, d\mu\,,
\end{equation*}
for any $G \in \Drspace$ and  $F \in \smoothcyl$.  Then $\mu$ is locally a Gibbs measure, i.e., 
for $\mu_0$-almost all  $u \in \mathbb{B}_R$,
\begin{equation}
\label{local_KMS_weight_equation}
\rho(u) = c_0e^{h^I(u)}, 
\end{equation}
for some normalization constant  $c_0\geq 0$ which constrains the mass of $\mu$ to be $1$.
\end{theorem}

\begin{remark}
Let us comment on the normalization constant $c_0$ from Theorem \ref{kms_implies_gibbs_theorem}. 
Since $\mu$ is a probability measure, we have
$$1=\int_{H^{-s}} d\mu = c_0 \int_{\mathbb{B}_R} \ee^{h^I}d\mu_0 + \int_{\mathbb{B}_{R}^c} d\mu.$$ So
$$
c_0 = \frac{1 - \mu(\mathbb{B}_{R}^c)}{\int_{\mathbb{B}_R} \ee^{h^I}d\mu_0}.
$$ 
In particular, if $\mu(\mathbb{B}_{R}^c) = 0$, $c_0$ is precisely the normalization constant for the measure $\mu$ and $\mu$ is equal to the truncated Gibbs measure  $\mu^{(1)}$ given by \eqref{local_Gibbs_measure_rigorous}.
\end{remark}

\medskip

 The result of Corollary \ref{almost_sure_WP_corollary} was known before. More precisely, when $d=1$, it was shown in \cite{Bou94} and when $d=2,3$, it was shown in \cite{Bou97}. Its interest lies in the fact that the result is recovered from a new method developed in \cite{AFS23}. However, to the best of our knowledge, all the other Theorems \ref{homogeneous_KMS_thm}, \ref{Gibbs_implies_local_KMS_theorem} and \ref{kms_implies_gibbs_theorem} are new. In particular, Theorems \ref{Gibbs_implies_local_KMS_theorem} and \ref{kms_implies_gibbs_theorem} lead to the characterization of all thermal equilibria of nonlinear focusing Hartree and Schr\"{o}dinger equations \eqref{vector_field_equation} on the torus $\T^d,   d=1,2,3$.

\subsection{Strategy of proofs}

As in \cite{AS21}, our analysis is based on Malliavin calculus and Gross-Sobolev spaces.  In order to prove Theorem \ref{homogeneous_KMS_thm}, which shows that local Gibbs measures are stationary solutions of the Liouville equation, we use a Gaussian integration by parts formula 
(see Proposition \ref{integration_by_parts_prop} below for a precise statement).
Before proving Theorem \ref{homogeneous_KMS_thm}, we first show an analogous result in the context of finite-dimensional Hamiltonian systems. The finite-dimensional result is given by Proposition \ref{homogeneous_finite_Liouville_lemma}. The proof of the latter gives insight on how to apply the integration by parts without having to worry about additional analytical issues. We observe that such an argument  extends to the infinite-dimensional setting by using Proposition \ref{integration_by_parts_prop} appropriately. A subtle point in the proof  lies in the presence of the cut-off on the (renormalized) mass. It is overcome by a suitable use of the Leibniz rule for the Poisson bracket and suitable conservation laws. The proof of Theorem \ref{Gibbs_implies_local_KMS_theorem} is also based on a Gaussian integration by parts.
The proof works in two steps. In the first step, we prove a corresponding result for the measures $\mu^{(\delta)}$ for $\delta \in (0,1)$. In the second step, we get the claimed result for $\mu^{(1)}$ and general $G \in \mathbb{D}^{1,2}_{\BB}(\mu_0)$ by means of an approximation argument.  The proof of Theorem \ref{kms_implies_gibbs_theorem} which shows that all local KMS states are local Gibbs measures (Theorem \ref{kms_implies_gibbs_theorem}) is more involved and it is based on  results of Aida \cite{Ai} and Kusuoka \cites{Kusuoka91,Kusuoka92} on irreducibility of Dirichlet forms over infinite-dimensional domains. First, we prove that the probability density $\rho$ in Theorem \ref{kms_implies_gibbs_theorem} satisfies a differential equation \eqref{Malliavin_diff_equ} on the domain $\mathbb{B}_R$ (see  Proposition \ref{local_KMS_differential_equation}).  Secondly, we use the result of Aida recalled in Proposition \ref{Aida_0_derivative_result} in order to solve the differential equation \eqref{Malliavin_diff_equ}. So, the problem is converted to showing that the domain $\mathbb{B}_R$ is an $H^1$-connected open set; see Lemma \ref{lem_Hconnected}. It is worth noticing that the calculations that we use along all  the proofs rely crucially on the normalizability and structure of the local Gibbs measure for the focusing Hartree and NLS equations, studied by Bourgain \cite{Bou97} when $d=2,3$. In Appendix \ref{Proof of Bourgain's large deviation estimate}, we revisit the proofs of these results by means of concentration inequalities for sums of subgaussian and subexponential random variables. The latter are summarized in Section \ref{Concentration inequalities}.

\subsection{Structure of paper}
In Section \ref{Liouville_section}, we prove Theorem \ref{homogeneous_KMS_thm}. We first consider, in Section \ref{Finite-dimensional setting}, a finite-dimensional version stated in Proposition \ref{homogeneous_finite_Liouville_lemma}. Then we analyse the infinite-dimensional setting and give the proof of Theorem \ref{homogeneous_KMS_thm} in Section \ref{Infinite-dimensional setting}.
Section \ref{Gibbs_implies_KMS_subsection} is devoted to the proof of Theorem \ref{Gibbs_implies_local_KMS_theorem} which shows that local Gibbs measures are local KMS states. In Section \ref{Local KMS states are local Gibbs measures}, we give the proof of
Theorem \ref{kms_implies_gibbs_theorem} which says that all local KMS states are locally Gibbs measures. In Appendix \ref{Malliavin derivative calculations and normalizability of the Gibbs measure for the focusing Hartree equation}, we review the algebra of Wick ordering  and give the details for Malliavin derivative calculations that we use in this paper. In Section \ref{Proof_of_Proposition_2.13_(i)}, we prove Proposition \ref{Hartree_equation_Malliavin_derivative} (i) which claims the differentiabililty of the nonlinearity $h^I$ in the sense of Malliavin. In Appendix \ref{Proof of Bourgain's large deviation estimate}, we revisit Bourgain's bound \eqref{Bourgain_large_deviation_estimate} and give a self-contained proof of this fact.

\section{Gibbs measures and the Liouville equation}
\label{Liouville_section}
The main goal of this section is to prove Theorem \ref{homogeneous_KMS_thm}, according to which local Gibbs measures \eqref{focusing_Gibbs} are stationary solutions to the Liouville equation \eqref{Liouville_equation}. In Section \ref{Finite-dimensional setting}, we analyse this relationship in the finite-dimensional setting. Here, we prove Proposition \ref{homogeneous_finite_Liouville_lemma}. The analysis of the infinite-dimensional setting and the proof of Theorem \ref{homogeneous_KMS_thm} are given in Section \ref{Infinite-dimensional setting}.

\subsection{Finite-dimensional setting}
\label{Finite-dimensional setting}
Let us first consider the relationship between local Gibbs measures and the Liouville equation in the setting of finite-dimensional Hamiltonian systems. 
This framework is a simplification of the infinite-dimensional setup that we wish to study in Section \ref{Infinite-dimensional setting} below. Nevertheless, the proof will be enlightening for the general analysis.

Let us consider $n \in \N$ fixed. We write 
\begin{equation}
\label{p_q_variables}\mathscr{C}^{\infty}
p=(p_1,\ldots,p_n)\,, \quad  q=(q_1,\ldots,q_n) \in \R^n\,. 
\end{equation}
The Hamiltonian equations of motion associated with a Hamiltonian $h \equiv h(p,q) \in \mathscr{C}^{\infty}(\R^{2n})$ are given by the following system of ODEs
\begin{equation}
\label{finite_hamiltonian}
\dot{p}_j(t) = \frac{\partial h}{\partial q_j}\,, \qquad\dot{q}_j(t) = -\frac{\partial h}{\partial p_j}\,,\qquad 1 \leq j \leq n\,.
\end{equation}
In what follows, we always assume that a global Hamiltonian flow for \eqref{finite_hamiltonian} exists and is uniquely defined. 
We write $dp \, dq := \prod_{j=1}^n dp_j \, dq_j$ in what follows. If
\begin{equation}
\label{partition_function_finite_dim}
z:=\int_{\R^{2n}} \mathrm{e}^{-h(p,q)} dp \, dq<\infty\,,
\end{equation}  
we can define the Gibbs measure $\mu \in \mathscr{P}(\R^{2n})$ associated with the Hamiltonian $h$ as
\begin{equation}
\label{Gibbs_measure_finite_dim}
d\mu := \frac{1}{z}\,\mathrm{e}^{-h(p,q)}\, dp \, dq\,.
\end{equation}
By Liouville's theorem (as the vector field on the right-hand side of the equations in \eqref{finite_hamiltonian} is divergence free) and the conservation of the Hamiltonian $h$ under the flow of \eqref{finite_hamiltonian}, we deduce that \eqref{Gibbs_measure_finite_dim} is invariant under the flow of \eqref{finite_hamiltonian}.

If condition \eqref{partition_function_finite_dim} is not satisfied, the above construction does not work. However, under suitable assumptions, we can work with a finite-dimensional variant of a local Gibbs measure. Let us now state the precise result.
\begin{lemma}[Local Gibbs measures in finite dimensions]
\label{local_Gibbs_finite_dim}
We define the mass $M:\R^{2n} \rightarrow \R$ associated with \eqref{finite_hamiltonian} by
\begin{equation}
\label{mass_finite_dim}
M(p,q):= \sum_{j=1}^{n} (p_j^2 +q_j^2)\,,
\end{equation}
Suppose that 
\begin{equation}
\label{{h,M}}
\{h,M\}=0\,, 
\end{equation}
where the Poisson bracket $\{\cdot,\cdot\}:\mathscr{C}^{\infty}(\R^{2n}) \times \mathscr{C}^{\infty}(\R^{2n}) \rightarrow \mathscr{C}^{\infty}(\R^{2n})$ in \eqref{{h,M}} is given by
\begin{equation}
\label{Poisson_bracket_finite_dim}
\{f,g\}:=\sum_{j=1}^{n} \biggl(\frac{\partial{f}}{\partial p_j}\,\frac{\partial{g}}{\partial q_j}-\frac{\partial{f}}{\partial q_j}\,\frac{\partial{g}}{\partial p_j}\biggl)\,.
\end{equation}
Let $\psi \in \mathscr{C}^{\infty}_{\rm{c}}(\R)$ be given.
Then $\mu_{\psi}$ given by
\begin{equation}
\label{mu_psi_definition}
d \mu_{\psi} := \frac{1}{z_{\psi}} \,\mathrm{e}^{-h(p,q)}\psi (M)\, dp \, dq \in \mathscr{P}(\R^{2n})\,,\quad z_{\psi}:= \int \mathrm{e}^{-h(p,q)}\,\psi(M)\, dp \, dq
\end{equation}
is a well-defined local Gibbs measure associated with $h$, which is invariant under the flow of \eqref{finite_hamiltonian}.
\end{lemma}
\begin{proof}
In order to see that $\mu_{\psi}$ given in \eqref{mu_psi_definition} is well-defined, we need to verify that $z_{\psi}<\infty$. This follows by using $h \in \mathscr{C}^{\infty}(\R^{2n})$ and \eqref{mass_finite_dim}, by which the integral defining $z_{\psi}$ in \eqref{mu_psi_definition} is taken over a compact set.
The proof of the invariance of $\mu_{\psi}$ under the flow of \eqref{finite_hamiltonian} is analogous to that of the invariance of \eqref{Gibbs_measure_finite_dim} when one assumes \eqref{partition_function_finite_dim}. The only change is that we now need to use the conservation of \eqref{mass_finite_dim} under the flow of \eqref{finite_hamiltonian}, which follows by \eqref{{h,M}}.
\end{proof}

\begin{example}
Invariance of local Gibbs measures in finite dimensions as in Lemma \ref{local_Gibbs_finite_dim} has been used to study 
the truncated system obtained from a focusing nonlinear Schr\"{o}dinger equation. We refer the reader to \cite[Section 3]{Bou94} for a detailed discussion; see also \cite[Step 1 of the proof of Proposition 2.2]{RS23} for a recent summary of the methods.
Here, one reduces the original infinite-dimensional flow to a finite dimensional one by projecting onto finitely many frequencies; see \cite[Equation (3.1)]{Bou94}. One then verifies \eqref{{h,M}} directly from the equation. 
\end{example}

Before stating the next result, let us set up some notation. We write 
\begin{equation}
\label{phase_space_decomposition}
\R^{2n}_{(p,q)} \simeq \R^n_{p} \oplus \R^{n}_{q}
\end{equation}
by splitting $\R^{2n}$ in the $p$ and $q$ variables from \eqref{p_q_variables}.
We let 
\begin{equation*}
\{e_1^{(1)},\ldots,e_n^{(1)};e_1^{(2)},\ldots,e_n^{(2)}\}
\end{equation*} 
denote the canonical orthonormal basis of $\R^{n}_p \oplus \R^{n}_q$ respectively. The complex structure on $\R^{2n}$ is given by the map $J: \R^n \oplus \R^n \rightarrow \R^n \times \R^n$ defined as 
\begin{equation*}
J u \oplus v=v \oplus -u\,. 
\end{equation*}
In particular, $J^2=-1$ and we view the map $J$ as multiplication by $-\ii$. We can hence view $\R^{2n} \simeq \mathbb{C}^{n}$ with the standard Hermitian inner product 
\begin{equation*}
\langle \xi,\eta \rangle_{\mathbb{C}^n} \equiv \sum_{j=1}^{n} \xi_j \,\overline{\eta_j}\,. 
\end{equation*}
Under these identifications, we obtain a Hermitian inner product on $\R^{2n}$ given by
\begin{equation}
\label{Hermitian_inner_product_R^{2n}}
\Big\langle \sum_{j=1}^{n} (a_j e_j^{(1)} +b_j e_j^{(2)}), \sum_{j=1}^{n} (\tilde{a}_je_j^{(1)}+ \tilde{b}_j e_j^{(2)}) \Big \rangle=\sum_{j=1}^{n} (a_j+\ii b_j) (\tilde{a}_j-\ii \tilde{b}_j)\,.
\end{equation}
In this setting, we now show that the local Gibbs measure \eqref{mu_psi_definition} is a stationary solution to an appropriate Liouville equation.
\begin{proposition}[Local Gibbs measures are stationary solutions to the Liouville equation-finite dimensions]
\label{homogeneous_finite_Liouville_lemma}
Suppose that $h \in \mathscr{C}^{\infty}(\R^{2n})$ is such that \eqref{{h,M}} holds. Suppose furthermore that $\psi \in \mathscr{C}^{\infty}(\R^{2n})$ is chosen. 
Let 
\begin{equation}
\label{vector_field_X_finite_dim}
X:\R^{2n} \rightarrow \R^{2n}\,,\qquad X:=\Bigl(\Bigl(\frac{\partial h}{\partial q_j}\Bigr)_{j=1,\ldots,n},\Bigr(-\frac{\partial h}{\partial p_j}\Bigr)_{j=1,\ldots,n}\Bigr)
\end{equation}
denote the vector field in \eqref{finite_hamiltonian}.
Then, recalling \eqref{mu_psi_definition}, we have that for all $\varphi \in \mathscr{C}^{\infty}_{\rm{c}}(\R^{2n})$
\begin{equation}
\label{homogeneous_finite_Liouville_lemma_1}
\int_{\R^{2n}} \Real \,\langle \nabla \varphi , X \rangle \, d\mu_{\psi} = 0\,.
\end{equation}
In \eqref{homogeneous_finite_Liouville_lemma_1}, the inner product $\langle \cdot, \cdot \rangle$ whose real part we take, is given by \eqref{Hermitian_inner_product_R^{2n}}. Furthermore, with the above identifications, we write 
\begin{equation}
\label{nabla_varphi_finite_dim}
\nabla \varphi=\sum_{j=1}^{n} \biggl (\frac{\partial \varphi}{\partial p_j}\, e_j^{(1)}+ \sum_{j=1}^{n} \frac{\partial \varphi}{\partial q_j}e_j^{(2)}\biggr)\,.
\end{equation}
We interpret \eqref{homogeneous_finite_Liouville_lemma_1} as saying that $\mu_{\psi}$ is a stationary solution to the Liouville equation with vector field $X$ given by \eqref{vector_field_X_finite_dim}.
\end{proposition}

\begin{proof}
By \eqref{mu_psi_definition}, \eqref{Hermitian_inner_product_R^{2n}}, \eqref{vector_field_X_finite_dim}, and \eqref{nabla_varphi_finite_dim}, we have
\begin{multline}
\label{homogeneous_finite_Liouville_lemma_2}
\int_{\R^{2n}} \Real \,\langle \nabla \varphi , X \rangle \, d\mu_{\psi} =\frac{1}{z_{\psi}} \,\int_{\R^{2n}} \textrm{Re}\,\langle \nabla \varphi , X \rangle\, \mathrm{e}^{-h}\, \psi(M) \, dp\,dq 
\\
= \frac{1}{z_{\psi}} \,\sum_{j=1}^{n} \int_{\R^{2n}}  \biggl(\frac{\partial \varphi}{\partial p_j}\,\frac{\partial h}{\partial q_j} - \frac{\partial \varphi}{\partial q_j}\,\frac{\partial h}{\partial p_j} \biggr) \, \ee^{-h}\, \psi(M)\, dp\, dq\,.
\end{multline}
We integrate by parts in \eqref{homogeneous_finite_Liouville_lemma_2} to rewrite this expression as
\begin{multline}
\label{homogeneous_finite_Liouville_lemma_3}
\frac{1}{z_{\psi}} \,\sum_{j=1}^{n} \int_{\R^{2n}} \varphi \biggl[ \biggl( -\frac{\partial^2 h}{\partial p_j \partial q_j}+\frac{\partial^2 h}{\partial q_j \partial p_j}\biggl) \,\ee^{-h}\, \psi(M)-\frac{\partial h}{\partial q_j}\,\frac{\partial}{\partial p_j}(\ee^{-h}\,\psi(M))
\\
+\frac{\partial h}{\partial p_j}\,\frac{\partial}{\partial q_j}(\ee^{-h}\,\psi(M)) \biggr]\, dp\,dq
\\
=\frac{1}{z_{\psi}} \,\int_{\R^{2n}} \varphi \,\{h,\ee^{-h}\,\psi(M)\}\, dp\,dq\,,
\end{multline}
where we recalled \eqref{Poisson_bracket_finite_dim}.
Using the Leibniz rule for the Poisson bracket, we can rewrite \eqref{homogeneous_finite_Liouville_lemma_3} as 
\begin{multline}
\label{homogeneous_finite_Liouville_lemma_4}
\frac{1}{z_{\psi}} \,\int_{\R^{2n}} \bigl(\varphi\, \{h,\ee^{-h}\}\, \psi(M)+ \varphi\,\{h,\psi(M)\}\, \ee^{-h}\bigr)\, dp\,dq
\\
=\frac{1}{z_{\psi}} \,\int_{\R^{2n}} \bigl(\varphi\, \{h,\ee^{-h}\}\, \psi(M)+ \varphi\,\psi'(M)\,\{h,M\}\, \ee^{-h}\bigr)\, dp\,dq=0
\,.
\end{multline}
In \eqref{homogeneous_finite_Liouville_lemma_4}, we used $\{h,\ee^{-h}\}=0$ and \eqref{{h,M}}.

\end{proof}
\begin{remark}
\label{integration_by_parts_principle}
In the proof of Proposition \ref{homogeneous_finite_Liouville_lemma}, a key idea was to use integration by parts to shift the derivative from the test function to the Hamiltonian. We will also use this principle in the infinite-dimensional setting.\end{remark}

\begin{remark}
\label{smoothness_assumptions_on_h}
We note that in the proof of Proposition \ref{homogeneous_finite_Liouville_lemma}, in terms of regularity of $h$, we only really need to assume that $h \in \mathscr{C}^2(\R^n)$.
\end{remark}

\subsection{Infinite-dimensional setting}
\label{Infinite-dimensional setting}
In this section, we extend the earlier analysis to the infinite-dimensional setting, and prove Theorem  \ref{homogeneous_KMS_thm}. Before proceeding, let us first recall an integration by parts formula \cite[Proposition A.1]{AS21}, which we will use in the analysis.

\begin{proposition}[Gaussian integration by parts formula]
\label{integration_by_parts_prop}
Suppose that $F \in \Dspace$ and $G \in \mathscr{C}^{\infty}_{\mathrm{b},\mathrm{cyl}}(H^{-s})$ or $F \in \mathscr{C}^{\infty}_{\mathrm{b},\mathrm{cyl}}(H^{-s})$ and $G \in \Dspace$. Then for any $\psi \in H^1$, we have
\begin{equation}
\label{integration_by_parts_prop_1}
\int_{H^{-s}} G(u) \langle \nabla F(u), \psi \rangle_{\mathcal{H}} \,d\free = \int_{H^{-s}} F(u)\left(-\langle \nabla G(u) , \psi \rangle_{\mathcal{H}} + G(u) \langle u, A\psi \rangle_{\mathcal{H}} \right) d\free\,.
\end{equation}
\end{proposition}
In our analysis, we will frequently use a special variant of Proposition \ref{integration_by_parts_prop}, which we now record for convenience. Given $j \in \N$ and $i=1,2$ and $F \in \mathbb{D}^{1,2}(\mu_0)$, we recall \eqref{ONB_convention} and write
\begin{equation}
\label{partial_derivative_convention}
\partial_j^{(i)} F(u):=\langle \nabla F(u), e_j^{(i)} \rangle_{\mathcal{H},\R}\,.
\end{equation}
We note that \eqref{partial_derivative_convention} is well-defined and that it corresponds to our convention 
for $F \in \mathscr{C}^{\infty}_{\mathrm{b},\mathrm{cyl}}(H^{-s})$ in \eqref{gradient_cylindrical_F}, in the sense that for $F$ as in \eqref{cylindrical_test_functions_1}, we have
\begin{equation*}
\langle \nabla F(u), e_j^{(i)} \rangle_{\mathcal{H},\R}=\partial_j^{(i)} \varphi(\pi_n(u))\,.
\end{equation*}

\begin{corollary}
\label{integration_by_parts_corollary}
Let $F$ and $G$ be as in the assumptions of Proposition \ref{integration_by_parts_prop}. For all $j \in \N$ and $i=1,2$, we have
\begin{equation*}
\int_{H^{-s}} \partial_j^{(i)}F(u)\, G(u) \,d\mu_0 = \int_{H^{-s}} F(u)\bigl(-\partial_j^{(i)}G(u)+ G(u) \langle u, Ae_j^{(i)}\rangle_{\mathcal{H},\R} \bigr) \,d\mu_0\,.
\end{equation*}
\end{corollary}
\begin{proof}
The claim follows from Proposition \ref{integration_by_parts_prop}. Namely, we take $\psi=e_j^{(i)} \in H^1$ and then we take real parts on both sides of \eqref{integration_by_parts_prop_1}. We recall that by assumption $F$ and $G$ are real-valued. Furthermore, we use \eqref{partial_derivative_convention}.
\end{proof}

Let us note a few more facts about the Malliavin derivative that we use in the analysis. The first is a Leibniz rule.

\begin{lemma}[Leibniz rule for the Malliavin derivative]
\label{Leibniz_rule}
Let $p_1,p_2,q \in [1,\infty)$ with $\frac{1}{q}=\frac{1}{p_1}+\frac{1}{p_2}$ be given. Suppose that $F \in \mathbb{D}^{1,p_1}(\mu_0)$ and $G \in \mathbb{D}^{1,p_2}(\mu_0)$. Then $FG \in \mathbb{D}^{1,q}(\mu_0)$, and we have
\begin{equation}
\label{Leibniz_rule_1}
\nabla(FG)=F \nabla G + G \nabla F\,.
\end{equation}
Moreover, if $F \in \mathbb{D}^{1,p_1}(\mu_0)$ and $G \in \mathscr{C}^{\infty}_{\mathrm{c},\mathrm{cyl}}(H^{-s})$, then $FG \in \mathbb{D}^{1,p_1}(\mu_0)$ and \eqref{Leibniz_rule_1} holds.
\end{lemma}
\begin{proof}
Let us first consider the claim when $F \in \mathbb{D}^{1,p_1}(\mu_0)$ and $G \in \mathbb{D}^{1,p_2}(\mu_0)$.
By H\"{o}lder's inequality, we know that 
\begin{equation}
\label{Leibniz_rule_2}
\|FG\|_{L^q(\mu_0)} \leq \|F\|_{L^{p_1}(\mu_0)}\,\|G\|_{L^{p_2}(\mu_0)}\,.
\end{equation}
We can find sequences $(F_n)_n$ and $(G_n)_n$ in $\mathscr{C}^{\infty}_{\mathrm{c},\mathrm{cyl}}(H^{-s})$ such that 
\begin{equation}
\label{Leibniz_rule_3}
F_n \stackrel{\mathbb{D}^{1,p_1}}{\longrightarrow} F\,,\qquad G_n \stackrel{\mathbb{D}^{1,p_2}}{\longrightarrow} G\,.
\end{equation}
Since $F_nG_n \in \mathscr{C}^{\infty}_{\mathrm{c},\mathrm{cyl}}(H^{-s})$, by \eqref{gradient_cylindrical_F}, we obtain that for all $n \in \N$
\begin{equation}
\label{Leibniz_rule_4}
\nabla(F_n G_n)=F_n \nabla G_n+G_n \nabla F_n\,.
\end{equation}
Using \eqref{Leibniz_rule_3}--\eqref{Leibniz_rule_4} and H\"{o}lder's inequality, we deduce that 
\begin{equation}
\label{Leibniz_rule_5}
\nabla(F_n G_n) \stackrel{L^q}{\longrightarrow} F\nabla G+G\nabla F\,.
\end{equation}
The claim when $F \in \mathbb{D}^{1,p_1}(\mu_0)$ and $G \in \mathbb{D}^{1,p_2}(\mu_0)$ follows from \eqref{Leibniz_rule_2}, \eqref{Leibniz_rule_5}, and Definition \ref{Malliavin_definition}. The claim when $F \in \mathbb{D}^{1,p_1}(\mu_0)$ and $G \in \mathscr{C}^{\infty}_{\mathrm{c},\mathrm{cyl}}(H^{-s})$ follows by analogous arguments.
\end{proof}

We also recall chain rule for the Malliavin derivative, whose proof can be found in \cite[Lemma A.2]{AS21}.
\begin{lemma}[Chain rule for the Malliavin derivative]
\label{chain_rule}
Suppose $\psi \in \mathscr{C}^1_{\mathrm{b}}(\R)$ and $F \in \mathbb{D}^{1,p}(\free)$ for some $p \in [1,\infty)$. Then $\psi(F) \in \mathbb{D}^{1,p}(\free)$ and
\begin{equation*}
\nabla \psi(F) = \psi'(F) \nabla F\,.
\end{equation*}
\end{lemma}

In what follows, we repeatedly use the following result, which follows directly from Lemma \ref{Malliavin_derivative_M_lemma} and Lemma \ref{chain_rule}.

\begin{lemma}
\label{Malliavin_derivative_M_lemma*}
Recalling \eqref{renormalized_mass} and Definition \ref{cut-off_chi}, for all $\delta \in (0,1]$ and $\BB>0$, the quantity $\chi_{\BB}^{(\delta)}(\mathcal{M})$
 belongs to $\mathbb{D}^{1,p}(\mu_{0})$ for all $p \in [1,\infty)$, and we have 
\begin{equation*}
\nabla \chi_{\BB}^{(\delta)}(\mathcal{M}(u))=2 \bigl(\chi_{\BB}^{(\delta)}\bigr)'(\mathcal{M}(u))\,u\,.
\end{equation*}
\end{lemma}

Let us finally note an elementary approximation result which gives us an approximation of the exponential function by continuously differentiable functions from with bounded derivatives. We denote the latter class of functions as $\mathscr{C}^1_{\mathrm{b}}(\R)$.

\begin{lemma}
\label{exponential_approximation}
Given $m \in \N$, we define the function $\theta_m: \R \rightarrow \R$ as 
\begin{equation*}
\theta_m(x) :=
\begin{cases}
\ee^{x} &\text{ for } x \leq m\,, \\
-\arctan(\ee^{m}(x-m)) + \ee^{m} &\text{ for } x>m\,.
\end{cases}
\end{equation*}
Then, the following properties hold.
\begin{itemize}
\item[(i)] $\theta_m \in \mathscr{C}^1_{\mathrm{b}}(\R)$.
\item[(ii)] For all $x \in \R$, we have 
\begin{equation*}
0 \leq \theta_m(x) \leq \ee^{x}\,,\qquad |\theta_m'(x)| \leq \ee^{x}\,.
\end{equation*}
\item[(iii)] For all $x \in \R$, we have
\begin{equation*}
\theta_m(x) \rightarrow \ee^{x}\,,\qquad \theta_m'(x) \rightarrow \ee^{x}
\end{equation*}
as $m \rightarrow \infty$.
\end{itemize}
\end{lemma}
\begin{proof}
The above result follows immediately by construction, and noting that for $x>m$, we have
\begin{equation*}
\theta_m'(x)=-\frac{\ee^m}{\ee^{2m}(x-m)^2+1}\,.
\end{equation*}
\end{proof}
Throughout the sequel, we omit the argument $u$ where possible to simplify the notation.  We now have all the tools at our disposal to prove Theorem \ref{homogeneous_KMS_thm}.

\begin{proof}[Proof of Theorem \ref{homogeneous_KMS_thm}]
Let us consider $\BB>0$ as in Proposition \ref{Hartree_equation_Malliavin_derivative} (ii), which we fix throughout the proof.
Following the general outline of the proof of Proposition \ref{homogeneous_finite_Liouville_lemma}, the proof of Theorem \ref{homogeneous_KMS_thm} is divided into three steps. In the first two steps, we follow ideas from the proof of Proposition \ref{homogeneous_finite_Liouville_lemma} to show that \eqref{homogeneous_KMS} holds for fixed $\delta \in (0,1)$. Step 1 corresponds to shifting derivatives from the test function $F$ by using Gaussian integration by parts. Step 2 consists of showing that the resulting expression is equal to zero. The analysis in these steps relies crucially on the differentiability of $\chi^{(\delta)}_{\BB}$ from Definition \ref{cut-off_chi}.
In the third step, we let $\delta \rightarrow 1$ to deduce \eqref{homogeneous_KMS} for $\mu^{(1)}$.

\paragraph{\textbf{Step 1: Shifting the derivative from the test function $F$ using Gaussian integration by parts}}
Let us consider $\delta \in (0,1)$. Similarly as in \eqref{homogeneous_finite_Liouville_lemma_2}--\eqref{homogeneous_finite_Liouville_lemma_3}, we want to use integration by parts (in this case given by Corollary \ref{integration_by_parts_corollary}) to shift the derivatives from $F$ to the Hamiltonian on the left-hand side of \eqref{homogeneous_KMS}. In the infinite-dimensional setting we are now considering, there are several analytical subtleties that were not present in the proof of Proposition \ref{homogeneous_finite_Liouville_lemma}. We recall \eqref{X_definition} and note that by Proposition \ref{Hartree_equation_Malliavin_derivative} (i), we only have good control in the Gross-Sobolev space of the interacting term (whereas the linear term loses two derivatives). To this end, we rewrite the left-hand side of \eqref{homogeneous_KMS} as 
\begin{equation}
\label{homogeneous_KMS_2}
\int_{H^{-s}} \langle \nabla F , -\ii Au \rangle_{\mathcal{H},\R}\,d \mu^{(\delta)}+\int_{H^{-s}} \langle \nabla F , \ii \nabla h^I \rangle_{\mathcal{H},\R}\,d \mu^{(\delta)}=:I+II\,,
\end{equation}
and treat each term separately in what follows.
Furthermore, we note that $\ee^{h^I}$ does not belong to a Gross-Sobolev space (without an additional cut-off present) due to the focusing nature of the nonlinearity. This is remedied by replacing $\ee^{h^I}$ with $\theta_m(h^I)$ for $\theta_m$ as in Lemma \ref{exponential_approximation} and letting $m \rightarrow \infty$; see \eqref{homogeneous_KMS_I_1} and \eqref{homogeneous_KMS_II_1} below for details; for the latter, we need an additional approximation of $h^I$ by functions $\mathscr{C}^{\infty}_{\mathrm{c},\mathrm{cyl}}(H^{-s})$ given by \eqref{G_m_convergence}. Once we have set up these approximations, we can carry out the analysis on each of the terms $I$ and $II$ in \eqref{homogeneous_KMS_2} separately by using the Gaussian integration by parts from Corollary \ref{integration_by_parts_corollary}. The final result of rewriting \eqref{homogeneous_KMS_2} using integration by parts is given by \eqref{homogeneous_KMS_sum_1} below.

Let us first analyse $I$. Recalling \eqref{local_Gibbs_measure_rigorous} and Lemma \ref{exponential_approximation}, we write 
\begin{multline}
\label{homogeneous_KMS_I_1}
I=\frac{1}{z^{(\delta)}}\,\int_{H^{-s}} \langle \nabla F , -\ii Au \rangle_{\mathcal{H},\R}\, \ee^{h^I}\chi_{\BB}^{(\delta)}(\mathcal{M})\,d\mu_{0}
\\
=\lim_{m} \frac{1}{z^{(\delta)}}\,\int_{H^{-s}} \langle \nabla F , -\ii A u \rangle_{\mathcal{H},\R}\, \theta_m(h^I)\,\chi_{\BB}^{(\delta)}(\mathcal{M}) \,d\mu_{0}\,.
\end{multline}
More precisely, we obtain \eqref{homogeneous_KMS_I_1} by using \eqref{homogeneous_KMS_I_2} with $q=2$, \eqref{integrability_of_weight} (from Proposition \ref{Hartree_equation_Malliavin_derivative} (ii)) with $p=2$, the Cauchy-Schwarz inequality, Lemma \ref{exponential_approximation} (ii)--(iii), and the dominated convergence theorem.

We know that for some $n \in \N$ (henceforth fixed throughout the rest of the proof) and $\varphi \in \mathscr{C}^{\infty}_{\mathrm{c}}(\R^{2n})$, we can write $F$ as in \eqref{cylindrical_test_functions_1}.
Recalling \eqref{ONB_convention} and \eqref{gradient_cylindrical_F}, we write 
\begin{equation}
\label{-iAu}
\langle \nabla F,-\ii A u \rangle_{\mathcal{H},\R}=\bigg \langle \nabla F, \sum_{j=1}^{n} \langle u, Ae_j^{(2)}\rangle_{\mathcal{H},\R}\,e_j^{(1)}-\langle u,Ae_j^{(1)}\rangle_{\mathcal{H},\R}\,e_j^{(2)} \bigg \rangle_{\mathcal{H},\R}
\end{equation}
for $u \in H^{-s}$.
For fixed $m \in \N$, we use \eqref{-iAu} and recall \eqref{gradient_cylindrical_F} again to write
\begin{multline}
\label{homogeneous_KMS_I_4_a}
\int_{H^{-s}} \langle \nabla F , -\ii Au \rangle_{\mathcal{H},\R}\, \theta_m(h^I)\,\chi_{\BB}^{(\delta)}(\mathcal{M})\,d\mu_{0}
\\
=
\int_{H^{-s}}\sum_{j=1}^{n} \bigl(\partial_j^{(1)}F\,\langle u, Ae_j^{(2)}\rangle_{\mathcal{H},\R}-\partial_j^{(2)}F\,\langle u,Ae_j^{(1)}\rangle_{\mathcal{H},\R}\bigr)\,
\\
\times
\theta_m(h^I)\,\chi_{\BB}^{(\delta)}(\mathcal{M})\,d\mu_0\,,
\end{multline}
which by Corollary \ref{integration_by_parts_corollary} equals
\begin{multline}
\label{homogeneous_KMS_I_4}
\sum_{j=1}^{n} \int_{H^{-s}} F \,\biggl[-\partial_j^{(1)} \Bigl( \langle u, Ae_j^{(2)}\rangle_{\mathcal{H},\R}\,\theta_m(h^I)\,\chi_{\BB}^{(\delta)}(\mathcal{M})\Bigr)
\\
+\partial_j^{(2)} \Bigl( \langle u, Ae_j^{(1)}\rangle_{\mathcal{H},\R}\,\theta_m(h^I)\,\chi_{\BB}^{(\delta)}(\mathcal{M})\Bigr)
+\langle u, Ae_j^{(1)} \rangle_{\mathcal{H},\R}\,\langle u,Ae_j^{(2)} \rangle_{\mathcal{H},\R}\, \theta_m(h^I)\,\chi_{\BB}^{(\delta)}(\mathcal{M})
\\
-
\langle u, Ae_j^{(2)} \rangle_{\mathcal{H},\R}\,\langle u,Ae_j^{(1)} \rangle_{\mathcal{H},\R}\, \theta_m(h^I)\,\chi_{\BB}^{(\delta)}(\mathcal{M})
\biggr]\,d\mu_0
\\
=\sum_{j=1}^{n} \int_{H^{-s}} F \,\biggl[-\partial_j^{(1)} \Bigl( \langle u, Ae_j^{(2)}\rangle_{\mathcal{H},\R}\,\theta_m(h^I)\,\chi_{\BB}^{(\delta)}(\mathcal{M})\Bigr)
\\
+\partial_j^{(2)} \Bigl( \langle u, Ae_j^{(1)}\rangle_{\mathcal{H},\R}\,\theta_m(h^I)\,\chi_{\BB}^{(\delta)}(\mathcal{M})\Bigr)
\biggr]\,d\mu_0
\,.
\end{multline}
In order to justify the application of Corollary \ref{integration_by_parts_corollary} in \eqref{homogeneous_KMS_I_4}, we should verify that for all $j=1,\ldots,n$ and $i=1,2$, we have 
\begin{equation}
\label{homogeneous_KMS_I_4_A}
\langle u, Ae_j^{(i)}\rangle_{\mathcal{H},\R}\,\theta_m(h^I)\,\chi_{\BB}^{(\delta)}(\mathcal{M}) \in \mathbb{D}^{1,2}(\mu_0)\,.
\end{equation}
We prove \eqref{homogeneous_KMS_I_4_A} for $i=1$. The case $i=2$ is analogous. Let us fix $j \in \{1,\ldots,n\}$. We first show that
\begin{equation}
\label{homogeneous_KMS_I_4_A_1}
\langle u, Ae_j^{(1)}\rangle_{\mathcal{H},\R}\,\theta_m(h^I)\,\chi_{\BB}^{(\delta)}(\mathcal{M}) \in L^2(\mu_0)\,.
\end{equation}
By \eqref{e_j_choice}, \eqref{ONB_convention}, and duality, we have 
\begin{equation}
\label{homogeneous_KMS_I_4_A_1a}
|\langle u, Ae_j^{(1)}\rangle_{\mathcal{H},\R}| = \lambda_j |\langle u, e_j^{(1)}\rangle_{\mathcal{H},\R}| \lesssim_{s,j} \|u\|_{H^{-s}}
\end{equation}
By Lemma \ref{exponential_approximation} (ii) and Definition \ref{cut-off_chi}, we have that
\begin{equation}
\label{homogeneous_KMS_I_4_A_1b}
|\theta_m(h^I)\,\chi_{\BB}^{(\delta)}(\mathcal{M})| \leq \ee^{h^I}\chi_{\BB}^{(\delta)}(\mathcal{M})\,.
\end{equation}
By \eqref{homogeneous_KMS_I_4_A_1a}--\eqref{homogeneous_KMS_I_4_A_1b} and H\"{o}lder's inequality, we obtain that 
\begin{multline}
\label{homogeneous_KMS_I_4_A_1c}
\bigl\|\langle u, Ae_j^{(1)}\rangle_{\mathcal{H},\R}\,\theta_m(h^I)\,\chi_{\BB}^{(\delta)}(\mathcal{M})\bigr\|_{L^2(\mu_0)} 
\\
\lesssim_{s,j} \|u\|_{L^4(\mu_0;H^{-s})}\, \bigl\|\ee^{h^I}\chi_{\BB}^{(\delta)}(\mathcal{M})\bigr\|_{L^4(\mu_0)}<\infty\,.
\end{multline}
For the last inequality in \eqref{homogeneous_KMS_I_4_A_1c}, we used Lemma \ref{Lemma_2.16_5} and \eqref{integrability_of_weight}. We hence deduce \eqref{homogeneous_KMS_I_4_A_1} from \eqref{homogeneous_KMS_I_4_A_1c}.

We now show that 
\begin{equation}
\label{homogeneous_KMS_I_4_A_2}
\nabla\bigl(\langle u, Ae_j^{(1)}\rangle_{\mathcal{H},\R}\,\theta_m(h^I)\,\chi_{\BB}^{(\delta)}(\mathcal{M})\bigr) \in L^2(\mu_0;H^{-s})\,.
\end{equation}
Let us note that, by \eqref{e_j_choice}, \eqref{ONB_convention}, and \eqref{gradient_cylindrical_F}, we have
\begin{equation}
\label{homogeneous_KMS_I_4_A_2i}
\nabla(\langle u, Ae_j^{(1)}\rangle_{\mathcal{H},\R})=\lambda_j \nabla(\langle u,e_j^{(1)}\rangle_{\mathcal{H},\R})= \lambda_j e_j^{(1)}\,.
\end{equation}
In particular,
\begin{equation}
\label{homogeneous_KMS_I_4_A_2i*}
-\partial_j^{(1)}\langle u, Ae_j^{(2)}\rangle_{\mathcal{H},\R}+\partial_j^{(2)}\langle u, Ae_j^{(2)}\rangle_{\mathcal{H},\R}
=0\,.
\end{equation}

We compute, by using Lemma \ref{Leibniz_rule}, \eqref{homogeneous_KMS_I_4_A_2i}, Lemma \ref{chain_rule},  \ref{Malliavin_derivative_M_lemma*}, Lemma \ref{exponential_approximation} (i), and Proposition \ref{Hartree_equation_Malliavin_derivative} (i) that
\begin{multline}
\label{homogeneous_KMS_I_4_A_2a}
\nabla\bigl(\langle u, Ae_j^{(1)}\rangle_{\mathcal{H},\R}\,\theta_m(h^I)\,\chi_{\BB}^{(\delta)}(\mathcal{M})\bigr)=\lambda_j \,\theta_m(h^I)\,\chi_{\BB}^{(\delta)}(\mathcal{M}) \,e_j^{(1)}
\\
+
\langle u, Ae_j^{(1)}\rangle_{\mathcal{H},\R}\,\theta_m'(h^I)\,\chi_{\BB}^{(\delta)}(\mathcal{M})\,\nabla h^I
+2\langle u, Ae_j^{(1)}\rangle_{\mathcal{H},\R}\,\theta_m(h^I)\,\bigl(\chi_{\BB}^{(\delta)}\bigr)'(\mathcal{M}) \,u\,.
\end{multline}

The use of Lemma \ref{Leibniz_rule} in \eqref{homogeneous_KMS_I_4_A_2a} is justified as we are applying the Malliavin gradient to a suitable product of Malliavin differentiable functions.
We estimate each of the terms on the right-hand side of \eqref{homogeneous_KMS_I_4_A_2a} separately.
By Lemma \ref{exponential_approximation} (ii) and \eqref{integrability_of_weight}, 
we have
\begin{equation}
\label{homogeneous_KMS_I_4_A_2b}
\|\theta_m(h^I)\,\chi_{\BB}^{(\delta)}(\mathcal{M}) \,e_j^{(1)}\|_{L^2(\mu_0;H^{-s})} \lesssim_{s,j} \|\ee^{h^I}\chi^{(\delta)}_{\BB}(\mathcal{M})\|_{L^2(\mu_0)}<\infty\,.
\end{equation}
By \eqref{homogeneous_KMS_I_4_A_1a} and Lemma \ref{exponential_approximation} (ii), we have that
\begin{multline}
\label{homogeneous_KMS_I_4_A_2d_1}
\|\langle u, Ae_j^{(1)}\rangle_{\mathcal{H},\R}\,\theta_m'(h^I)\,\chi_{\BB}^{(\delta)}(\mathcal{M})\,\nabla h^I\|_{L^2(\mu_0;H^{-s})} 
\\
\lesssim_{s,j} \bigl\|\|u\|_{H^{-s}}\,\ee^{h^I}\chi_{\BB}^{(\delta)}(\mathcal{M})\,\nabla h^I\bigr\|_{L^2(\mu_0;H^{-s})}\,,
\end{multline}
which by H\"{o}lder's inequality is
\begin{equation}
\label{homogeneous_KMS_I_4_A_2d}
\leq \|u\|_{L^8(\mu_0;H^{-s})}\, \|\ee^{h^I}\chi_{\BB}^{(\delta)}(\mathcal{M})\|_{L^8(\mu_0)}\,\|\nabla h^I\|_{L^4(\mu_0;H^{-s})}<\infty\,.
\end{equation}
For the last inequality in \eqref{homogeneous_KMS_I_4_A_2d}, we used Lemma \ref{Lemma_2.16_5}, \eqref{integrability_of_weight}, and Proposition \ref{Hartree_equation_Malliavin_derivative} (i).

By \eqref{homogeneous_KMS_I_4_A_1a} and Lemma \ref{exponential_approximation} (ii), we have
\begin{multline}
\label{homogeneous_KMS_I_4_A_2e_1}
\bigl\|\langle u, Ae_j^{(1)}\rangle_{\mathcal{H},\R}\,\theta_m(h^I)\,\bigl(\chi_{\BB}^{(\delta)}\bigr)'(\mathcal{M})\,u\bigr\|_{L^2(\mu_0;H^{-s})} 
\\
\lesssim_{s,j} \bigl\|\|u\|_{H^{-s}}\,\ee^{h^I}\bigl(\chi_{\BB}^{(\delta)}\bigr)'(\mathcal{M})\,u\bigr\|_{L^2(\mu_0;H^{-s})}
=\bigl\|\|u\|_{H^{-s}}^2\,\ee^{h^I}\bigl(\chi_{\BB}^{(\delta)}\bigr)'(\mathcal{M})\bigr\|_{L^2(\mu_0)}
\,,
\end{multline}
which by H\"{o}lder's inequality is
\begin{equation}
\label{homogeneous_KMS_I_4_A_2e}
\leq \|u\|_{L^6(\mu_0;H^{-s})}^2\, \bigl\|\ee^{h^I}\bigl(\chi_{\BB}^{(\delta)}\bigr)'(\mathcal{M})\bigr\|_{L^6(\mu_0)}<\infty\,.
\end{equation}
In order to obtain the last inequality in \eqref{homogeneous_KMS_I_4_A_2e}, we used Lemma \ref{Lemma_2.16_5} and the observation that for all $p \in [1,\infty)$, we have
\begin{equation}
\label{integrability_of_weight'}
\bigl\|\ee^{h^I}\bigl(\chi_{\BB}^{(\delta)}\bigr)'(\mathcal{M})\bigr\|_{L^p(\mu_0)}<\infty\,.
\end{equation}
In order to deduce \eqref{integrability_of_weight'}, we recall Definition \ref{cut-off_chi} and argue analogously as for the proof of \eqref{integrability_of_weight} when $d=2,3$; the main point here is that the proof relies on the fact that the support of $\chi'_{\BB}$ is contained in $[-\BB,\BB]$; see \eqref{Integral_1} above. When $d=1$, the result follows from \cite[Lemma 3.10]{Bou94}; see also \cite[Appendix A]{RS22}. By combining \eqref{homogeneous_KMS_I_4_A_2b}--\eqref{homogeneous_KMS_I_4_A_2e}, we obtain \eqref{homogeneous_KMS_I_4_A_2}. From \eqref{homogeneous_KMS_I_4_A_1} and \eqref{homogeneous_KMS_I_4_A_2}, we obtain \eqref{homogeneous_KMS_I_4_A}, hence justifying the calculations used to obtain \eqref{homogeneous_KMS_I_4}.

Using Lemma \ref{Leibniz_rule}, Lemma \ref{chain_rule}, and \eqref{homogeneous_KMS_I_4_A_2i}, we can hence rewrite \eqref{homogeneous_KMS_I_4} as
\begin{multline}
\label{homogeneous_KMS_I_5}
\sum_{j=1}^{n} \int_{H^{-s}} F\,\Bigl[-\langle u,Ae_j^{(2)} \rangle_{\mathcal{H},\R}\,\partial_j^{(1)} \chi_{\BB}^{(\delta)}(\mathcal{M})\,\theta_m(h^I)
\\
+\langle u,Ae_j^{(1)} \rangle_{\mathcal{H},\R}\,\partial_j^{(2)} \chi_{\BB}^{(\delta)}(\mathcal{M})\,\theta_m(h^I)
-\langle u, Ae_j^{(2)} \rangle_{\mathcal{H},\R} \,\theta_m'(h^I)\, \partial_j^{(1)} h^I\,\chi_{\BB}^{(\delta)}(\mathcal{M})
\\
+\langle u, Ae_j^{(1)} \rangle_{\mathcal{H},\R} \,\theta_m'(h^I)\, \partial_j^{(2)} h^I\,\chi_{\BB}^{(\delta)}(\mathcal{M})
\Bigr]\,d \mu_0\,.
\end{multline}

We now want to let $m \rightarrow \infty$ in \eqref{homogeneous_KMS_I_5}. H\"{o}lder's inequality and the arguments as for \eqref{homogeneous_KMS_I_4_A_2e_1}--\eqref{homogeneous_KMS_I_4_A_2e} show that 
\begin{equation}
\label{homogeneous_KMS_I_7}
\langle u,Ae_j^{(2)} \rangle_{\mathcal{H},\R}\,\partial_j^{(1)} \chi_{\BB}^{(\delta)}(\mathcal{M})\,\ee^{h^I}=
2 \langle u,Ae_j^{(2)} \rangle_{\mathcal{H},\R}\, \bigl(\chi_{\BB}^{(\delta)}\bigr)'(\mathcal{M})\,\langle u, e_j^{(1)} \rangle_{\mathcal{H},\R}\,\ee^{h^I}
 \in L^1(\mu_0)\,.
\end{equation}
Analogously, we have 
\begin{equation}
\label{homogeneous_KMS_I_7*}
\langle u,Ae_j^{(1)} \rangle_{\mathcal{H},\R}\,\partial_j^{(2)} \chi_{\BB}^{(\delta)}(\mathcal{M})\,\ee^{h^I} \in L^1(\mu_0)\,.
\end{equation}
Furthermore, the arguments as for \eqref{homogeneous_KMS_I_4_A_2d_1}--\eqref{homogeneous_KMS_I_4_A_2d} show that 
\begin{equation}
\label{homogeneous_KMS_I_8}
\langle u, Ae_j^{(2)} \rangle_{\mathcal{H},\R} \,\ee^{h^I}\, \partial_j^{(1)} h^I\chi_{\BB}^{(\delta)}(\mathcal{M})\,,\,\langle u, Ae_j^{(1)} \rangle_{\mathcal{H},\R} \,\ee^{h^I}\, \partial_j^{(2)} h^I\chi_{\BB}^{(\delta)}(\mathcal{M}) \in L^1(\mu_0)\,.
\end{equation}
We use \eqref{homogeneous_KMS_I_7}--\eqref{homogeneous_KMS_I_8}, the fact that $F \in \mathscr{C}^{\infty}_{\mathrm{c},\mathrm{cyl}}(H^{-s})$, Lemma \ref{exponential_approximation} (ii)--(iii), and the dominated convergence theorem in \eqref{homogeneous_KMS_I_5} to rewrite \eqref{homogeneous_KMS_I_1} as
\begin{multline}
\label{homogeneous_KMS_I_9}
I=\frac{1}{z^{(\delta)}}\sum_{j=1}^{n} \int_{H^{-s}} F\,\Big[-\langle u,Ae_j^{(2)} \rangle_{\mathcal{H},\R}\,\partial_j^{(1)} \chi_{\BB}^{(\delta)}(\mathcal{M})\,\ee^{h^I}
\\
+\langle u,Ae_j^{(1)} \rangle_{\mathcal{H},\R}\,\partial_j^{(2)} \chi_{\BB}^{(\delta)}(\mathcal{M})\,\ee^{h^I}
-\langle u, Ae_j^{(2)} \rangle_{\mathcal{H},\R} \,\ee^{h^I}\, \partial_j^{(1)} h^I\chi_{\BB}^{(\delta)}(\mathcal{M})
\\+\langle u, Ae_j^{(1)} \rangle_{\mathcal{H},\R} \,\ee^{h^I}\, \partial_j^{(2)} h^I\chi_{\BB}^{(\delta)}(\mathcal{M})
\Bigr]\,d \mu_0\,.
\end{multline}

We now analyse term $II$ from \eqref{homogeneous_KMS_2}. 
By Proposition \ref{Hartree_equation_Malliavin_derivative} (i), we know that $h^I \in \mathbb{D}^{1,4}(\mu_0)$. Hence, we can find a sequence $(G_m)_{m}$ in $\mathscr{C}^{\infty}_{\mathrm{c},\mathrm{cyl}}(H^{-s})$ such that 
\begin{equation}
\label{G_m_convergence}
G_m \rightarrow h^I \quad \mbox{in} \quad \mathbb{D}^{1,4}(\mu_0)\,.
\end{equation}
Similarly as in \eqref{homogeneous_KMS_I_1}, we write
\begin{multline}
\label{homogeneous_KMS_II_1}
II=\frac{1}{z^{(\delta)}}\,\int_{H^{-s}} \langle \nabla F , \ii \nabla h^I \rangle_{\mathcal{H},\R}\, \ee^{h^I}\chi_{\BB}^{(\delta)}(\mathcal{M})\,d\mu_{0}
\\
=\lim_{m} \frac{1}{z^{(\delta)}}\,\int_{H^{-s}} \langle \nabla F , \ii \nabla G_m \rangle_{\mathcal{H},\R}\, \theta_m(h^I)\,\chi_{\BB}^{(\delta)}(\mathcal{M}) \,d\mu_{0}\,.
\end{multline}
Let us justify \eqref{homogeneous_KMS_II_1}. We first note that 
\begin{multline}
\label{homogeneous_KMS_II_2}
\bigl\|\langle \nabla F , \ii \nabla h^I \rangle_{\mathcal{H},\R}\bigr\|_{L^2(\mu_0)} \leq \big\| \|\nabla F\|_{H^s}\,\|\nabla h^I\|_{H^{-s}} \bigr\|_{L^2(\mu_0)}
\\
\leq \|\nabla F\|_{L^4(\mu_0;H^{s})}\,\|\nabla h^I\|_{L^4(\mu_0;H^{-s})} \lesssim_{F,s} \|\nabla h^I\|_{L^4(\mu_0;H^{-s})} 
<\infty\,,
\end{multline}
which follows by using duality, H\"{o}lder's inequality, and Proposition \ref{Hartree_equation_Malliavin_derivative} (i).
Arguing as for \eqref{homogeneous_KMS_II_2} and recalling \eqref{G_m_convergence}, we also obtain
\begin{equation}
\label{homogeneous_KMS_II_3}
\bigl\|\langle \nabla F , \ii \nabla (G_m-h^I) \rangle_{\mathcal{H},\R}\bigr\|_{L^2(\mu_0)}  \lesssim_{F,s} \|\nabla (G_m-h^I)\|_{L^4(\mu_0;H^{-s})} \rightarrow 0\,.
\end{equation}
We have
\begin{multline}
\label{homogeneous_KMS_II_3_B}
\biggl|\int_{H^{-s}} \Big(\langle \nabla F , \ii \nabla G_m \rangle_{\mathcal{H},\R}\, \theta_m(h^I)\,\chi_{\BB}^{(\delta)}(\mathcal{M}) 
- \langle \nabla F , \ii \nabla h^I \rangle_{\mathcal{H},\R}\, \ee^{h^I}\chi_{\BB}^{(\delta)}(\mathcal{M}) \Bigr)\,d\mu_{0}\biggr|
\\
\leq \biggl| \int_{H^{-s}} \langle \nabla F , \ii \nabla (G_m-h^I) \rangle_{\mathcal{H},\R}\, \theta_m(h^I)\,\chi_{\BB}^{(\delta)}(\mathcal{M})\,d\mu_0 \biggr|
\\
+\biggl| \int_{H^{-s}} \langle \nabla F , \ii \nabla h^I \rangle_{\mathcal{H},\R}\, \bigl(\theta_m(h^I)-\ee^{h^I}\bigr)\,\chi_{\BB}^{(\delta)}(\mathcal{M})\, d\mu_0\biggr|\,.
\end{multline}
By using the Cauchy-Schwarz inequality and Lemma \ref{exponential_approximation} (ii), the first term on the right-hand side of \eqref{homogeneous_KMS_II_3_B} is 
\begin{equation*}
\leq \bigl\|\langle \nabla F , \ii \nabla (G_m-h^I) \rangle_{\mathcal{H},\R}\bigr\|_{L^2(\mu_0)} \,\bigl\|\ee^{h^I} \chi_{\BB}^{(\delta)}(\mathcal{M})\bigr\|_{L^2(\mu_0)}\,,
\end{equation*}
which tends to zero as $m \rightarrow \infty$ by \eqref{homogeneous_KMS_II_3} and \eqref{integrability_of_weight}. 
Similarly, by the Cauchy-Schwarz inequality, the second term on the right-hand side of \eqref{homogeneous_KMS_II_3_B} is
\begin{equation*}
\leq \bigl\|\langle \nabla F , \ii \nabla h^I \rangle_{\mathcal{H},\R}\bigr\|_{L^2(\mu_0)} \,\bigl\|\bigl(\theta_m(h^I)-\ee^{h^I}\bigr)\,\chi_{\BB}^{(\delta)}(\mathcal{M})\bigr\|_{L^2(\mu_0)}\,,
\end{equation*}
which tends to zero by \eqref{homogeneous_KMS_II_2}, Lemma \ref{exponential_approximation} (ii)--(iii), \eqref{integrability_of_weight} with $p=2$, and the dominated convergence theorem. In particular, we deduce that the right-hand side of \eqref{homogeneous_KMS_II_3_B} tends to zero as $m \rightarrow \infty$. We hence obtain \eqref{homogeneous_KMS_II_1}, as was claimed.

For fixed $m \in \N$, we write
\begin{multline*}
\int_{H^{-s}} \langle \nabla F , \ii \nabla G_m \rangle_{\mathcal{H},\R}\, \theta_m(h^I)\,\chi_{\BB}^{(\delta)}(\mathcal{M})\,d\mu_{0}
\\
=
- \int_{H^{-s}}\sum_{j=1}^{n} \bigl(\partial_j^{(1)}F\,\partial_j^{(2)}G_m-\partial_j^{(2)}F\,\partial_j^{(1)}G_m\bigr)\,
\theta_m(h^I)\,\chi_{\BB}^{(\delta)}(\mathcal{M})\,d\mu_0\,,
\end{multline*}
which by Corollary \ref{integration_by_parts_corollary} is 
\begin{multline}
\label{homogeneous_KMS_II_4_a*}
=\sum_{j=1}^{n}\int_{H^{-s}} F \Bigl[ 
\partial_j^{(1)} \partial_j^{(2)} G_m \, \theta_m(h^I) \,\chi_{\BB}^{(\delta)}(\mathcal{M})+\partial_j^{(2)} G_m\, \partial_j^{(1)} \bigl[\theta_m(h^I)\, \chi_{\BB}^{(\delta)}(\mathcal{M}) \bigr]
\\
-\partial_j^{(2)}G_m\,\theta_m(h^I)\,\chi_{\BB}^{(\delta)}(\mathcal{M})\,\langle u, Ae_j^{(1)}\rangle_{\mathcal{H},\R}
\\
-\partial_j^{(2)} \partial_j^{(1)} G_m \, \theta_m(h^I) \,\chi_{\BB}^{(\delta)}(\mathcal{M})-\partial_j^{(1)} G_m\, \partial_j^{(2)} \bigl[\theta_m(h^I)\, \chi_{\BB}^{(\delta)}(\mathcal{M}) \bigr]
\\
+\partial_j^{(1)}G_m\,\theta_m(h^I)\,\chi_{\BB}^{(\delta)}(\mathcal{M})\,\langle u, Ae_j^{(2)}\rangle_{\mathcal{H},\R}
\Bigr]\,d \mu_0\,.
\end{multline}
Since $G_m \in \mathscr{C}^{\infty}_{\mathrm{c},\mathrm{cyl}}(H^{-s})$, the expression in \eqref{homogeneous_KMS_II_4_a*} is
\begin{multline}
\label{homogeneous_KMS_II_4_a} 
=\sum_{j=1}^{n}\int_{H^{-s}} F \Bigl[ 
\partial_j^{(2)} G_m\, \partial_j^{(1)} \bigl[\theta_m(h^I)\, \chi_{\BB}^{(\delta)}(\mathcal{M}) \bigr]
\\
-\partial_j^{(2)}G_m\,\theta_m(h^I)\,\chi_{\BB}^{(\delta)}(\mathcal{M})\,\langle u, Ae_j^{(1)}\rangle_{\mathcal{H},\R}
-\partial_j^{(1)} G_m\, \partial_j^{(2)} \bigl[\theta_m(h^I)\, \chi_{\BB}^{(\delta)}(\mathcal{M}) \bigr]
\\
+\partial_j^{(1)}G_m\,\theta_m(h^I)\,\chi_{\BB}^{(\delta)}(\mathcal{M})\,\langle u, Ae_j^{(2)}\rangle_{\mathcal{H},\R}
\Bigr]\,d \mu_0\,.
\end{multline}
In order to justify the application of Corollary \ref{integration_by_parts_corollary} in \eqref{homogeneous_KMS_II_4_a*}, we note that for all $j=1,\ldots,n$ and $i=1,2$, we have 
\begin{equation}
\label{homogeneous_KMS_II_4_b}
\partial_j^{(i)}G_m\,\theta_m(h^I)\,\chi_{\BB}^{(\delta)}(\mathcal{M}) \in \mathbb{D}^{1,2}(\mu_0)\,.
\end{equation}
We show \eqref{homogeneous_KMS_II_4_b} by arguing similarly as for \eqref{homogeneous_KMS_I_4_A}.  
Using the fact that $\partial_j^{(i)}G_m \in L^{\infty}(\mu_0)$ combined with \eqref{homogeneous_KMS_I_4_A_1b} and \eqref{integrability_of_weight} with $p=2$, we deduce that 
\begin{equation}
\label{homogeneous_KMS_II_4_c}
\partial_j^{(i)}G_m\,\theta_m(h^I)\,\chi_{\BB}^{(\delta)}(\mathcal{M}) \in L^2(\mu_0)\,.
\end{equation}
We now show that
\begin{equation}
\label{homogeneous_KMS_II_4_d}
\nabla\bigl(\partial_j^{(i)}G_m\,\theta_m(h^I)\,\chi_{\BB}^{(\delta)}(\mathcal{M})\bigr) \in L^2(\mu_0;H^{-s})\,.
\end{equation}
By Lemma \ref{Leibniz_rule}, Lemma \ref{chain_rule}, Lemma \ref{Malliavin_derivative_M_lemma*}, Lemma \ref{exponential_approximation}, and Proposition \ref{Hartree_equation_Malliavin_derivative} (i), we compute
\begin{multline}
\label{homogeneous_KMS_II_4_e}
\nabla\bigl(\partial_j^{(i)}G_m\,\theta_m(h^I)\,\chi_{\BB}^{(\delta)}(\mathcal{M})\bigr)
\\
=\theta_m(h^I)\,\chi_{\BB}^{(\delta)}(\mathcal{M})\,\nabla \partial_j^{(i)}G_m+\partial_j^{(i)}G_m\,\theta_m'(h^I)\,\chi_{\BB}^{(\delta)}(\mathcal{M})\,\nabla h^I
\\
+2 \partial_j^{(i)}G_m\,\theta_m(h^I)\,\bigl(\chi_{\BB}^{(\delta)}\bigr)'(\mathcal{M})\,u\,.
\end{multline}
We now show that all of the terms on the right-hand side of \eqref{homogeneous_KMS_II_4_e} belong to $L^2(\mu_0;H^{-s})$. The application of Lemma \ref{Leibniz_rule} is justified as in \eqref{homogeneous_KMS_I_4_A_2a} above.
We estimate each term separately.
By H\"{o}lder's inequality, Lemma \ref{exponential_approximation} (ii), and \eqref{integrability_of_weight} from  Proposition \ref{Hartree_equation_Malliavin_derivative} (ii), we have that
\begin{multline}
\label{homogeneous_KMS_II_4_e1}
\|\theta_m(h^I)\,\chi_{\BB}^{(\delta)}(\mathcal{M})\,\nabla \partial_j^{(i)}G_m\|_{L^2(\mu_0;H^{-s})} 
\\
\leq \|\ee^{h^I}\chi_{\BB}^{(\delta)}(\mathcal{M})\|_{L^4(\mu_0)}\,\|\nabla \partial_j^{(i)}G_m\|_{L^4(\mu_0;H^{-s})} <\infty\,.
\end{multline}
By H\"{o}lder's inequality, Lemma \ref{exponential_approximation} (ii), and Proposition \ref{Hartree_equation_Malliavin_derivative}, we have
\begin{multline}
\label{homogeneous_KMS_II_4_e2}
\|\partial_j^{(i)}G_m\,\theta_m'(h^I)\,\chi_{\BB}^{(\delta)}(\mathcal{M})\,\nabla h^I\|_{L^2(\mu_0;H^{-s})}
\\
\leq \|\partial_j^{(i)}G_m\|_{L^4(\mu_0)}\, \|\ee^{h^I}\chi_{\BB}^{(\delta)}(\mathcal{M})\|_{L^8(\mu_0)}\,\|\nabla h^I\|_{L^8(\mu_0;H^{-s})}<\infty\,.
\end{multline}
By H\"{o}lder's inequality, Lemma \ref{exponential_approximation} (ii), \eqref{integrability_of_weight'}, and Lemma \ref{Lemma_2.16_5}, we have
\begin{multline}
\label{homogeneous_KMS_II_4_e3}
\bigl\|\partial_j^{(i)}G_m\,\theta_m(h^I)\,\bigl(\chi_{\BB}^{(\delta)}\bigr)'(\mathcal{M})\,u\bigr\|_{L^2(\mu_0;H^{-s})}
\\
\leq \|\partial_j^{(i)}G_m\|_{L^4(\mu_0)}\, \bigl\|\ee^{h^I}\bigl(\chi_{\BB}^{(\delta)}\bigr)'(\mathcal{M})\bigr\|_{L^8(\mu_0)}\,\|u\|_{L^8(\mu_0;H^{-s})}<\infty\,.
\end{multline}
In \eqref{homogeneous_KMS_II_4_e1}--\eqref{homogeneous_KMS_II_4_e3}, we also used the fact that $G_m \in \mathscr{C}^{\infty}_{\mathrm{c},\mathrm{cyl}}(H^{-s})$ to note that 
\begin{equation*}
\|\nabla \partial_j^{(i)}G_m\|_{L^2(\mu_0;H^{-s})}<\infty\,,\quad \|\partial_j^{(i)}G_m\|_{L^4(\mu_0)}<\infty\,.
\end{equation*}
We obtain \eqref{homogeneous_KMS_II_4_d} from \eqref{homogeneous_KMS_II_4_e1}--\eqref{homogeneous_KMS_II_4_e3}. Hence, we obtain \eqref{homogeneous_KMS_II_4_b} from \eqref{homogeneous_KMS_II_4_c}--\eqref{homogeneous_KMS_II_4_d}.

By minor modifications of the argument used to prove \eqref{homogeneous_KMS_I_4_A}, we obtain
\begin{equation}
\label{homogeneous_KMS_I_4_A*}
\theta_m(h^I)\, \chi_{\BB}^{(\delta)}(\mathcal{M})  \in \mathbb{D}^{1,2}(\mu_0)\,. 
\end{equation}
We omit the details.
In particular, we can use Lemma \ref{Leibniz_rule}, followed by Lemma \ref{chain_rule}, Lemma \ref{Malliavin_derivative_M_lemma*}, and Lemma \ref{exponential_approximation} to rewrite \eqref{homogeneous_KMS_II_4_a} as

\begin{multline}
\label{homogeneous_KMS_II_4_f}
\sum_{j=1}^{n}\int_{H^{-s}} F \Bigl[ 
\partial_j^{(2)} G_m\,  \theta_m'(h^I)\,\partial_j^{(1)}h^I \chi_{\BB}^{(\delta)}(\mathcal{M})+\partial_j^{(2)} G_m\,\theta_m(h^I)\, \partial_j^{(1)} \chi_{\BB}^{(\delta)}(\mathcal{M}) 
\\
-\partial_j^{(2)}G_m\,\theta_m(h^I)\,\chi_{\BB}^{(\delta)}(\mathcal{M})\,\langle u, Ae_j^{(1)}\rangle_{\mathcal{H},\R}
\\
-\partial_j^{(1)} G_m\, \theta_m'(h^I)\, \partial_j^{(2)}h^{I}\, \chi_{\BB}^{(\delta)}(\mathcal{M})
-\partial_j^{(1)} G_m\, \theta_m(h^I)\,\partial_j^{(2)} \chi_{\BB}^{(\delta)}(\mathcal{M})
\\
+\partial_j^{(1)}G_m\,\theta_m(h^I)\,\chi_{\BB}^{(\delta)}(\mathcal{M})\,\langle u, Ae_j^{(2)}\rangle_{\mathcal{H},\R}
\Bigr]\,d \mu_0\,.
\end{multline}
We now want to let $m \rightarrow \infty$ in \eqref{homogeneous_KMS_II_4_f}.
Let us fix $j \in \{1,\ldots,n\}$. We have
\begin{multline}
\label{homogeneous_KMS_II_4_f_1}
\biggl|\int_{H^{-s}} F
\partial_j^{(2)} G_m\,  \theta_m'(h^I)\,\partial_j^{(1)}h^I \chi_{\BB}^{(\delta)}(\mathcal{M})\,d\mu_0-
\int_{H^{-s}} F
\partial_j^{(2)} h^I\,  \ee^{h^I}\,\partial_j^{(1)}h^I \chi_{\BB}^{(\delta)}(\mathcal{M})\,d\mu_0\biggr|
\\
\leq
\biggl|\int_{H^{-s}} F
\bigl(\partial_j^{(2)} G_m-\partial_j^{(2)}h^I\bigr)\,  \theta_m'(h^I)\,\partial_j^{(1)}h^I \chi_{\BB}^{(\delta)}(\mathcal{M})\,d\mu_0\biggr|
\\
+\biggl|\int_{H^{-s}} F
\partial_j^{(2)} h^I\,  \bigl(\theta_m'(h^I)-\ee^{h^I}\bigr)\,\partial_j^{(1)}h^I \chi_{\BB}^{(\delta)}(\mathcal{M})\,d\mu_0\biggr|\,.
\end{multline}
We argue as for \eqref{homogeneous_KMS_II_4_e2} to deduce that the first term on the right-hand side of \eqref{homogeneous_KMS_II_4_f_1} is
\begin{equation*}
\leq \|F\|_{L^2(\mu_0)}\, \|\partial_j^{(2)}(G_m-h^I)\|_{L^4(\mu_0)}\, \|\ee^{h^I}\chi_{\BB}^{(\delta)}(\mathcal{M})\|_{L^8(\mu_0)}\,\|\nabla h^I\|_{L^8(\mu_0;H^{-s})}\,,
\end{equation*}
which converges to zero as $m \rightarrow \infty$ by \eqref{G_m_convergence}. Here, we also recall $F \in \mathscr{C}^{\infty}_{\mathrm{c},\mathrm{cyl}}(H^{-s})$. The second term on the right-hand side of \eqref{homogeneous_KMS_II_4_f_1} is 
\begin{equation*}
\leq \|F\|_{L^2(\mu_0)}\, \|\partial_j^{(2)}h^I\|_{L^4(\mu_0)}\, \|(\theta_m'(h^I)-\ee^{h^I})\chi_{\BB}^{(\delta)}(\mathcal{M})\|_{L^8(\mu_0)}\,\|\nabla h^I\|_{L^8(\mu_0;H^{-s})}\,,
\end{equation*}
which converges to zero as $m \rightarrow \infty$, similarly as for the first term. Here, in addition, we use Lemma \ref{exponential_approximation} (ii)--(iii) and the dominated convergence theorem. Using these observations in \eqref{homogeneous_KMS_II_4_f_1}, and noting that we can reverse the roles of the $\partial_j^{(1)}$ and $\partial_j^{(2)}$ by symmetry, we deduce that
\begin{multline}
\label{homogeneous_KMS_II_5}
\lim_m \int_{H^{-s}} F \Bigl[ 
\partial_j^{(2)} G_m\,  \theta_m'(h^I)\,\partial_j^{(1)}h^I \chi_{\BB}^{(\delta)}(\mathcal{M})-
\partial_j^{(1)} G_m\,  \theta_m'(h^I)\,\partial_j^{(2)}h^I \chi_{\BB}^{(\delta)}(\mathcal{M})\Bigr]\,d \mu_0
\\
=\int_{H^{-s}} F \Bigl[ 
\partial_j^{(2)} h^I\,  \ee^{h^I}\,\partial_j^{(1)}h^I \chi_{\BB}^{(\delta)}(\mathcal{M})-
\partial_j^{(1)} h^I\,  \ee^{h^I}\,\partial_j^{(2)}h^I \chi_{\BB}^{(\delta)}(\mathcal{M})\Bigr]\,d \mu_0
=0\,.
\end{multline}
By a minor modification of the argument used to prove \eqref{homogeneous_KMS_II_5} (in particular, we use the estimates from \eqref{homogeneous_KMS_II_4_e3} instead of those from \eqref{homogeneous_KMS_II_4_e2}), we deduce that 
\begin{multline}
\label{homogeneous_KMS_II_6}
\lim_m \int_{H^{-s}} F \Bigl[ 
\partial_j^{(2)} G_m\,  \theta_m(h^I)\, \partial_j^{(1)}\chi_{\BB}^{(\delta)}(\mathcal{M})-
\partial_j^{(1)} G_m\,  \theta_m(h^I)\,\partial_j^{(2)}\chi_{\BB}^{(\delta)}(\mathcal{M})\Bigr]\,d \mu_0
\\
=\int_{H^{-s}} F \Bigl[ 
\partial_j^{(2)} h^I\,  \ee^{h^I}\,\partial_j^{(1)} \chi_{\BB}^{(\delta)}(\mathcal{M})-
\partial_j^{(1)} h^I\,  \ee^{h^I}\,\partial_j^{(2)} \chi_{\BB}^{(\delta)}(\mathcal{M})\Bigr]\,d \mu_0\,.
\end{multline}
Likewise, by similar arguments (now using \eqref{homogeneous_KMS_I_4_A_1c} instead of \eqref{homogeneous_KMS_II_4_e2}), we obtain
\begin{multline}
\label{homogeneous_KMS_II_7}
\lim_m \int_{H^{-s}} F \Bigl[-\partial_j^{(2)}G_m\,\theta_m(h^I)\,\chi_{\BB}^{(\delta)}(\mathcal{M})\,\langle u, Ae_j^{(1)}\rangle_{\mathcal{H},\R}
\\
+\partial_j^{(1)}G_m\,\theta_m(h^I)\,\chi_{\BB}^{(\delta)}(\mathcal{M})\,\langle u, Ae_j^{(2)}\rangle_{\mathcal{H},\R}
\Bigr]\,d \mu_0
\\
= \int_{H^{-s}} F \Bigl[-\partial_j^{(2)}h^I\,\ee^{h^I}\chi_{\BB}^{(\delta)}(\mathcal{M})\,\langle u, Ae_j^{(1)}\rangle_{\mathcal{H},\R}
\\
+\partial_j^{(1)}h^I\,\ee^{h^I}\chi_{\BB}^{(\delta)}(\mathcal{M})\,\langle u, Ae_j^{(2)}\rangle_{\mathcal{H},\R}
\Bigr]\,d \mu_0\,.
\end{multline} 
Using \eqref{homogeneous_KMS_II_4_f} and \eqref{homogeneous_KMS_II_5}--\eqref{homogeneous_KMS_II_7}, we can rewrite \eqref{homogeneous_KMS_II_1} as 
\begin{multline}
\label{homogeneous_KMS_II_8}
II=\frac{1}{z^{(\delta)}} \sum_{j=1}^{n}  \int_{H^{-s}} F \Bigl[\partial_j^{(2)} h^I\,  \ee^{h^I}\,\partial_j^{(1)} \chi_{\BB}^{(\delta)}(\mathcal{M})-
\partial_j^{(1)} h^I\,  \ee^{h^I}\,\partial_j^{(2)} \chi_{\BB}^{(\delta)}(\mathcal{M})
\\
-\partial_j^{(2)}h^I\,\ee^{h^I}\chi_{\BB}^{(\delta)}(\mathcal{M})\,\langle u, Ae_j^{(1)}\rangle_{\mathcal{H},\R}
+\partial_j^{(1)}h^I\,\ee^{h^I}\chi_{\BB}^{(\delta)}(\mathcal{M})\,\langle u, Ae_j^{(2)}\rangle_{\mathcal{H},\R}
\Bigr]\,d \mu_0\,.
\end{multline}
By using \eqref{homogeneous_KMS_2}, \eqref{homogeneous_KMS_I_9}, and \eqref{homogeneous_KMS_II_8}, we deduce that 
\begin{multline}
\label{homogeneous_KMS_sum_1}
\int_{H^{-s}} \langle \nabla F , X(u) \rangle_{\mathcal{H},\R} \, d\mu=\frac{1}{z^{(\delta)}} \sum_{j=1}^{n} \int_{H^{-s}}F\,\Big[-\langle u,Ae_j^{(2)} \rangle_{\mathcal{H},\R}\,\partial_j^{(1)} \chi_{\BB}^{(\delta)}(\mathcal{M})\,\ee^{h^I}
\\
+\langle u,Ae_j^{(1)} \rangle_{\mathcal{H},\R}\,\partial_j^{(2)} \chi_{\BB}^{(\delta)}(\mathcal{M})\,\ee^{h^I}+\partial_j^{(2)} h^I\,  \ee^{h^I}\,\partial_j^{(1)} \chi_{\BB}^{(\delta)}(\mathcal{M})
\\
-
\partial_j^{(1)} h^I\,  \ee^{h^I}\,\partial_j^{(2)} \chi_{\BB}^{(\delta)}(\mathcal{M})\Bigr]\,d \mu_0\,.
\end{multline}

\paragraph{\textbf{Step 2: Proof that the expression in \eqref{homogeneous_KMS_sum_1} equals to zero}}

We now want to show that the expression \eqref{homogeneous_KMS_sum_1}, that we obtained by Gaussian integration by parts, is equal to zero. The proof is similar in spirit to that of \eqref{homogeneous_finite_Liouville_lemma_4}; see Remark \ref{comment_on_proof} (ii)--(iii) below for a detailed comparison.
Recalling \eqref{nonlinear_Hamiltonian} and noting that by \eqref{nonlinear_Hamiltonian_B}, we have
\begin{equation}
\label{h_0_derivative}
\partial_j^{(i)}h_0=\langle u, Ae_j^{(i)} \rangle_{\mathcal{H},\R}
\end{equation}
for all $j=1,\ldots,n$ and $i=1,2$, we can rewrite \eqref{homogeneous_KMS_sum_1} as
\begin{multline}
\label{homogeneous_KMS_sum_2}
\frac{1}{z_R} \int_{H^{-s}} F \sum_{j=1}^{n} \Bigl(\partial_j^{(1)}h \,\partial_j^{(2)} \chi_{\BB}^{(\delta)}(\mathcal{M})-\partial_j^{(2)}h\, \partial_j^{(1)} \chi_{\BB}^{(\delta)}(\mathcal{M})\Bigr)\,\ee^{h^I}\,d \mu_0
\\
=\frac{1}{z^{(\delta)}} \int_{H^{-s}} F \langle \mathrm{\Pi}_n \nabla h,-\ii \mathrm{\Pi}_n \nabla \chi_{\BB}^{(\delta)}(\mathcal{M})\rangle_{\mathcal{H},\R}\,\ee^{h^I}\,d \mu_0\,,
\end{multline}
where $\mathrm{\Pi}_n: H^{-s} \rightarrow  L^2 \equiv \mathcal{H}$ is the projection given by
\begin{equation}
\label{projection_Pi_n}
\mathrm{\Pi}_n (\cdot) := \sum_{j=1}^n \langle \cdot, e_j^{(1)} \rangle_{\mathcal{H},\R} \,e^{(1)}_j +  \langle \cdot, e_j^{(2)} \rangle_{\mathcal{H},\R}\,e^{(2)}_j\,.
\end{equation}
Using \eqref{nonlinear_Hamiltonian_B}, Lemma \ref{Malliavin_derivative_M_lemma*}, and \eqref{projection_Pi_n}, we write
\begin{multline}
\label{homogeneous_KMS_sum_3}
\bigl \langle \mathrm{\Pi}_n \nabla h_0,-\ii \mathrm{\Pi}_n \nabla \chi_{\BB}^{(\delta)}(\mathcal{M})\bigr\rangle_{\mathcal{H},\R}=
2\bigl\langle \mathrm{\Pi}_n A u,-\ii \mathrm{\Pi}_n \bigl[\bigl(\chi_{\BB}^{(\delta)}\bigr)'(\mathcal{M})u\bigr]\bigr\rangle_{\mathcal{H},\R}
\\
=2\bigl\langle A^{1/2} \,\mathrm{\Pi}_n u,-\ii  \bigl(\chi_{\BB}^{(\delta)}\bigr)'(\mathcal{M}) A^{1/2} \,\mathrm{\Pi}_n u\bigr\rangle_{\mathcal{H},\R}=0\,.
\end{multline}
In \eqref{homogeneous_KMS_sum_3}, we used the fact that $A$ given by Assumption \ref{A_choice} is a positive operator which commutes with $\Pi_n$ and that $(\chi_{\BB}^{(\delta)})'$ is a real-valued function, which follows from Definition \ref{cut-off_chi}.

Let us suppose for now that $h^I$ is given as in Assumption \ref{Assumption_on_V} (ii). Arguing similarly as in \eqref{homogeneous_KMS_sum_3}, and using the identity \eqref{nabla_h^I_formula} from Remark \ref{formula_for_Malliavin_derivative_Hartree} below, we compute
\begin{multline}
\label{homogeneous_KMS_sum_4}
\bigl\langle \mathrm{\Pi}_n \nabla h^I,-\ii \mathrm{\Pi}_n \nabla \chi_{\BB}^{(\delta)}(\mathcal{M})\bigr\rangle_{\mathcal{H},\R}=
2 \bigl\langle \mathrm{\Pi}_n \nabla h^I,-\ii  \bigl(\chi_{\BB}^{(\delta)}\bigr)'(\mathcal{M}) \mathrm{\Pi}_n u \bigr\rangle_{\mathcal{H},\R}
\\
=2\bigl \langle (V*\,:|u|^2:\,)\mathrm{\Pi}_n u,-\ii  \bigl(\chi_{\BB}^{(\delta)}\bigr)'(\mathcal{M})  \,\mathrm{\Pi}_n u\bigr\rangle_{\mathcal{H},\R}=0\,.
\end{multline}
In \eqref{homogeneous_KMS_sum_4}, we used the assumption that $V$ was real-valued.
By using \eqref{nonlinear_Hamiltonian} and \eqref{homogeneous_KMS_sum_3}--\eqref{homogeneous_KMS_sum_4}, we deduce that $\eqref{homogeneous_KMS_sum_2}=0$ for $h^I$ as in Assumption \ref{Assumption_on_V} (ii).

The analysis for $h^I$ as in Assumption \ref{Assumption_on_V} (i) is similar. Namely, for $h^I$ is as in \eqref{local_nonlinearity_1D}, instead of 
\eqref{homogeneous_KMS_sum_4}, we use
\begin{equation}
\label{homogeneous_KMS_sum_4*}
\bigl \langle \mathrm{\Pi}_n \nabla h^I,-\ii \mathrm{\Pi}_n \nabla \chi_{\BB}^{(\delta)}(\mathcal{M})\bigr \rangle_{\mathcal{H},\R}=
2\bigl\langle |u|^{r-1} \mathrm{\Pi}_n u,-\ii  \bigl(\chi_{\BB}^{(\delta)}\bigr)'(\mathcal{M})  \,\mathrm{\Pi}_n u\bigr\rangle_{\mathcal{H},\R}=0\,.
\end{equation}
Here, we recalled \cite[Proposition 5.1, proof of part (3)]{AS21}.
Likewise, for $h^I$ as in \eqref{nonlocal_nonlinearity_1D}, we use
\begin{equation}
\label{homogeneous_KMS_sum_4**}
\langle \mathrm{\Pi}_n \nabla h^I,-\ii \mathrm{\Pi}_n \nabla \chi_{\BB}^{(\delta)}(\mathcal{M})\rangle_{\mathcal{H},\R}=
2\bigl\langle (V*\,|u|^2\,)\mathrm{\Pi}_n u,-\ii  \bigl(\chi_{\BB}^{(\delta)}\bigr)'(\mathcal{M})  \,\mathrm{\Pi}_n u \bigr\rangle_{\mathcal{H},\R}=0\,,
\end{equation}
since $V$ is real-valued. Here, we recalled \cite[Equation (5.15)]{AS21}. We hence conclude that \eqref{homogeneous_KMS_sum_1} equals to zero, and thus that \eqref{homogeneous_KMS} holds for all $\delta \in (0,1)$.

\paragraph{\textbf{Step 3: Letting $\delta \rightarrow 1$ and obtaining \eqref{homogeneous_KMS} for $\delta=1$.}}

By recalling \eqref{local_Gibbs_measure_rigorous}, and by using \eqref{homogeneous_KMS} for $\delta \in (0,1)$ proved above, the claim for $\delta=1$ follows provided that we show that as $\delta \rightarrow 1$
\begin{equation}
\label{Step_3}
\langle \nabla F ,  -\ii A u + \ii \nabla h^I \rangle_{\mathcal{H},\R} \, \ee^{h^I} \chi_{\BB}^{(\delta)}(\mathcal{M}) \rightarrow \langle \nabla F ,  -\ii A u + \ii \nabla h^I \rangle_{\mathcal{H},\R} \, \ee^{h^I} \chi_{\BB}^{(1)}(\mathcal{M})
\end{equation}
in $L^1(\mu_0)$.

By H\"{o}lder's inequality, Definition \ref{cut-off_chi}, \eqref{integrability_of_weight}, \eqref{homogeneous_KMS_I_2} and the dominated convergence theorem, we have
\begin{multline}
\label{rough_cut-off_Liouville_eq_2}
\bigl\|\langle \nabla F,-\ii Au \rangle_{\mathcal{H},\R} \,\ee^{h^I} \bigl(\chi_{\BB}^{(\delta)}(\mathcal{M})-\chi_{\BB}^{(1)}(\mathcal{M})\bigr)\bigr\|_{L^1(\mu_0)}
\\
\leq \|\langle \nabla F , -\ii Au \rangle_{\mathcal{H},\R}\|_{L^{\infty}(\mu_0)}\, \bigl\|\ee^{h^I}\bigl(\chi_{\BB}^{(\delta)}(\mathcal{M})-\chi_{\BB}^{(1)}(\mathcal{M})\bigr)\bigr\|_{L^1(\mu_0)} \rightarrow 0\,\,\,\mbox{as } \,\delta \rightarrow 1
\,.
\end{multline}
Similarly, by using duality, H\"{o}lder's inequality, $F \in \mathscr{C}^{\infty}_{\mathrm{c},\mathrm{cyl}}(H^{-s})$, and Proposition \ref{Hartree_equation_Malliavin_derivative} (i), we have
\begin{multline}
\label{rough_cut-off_Liouville_eq_3}
\bigl\|\langle \nabla F,\ii \nabla h^I \rangle_{\mathcal{H},\R} \,\ee^{h^I}\bigl(\chi_{\BB}^{(\delta)}(\mathcal{M})-\chi_{\BB}^{(1)}(\mathcal{M})\bigr)\bigr\|_{L^1(\mu_0)}
\leq \|\nabla F\|_{L^{\infty}(\mu_0;H^{s})} \times\,
\\
\times \|\nabla h^I\|_{L^2(\mu_0;H^{-s})} \, \bigl\|\ee^{h^I}\bigl(\chi_{\BB}^{(\delta)}(\mathcal{M})-\chi_{\BB}^{(1)}(\mathcal{M})\bigr)\bigr\|_{L^2(\mu_0)} \rightarrow 0\,\,\,\mbox{as } \,\delta \rightarrow 1
\,.
\end{multline}
We obtain \eqref{Step_3} from \eqref{rough_cut-off_Liouville_eq_2}--\eqref{rough_cut-off_Liouville_eq_3}. The claim now follows.
\end{proof}

\begin{remark}
\label{dominated_convergence_theorem_remark}
The arguments used in the proof of \eqref{rough_cut-off_Liouville_eq_2} more generally show that for all $p \in [1,\infty)$
\begin{equation}
\label{dominated_convergence_theorem_remark_1}
 \lim_{\delta \rightarrow 1}  \bigl\|\ee^{h^I} \bigl(\chi_{\BB}^{(\delta)}(\mathcal{M})-\chi_{\BB}^{(1)}(\mathcal{M})\bigr)\bigr\|_{L^p(\mu_0)} =0\,.
\end{equation}
\end{remark}

\begin{remark}
\label{comment_on_proof}
Let us comment on the parallels in the proofs of Proposition \ref{homogeneous_finite_Liouville_lemma} and Theorem  \ref{homogeneous_KMS_thm}. The parallels can only be seen in Steps 1 and 2, where $\delta \in (0,1)$. Throughout we recall \eqref{nonlinear_Hamiltonian}.
\begin{itemize}
\item[(i)]
The cancellation 
\begin{equation*}
\sum_{j=1}^{n} \biggl(-\frac{\partial^2 h}{\partial p_j \partial q_j}+\frac{\partial^2 h}{\partial q_j \partial p_j}\biggr)=0
\end{equation*}
in \eqref{homogeneous_finite_Liouville_lemma_3} corresponds to the cancellation \eqref{homogeneous_KMS_I_4_A_2i*} in \eqref{homogeneous_KMS_I_4} for $h_0$ (by recalling \eqref{h_0_derivative}), and \eqref{homogeneous_KMS_II_4_a*}--\eqref{homogeneous_KMS_II_4_a} for $h^I$.
\item[(ii)] By again recalling \eqref{h_0_derivative}, the identity $\{h,\ee^{-h}\}=0$ from \eqref{homogeneous_finite_Liouville_lemma_4} corresponds to the cancellation between the sums of the last two terms in \eqref{homogeneous_KMS_I_9} and \eqref{homogeneous_KMS_II_8} respectively and to the cancellation in \eqref{homogeneous_KMS_II_5}.
\item[(iii)] By \eqref{poisson_general}, it follows that the identity $\{h,\psi(M)\}=0$ from \eqref{homogeneous_finite_Liouville_lemma_4} formally corresponds to 
\begin{equation}
\label{homogeneous_infinite_Liouville_lemma_4*}
\bigl\langle \mathrm{\Pi}_n \nabla h,-\ii \mathrm{\Pi}_n \nabla \chi_{\BB}^{(\delta)}(\mathcal{M})\bigr\rangle_{\mathcal{H},\R}=0\,,
\end{equation}
which follows from \eqref{homogeneous_KMS_sum_3}--\eqref{homogeneous_KMS_sum_4**}. However, \eqref{homogeneous_infinite_Liouville_lemma_4*} does not come from a conserved quantity in finite dimensions, but from the precise form of $\nabla h^I$ and $\nabla \mathcal{M}$, which allows us to deduce that the inner product \eqref{homogeneous_infinite_Liouville_lemma_4*} vanishes.
\end{itemize}

\end{remark}

We note the following lemma, which follows from the proof of Theorem \ref{homogeneous_KMS_thm} given above.

\begin{lemma}
\label{exponential_Malliavin_differentiability}
For all $\delta \in (0,1)$ and $p \in [1,\infty)$, we have $\ee^{-h^I} \chi^{(\delta)}_{\BB}(\mathcal{M}) \in \mathbb{D}^{1,p}(\mu_0)$. Moreover, 
\begin{equation}
\label{exponential_Malliavin_differentiability_1}
\nabla \bigl(\ee^{-h^I} \chi^{(\delta)}_{\BB}(\mathcal{M})\bigr)=-\ee^{-h^I}\chi^{(\delta)}_{\BB}(\mathcal{M}) \nabla h^I+2\ee^{-h^I} \bigl(\chi^{(\delta)}_{\BB}\bigr)'(\mathcal{M})\,u\,.
\end{equation}
\end{lemma}
\begin{proof}
Similarly as for \eqref{homogeneous_KMS_I_4_A*}, we have that for all $m \in \N$
\begin{equation}
\label{exponential_Malliavin_differentiability_2}
\theta_m(-h^I)\, \chi_{\BB}^{(\delta)}(\mathcal{M})  \in \mathbb{D}^{1,p}(\mu_0)\,. 
\end{equation}
Namely, by Assumption \ref{Assumption_on_V}, we can replace $h^I$ by $-h^I$ in \eqref{homogeneous_KMS_I_4_A*} by the same proof.

Moreover, from \eqref{exponential_Malliavin_differentiability_2}, Lemmas \ref{Leibniz_rule}--\ref{Malliavin_derivative_M_lemma*}, and Lemma \ref{exponential_approximation} (i), we compute
\begin{equation}
\label{exponential_Malliavin_differentiability_3}
\nabla \bigl(\theta_m(-h^I)\, \chi_{\BB}^{(\delta)}(\mathcal{M}) \bigr)=-\theta_m'(-h^I) \chi^{(\delta)}_{\BB}(\mathcal{M})\nabla h^I+2\theta_m(-h^I) \bigl(\chi^{(\delta)}_{\BB}\bigr)'(\mathcal{M})\,u\,. 
\end{equation}
By using \eqref{integrability_of_weight} with $h^I$ replaced by $-h^I$ (which is possible to do by Assumption \ref{Assumption_on_V}), Lemma \ref{exponential_approximation} (ii), and the dominated convergence theorem, we have that
\begin{equation}
\label{exponential_Malliavin_differentiability_4}
\lim_m \|(\theta_m(-h^I)-\ee^{-h^I})\, \chi_{\BB}^{(\delta)}(\mathcal{M})\|_{L^p(\mu_0)}=0\,.
\end{equation}
We deduce the claim from \eqref{exponential_Malliavin_differentiability_4} once we prove that the right-hand side of \eqref{exponential_Malliavin_differentiability_3} converges to the right-hand side of \eqref{exponential_Malliavin_differentiability_1} as $m \rightarrow \infty$ in $L^p(\mu_0;H^{-s})$. To this end, we use H\"{o}lder's inequality and estimate
\begin{multline}
\bigl\|-\theta_m'(-h^I) \chi^{(\delta)}_{\BB}(\mathcal{M})\nabla h^I+2\theta_m(h^I) \bigl(\chi^{(\delta)}_{\BB}\bigr)'(\mathcal{M})\,u
\\
+
\ee^{-h^I}\chi^{(\delta)}_{\BB}(\mathcal{M}) \nabla h^I-2\ee^{h^I} \bigl(\chi^{(\delta)}_{\BB}\bigr)'(\mathcal{M})\,u\bigr\|_{L^p(\mu_0;H^{-s})}
\\
\leq \|(\theta_m'(-h^I)-\ee^{-h^I}) \chi^{(\delta)}_{\BB}(\mathcal{M})\|_{L^{2p}(\mu_0)}\,\|\nabla h^I\|_{L^{2p}(\mu_0;H^{-s})} 
\\
+2\bigl\|(\theta_m(-h^I)-\ee^{-h^I}) \bigl(\chi^{(\delta)}_{\BB}\bigr)'(\mathcal{M})\bigr\|_{L^{2p}(\mu_0)}\,\|u\|_{L^{2p}(\mu_0;H^{-s})}\,,
\end{multline}
which tends to zero as $m \rightarrow \infty$ by using \eqref{integrability_of_weight} and \eqref{integrability_of_weight'} both with $h^I$ replaced by $-h^I$, Proposition \ref{Hartree_equation_Malliavin_derivative} (i), Lemma \ref{Lemma_2.16_5}, Lemma \ref{exponential_approximation} (ii)--(iii), and the dominated convergence theorem.
\end{proof}

\begin{remark}
\label{exponential_Malliavin_differentiability_remark}
When $h^I \geq 0$, one does not need to use the cut-off in Lemma \ref{exponential_Malliavin_differentiability}. Similar arguments as in the proof of Lemma \ref{exponential_Malliavin_differentiability} show that then
$\ee^{-h^I} \in \mathbb{D}^{1,p}(\mu_0)$ and $\nabla \ee^{-h^I}=-\ee^{-h^I}\nabla h^I$. We do not use this in what follows.
\end{remark}

\section{Local Gibbs measures are local KMS states. Proof of Theorem \ref{Gibbs_implies_local_KMS_theorem} }
\label{Gibbs_implies_KMS_subsection}
In this section, we prove that local Gibbs measures satisfy a local KMS condition in the sense mentioned above. This is the content of Theorem \ref{Gibbs_implies_local_KMS_theorem}. The proof is similar to that of \cite[Proposition 5.7]{AS21}, although it is more subtle as we also work in dimensions $d=2,3$. We note that the notion of a local KMS state in Definition \ref{local_KMS_definition} is slightly different from the corresponding one in \cite[Section 5.3]{AS21}. This convention turns out to be more convenient for our setting.
The analysis in \cite[Section 5.3]{AS21} is done for $d=1$ and therefore does not require one to work with the renormalized mass. In this section, we are also considering the cases $d=2$ and $d=3$. By \eqref{renormalized_mass}, this requires us to work with the renormalized mass.

Before proceeding to the proof of Theorem \ref{Gibbs_implies_local_KMS_theorem}, we note that the free Gibbs measure \ref{free_Gibbs_prop} with $A$ as in Assumption \ref{A_choice} is a local KMS state in the sense of Definition  \ref{local_KMS_definition} above, with vector field $X_0=-\ii A$ as in \eqref{X^I_definition}. In particular, this justifies our localization convention that $G \in \mathbb{D}^{1,2}_{\BB}(\mu_0)$  in Definition \ref{local_KMS_definition} (which is slightly different from those used in \cite[Section 5.3]{AS21}). 
We prove the following result.

\begin{proposition}
\label{local_KMS_approximation_lemma}
Let $A$ be as in Assumption \ref{A_choice}, $\mu_0$ is the Gaussian measure given by Proposition \ref{free_Gibbs_prop}. For all $F \in \mathscr{C}^{\infty}_{\mathrm{c},\mathrm{cyl}}(H^{-s})$ and for all $G \in \mathbb{D}^{1,2}(\mu_0)$, we have
\begin{equation}
\label{local_KMS_approximation_lemma_1}
\int_{H^{-s}} \{F,G\}\,d \mu_0= \int_{H^{-s}} \langle \nabla F(u),-\ii A u\rangle_{\mathcal{H},\R}\, G(u)\, d \mu_0\,.\end{equation}
\end{proposition}

\begin{remark}
In particular, from Proposition \ref{local_KMS_approximation_lemma}, we obtain that the identity \eqref{local_KMS_approximation_lemma_1} holds for all $F \in \mathscr{C}^{\infty}_{\mathrm{c},\mathrm{cyl}}(H^{-s})$ and $G \in \mathbb{D}^{1,2}(\mu_0)$ (in this case, we do not need to consider $G \in \mathbb{D}^{1,2}_{\BB}(\mu_0)$ as in Theorem \ref{Gibbs_implies_local_KMS_theorem}).
\end{remark}

\begin{proof}[Proof of Proposition \ref{local_KMS_approximation_lemma}]
By \cite[Theorem 4.11]{AS21} (with $h^I \equiv 0$), we obtain that $\mu_0$ satisfies the KMS condition in the sense of Definition \ref{KMS_definition_rigorous} with vector field $X_0=-
\ii A$. In particular, we have that \eqref{local_KMS_approximation_lemma_1} holds for all $F,G \in \mathscr{C}^{\infty}_{\mathrm{c},\mathrm{cyl}}(H^{-s})$. 
We now show how to obtain the general claim by density.

For $G \in \mathbb{D}^{1,2}(\mu_0)$, we find a sequence $(G_m)_{m}$ in $\mathscr{C}^{\infty}_{\mathrm{c},\mathrm{cyl}}(H^{-s})$ such that $G_m \rightarrow G$ in $\mathbb{D}^{1,2}(\mu_0)$. We recall \eqref{poisson_general} and use duality followed by the Cauchy-Schwarz inequality to obtain that
\begin{multline}
\label{local_KMS_approximation_lemma_2}
\biggl|\int_{H^{-s}} \{F,G_m\}\,d \mu_0-\int_{H^{-s}} \{F,G\}\,d \mu_0\biggl| 
\\
\leq \|\nabla F\|_{L^2(\mu_0;H^s)}\,\|\nabla G_m - \nabla G\|_{L^2(\mu_0;H^{-s})} \rightarrow 0 \,\,\, \mbox{as} \,\,\, m \rightarrow \infty\,. 
\end{multline}
Similarly, using \eqref{X^I_definition}, the self-adjointness of $A$, duality, H\"{o}lder's inequality, and Lemma \ref{Lemma_2.16_5}, we have 
\begin{multline}
\label{local_KMS_approximation_lemma_3}
\biggl|\int_{H^{-s}} \langle \nabla F,\ii A u\rangle_{\mathcal{H},\R}\, G_m\, d \mu_0-\int_{H^{-s}} \langle \nabla F,\ii A u\rangle_{\mathcal{H},\R}\, G(u)\, d \mu_0\biggl| 
\\
\leq
\|A \nabla F\|_{L^4(\mu_0;H^s)}\,\|u\|_{L^4(\mu_0;H^{-s})}\,\|G_m-G\|_{L^2(\mu_0)} \rightarrow 0 \,\,\, \mbox{as} \,\,\, m \rightarrow \infty\,. 
\end{multline}
For \eqref{local_KMS_approximation_lemma_2}--\eqref{local_KMS_approximation_lemma_3}, we also used the assumption that $F\in \mathscr{C}^{\infty}_{\mathrm{c},\mathrm{cyl}}(H^{-s})$. The general claim then follows from  \eqref{local_KMS_approximation_lemma_2}--\eqref{local_KMS_approximation_lemma_3}.
\end{proof}

Let us now present the proof of Theorem \ref{Gibbs_implies_local_KMS_theorem}.

\begin{proof}[Proof of Theorem \ref{Gibbs_implies_local_KMS_theorem}]
Let us consider $F$ of the form \eqref{cylindrical_test_functions_1} for $n \in \N$ and $\varphi \in \mathscr{C}^{\infty}_{\mathrm{c}}(\R^{2n})$ and let us consider $G \in \mathbb{D}^{1,2}_{\BB}(\mu_0)$ (defined in \eqref{D^{1,2}_R} above), all of which we fix throughout the proof. Recalling \eqref{B_R'}--\eqref{D^{1,2}_R}, we can find $\BB' \in (0,\BB)$ such that 
\begin{equation}
\label{Theorem_2_R'_choice}
G(u) = 0\,\,\mbox{$\mu_0$-almost surely on} \,\, \mathbb{B}_{\BB'}^{c}\,.
\end{equation}
With $\BB'$ as in \eqref{Theorem_2_R'_choice}, we show that
\begin{equation}
\label{Theorem_2_delta_approximation}
\int_{H^{-s}} \{F,G\}\,d \mu^{(\delta)}= \int_{H^{-s}} \langle \nabla F,-\ii A u +\ii \nabla h^I \rangle_{\mathcal{H},\R}\, G\, d \mu^{(\delta)}
\end{equation}
for all 
\begin{equation}
\label{Theorem_2_delta_choice}
\delta \in (\BB'/\BB,1]\,.
\end{equation}
We divide the proof in two steps. In the first step, we prove \eqref{Theorem_2_delta_approximation} for $\delta \in (\BB'/\BB,1)$. Here, we need to use the differentiability properties of $\chi_{\BB}^{(\delta)}$. In the second step, we let $\delta \rightarrow 1$ and obtain \eqref{Theorem_2_delta_approximation} with $\delta=1$, which is the claim of the theorem.

\paragraph{\textbf{Step 1: Proof of \eqref{Theorem_2_delta_approximation} for $\delta \in (\BB'/\BB,1)$}}
We use the Gaussian integration by parts formula given by Corollary \ref{integration_by_parts_corollary} above. In order to rigorously justify the use of the latter, we use an approximation argument. To this end, we consider a sequence $(G_k)_{k}$ in $\mathscr{C}^{\infty}_{\mathrm{c},\mathrm{cyl}}(H^{-s})$ such that 
\begin{equation}
\label{G_k_convergence_2}
G_k \rightarrow G \quad \mbox{in} \quad \mathbb{D}^{1,2}(\mu_0)\,.
\end{equation}
We recall \eqref{well-definedness_remark_2_b} and \eqref{homogeneous_KMS_I_2**} respectively (both with $G$ replaced by $G-G_k$) and deduce that \eqref{G_k_convergence_2} implies
\begin{equation}
\label{well-definedness_remark_2_b_2}
\lim_k \int_{H^{-s}} \{F,G_k\}\,d \mu^{(\delta)} = \int_{H^{-s}} \{F,G\}\,d \mu^{(\delta)} 
\end{equation}
and
\begin{equation}
\label{homogeneous_KMS_I_2**_2}
\lim_k \int_{H^{-s}} \langle \nabla F,-\ii A u +\ii \nabla h^I \rangle_{\mathcal{H},\R}\, G_k\, d \mu^{(\delta)} =
\int_{H^{-s}} \langle \nabla F,-\ii A u +\ii \nabla h^I \rangle_{\mathcal{H},\R}\, G\, d \mu^{(\delta)} \,.
\end{equation}
We recall \eqref{local_Gibbs_measure_rigorous} and \eqref{poisson_general}, and consider for fixed $k \in \N$ the quantity
\begin{multline}
\label{Theorem_2_1}
\int_{H^{-s}} \{F,G_k\}\,\ee^{h^{I}}\chi_{\BB}^{(\delta)}(\mathcal{M})\,d \mu_0
\\
=
\sum_{j=1}^{n} \int_{H^{-s}} \bigl(\partial_j^{(1)}F\, \partial_j^{(2)}G_k - \partial_j^{(1)} G_k \,\partial_j^{(2)}F \bigr)\,\ee^{h^I} \chi_{\BB}^{(\delta)}(\mathcal{M})\,d\mu_0
\\
= \lim_m\, \sum_{j=1}^{n} \int_{H^{-s}} \bigl(\partial_j^{(1)}F\, \partial_j^{(2)}G_k- \partial_j^{(1)} G_k \,\partial_j^{(2)}F \bigr)\,\theta_m(h^I)\, \chi_{\BB}^{(\delta)}(\mathcal{M})\,d\mu_0\,.
\end{multline}
The last step in \eqref{Theorem_2_1} holds since $F, G_k \in \mathscr{C}^{\infty}_{\mathrm{c},\mathrm{cyl}}(H^{-s})$, by using Lemma \ref{exponential_approximation} (ii)--(iii), \eqref{integrability_of_weight},  and the dominated convergence theorem.

By \eqref{homogeneous_KMS_I_4_A*}, the fact that $F \in \mathscr{C}^{\infty}_{\mathrm{c},\mathrm{cyl}}(H^{-s})$, and Lemma \ref{Leibniz_rule}, we have that for all $m \in \N$, $1 \leq j \leq n$, $1 \leq i \leq 2$
\begin{equation}
\label{Theorem_2_3}
\partial_j^{(i)}F\, \theta_m(h^I)\,\chi_{\BB}^{(\delta)}(\mathcal{M}) \in \mathbb{D}^{1,2}(\mu_0)\,.
\end{equation}
For fixed $m \in \N$, we use Corollary \ref{integration_by_parts_corollary} 
to write
\begin{align}
\notag
\sum_{j=1}^{n} \int_{H^{-s}} \bigl(\partial_j^{(1)}F\, \partial_j^{(2)}G_k- \partial_j^{(1)} G_k \,\partial_j^{(2)}F \bigr)\,\theta_m(h^I)\, \chi_{\BB}^{(\delta)}(\mathcal{M})\,d\mu_0&
\\
\notag
=\sum_{j=1}^{n} \int_{H^{-s}} G_k\,\Bigl[-\partial_j^{(2)} \partial_j^{(1)}F\,\theta_m(h^I)\,\chi_{\BB}^{(\delta)}(\mathcal{M})
-\partial_j^{(1)}F \,\partial_j^{(2)}\bigl(\theta_m(h^I)\,\chi_{\BB}^{(\delta)}(\mathcal{M})\bigr)&
\\
\notag
+\partial_j^{(1)} \partial_j^{(2)}F\,\theta_m(h^I)\,\chi_{\BB}^{(\delta)}(\mathcal{M})
+\partial_j^{(2)}F \,\partial_j^{(1)}\bigl(\theta_m(h^I)\,\chi_{\BB}^{(\delta)}(\mathcal{M})\bigr)&
\\
\label{Theorem_2_4}
+\partial_j^{(1)}F\,\theta_m(h^I)\,\chi_{\BB}^{(\delta)}(\mathcal{M})\,\langle u, Ae_j^{(2)}\rangle_{\mathcal{H},\R}-
\partial_j^{(2)}F\,\theta_m(h^I)\,\chi_{\BB}^{(\delta)}(\mathcal{M})\,\langle u, Ae_j^{(1)}\rangle_{\mathcal{H},\R}\Bigr]
\,d\mu_0\,.&
\end{align}
The above application of Corollary \ref{integration_by_parts_corollary} is justified by \eqref{Theorem_2_3} and the fact that $G_k \in \mathscr{C}^{\infty}_{\mathrm{c},\mathrm{cyl}}(H^{-s})$. We use the fact that $F \in \mathscr{C}^{\infty}_{\mathrm{c},\mathrm{cyl}}(H^{-s})$, combined with Lemma \ref{Leibniz_rule}, Lemma \ref{chain_rule}, and Lemma \ref{Malliavin_derivative_M_lemma*} to write \eqref{Theorem_2_4} as
\begin{multline}
\label{Theorem_2_5}
\sum_{j=1}^{n} \int_{H^{-s}} G_k\,\Bigl[-\partial_j^{(1)}F\,\theta'_m(h^I)\,\partial_j^{(2)}h^I\,\chi_{\BB}^{(\delta)}(\mathcal{M})
\\
-2\partial_j^{(1)}F\,\theta_m(h^I)\,\bigl(\chi_{\BB}^{(\delta)}\bigr)'(\mathcal{M})\,\langle u, e_j^{(2)}\rangle_{\mathcal{H},\R}
\\
+\partial_j^{(2)}F\,\theta'_m(h^I)\,\partial_j^{(1)}h^I\,\chi_{\BB}^{(\delta)}(\mathcal{M})+2\partial_j^{(2)}F\,\theta_m(h^I)\,\bigl(\chi_{\BB}^{(\delta)}\bigr)'(\mathcal{M})\,\langle u, e_j^{(1)}\rangle_{\mathcal{H},\R}
\\
+\partial_j^{(1)}F\,\theta_m(h^I)\,\chi_{\BB}^{(\delta)}(\mathcal{M})\,\langle u, Ae_j^{(2)}\rangle_{\mathcal{H},\R}-
\partial_j^{(2)}F\,\theta_m(h^I)\,\chi_{\BB}^{(\delta)}(\mathcal{M})\,\langle u, Ae_j^{(1)}\rangle_{\mathcal{H},\R}\Bigr]
\,d\mu_0\,.
\end{multline}
By applying the same arguments as for \eqref{homogeneous_KMS_II_5}--\eqref{homogeneous_KMS_II_7} and recalling \eqref{X^I_definition}, we deduce that
\begin{multline}
\lim_m \eqref{Theorem_2_5}=
\\
\label{Theorem_2_6}
\sum_{j=1}^{n} \int_{H^{-s}} G_k\,\Bigl[-\partial_j^{(1)}F\,\partial_j^{(2)}h^I\,\ee^{h^I}\chi_{\BB}^{(\delta)}(\mathcal{M})+\partial_j^{(2)}F\,\partial_j^{(1)}h^I\,\ee^{h^I}\chi_{\BB}^{(\delta)}(\mathcal{M})
\\
+\partial_j^{(1)}F\,\ee^{h^I}\chi_{\BB}^{(\delta)}(\mathcal{M})\,\langle u, Ae_j^{(2)}\rangle_{\mathcal{H},\R}-
\partial_j^{(2)}F\,\ee^{h^I}\chi_{\BB}^{(\delta)}(\mathcal{M})\,\langle u, Ae_j^{(1)}\rangle_{\mathcal{H},\R}
\\
-2\partial_j^{(1)}F\,\ee^{h^I}\bigl(\chi_{\BB}^{(\delta)}\bigr)'(\mathcal{M})\,\langle u, e_j^{(2)}\rangle_{\mathcal{H},\R}+2\partial_j^{(2)}F\,\ee^{h^I}\chi'_{\BB}(\mathcal{M})\,\langle u, e_j^{(1)}\rangle_{\mathcal{H},\R}\Bigr]
\,d\mu_0
\\
=\int_{H^{-s}} \langle \nabla F, -\ii Au+\ii \nabla h^I \rangle_{\mathcal{H},\R}\,G_k\,d \mu_0
\\
+\sum_{j=1}^{n} \int_{H^{-s}} G_k\,\Bigl[-2\partial_j^{(1)}F\,\ee^{h^I}\bigl(\chi_{\BB}^{(\delta)}\bigr)'(\mathcal{M})\,\langle u, e_j^{(2)}\rangle_{\mathcal{H},\R}
\\
+2\partial_j^{(2)}F\,\ee^{h^I}\bigl(\chi_{\BB}^{(\delta)}\bigr)'(\mathcal{M})\,\langle u, e_j^{(1)}\rangle_{\mathcal{H},\R}\Bigr]
\,d\mu_0\,.
\end{multline}
Recalling \eqref{well-definedness_remark_2_b_2}--\eqref{homogeneous_KMS_I_2**_2},  \eqref{local_Gibbs_measure_rigorous}, \eqref{Theorem_2_1}, and \eqref{Theorem_2_6}, the claim of Step 1 follows if we show that for all $1\leq j \leq n$, we have
\begin{multline}
\label{Theorem_2_7}
\lim_k \int_{H^{-s}} G_k\,\Bigl[-2\partial_j^{(1)}F\,\ee^{h^I}\bigl(\chi_{\BB}^{(\delta)}\bigr)'(\mathcal{M})\,\langle u, e_j^{(2)}\rangle_{\mathcal{H},\R}
\\
+2\partial_j^{(2)}F\,\ee^{h^I}\bigl(\chi_{\BB}^{(\delta)}\bigr)'(\mathcal{M})\,\langle u, e_j^{(1)}\rangle_{\mathcal{H},\R}\Bigr]
\,d\mu_0=0\,.
\end{multline}
Let us fix $k \in \N$. By Definition \ref{cut-off_chi}, we have that 
\begin{multline}
\label{Theorem_2_8}
\int_{H^{-s}} G_k\,\partial_j^{(1)}F\,\ee^{h^I}\bigl(\chi_{\BB}^{(\delta)}\bigr)'(\mathcal{M})\,\langle u, e_j^{(2)}\rangle_{\mathcal{H},\R}\,d \mu_0
\\
=\int_{H^{-s}} G_k \Phi_{\BB,\BB'}^{(\delta)}(\mathcal{M})\,\partial_j^{(1)}F\,\ee^{h^I}\bigl(\chi_{\BB}^{(\delta)}\bigr)'(\mathcal{M})\,\langle u, e_j^{(2)}\rangle_{\mathcal{H},\R}\,d \mu_0\,,
\end{multline}
where $\Phi_{\BB,\BB'}^{(\delta)} \in \mathscr{C}^{\infty}(\R)$ is chosen in such a way that
\begin{equation}
\label{Theorem_2_9}
\Phi_{\BB,\BB'}^{(\delta)}(x) =
\begin{cases}
1 \text{ for } |x| \geq  \BB \delta\,, \\
0 \text{ for } |x|  \leq \BB'\,.
\end{cases}
\end{equation}
It is possible to find such a function by \eqref{Theorem_2_delta_choice}.
By H\"{o}lder's inequality, we get that the expression on the right-hand side of \eqref{Theorem_2_8} is in absolute value 
\begin{equation}
\label{Theorem_2_10}
\leq \|G_k \Phi_{\BB,\BB'}^{(\delta)}(\mathcal{M})\|_{L^2(\mu_0)}\,\|\partial_j^{(1)}F\|_{L^{\infty}(\mu_0)}\,\|\ee^{h^I} (\chi_{\BB}^{(\delta)})'(\mathcal{M})\|_{L^4(\mu_0)}\,\|\langle u, e_j^{(2)}\rangle_{\mathcal{H},\R}\|_{L^4(\mu_0)}\,.
\end{equation}
By \eqref{Theorem_2_R'_choice}, \eqref{G_k_convergence_2}, and \eqref{Theorem_2_9}, we have
\begin{equation}
\label{Theorem_2_11}
\lim_k \|G_k \Phi_{\BB,\BB'}^{(\delta)}(\mathcal{M})\|_{L^2(\mu_0)}=0\,.
\end{equation}
More precisely, for \eqref{Theorem_2_11}, we used 
\begin{equation*}
\|G_k \Phi_{\BB,\BB'}^{(\delta)}(\mathcal{M})\|_{L^2(\mu_0)} =\|(G_k-G) \Phi_{\BB,\BB'}^{(\delta)}(\mathcal{M})\|_{L^2(\mu_0)} \lesssim \|G_k-G\|_{L^2(\mu_0)}\,,
\end{equation*}
which holds by \eqref{Theorem_2_R'_choice} and \eqref{Theorem_2_9}.

On the other hand, we also know that 
\begin{equation}
\label{Theorem_2_12}
\|\partial_j^{(1)}F\|_{L^{\infty}(\mu_0)}\,\|\ee^{h^I} (\chi_{\BB}^{(\delta)})'(\mathcal{M})\|_{L^4(\mu_0)}\,\|\langle u, e_j^{(2)}\rangle_{\mathcal{H},\R}\|_{L^4(\mu_0)}<\infty\,.
\end{equation}
Namely, the first factor in \eqref{Theorem_2_12} is finite since $F \in \mathscr{C}^{\infty}_{\mathrm{c},\mathrm{cyl}}(H^{-s})$. The second one is finite by \eqref{integrability_of_weight'}. For the finiteness of the third factor, we use duality to note that 
\begin{equation*}
|\langle u, e_j^{(2)}\rangle_{\mathcal{H},\R}| \lesssim_j \|u\|_{H^{-s}}\,,
\end{equation*}
followed by Lemma \ref{Lemma_2.16_5}. From \eqref{Theorem_2_8} and \eqref{Theorem_2_10}--\eqref{Theorem_2_12}, we hence deduce that 
\begin{equation}
\label{Theorem_2_convergence_1}
\lim_k \int_{H^{-s}} G_k\,\partial_j^{(1)}F\,\ee^{h^I}\bigl(\chi_{\BB}^{(\delta)}\bigr)'(\mathcal{M})\,\langle u, e_j^{(2)}\rangle_{\mathcal{H},\R}\,d \mu_0=0\,.
\end{equation}
By analogous arguments, we obtain 
\begin{equation}
\label{Theorem_2_convergence_2}
\lim_k \int_{H^{-s}} G_k\,\partial_j^{(2)}F\,\ee^{h^I}\bigl(\chi_{\BB}^{(\delta)}\bigr)'(\mathcal{M})\,\langle u, e_j^{(1)}\rangle_{\mathcal{H},\R}\,d \mu_0=0\,.
\end{equation}
We deduce \eqref{Theorem_2_7} from \eqref{Theorem_2_convergence_1}--\eqref{Theorem_2_convergence_2}. The claim of Step 1 now follows.

\paragraph{\textbf{Step 2: Proof of \eqref{Theorem_2_delta_approximation} for $\delta=1$}}

By recalling \eqref{z_R_definition}, and applying Definition \ref{cut-off_chi}, \eqref{integrability_of_weight}
and the dominated convergence theorem, we obtain that 
\begin{equation}
\label{partition_function_convergence}
\lim_{\delta \rightarrow 1} z^{(\delta)} = z^{(1)}\,.
\end{equation}
By \eqref{local_Gibbs_measure_rigorous}, \eqref{partition_function_convergence}, and the result of Step 1, the claim follows if we show that 
\begin{equation}
\label{Step_2_A}
\lim_{\delta \rightarrow 0} \int_{H^{-s}} \{F,G\}\,\ee^{h^I} \chi_{\BB}^{(\delta)}(\mathcal{M})\,d \mu_0=
\int_{H^{-s}} \{F,G\}\,\ee^{h^I} \chi_{\BB}^{(1)}(\mathcal{M})\,d \mu_0
\end{equation}
and 
\begin{multline}
\label{Step_2_B}
\lim_{\delta \rightarrow 0} \int_{H^{-s}} \langle \nabla F,-\ii A u +\ii \nabla h^I \rangle_{\mathcal{H},\R}\, G\,\ee^{h^I} \chi_{\BB}^{(\delta)}(\mathcal{M})\,d \mu_0
\\
=\int_{H^{-s}} \langle \nabla F,-\ii A u +\ii \nabla h^I \rangle_{\mathcal{H},\R}\, G\,\ee^{h^I} \chi_{\BB}^{(1)}(\mathcal{M})\,d \mu_0\,.
\end{multline}
In order to show \eqref{Step_2_A}, we argue as in \eqref{well-definedness_remark_2_b} and use \eqref{dominated_convergence_theorem_remark_1} to observe that
\begin{equation*}
\biggl|\int_{H^{-s}} \{F,G\}\,\ee^{h^I} \bigl(\chi_{\BB}^{(\delta)}(\mathcal{M})-\chi_{\BB}^{(1)}(\mathcal{M})\bigr)\,d \mu_0\biggr|
\\
\lesssim_{F,s} \|\nabla G(u)\|_{L^2(\mu_0;H^{-s})}\,\bigl\|\ee^{h^I}\bigl(\chi_{\BB}^{(\delta)}(\mathcal{M})-\chi_{\BB}^{(1)}(\mathcal{M})\bigr)\bigr\|_{L^2(\mu_0)}\,\,\,  \mbox{as}\,\,\, \delta \rightarrow 1\,.
\end{equation*}
Similarly, in order to show \eqref{Step_2_B}, we argue as for \eqref{homogeneous_KMS_I_2**} and use \eqref{dominated_convergence_theorem_remark_1} to observe that as $\delta \rightarrow 1$
\begin{multline}
\label{Step_2_B_1}
\biggl|\int_{H^{-s}} \bigl|\langle \nabla F(u),-\ii A u +\ii \nabla h^I(u) \rangle_{\mathcal{H},\R}\, G(u)\, \ee^{h^I} \bigl(\chi_{\BB}^{(\delta)}(\mathcal{M})-\chi_{\BB}^{(1)}(\mathcal{M})\bigr) d \mu_0\biggr|
\\
\lesssim_{F,s} \bigl(1+\|\nabla h^I\|_{L^q(\mu_0;H^{-s})}\bigr)\,\|G\|_{L^2(\mu_0)}\,
\bigl\|\ee^{h^I} \bigl(\chi_{\BB}^{(\delta)}(\mathcal{M})-\chi_{\BB}^{(1)}(\mathcal{M})\bigr)\bigr\|_{L^p(\mu_0)}\rightarrow 0\,.
\end{multline}
In \eqref{Step_2_B_1}, we recall \eqref{p,q_choice}.
We hence obtain \eqref{Step_2_A}--\eqref{Step_2_B}, and the claim follows.
\end{proof}

\begin{remark}
We are crucially using the assumption that $G \in \mathbb{D}^{1,2}_{\BB}(\mu_0)$ in order to deduce \eqref{Theorem_2_7}.
\end{remark}
\begin{remark}
\label{Theorem_2_assumption_remark}
Let us comment on the difference between the assumptions on the space in which the localized function lies in Definition \ref{local_KMS_definition} and in \cite[Section 5.3]{AS21}. We note that in the latter, one considers $d=1$ and the localized function (which in \cite{AS21} is $F$) is taken to lie in the space $\mathscr{C}^2_{\mathrm{b}}(\mathcal{H})$, which is the space of twice continuously (Fr\'{e}chet) differentiable real-valued functions with bounded second derivative. 

In Definition \ref{local_KMS_definition} above, the localized function is $G$, which is assumed to be in $\mathbb{D}^{1,2}(\mu_0)$. 
When $d=1$, we show that
\begin{equation}
\label{Theorem_2_assumption_remark_1}
\mathscr{C}^1_{\mathrm{b}}(\mathcal{H}) \subset \mathbb{D}^{1,2}(\mu_0)\,.
\end{equation}
Let us note that, when $d=1$, we can take $s=0$ in Asssumption \ref{choice_of_s_assumption}, in which case
 $H^{-s}=\mathcal{H}$.
Hence, by \eqref{Theorem_2_assumption_remark_1}, we can consider $G \in \mathscr{C}^1_{\mathrm{b}}(\mathcal{H})$ satisfying the localization in $\mathcal{M}$ as for \eqref{D^{1,2}_R} in Definition \ref{local_KMS_definition}, and subsequently in Theorem \ref{Gibbs_implies_local_KMS_theorem}. 

Let us now prove \eqref{Theorem_2_assumption_remark_1}. Consider $G \in \mathscr{C}^1_{\mathrm{b}}(\mathcal{H})$. By assumption, this means that $G: \mathcal{H} \rightarrow \R$ is continuous and bounded, and for all $u \in \mathcal{H}$, there exists $DG(u) \in \mathcal{H}$ such that 
\begin{equation*}
\lim_{v \rightarrow 0} \frac{|G(u+v)-G(u)-\langle DG(u),v\rangle_{\mathcal{H},\R}|}{\|v\|_{\mathcal{H}}}=0\,,
\end{equation*}
where the map $u \in \mathcal{H} \mapsto DG(u) \in \mathcal{H}$ is continuous and bounded, in the sense that 
\begin{equation}
\label{boundedness_of_G}
\sup_{u \in \mathcal{H}} \|DG(u)\|<\infty\,.
\end{equation}
By continuity, we obtain that $G$ is $\mu_0$-measurable.

Since $(e_j^{(i)})$ is an orthonormal basis of $(\mathcal{H},\langle \cdot, \cdot \rangle_{\mathcal{H},\R})$ and $DG(u) \in \mathcal{H}$, we obtain that
\begin{equation}
\label{Theorem_2_assumption_remark_3}
DG(u)=\sum_{j=1}^{\infty} \sum_{i=1}^{2} \langle DG(u), e_j^{(i)}\rangle_{\mathcal{H},\R}\,e_j^{(i)}\,,
\end{equation}
where the sum converges in $\mathcal{H}$.
Recalling \eqref{projection_Pi_n}, we define for $n \in \N$
\begin{equation}
\label{Theorem_2_assumption_remark_4}
G_n:=G \circ \Pi_n\,.
\end{equation}
Each $G_n$ is continuous, hence $\mu_0$-measurable.
From \eqref{Theorem_2_assumption_remark_4}, \eqref{projection_Pi_n}, as well as the continuity and boundedness of $G$, we deduce that 
\begin{equation}
\label{Theorem_2_assumption_remark_5}
\lim_n |G_n(u)-G(u)|=0\,\quad \forall u \in \mathcal{H}\,,\qquad \sup_n \sup_{u \in \mathcal{H}} |G_n(u)| <\infty\,.
\end{equation}
By using \eqref{Theorem_2_assumption_remark_5} and the dominated convergence theorem, we have that
\begin{equation}
\label{G_n_convergence_remark}
G_n \rightarrow G \quad\mbox{in}\quad L^2(\mu_0)\,.
\end{equation}
From \eqref{Theorem_2_assumption_remark_4}, we have
\begin{equation}
\label{Theorem_2_assumption_remark_6}
DG_n(u)=\sum_{j=1}^{n} \sum_{i=1}^{2} \langle DG(\Pi_n(u)), e_j^{(i)} \rangle_{\mathcal{H},\R}\,e_j^{(i)}\,.
\end{equation}
From \eqref{Theorem_2_assumption_remark_3}, \eqref{Theorem_2_assumption_remark_6}, and Plancherel's theorem, we obtain that for all $u \in \mathcal{H}$
\begin{multline}
\label{Theorem_2_assumption_remark_7}
\|DG_n(u)-DG(u)\|_{\mathcal{H}}^2=\sum_{j=1}^{n} \sum_{i=1}^{2} \bigl|\langle DG(\mathrm{\Pi}_n (u))-DG(u),e_j^{(i)}\rangle_{\mathcal{H},\R}\bigr|^2
\\
+\sum_{j>n}\sum_{i=1}^{2}\bigl|\langle DG(u),e_j^{(i)}\rangle_{\mathcal{H},\R}\bigr|^2\,.
\end{multline}
Let us note that the right-hand side of \eqref{Theorem_2_assumption_remark_7} tends to zero as $n \rightarrow \infty$.
Namely, by Plancherel's theorem, \eqref{projection_Pi_n}, and the continuity of $DG$, we get that the first term is
\begin{equation}
\label{Theorem_2_assumption_remark_7_1}
\leq \|DG(\mathrm{\Pi}_n(u))-DG(u)\|_{\mathcal{H}}^2 \rightarrow 0 \quad \mbox{as} \quad n \rightarrow \infty\,.
\end{equation}
For the second term, we use $DG(u) \in \mathcal{H}$ and Plancherel's theorem to deduce that
\begin{equation}
\label{Theorem_2_assumption_remark_7_2}
\lim_n\, \sum_{j>n}\sum_{i=1}^{2}\bigl|\langle DG(u),e_j^{(i)}\rangle_{\mathcal{H},\R}\bigr|^2=0\,.
\end{equation}
From \eqref{Theorem_2_assumption_remark_7}--\eqref{Theorem_2_assumption_remark_7_2}, we deduce that for all $u \in \mathcal{H}$
\begin{equation}
\label{Theorem_2_assumption_remark_8}
\lim_n \|DG_n(u)-DG(u)\|_{\mathcal{H}}=0\,.
\end{equation}
From \eqref{Theorem_2_assumption_remark_6} and Plancherel's theorem, we deduce that for all $n \in \N$ and $u \in \mathcal{H}$, we have
\begin{equation*}
\|DG_n(u)\|_{\mathcal{H}} \leq \|DG(\mathrm{\Pi}_n(u))\|_{\mathcal{H}}\,,
\end{equation*}
which by \eqref{projection_Pi_n} and \eqref{boundedness_of_G} implies that 
\begin{equation}
\label{Theorem_2_assumption_remark_9}
\sup_{u \in \mathcal{H}} \sup_n \|DG_n(u)\|_{\mathcal{H}}<\infty\,.
\end{equation}
By \eqref{G_n_convergence_remark}, \eqref{Theorem_2_assumption_remark_8}, \eqref{Theorem_2_assumption_remark_9}, and the dominated convergence theorem, we can deduce that $G \in \mathbb{D}^{1,2}(\mu_0)$ and hence that \eqref{Theorem_2_assumption_remark_1} holds, provided that we show
\begin{equation}
\label{Theorem_2_assumption_remark_10}
G_n \in \mathbb{D}^{1,2}(\mu_0)\,,\qquad \nabla G_n=DG_n\,,
\end{equation}
for all $n \in \N$ (here $\nabla$ denotes the Malliavin derivative).
Furthermore, from  \eqref{G_n_convergence_remark}, \eqref{Theorem_2_assumption_remark_8} and \eqref{Theorem_2_assumption_remark_10}, we obtain that $DG=\nabla G$.

We now show \eqref{Theorem_2_assumption_remark_10}. Let us fix $n \in \N$. By \eqref{Theorem_2_assumption_remark_4}, \eqref{projection_Pi_n}, and \eqref{projection_map} we can write 
\begin{equation}
\label{Theorem_2_assumption_remark_11}
G_n=\psi \circ \pi_n\,,
\end{equation}
where $\psi:\R^n \rightarrow \R$ is defined by
\begin{equation}
\label{Theorem_2_assumption_remark_12}
\psi(x_1^{(1)}, \ldots, x_n^{(1)};x_1^{(2)},\ldots,x_n^{(2)})=G\Biggl(\sum_{j=1}^{n} \sum_{i=1}^{2} x_i^{(j)} e_i^{(j)}\Biggr)\,.
\end{equation}
Furthermore, since $G \in \mathscr{C}^1_{\mathrm{b}}(\mathcal{H})$, we obtain that $\psi \in \mathscr{C}^1_{\mathrm{b}}(\R^{2n})$. We can then find a sequence $(\psi_k)_k$ in $\mathscr{C}^{\infty}_{\mathrm{c}}(\R^{2n})$ such that 
\begin{equation}
\label{sequence_psi}
\supp\, \psi_k \subset \supp\, \psi\,,\qquad \psi_k \rightarrow \psi\quad\mbox{in}\quad \mathscr{C}^1(\R^{2n})\,.
\end{equation}
For $k \in \N$, we define 
\begin{equation}
\label{sequence_psi_2}
G_{n,k}:=\psi_k \circ \pi_n \in \mathscr{C}^{\infty}_{\mathrm{c},\mathrm{cyl}}(\mathcal{H})\,.
\end{equation}
By \eqref{gradient_cylindrical_F}, \eqref{Theorem_2_assumption_remark_11}, \eqref{sequence_psi}--\eqref{sequence_psi_2}, and the dominated convergence theorem, we deduce that as $k \rightarrow \infty$
\begin{equation}
\label{G_{n,k}_convergence}
G_{n,k} \rightarrow G_n \quad\mbox{in}\quad \mathbb{D}^{1,2}(\mu_0)\,.
\end{equation}
The first claim in \eqref{Theorem_2_assumption_remark_10} follows from \eqref{G_{n,k}_convergence}. Furthermore, from \eqref{sequence_psi_2} followed by \eqref{gradient_cylindrical_F}, we have that
\begin{equation*}
\nabla G_{n,k}=\sum_{j=1}^{n} \sum_{i=1}^{2} \partial_j^{(i)} \psi_k(\pi_n(u)) e_j^{(i)}\,,
\end{equation*}
which by \eqref{sequence_psi}, the dominated convergence theorem, and Plancherel's theorem converges in $L^2(\mu_0; \mathcal{H})$ to 
\begin{equation}
\label{sequence_psi_3}
\sum_{j=1}^{n} \sum_{i=1}^{2} \partial_j^{(i)} \psi(\pi_n(u)) e_j^{(i)}=DG_n(u)\,.
\end{equation}
For the last equality in \eqref{sequence_psi_3}, we used \eqref{Theorem_2_assumption_remark_11}. By \eqref{G_{n,k}_convergence}, the quantity \eqref{sequence_psi_3} is also equal to $\nabla G_n(u)$. We hence obtain the second claim in \eqref{Theorem_2_assumption_remark_10}.
\end{remark}

To conclude this section, we remark how our analysis extends to the setting of local Gibbs measures studied in the recent paper \cite{RS23}.
\begin{remark}
\label{nonlocal_quintic_NLS_remark}
In \cite{RS23}, the second and third author consider local Gibbs measures for the following nonlocal (focusing) quintic NLS on $\T$:
\begin{align}
\notag
\mathrm{i}\partial_t u + (\Delta - 1)u =&-\frac{1}{3}\,\int_{\T} \int_{\T} V(x-y)\,V(x-z)\,|u(y)|^2\,|u(z)|^2\,u(x)\,dy\,dz
\\
\label{nonlocal_quintic_NLS}
&-\frac{2}{3}\,\int_{\T} \int_{\T} \, V(x-y)\,V(y-z)\,|u(y)|^2\,|u(z)|^2\,u(x)\,dy\,dz\,.
\end{align}
Here, $V \in L^1(\T)$ is assumed to be even and real-valued \cite[Assumption 1.2 (i)]{RS23}. If we formally set $V$ to be the Dirac delta function, \eqref{nonlocal_quintic_NLS} would correspond to the focusing quintic NLS. The main aim of \cite{RS23} is to obtain local Gibbs measures arising from the Hamiltonian \eqref{nonlocal_quintic_NLS} as suitable mean-field limits of (truncated) many-body quantum Gibbs states; see \cite[Theorem 1.9]{RS23} for a precise statement. We now study these local Gibbs measures in the context of local KMS states and see that the above analysis applies. In this remark, we just outline the main ideas. For further details, we refer the reader to \cite{RS23}.

By arguing as for \cite[Lemma 1.3]{RS23}, one can show that  \eqref{nonlocal_quintic_NLS} is a Hamiltonian system with Hamiltonian
\begin{multline}
\label{nonlocal_quintic_NLS_Hamiltonian}
h(u)=\frac{1}{2}\int_{\T} \bigl(|\nabla u(x)|^2 + |u(x)|^2\bigr)\,dx  \\ 
- \frac{1}{6}\int_{\T} \int_{\T} \int_{\T} V(x-y)\,V(x-z)\,|u(x)|^2\,|u(y)|^2\,|u(z)|^2\,dx\,dy\,dz\,.
\end{multline}
Instead of taking $h^I$ as in Assumption \ref{Assumption_on_V}, by \eqref{nonlocal_quintic_NLS_Hamiltonian}, we take
\begin{equation}
\label{nonlocal_quintic_NLS_interaction}
h^I(u):= \frac{1}{6}\int_{\T} \int_{\T} \int_{\T} V(x-y)\,V(x-z)\,|u(x)|^2\,|u(y)|^2\,|u(z)|^2\,dx\,dy\,dz\,.
\end{equation}
With $h^I$ as in \eqref{nonlocal_quintic_NLS_interaction}, it is shown in \cite[Proposition 2.2 (i)]{RS23} that the local Gibbs measure \eqref{local_Gibbs_measure_rigorous} is well-defined provided that we take $\BB>0$ sufficiently small. The proof \cite[Proposition 2.2 (i)]{RS23} also shows that the analogue of Proposition \ref{Hartree_equation_Malliavin_derivative} (ii) and \eqref{integrability_of_weight'}
 hold for $h^I$ given by \eqref{nonlocal_quintic_NLS_interaction}.
 
By a suitable modification of the arguments from the proof of Proposition \ref{Hartree_equation_Malliavin_derivative} (i) given in Section \ref{Proof_of_Proposition_2.13_(i)} below, one can show that
\begin{multline}
\label{nonlocal_quintic_NLS_nabla_h^I}
\nabla h^I(u)=\frac{1}{3}\,\int_{\T} \int_{\T} V(x-y)\,V(x-z)\,|u(y)|^2\,|u(z)|^2\,u(x)\,dy\,dz
\\
+\frac{2}{3}\,\int_{\T} \int_{\T} \, V(x-y)\,V(y-z)\,|u(y)|^2\,|u(z)|^2\,u(x)\,dy\,dz\,.
\end{multline}
We note that 
\begin{equation}
\label{nonlocal_quintic_NLS_Gross-Sobolev}
h^I \in \mathbb{D}^{1,p}(\mu_0)\,\,\mbox{for all}\,\, p \in [1,\infty)\,,
\end{equation}
thus showing that \eqref{nonlocal_quintic_NLS_interaction} satisfies the analogue of Proposition \ref{Hartree_equation_Malliavin_derivative} (i) (and also rigorously justifying the calculations used to obtain \eqref{nonlocal_quintic_NLS_nabla_h^I}).

We now show \eqref{nonlocal_quintic_NLS_Gross-Sobolev}. Let $p \in [1,\infty)$ be given.
We first show that 
\begin{equation}
\label{nonlocal_quintic_NLS_Gross-Sobolev_1}
h^I \in L^p(\mu_0)\,.
\end{equation}
Let us rewrite \eqref{nonlocal_quintic_NLS_interaction} as 
\begin{equation*}
h^I(u)= \frac{1}{6}\int_{\T} (V*|u|^2)^2(x)\,|u(x)|^2 dx\,,
\end{equation*}
and use H\"{o}lder's inequality followed by Young's inequality to obtain that
\begin{equation} 
\label{nonlocal_quintic_NLS_Gross-Sobolev_1_a}
|h^I(u)| \leq \frac{1}{6}\, \|V\|_{L^1}^2\,\|u\|_{L^6}^6\,.
\end{equation}
We obtain \eqref{nonlocal_quintic_NLS_Gross-Sobolev_1} from \eqref{nonlocal_quintic_NLS_Gross-Sobolev_1_a}
and from the analogue of \eqref{nonlocal_quintic_NLS_Gross-Sobolev_1} for the local quintic NLS (given by  Proposition \ref{Hartree_equation_Malliavin_derivative} (i)).

We now show that 
\begin{equation}
\label{nonlocal_quintic_NLS_Gross-Sobolev_2}
\nabla h^I \in L^p(\mu_0;H^{-s})\,.
\end{equation}
Let us write \eqref{nonlocal_quintic_NLS_nabla_h^I} as
\begin{equation}
\label{nonlocal_quintic_NLS_nabla_h^I_1}
\nabla h^I(u)=\mathcal{N}_1(u)+\mathcal{N}_2(u)\,,
\end{equation}
where
\begin{equation}
\label{nonlocal_quintic_NLS_nabla_h^I_2}
\mathcal{N}_1(u):=\frac{1}{3}\,(V*|u|^2)^2\,u\,,\qquad \mathcal{N}_2(u):=\frac{2}{3}\,\Bigl(V* \Bigl\{|u|^2 \bigl[V*|u|^2\bigr] \Bigr\}\Bigr)u\,.
\end{equation}
Let us now choose the function $w$ as
\begin{equation}
\label{nonlocal_quintic_NLS_nabla_h^I_3}
w=\mathcal{F}^{-1}(|\widehat{u}|)\,,
\end{equation}
where $\mathcal{F}$ here denotes the Fourier transform \eqref{Fourier_transform}.
By \eqref{nonlocal_quintic_NLS_nabla_h^I_2}--\eqref{nonlocal_quintic_NLS_nabla_h^I_3} and a direct calculation, it follows that 
\begin{equation}
\label{nonlocal_quintic_NLS_nabla_h^I_4}
\bigl|\widehat{\mathcal{N}_1(u)}\bigr|+\bigl|\widehat{\mathcal{N}_2(u)}\bigr| \leq \|V\|_{L^1}^2 \,\widehat{|w|^4 w}\,.
\end{equation}
(For more details on \eqref{nonlocal_quintic_NLS_nabla_h^I_4}, we refer the reader to the proof of \cite[Lemma 2.3]{RS23}).
From \eqref{nonlocal_quintic_NLS_nabla_h^I_1}, \eqref{nonlocal_quintic_NLS_nabla_h^I_3}--\eqref{nonlocal_quintic_NLS_nabla_h^I_4}, Plancherel's theorem, and arguments analogous to the proof of \cite[Proposition 5.1 (iii)]{AS21}, we hence deduce that for some $\zeta \in (-s,\frac{1}{2})$
\begin{multline}
\label{nonlocal_quintic_NLS_nabla_h^I_5}
\|\nabla h^I\|_{L^p(\mu_0;H^{-s})} \leq \|V\|_{L^1}^2 \||w|^4w\|_{L^p(\mu_0;H^{-s})}
\\
\lesssim_{\zeta} \|V\|_{L^1}^2 \, \|\|w\|_{H^{\zeta}}^5\|_{L^p(\mu_0)}= \|V\|_{L^1}^2 \,\|\|u\|_{H^{\zeta}}^5\|_{L^p(\mu_0)}
<\infty\,.
\end{multline}
More precisely, in the second inequality of \eqref{nonlocal_quintic_NLS_nabla_h^I_5}, we used the product estimate given in \cite[Lemma B.1]{AS21}. We omit the details. In particular, we have that \eqref{nonlocal_quintic_NLS_Gross-Sobolev_2} follows from the analogous statement when $V=\delta$. We hence deduce \eqref{nonlocal_quintic_NLS_Gross-Sobolev} from \eqref{nonlocal_quintic_NLS_Gross-Sobolev_1} and \eqref{nonlocal_quintic_NLS_Gross-Sobolev_2}. In summary, the results from Proposition \ref{Hartree_equation_Malliavin_derivative} hold when $h^I$ is given by \eqref{nonlocal_quintic_NLS_interaction}.

By recalling \eqref{projection_Pi_n}, using \eqref{nonlocal_quintic_NLS_nabla_h^I_1}--\eqref{nonlocal_quintic_NLS_nabla_h^I_2}, and arguing as for \eqref{homogeneous_KMS_sum_4}, we get that for $h^I$ as in \eqref{nonlocal_quintic_NLS_interaction},

\begin{multline}
\label{homogeneous_KMS_sum_4_nonlocal_quintic}
\bigl \langle \mathrm{\Pi}_n \nabla h^I,-\ii \mathrm{\Pi}_n \nabla \chi_{\BB}^{(\delta)}(\mathcal{M})\bigr \rangle_{\mathcal{H},\R}
=\frac{2}{3}\,\bigl\langle (V*\,|u|^2)^2\mathrm{\Pi}_n u,-\ii  \bigl(\chi_{\BB}^{(\delta)}\bigr)'(\mathcal{M})  \,\mathrm{\Pi}_n u\bigr \rangle_{\mathcal{H},\R}
\\
+\frac{4}{3}\,\Big\langle V*\Bigl\{|u|^2\bigl[V*\,|u|^2\bigr]\Bigr\}\mathrm{\Pi}_n u,-\ii  \bigl(\chi_{\BB}^{(\delta)}\bigr)'\mathcal{M})  \,\mathrm{\Pi}_n u\Big\rangle_{\mathcal{H},\R}
=0\,.
\end{multline}
In \eqref{homogeneous_KMS_sum_4_nonlocal_quintic}, we used the assumption that $V$ was real-valued, as well as the fact that $(\chi_{\BB}^{(\delta)})'$ is real-valued, which follows from from Definition \ref{cut-off_chi}. By replacing \eqref{homogeneous_KMS_sum_4} with \eqref{homogeneous_KMS_sum_4_nonlocal_quintic}, the rest of the proof of Theorem \ref{homogeneous_KMS_thm} shows that the local Gibbs measure \eqref{local_Gibbs_measure_rigorous} associated with interaction \eqref{nonlocal_quintic_NLS_interaction} (and sufficiently small $\BB$) is a stationary solution to the Liouville equation in the sense of \eqref{homogeneous_KMS} for $\delta \in (0,1]$. Likewise, the proof of Theorem \ref{Gibbs_implies_local_KMS_theorem} shows that this local Gibbs measure $\mu^{(1)}$ is a local KMS state.
\end{remark}

\section{Local KMS states are local Gibbs measures}
\label{Local KMS states are local Gibbs measures}

In this section, we characterize all possible local KMS equilibrium  states of focusing NLS equations on the torus.  In particular, we prove Theorem \ref{kms_implies_gibbs_theorem}.  For reader's convenience, we recall the following notations for $R>0$,
\begin{equation}
\mathbb{B}_{R}=\{u\in H^{-s}: |\mathcal{M}(u)|<R\}, \qquad \mathbb{B}_{R}^c=H^{-s}\setminus \mathbb{B}_{R},
\end{equation}
and  
\begin{equation}
\label{}
\Drspace:=\{ G \in \Dspace:  \exists R'\in(0,R) \text{ s.t }  
G(u) = 0 \text{ for } \text{$\mu_0$-a.s. } u \in \mathbb{B}_{R'}^c \}.
\end{equation}

\begin{proposition}
\label{local_KMS_differential_equation}
Let $\mu$ be a Borel probability measure in $\mathscr{P}(H^{-s})$. Suppose that $d \mu = \rho \, d\mu_{0}$ with $\rho \in \mathbb{D}^{1,2}(\mu_{0})\cap L^4(\mu_0)$ and that $\mu$ satisfies a local KMS condition, i.e., there exists $R>0$ such that
\begin{equation*}
\label{local_KMS_condition}
    \int_{H^{-s}} \{F,G\} \, d\mu = \int_{H^{-s}} {\rm Re}\langle \nabla F(u), X(u) \rangle \,G(u) \, d\mu\,,
\end{equation*}
for any  $G \in \Drspace$ and $F \in \smoothcyl$. 
Then $\rho$ satisfies the following differential equation 
\begin{equation}
\label{Malliavin_diff_equ}
\nabla \rho(u) - \rho(u) \nabla h^I(u) = 0
\end{equation}
in the Malliavin sense for $\mu_0$-almost all $u\in \mathbb{B}_{R}$.
\end{proposition}

\begin{proof}
The proof is inspired by \cite[Proposition 4.12]{AS21} and relies on the fact that the Gaussian measure $\mu_0$ satisfies the global KMS condition in Proposition \ref{local_KMS_approximation_lemma}. Let $G\in \Drspace$ be such that for some $R'\in(0,R)$, we have  $G(u) = 0$ for $\mu_0$-almost all $u \in \mathbb{B}_{R'}^c$.  Since $\rho \in \mathbb{D}^{1,2}(\mu_{0})$, we can find $\rho_n \to \rho$ in $\mathbb{D}^{1,2}(\mu_{0})$ with $\rho_n \in \smoothcyl$. Recall here that the Poisson bracket $\{F,G\}$ makes sense and belongs to $L^2(\mu_0)$. Hence, using $\rho_n \to \rho$, Cauchy-Schwarz, and that $d \mu = \rho \, d\mu_{0}$, we have 
\begin{align*}
\lim_n \int_{H^{-s}} \{F,G\} \rho_n \, d\mu_{0} = \int_{H^{-s}} \{F,G\} \, d\mu . 
\end{align*}
Similarly, we have
\begin{align*}
\lim_n \int_{H^{-s}} \{F,\rho_n\} G(u) \, d\mu_{0} = \int_{H^{-s}} {\rm Re} \langle \nabla F, -i \nabla \rho \rangle G(u) \, d\mu_{0}. 
\end{align*}
Then, using  the Leibniz rule (Lemma \ref{Leibniz_rule}),  we obtain
\begin{align*}
\lim_n \int_{H^{-s}} \{F,G\rho_n\} \, d\mu_{0} &= \lim_n \int_{H^{-s}} \{F,G\} \rho_n \, d\mu_{0} + \lim_n \int_{H^{-s}} \{F,\rho_n\} G(u) \, d\mu_{0} \\
&= \int_{H^{-s}} \{F,G\} \, d\mu + \int_{H^{-s}} {\rm Re} \langle \nabla F, -i \nabla \rho \rangle G(u) \, d\mu_{0}. 
\end{align*}
On the other hand, recall that $\rho_n \in \smoothcyl$ which yields $G\rho_n \in \Dspace$. Therefore, we can apply   Proposition  \ref{local_KMS_approximation_lemma} and write
\begin{align*}
\lim_n \int_{H^{-s}} \{F,G\rho_n\} \, d \mu_{0} &=   \lim_n \int_{H^{-s}} {\rm Re} \langle \nabla F, X_0 \rangle G(u) \rho_n \, d\mu_{ 0} \\
&=  \int_{H^{-s}} {\rm Re} \langle \nabla F, X_0 \rangle G(u) \, d\mu.
\end{align*}
Combining the above equalities gives
\begin{align*}
\int_{H^{-s}} {\rm Re} \langle \nabla F, X_0 \rangle G(u) \, d\mu = \int_{H^{-s}} \{F,G\} \, d\mu + \int_{H^{-s}} {\rm Re} \langle \nabla F, -i \nabla \rho \rangle G(u) \, d\mu_{0}.
\end{align*}
Recalling that $\mu$ satisfies the local KMS condition for $X = X_0 + X^I$ (with $X^I=i \nabla h^I$) and using that ${\rm Re} \langle f, g \rangle \rho = {\rm Re} \langle f, g \rho \rangle$ since $\rho$ is real valued, we have for any $G\in \Drspace$
\begin{equation*}
\int_{H^{-s}} {\rm Re} \langle \nabla F,  \rho X^I-i \nabla \rho \rangle G(u) \, d\mu_{0}=0.
\end{equation*}
We choose  $G=\chi_{R'}^{(\delta)}(\mathcal{M})  \tilde G$ with $\tilde G \in \smoothcyl$ and $\chi_{R'}^{(\delta)}$ is the cutoff function in Definition \ref{cut-off_chi} such that $\delta\in(0,1)$ and $R'\in(0,R)$. Using the support properties of $\chi_{R'}^{(\delta)}$ with   Lemmas \ref{Leibniz_rule} and \ref{Malliavin_derivative_M_lemma*},  we show that any $G=\chi_{R'}^{(\delta)}(\mathcal{M}) \tilde G$ belongs to the space $\Drspace$.  Therefore, we have that for any $\tilde G\in\smoothcyl$ 
\begin{equation}
\label{sec.kmstogibbs.eq.1}
\int_{H^{-s}} {\rm Im} \langle \nabla F,  \rho \nabla h^I-\nabla \rho \rangle \chi_{R'}^{(\delta)}(\mathcal{M}) \, \tilde G(u) \, d\mu_{0}=0. 
\end{equation}
Recalling that we have, by Proposition \ref{Hartree_equation_Malliavin_derivative}, 
$h^I\in \mathbb{D}^{1,4}(\free)$, and that $\rho\in \Dspace \cap L^4(\free)$ by assumption. Hence, we have that ${\rm Re} \langle \nabla F,  \rho \nabla h^I-\nabla \rho \rangle \in L^2(\mu_0)$. 
Since the space $\smoothcyl$ is dense in $L^2(\free)$, by orthogonality we deduce from \eqref{sec.kmstogibbs.eq.1}  that
\begin{equation}
\label{sec.kmstogibbs.eq.2}
{\rm Im} \langle \nabla F,  \rho \nabla h^I-\nabla \rho \rangle \, \chi_{R'}^{(\delta)}(\mathcal{M})=0 \quad \text{ in } L^2(\free),
\end{equation}
for any $F\in\smoothcyl$. Taking, for $i \in \{1,2\}$, the function 
$$
F(\cdot)\equiv \psi_k(\langle \cdot, e_j^{(i)} \rangle)=\langle \cdot, e_j^{(i)} \rangle \, \varphi\left( \frac{\langle \cdot, e_j^{(i)} \rangle}{k}\right)
$$
with $\varphi\in C_{\rm c}^\infty(\R)$ such that $0\leq \varphi\leq 1$ and 
 $\varphi(x)=1$ on  $|x|\leq 1$ and $0$ on $|x|\geq 2$. 
This gives in  $L^2(\free)$,
\begin{equation}
\psi_k'(\langle \cdot, e_j^{(i)} \rangle) \; {\rm Im}\langle e_j^{(i)},  \rho \nabla h^I-\nabla \rho \rangle \, \chi_{R'}^{(\delta)}(\mathcal{M})=0,
\end{equation}
Letting  $k\to \infty$, we obtain by dominated convergence
\begin{equation}
\|{\rm Im} \langle e_j^{(i)},  \rho \nabla h^I-\nabla \rho \rangle \, \chi_{R'}^{(\delta)}(\mathcal{M})\|_{L^2(\mu_0)}=0. 
\end{equation}
Here we have used that  $\rho \nabla h^I-\nabla \rho \in L^2(\mu_0; H^{-s})$,  $\chi_{R'}^{(\delta)}(\mathcal{M})\in L^\infty(\free)$ and $\psi_k' \to 1$ pointwise and $\psi_k'$ is bounded uniformly in $k$.  Using the complex structure of the space $\mathcal H$ with the convention \eqref{ONB_convention}, we obtain  for all $j\in\N$
\begin{equation}
\|\langle e_j,  \rho \nabla h^I-\nabla \rho \rangle \, \chi_{R'}^{(\delta)}(\mathcal{M})\|_{L^2(\free)}=0. 
\end{equation}
Thus by monotone convergence, we  have
\begin{align*}
 \int_{H^{-s}} \sum_{j\in \N} \lambda_j^{-s} | \langle e_j,  \rho \nabla h^I-\nabla \rho \rangle  \, \chi_{R'}^{(\delta)}(\mathcal{M})|^2 \, d\mu_0 & = \int_{H^{-s}}\|\rho \nabla h^I-\nabla \rho\|_{H^{-s}}^2 
\chi_{R'}^{(\delta)}(\mathcal{M})^2 \, d\mu_0 \\ &=0. 
\end{align*}
Taking now $R' \to R$ and then $\delta \to 1$ we conclude, by dominated convergence,  that 
\begin{align*}
\lim_{\delta \to 1} \lim_{R'\to R}\int_{H^{-s}}\|\rho \nabla h^I-\nabla \rho\|_{H^{-s}}^2 
\chi_{R'}^{(\delta)}(\mathcal{M})^2 \, d\mu_0 &=\int_{|\mathcal{M}(u)|<R}\|\rho \nabla h^I-\nabla \rho\|_{H^{-s}}^2 \, d\mu_0 \\ &=0. 
\end{align*}
Hence, we have
\begin{equation*}
\nabla \rho(u) - \rho(u) \nabla h^I(u) = 0
\end{equation*}
$\mu_0$-almost  all  $u \in \mathbb{B}_R$.
\end{proof}

\medskip

In the following we appeal to a remarkable result  proved by Shigeki Aida \cite{Ai} (see also Seiichiro Kusuoka's work \cites{Kusuoka91,Kusuoka92}). It is a key argument in our analysis. Before recalling  this, we briefly comment  on the construction of Sobolev classes  on abstract Wiener spaces (or Gaussian probability spaces). There are essentially three standard classes obtained by different ways: 
\begin{itemize}
    \item[(i)] Completion of smooth cylindrical functions or cylindrical polynomials with respect to some $L^p(\mu_0)$ norms involving the gradient ($\mu_0$ is a Gaussian measure as in Proposition \ref{free_Gibbs_prop}). This yields the spaces $W^{p,r}$  sometimes called Gross-Stroock Sobolev spaces and studied for instance by Paul Malliavin in \cite{Mall97} and by Vladimir  Bogachev in \cite{Bog98}.
    \item[(ii)] Using generalized derivatives. This gives the spaces ${D}^{p,r}$ used in Aida's paper \cite{Ai}. 
    \item[(iii)] Using  Ornstein–Uhlenbeck operators and semigroups. This leads to the class ${H}^{p,r}$ studied by Shinzo Watanabe in \cite{Wat84}.  
\end{itemize} 
Unfortunately there is no consensus in the literature neither  on the notations nor on the norms of the above Sobolev classes. Here in this brief review,  we use the conventions from  \cite{Bog98}. 
 Mainly it is known for $p\in (1,\infty)$ and $r\in\mathbb{N}$  that ${H}^{p,r}$,  ${W}^{p,r}$  and ${D}^{p,r}$ are the same spaces equipped with equivalent norms. For the discussion of these issues, we refer the reader to the monographs by Vladimir  Bogachev \cite{Bog98} and  Shinzo Watanabe \cite{Wat84}.   It is worth noticing also that the Gross-Sobolev spaces $\mathbb{D}^{1,p}(\mu_0)$, we are using here, fall into the category (i) but equipped with slightly different norms compared to \cite{Bog98}.  In particular, with the exponent $s$ satisfying \eqref{eq.cond.spower}, $\mathbb{D}^{1,2}(\mu_0)$ is included or equal to the Sobolev space used by Aida in \cite{Ai}. \\

Below we reformulate in our context the result of Aida \cite[Corollary 4]{Ai}. 
\begin{proposition}[S.~Aida]
\label{Aida_0_derivative_result} 
Let $\phi:H^{-s}\to \R$ be a measurable function which is an $H^1$-continuous function, i.e., 
for any $u\in H^{-s}$ the function $\phi(u+\cdot):H^{1} \to \R$  is continuous. Let $\mathcal O$  denote the domain $\{u\in H^{-s}: \phi(u) >0\}$. Suppose that  $\mu_0(\mathcal O)>0$ and 
that $\mathcal O$ is  $H^1$-connected, i.e., for any $u\in \mathcal O$ the subset $\mathcal O(u)=\{w\in H^1: 
u+w\in \mathcal O\}$ is a connected  set of $(H^1, \|\cdot\|_{H^1})$. Assume that 
$G\in \Dspace$ such that $\nabla G=0$ for $\mu_0$-a.e. $u\in \mathcal O$, then $G$ is constant on $\mathcal O$. \end{proposition}

\begin{lemma}\label{lem.borelrepr.mass}
Let $\mathcal{M}(\cdot)$ be the renormalized mass in the case $d=2,3$, given by \eqref{renormalized_mass} and Lemma \ref{Wick_ordered_mass_lemma} and belonging to $L^p(\mu_0)$ for all  $p\in [1,\infty)$. There exists a Borel measurable representative of $\mathcal M$ denoted  $\widetilde{\mathcal{M}}: H^{-s} \to \mathbb R$ and there exists a Borel set  $\Omega_0$ in $H^{-s}$ 
of full measure (i.e. $\mu_0(\Omega_0)=1$) satisfying  $\Omega_0+H^1=\Omega_0$ and such that we have for all $u\in \Omega_0$ and all $w\in H^1$,
\begin{equation}\label{eq.lemborelrep.iden}
    \widetilde{\mathcal{M}}(u+w)= \widetilde{\mathcal{M}}(u)+2\Real\langle u, w\rangle_{L^2}+ \|w\|^2_{L^2} \,.
\end{equation}
Moreover, for an arbitrary $R>0$,  we can set $\widetilde{\mathcal{M}}(u)=R$ for all $u\notin \Omega_0$. 
\end{lemma}
\begin{proof}
Consider the sequence of Borel measurable  functions    
\begin{eqnarray*}
    \mathcal M_n: u \in H^{-s}& \longmapsto & \mathcal M_n(u)= \sum_{|k|\leq n}\left( |\hat u(k)|^2-\frac{1}{\langle k \rangle^2}\right).
\end{eqnarray*}
Since $(\mathcal M_n)_n$ converges in $L^2(\mu_0)$, there exists a subsequence $(\mathcal M_{n_j})_j$ which converges almost surely on $H^{-s}$. We denote 
\[
\begin{array}{rcl}
    u \in H^{-s}& \longmapsto &  {\mathcal M}_\infty(u)= \begin{cases}
        \lim_j \mathcal M_{n_j}(u) & \text{if convergence holds,} \quad \\
        R & \text{if not.}
    \end{cases}
\end{array}
\]
On the other hand, we know that the set of all points of convergence is a Borel set,  
\begin{eqnarray*}
    \Omega &= & \bigl\{u\in H^{-s}: (\mathcal M_{n_j}(u))_j \text{ converges }\bigr\} \\
    &= & \bigcap_{\varepsilon\in\mathbb Q, \varepsilon>0} \bigcup_{N\in\mathbb N} \bigcap_{j,j'>N} \bigl\{u\in H^{-s}: | \mathcal M_{n_j}(u)- \mathcal M_{n_{j'}}(u) |<\varepsilon\bigr\}.
\end{eqnarray*}
Hence, ${\mathcal M}_\infty$ is a Borel measurable representative of $\mathcal M$ and 
$\mu_0(\Omega)=1$.  According to the Cameron-Martin theorem, we know that the  Gaussian measures $\mu_0$ and $(T_w)_{\sharp}\mu_0$ are equivalent for any translation $T_w$ by a vector $w$ in $H^1$ (see e.g. \cite[Corollary 2.4.3]{Bog98}).  
Therefore we have for any $w\in H^1$,
$$
\mu_0(\Omega \cap \Omega-w)=1. 
$$
Since $H^1$ is separable we can find a countable dense set $\mathcal D$ in $H^1$ and 
build the following Borel set of full measure
\begin{equation}
    \Omega_0=\bigcap_{w\in \mathcal D} \Omega \cap (\Omega-w). 
\end{equation}
This definition gives that $u\in \Omega_0$ if and only if $u\in\Omega$ and $u+w\in\Omega$  for all $w\in \mathcal D$. Moreover, for any $u\in\Omega_0$ and $w\in  \mathcal D$,
\begin{equation}\label{eq.borelrep.3}
\begin{array}{cl}
    {\mathcal M}_\infty(u+w)    &=\lim_j {\mathcal M}_{n_j}(u+w)  \\
     & =\lim_j {\mathcal M}_{n_j}(u)+ 2 \Real\langle u_{n_j},w_{n_j}\rangle_{L^2} +\|w_{n_j}\|^2_{L^2} \\
      & ={\mathcal M}_\infty(u)+ 2 \Real\langle u,w\rangle_{L^2} +\|w\|^2_{L^2},
\end{array}
\end{equation}
where 
\begin{equation}
    \label{eq.borelrep.ad1}
    u_{n}=P_{n} u, \qquad w_{n}=P_{n} w,
\end{equation} 
and $P_{n}$ is the projector in \eqref{P_n}. In \eqref{eq.borelrep.3}, the inner product $\langle u, w \rangle_{L^2}$ is well defined since $s \leq 1$ by Assumption \ref{choice_of_s_assumption}.
We also have for any $x,y\in\mathcal D$, 
\begin{equation}\label{eq.borelrepr.4}
    {\mathcal M}_{n}(u+x)-{\mathcal M}_{n}(u+y)=2\Real\langle u_{n}, x-y\rangle_{L^2}
+ \|P_{n}x\|^2_{L^2}-\|P_{n}y\|^2_{L^2}. 
\end{equation}
Now, we claim that  for any $u\in\Omega_0$ and $x\in H^1$, the subsequence $({\mathcal M}_{n_j}(u+x))_j$ converges. Indeed, we have 
\begin{eqnarray}
\notag
   |{\mathcal M}_{n_j}(u+x) -{\mathcal M}_{n_{j'}}(u+x) |\leq &  |{\mathcal M}_{n_j}(u+x) -{\mathcal M}_{n_j}(u+w) | \\ 
   \label{5.14_terms}
   &+ |{\mathcal M}_{n_j}(u+w) -{\mathcal M}_{n_{j'}}(u+w) | \\ 
   \notag
   &+|{\mathcal M}_{n_{j'}}(u+w) -{\mathcal M}_{n_{j'}}(u+x) | .
\end{eqnarray}
So, choosing $w\in\mathcal D$ sufficiently close to $x$ and using the computation \eqref{eq.borelrepr.4} and the convergence \eqref{eq.borelrep.3}, 
one proves that the left-hand side of \eqref{5.14_terms} converges to $0$ as $j,j'\to \infty$. So, we conclude that for any $u\in\Omega_0$, we have $u+x\in \Omega$ for all $x\in H^1$ since the subsequence $({\mathcal M}_{n_j}(u+x))_j$ converges. Therefore from the definition of $\Omega_0$ we obtain  $u+x\in \Omega_0$ for all $x\in H^1$. This means that $\Omega_0+H^1=\Omega_0$ and we can take a Borel representative for $\mathcal M$ as  $\widetilde{\mathcal{M}}=\mathbbm{1}_{\Omega_0} \, {\mathcal M}_\infty + R \, \mathbbm{1}_{H^{-s}\setminus \Omega_0}$ which automatically satisfies the identity \eqref{eq.lemborelrep.iden}. 

\end{proof}

\begin{lemma}
\label{lem_Hconnected}
Let $R>0$ and consider the function $\phi:H^{-s}\to \R$  defined as
\begin{equation}\label{eq.connect.7}
\phi(u)=
\begin{cases}
    R-\|u\|^2_{L^2} & \text{if} \quad d=1, \\
    R-|\widetilde{\mathcal{M}}(u)| & \text{if} \quad d=2,3. 
\end{cases}
\end{equation}
Then $\phi$ is  $H^1$-continuous and the set $\mathcal O=\{u\in H^{-s}: \phi(u) >0\}$ is   $H^1$-connected  as in Proposition \ref{Aida_0_derivative_result}. 
\end{lemma}

\begin{proof}
We treat the dimension $d=1$ separately. 

\medskip\noindent
$\bullet$ \underline{$d=1$}: In this case  $\mathcal M(u)=\|u\|^2_{L^2}$ and  according  to \eqref{eq.cond.spower} we can choose $s=0$. Furthermore,  we have  $\phi(u)=R-\|u\|^2_{L^2}$  and  the set 
$$
\mathcal O=\mathbb{B}_R=\{u\in L^2(\mathbb T) : \|u\|^2_{L^2} <R\}.
$$
So, it follows that $\phi(u+\cdot):H^{1}\to \R$  is continuous since, for any $u\in L^2$ and $w\in H^1$, we can write
$$
\phi(u+w)=R-(\|u\|_{L^2}^2+ {\rm Re}\langle u, w\rangle_{L^2}+ \|w\|^2_{L^2}). 
$$
On the other hand, for any  $u\in \mathcal O$, the set 
$$
\mathcal O(u)=\{w\in H^1 : \|u+w\|^2_{L^2}<R\}
$$
is  $H^1$-connected  since it  is convex. We remark that  $\mathcal O(u)$ is open as a consequence of the continuity of the function  $w\in H^1\mapsto \phi(u+w)$.

\medskip\noindent
$\bullet$ \underline{$d=2,3$}: In this case  $\mathcal M(\cdot)$ is the renormalized mass given by \eqref{renormalized_mass} and the exponent  $s$ satisfies \eqref{eq.cond.spower}. Let $\widetilde{\mathcal{M}}$ and $\Omega_0$ be respectively the Borel function and Borel set  in Lemma \ref{lem.borelrepr.mass}. According to the proof of Lemma \ref{lem.borelrepr.mass}, we have the following convergence for all $u\in \Omega_0$,
\begin{equation}
    \widetilde{\mathcal{M}}(u)=\lim_{j\to\infty} \sum_{|k|\leq n_j}\left( |\hat u(k)|^2-\frac{1}{\langle k \rangle^2}\right),
\end{equation}
and for all $u\notin \Omega_0$ we set $\widetilde{\mathcal{M}}(u)=R$. 
 In particular, for any $u\notin \Omega_0$, the map $w\in H^1\mapsto\widetilde{\mathcal{M}}(u+w)=R$ is constant since the complement of $\Omega_0$ is also stable by translation of elements of $H^1$. For $u\in \Omega_0$ and any $w\in H^1$, we have 
\begin{equation}
   \widetilde{\mathcal{M}}(u+w)= \widetilde{\mathcal{M}}(u)+ 2 \Real \langle u, w\rangle_{L^2}+ \|w\|^2_{L^2}.
\end{equation}
This proves the continuity with respect to $w$ since  $u\in H^{-s}$ with $s\leq 1$. Hence, for any $u\in H^{-s}$, the map $w\in H^1\mapsto \widetilde{\mathcal{M}}(u+w) $ is continuous.   

\medskip
Now for  $u\in \mathcal O=\{u\in H^{-s} : |\widetilde{\mathcal{M}}(u)|<R\}\subset\Omega_0 $, consider the set
\begin{equation}
    \mathcal O(u)=\{ w\in H^1 : |\widetilde{\mathcal{M}}(u+w)|<R\}.
\end{equation}
Again  $ \mathcal O(u)$ is open since $\phi(u+\cdot)$ is $H^1$-continuous. Unlike the case $d=1$, this set is not convex. Nevertheless, we prove that $ \mathcal O(u)$ is $H^1$-path connected. First, since $ \mathcal O(u)$ is open, any $w \in\mathcal O(u)$ can be continuously connected by a path within  $ \mathcal O(u)$ to 
\begin{equation}
\label{eq.w_n}
    w_n=\sum_{|k|\leq n} \langle w, e_k\rangle_{L^2}  \, e_k
\end{equation}
for some $n\in \mathbb N$ large enough (See \eqref{eigenfunctions_torus}--\eqref{u_n} for the definition of $w_n$). In this case, the path can be taken as a straight line. Secondly, we remark that the origin $0_{H^1}$ belongs to $ \mathcal O(u)$. Here $\{e_k\}_k$ is the ONB in \eqref{eigenfunctions_torus}. Then we will build a continuous path inside $ \mathcal O(u)$ connecting the point $w_n\in \mathcal O(u)$ to $0_{H^1}$. To explain in detail this step, we distinguish several cases: 

\medskip\noindent
1) If $w_n\in \mathcal O(u)$ is orthogonal to $u$ with respect to the real inner product, i.e. $\Real\langle u,w_n\rangle_{L^2}=0$: Then the continuous path $t\in[0,1]\mapsto tw_n$ connects the two points $0_{H^1}$ and $w_n$ without leaving $\mathcal O(u)$. In fact, we check by \eqref{eq.lemborelrep.iden} that for all $t\in[0,1]$,
$$
\widetilde{\mathcal{M}}(u+t w_n)=\widetilde{\mathcal{M}}(u)+t^2 \|w_n\|^2_{L^2}\in (-R,R),
$$
since $\widetilde{\mathcal{M}}(u)$ and $\widetilde{\mathcal{M}}(u+w_n)$ belong to  $(-R,R)$ and $t\mapsto \widetilde{\mathcal{M}}(u+t w_n)$ is an increasing function. 
 
\medskip\noindent
2) If $w_n\in \mathcal O(u)$ is parallel to $u_n=\sum_{|k|\leq n} \langle u, e_k\rangle_{L^2}  \,e_k$, such that there exists $\lambda \in [-1,+\infty)$ and $\lambda\neq 0$,  satisfying $w_n=\lambda u_n$:  Then, in this case, we  can write 
    $$
 \widetilde{\mathcal{M}}(u+t w_n)=    \widetilde{\mathcal{M}}(u)+ t \lambda (2  + t \lambda ) \|u_n\|^2_{L^2}.
    $$
Analysing the function $t\mapsto  \widetilde{\mathcal{M}}(u+t w_n)$, we see that it is monotone on $[0,1]$  whenever 
$\lambda\in[-1,+\infty) $. Therefore,  the continuous path  $t\in [0,1]\mapsto t w_n$ satisfies  $\widetilde{\mathcal{M}}(u+t w_n)\in (-R,R)$ for all $t\in [0,1]$. 

\medskip\noindent
3) If  $w_n\in \mathcal O(u)$ is parallel to $u_n=\sum_{|k|\leq n} \langle u, e_k\rangle_{L^2}  \,e_k$, such that there exists $\lambda \in (-\infty,-2]$,  satisfying $w_n=\lambda u_n$: In this case  we do not have  monotonicity of the function $t\in [0,1]\mapsto t \lambda (2+t\lambda)$. So, we need to use a specific curve instead of a straight line. First, we remark that there exists $w^\perp\in H^1$ which is orthogonal to $u$ (i.e., $\Real\langle w^\perp,u\rangle=0$) and $w^\perp\neq0$. Then we construct the following path
\begin{equation}\label{eq.path.connect}
  \begin{array}{ccc}
 \gamma : [0,1] & \longrightarrow & H^1     \\
     t & \longmapsto  & f(t) w^\perp+ t w_n
\end{array}   
\end{equation}
where 
\[ f(t) =
  \begin{cases}
    \frac{\|u_n\|_{L^2}}{\|w^\perp\|_{L^2}} \sqrt{-t \lambda (2+ t\lambda )}     & \quad \text{ if } \, t \lambda (2+t\lambda)<0,\\
    0 & \quad \text{ if } \, t \lambda (2+t\lambda) \geq 0.
  \end{cases}
\]
For all $t\in[0,1]$, we have
$$
\widetilde{\mathcal{M}}(u+\gamma(t))=\widetilde{\mathcal{M}}(u)+ t\lambda (2+t\lambda) \|u_n\|^2_{L^2}+ f(t)^2 \|w^\perp\|^2_{L^2}.
$$
Hence,  we can check that for all $t\in [0,1]$ such that  $t \lambda (2+t\lambda)<0$, we have
$$
\widetilde{\mathcal{M}}(u+\gamma(t))= \widetilde{\mathcal{M}}(u)\in (-R,R),
$$
while in the other case when $t \lambda (2+t\lambda)\geq 0$ with $t\in [0,1]$, we have
$$
-R< \widetilde{\mathcal{M}}(u) \leq \widetilde{\mathcal{M}}(u+\gamma(t))= 
\widetilde{\mathcal{M}}(u)+ t \lambda (2+t\lambda) \|u_n\|^2_{L^2} \leq \widetilde{\mathcal{M}}(u+w_n)<R. 
$$
Since $f(0)=f(1)=0$, $\gamma$ is a continuous path staying  inside  $ \mathcal O(u)$ and satisfying $\gamma(0)=0_{H^1}$ and $\gamma(1)=w_n$. 

\medskip\noindent
4) If  $w_n\in \mathcal O(u)$ is parallel to $u_n=\sum_{|k|\leq n} \langle u, e_k\rangle_{L^2}  \,e_k$, such that there exists $\lambda \in (-2,-1)$,  satisfying $w_n=\lambda u_n$: This case is similar to (3) but we need to adjust the argument further. 
We remark that there exists $t_0\in (0,1)$ such that we have for all $t\in [t_0,1]$,
\begin{equation}
    \label{eq.connect.lem.1}
    \widetilde{\mathcal{M}}(u)+t\lambda (2+ t\lambda) \|u_n\|^2_{L^2} \in (-R,R),
\end{equation}
with the function $t\mapsto t\lambda (2+ t\lambda) $ increasing over $[t_0,1]$. Indeed, since $\mathcal{M}_0(u+w_n)\in (-R,R)$, we just need to take $t_0$ close enough to $1$.  Let 
\begin{equation}
    \label{eq.connect.lem.2}
\theta (t)= \frac{1}{\|w^\perp\|_{L^2}} \sqrt{R-\widetilde{\mathcal{M}}(u)- t\lambda (2+t\lambda) \|u_n\|^2_{L^2} }.
\end{equation}
Then we check that $\theta$ is a decreasing continuous function with $\theta(t) >0 $ for all $t\in [t_0,1]$. So, we choose a path $\gamma : t\in [0,1] \mapsto  f(t) w^\perp+ t w_n$ such that $f$ is given as 
\begin{align*}
     f(t) = 
  \begin{cases}
  \displaystyle 
    \frac{\|u_n\|_{L^2}}{\|w^{\perp}\|_{L^2}} \sqrt{-t \lambda (2+ t\lambda )}     & \quad \text{ if } \, t \leq t_0,\vspace{.1in}\\ \displaystyle 
     \frac{(1-t)\theta(t)} {(1-t_0) \theta(t_0)} f(t_0) & \quad \text{ if } \, t > t_0.
  \end{cases}
\end{align*}
Hence, $f$ is continuous with $f(1)=f(0)=0$ and satisfies, by Lemma \ref{lem.borelrepr.mass} and \eqref{eq.lemborelrep.iden}, for all $t\in [0,t_0]$
$$
\widetilde{\mathcal{M}}(u+\gamma(t))=\widetilde{\mathcal{M}}(u) \in (-R,R). 
$$
Moreover,  we claim that for  $t\in [t_0,1]$, 
\begin{equation}
    \label{eq.hconnect.3}
    -R<\widetilde{\mathcal{M}}(u+\gamma(t))=\widetilde{\mathcal{M}}(u)+t\lambda (2+ t\lambda) \|u_n\|^2_{L^2} + f(t)^2 \|w^\perp\|^2_{L^2}<R.
\end{equation}
In fact, by \eqref{eq.connect.lem.1}, we  have $f(t_0)<\theta(t_0)$. This shows that for all $t\in [t_0,1]$,
$$
f(t)< \theta(t). 
$$
The latter inequality implies that $\widetilde{\mathcal{M}}(u+\gamma(t))<R$ for all $t\in[t_0,1]$. On the other hand, by \eqref{eq.connect.lem.1} we have for all $t\in [t_0,1]$,
$$
-R< \widetilde{\mathcal{M}}(u)+t\lambda (2+ t\lambda) \|u_n\|^2_{L^2}. 
$$
This shows that $\gamma$ connects continuously $0_{H^1}$ to $w_n$ within  $\mathcal O(u)$.

\medskip\noindent
5) The general case is amenable to one of the above situations: In fact, take $w_n$ as before and decompose it into a sum of parallel and  orthogonal components with respect to $u_n$. This leads to 
$$
w_n=\lambda u_n+ w^\perp,
$$
for some $\lambda \in \mathbb R$  and $w^\perp\in H^1$. The cases $\lambda=0$ or $w^\perp=0$ are covered by the above discussion. So, we can assume $\lambda\neq 0$ and $w^\perp\neq 0$ and consider a path of the form $\gamma(t)=\lambda t u_n+ f(t) w^\perp$. Knowing that $\widetilde{\mathcal{M}}(u+w_n)$ and $\widetilde{\mathcal{M}}(u)$ are in $(-R,R)$,  we look for a continuous function $f$ such that for all $t\in[0,1]$,
\begin{equation}
    \label{eq.connect.5}
\widetilde{\mathcal{M}}(u+\gamma(t))=\widetilde{\mathcal{M}}(u)+ t\lambda (2+t\lambda) \|u_n\|^2_{L^2}+
f(t)^2 \|w^\perp\|_{L^2}^2 \in (-R,R) \,,
\end{equation}
with  $f(0)=0$ and $f(1)=1$. 
Following the same arguments as in the cases 1)--4), we can construct such a function $f$ according to the value of $\lambda$ and the behaviour of the function $t\in [0,1]\mapsto t\lambda (2+t\lambda)$. The latter is increasing when $\lambda \geq 0$, decreasing when $-1 \leq \lambda \leq 0$, and has a global minimum at $-\frac{1}{\lambda}\in (0,1)$ when $\lambda \in (-\infty, -1)$.  
But instead of giving an explicit construction of $f$ as before, we provide a rough argument justifying of the existence of such a function. Set 
$$
F_1(t)=\frac{-R-\widetilde{\mathcal{M}}(u)-\lambda t (2+\lambda t)\|u_n\|^2_{L^2}}{\|w^\perp\|^2_{L^2}}
$$
and 
$$
 F_2(t)=\frac{R-\widetilde{\mathcal{M}}(u)-\lambda t (2+\lambda t)\|u_n\|^2_{L^2}}{\|w^\perp\|^2_{L^2}}.
$$
Therefore we  require  for all $t\in [0,1]$
$$
F_1(t)< f(t)^2<F_2(t).
$$
In order to prove the existence of $f$ satisfying \eqref{eq.connect.5}, it is enough to show that  we have for all $t\in [0,1]$,
\begin{equation}
    \label{eq.connect.6}
    F_1(t) < F_2(t), \quad F_1(0)<0< F_2(0), \quad F_1(1)<1< F_2(1),  \; \text{ and } \; F_2(t)>0. 
\end{equation}
The first inequality follows immediately by construction of $F_1$ and $F_2$. Using  the fact that $\widetilde{\mathcal{M}}(u+w_n)$ and $\widetilde{\mathcal{M}}(u)$ belong to $(-R,R)$, we check the second and third inequalities in \eqref{eq.connect.6}. For the last statement $F_2(t)>0$, one needs  to take into account the variation of the function $t\in [0,1]\mapsto t\lambda (2+t\lambda)$. Accordingly, we have 
\begin{itemize}
    \item  If $\lambda > 0$ then $\min_{t\in[0,1]} F_2(t)=F_2(1)>0$, \vspace{.05in}
    \item If $-1 \leq \lambda \leq 0$ then  $\min_{t\in[0,1]} F_2(t)=F_2(0)>0$,  \vspace{.05in}
    \item If $-2 < \lambda < -1$ then $\min_{t\in[0,1]} F_2(t)= F_2(0)>0$, \vspace{.05in}
    \item If $\lambda \leq -2$ then $\min_{t\in[0,1]} F_2(t)=F_2(1)>0$. 
\end{itemize}
This  shows that  $ \mathcal O(u)$ is $H^1$-path connected. 

\end{proof}

\begin{lemma}
For all $R>0$, we have  
    \begin{equation}\label{eq.mes_diff_zero}
        \mu_0(\{u\in H^{-s}: \phi(u) >0\})=\mu_0(\{u\in H^{-s}: |\mathcal M(u)|<R\}) \neq 0,
    \end{equation}
where  $\phi$ is the Borel measurable function given in \eqref{eq.connect.7}.
\end{lemma}
\begin{proof}
We discuss only the case $d=2,3$ ($d=1$ is simpler).  We know, by Lemma \ref{Wick_ordered_mass_lemma} and the computation \eqref{Integral_n_Cauchy}, that  
$\mathcal{M}\in L^2(\mu_0)$ and $\|\mathcal{M}\|_{L^2(\mu_0)}\neq 0$. 
To prove \eqref{eq.mes_diff_zero} for all $R>0$, we use Lemma \ref{app.lem.partition_funct} based on  Bernstein’s inequality. Indeed, we have 
\begin{eqnarray*}
    \mu_0(\{u\in H^{-s}: |\mathcal M(u)|< R\})&=&\mu_0\left( \left|\sum_{\ell\in\Z^d} \frac{|\langle u,e_\ell\rangle_{L^2}|^2-1}{\langle \ell\rangle^2}\right|<R \right)\\
    &=& \mathbb{P}_{\omega}\left(  \left|\sum_{\ell\in\Z^d} \frac{|g_\ell(\omega)|^2-1}{\langle \ell\rangle^2}\right|<R\right)
    >0. 
\end{eqnarray*}
See Appendix \ref{Malliavin derivative calculations and normalizability of the Gibbs measure for the focusing Hartree equation} for the above notations $\mathbb P_\omega$ and the complex Gaussian random variables $g_\ell(\omega)$. 
\end{proof}

 \medskip
 \noindent
{\bf Proof of Theorem \ref{kms_implies_gibbs_theorem}:} Collecting  Proposition \ref{local_KMS_differential_equation}, Proposition \ref{Aida_0_derivative_result} and Lemma \ref{lem_Hconnected}, we are now able to prove our third main Theorem \ref{kms_implies_gibbs_theorem}. 
\begin{proof}[Proof of Theorem \ref{kms_implies_gibbs_theorem}]
Let $\chi^{(\delta)}_R$ be the cutoff function in \eqref{chi} for $R>0$ and $\delta\in(0,1)$. Using  the Leibniz rule Lemma \ref{Leibniz_rule}  and the chain rule Lemma  \ref{chain_rule}, we have
\begin{equation}
\label{eq:nabla:kmscondition}
\nabla \left(e^{-h^I}  \chi_R^{(\delta)}(\mathcal{M})  ~\varrho \right)= \left(-
\nabla h^I ~  \varrho +\nabla \varrho \right) e^{-h^I} ~ \chi_R^{(\delta)}(\mathcal{M})+ \varrho ~
 e^{-h^I}  \big(\chi_R^{(\delta)}\big)'(\mathcal{M}) \nabla \mathcal{M}.
\end{equation}
Here we have also used Lemma \ref{exponential_Malliavin_differentiability} with H\"older's inequality to see that the identity \eqref{eq:nabla:kmscondition} holds true in $L^1(\mu_0,H^{-s})$. 
The support properties of $\chi_R^{(\delta)}$ imply that   $\big(\chi_R^{(\delta)}\big)'(\mathcal{M})=0$  for all $u\in \mathbb{B}_{\delta R}$. Hence,  using the differential equation \eqref{Malliavin_diff_equ}, we obtain  for $\mu_0$-almost all $u\in \mathbb{B}_{\delta R}$,
\begin{equation}
\nabla \left(e^{-h^I} \chi_R^{(\delta)}(\mathcal{M}) ~\varrho \right)(u)= 0.
\end{equation}
Since by Lemma \ref{lem_Hconnected} the set $\mathbb{B}_{\delta R}$ satisfies the assumptions of Proposition  \ref{Aida_0_derivative_result}, we deduce that  for $\mu_0$-almost all $u\in \mathbb{B}_{\delta R}$,
\begin{equation}
e^{-h^I(u)} \varrho(u)= e^{-h^I(u)} \varrho(u) ~ \chi_R^{(\delta)}(\mathcal{M}(u))= c,
\end{equation}
for some constant $c\geq 0$. Thus, we conclude that for $\mu_0$-almost all $u\in \mathbb{B}_{R}=\cup_{\delta\in(0,1)\cap \mathbb Q} \mathbb{B}_{\delta R}$
\begin{equation}
\varrho(u)= c e^{h^I(u)},
\end{equation}
with $c$ a non-negative constant.
\end{proof}

\appendix

\section{Wick ordering and Malliavin derivative calculations}
\label{Malliavin derivative calculations and normalizability of the Gibbs measure for the focusing Hartree equation}

\subsection{Review of Wick ordering}
\label{Review of Wick ordering}
Throughout the appendix, we consider $A$ as in Assumption \ref{A_choice}. We also fix $s$ as in Assumption \eqref{choice_of_s_assumption} (recall Lemma \ref{A_choice_lemma}). We give a brief review of the facts from Wick ordering that are used in the analysis. For notational simplicity, in the appendix, we slightly modify our convention from \eqref{Hilbert_space_ONB}--\eqref{e_j_choice} and \eqref{ONB_convention} above.  In particular, we consider eigenfunctions that are indexed by elements of $\Z^d$ instead of by $\N$. Namely, for $k \in \Z^d$, we let 
\begin{equation}
\label{eigenfunctions_torus}
e_k(x) \equiv \ee^{2\pi \ii k \cdot x}\,, \qquad f_k(x) \equiv \ii \ee^{2\pi \ii k \cdot x}\,.
\end{equation}
With notation as in \eqref{eigenfunctions_torus}, for $n \in \N$ we define $P_n:H^{-s} \rightarrow \mathrm{span}_{\mathbb{C}} \{(e_k)_{|k| \leq n}\}$ as
\begin{equation} 
\label{P_n}
P_n=\sum_{|k| \leq n} |e_k \rangle \langle e_k|\,.
\end{equation}
With notation as in \eqref{P_n}, we henceforth use the shorthand notation 
\begin{equation}
\label{u_n}
u_n=P_n u\,.
\end{equation}
When performing the calculations in the sequel, we don't work with the measure $\mu_0$ from Proposition \ref{free_Gibbs_prop} above directly. Instead, under a suitable identification, we work with random Fourier series.
To this end,  
we introduce the probability space $(\mathbb{C}^\mathbb{N}, \mathcal{G}, \mathbb{P})$. Here, $\mathcal{G}$ is the product sigma-algebra and $\nu:=\bigotimes_{k \in \N} \nu_{k}$ with $\nu_{k}(z):=\frac{1}{\pi} \,\ee^{-|z|^2} d z$, where $d z$ is the Lebesgue measure on $\mathbb{C}$. We denote the probability space $\mathbb{C}^\mathbb{N}$ by $\Omega$ and its elements by $\omega \equiv (g_k(\omega))_{k \in \mathbb{N}}$.

With the setup as above, we view the Gaussian measure $\mu_{0}$ as the probability measure induced by the map

\begin{equation} 
\label{map_varphi_omega}
\omega \in \Omega \mapsto u(x) \;\equiv\; u^{\omega}(x)=
\sum_{k \in \mathbb{Z}^{d}} \frac{{g}_{k}(\omega)}{\langle k \rangle}\,\ee^{2 \pi \mathrm{i} k \cdot x}\,,
\end{equation}
where in \eqref{map_varphi_omega} and in the sequel, we use the convention that 
\begin{equation}
\label{Japanese_bracket}
\langle k \rangle:=\sqrt{4\pi^2|k|^2+1}\,.
\end{equation}
In what follows, we sometimes write $\mathbb{P}_{\omega} \equiv \mathbb{P}$ and $\mathbb{E}_{\omega}[\cdot] \equiv \mathbb{E}[\cdot]$ for expectation with respect to $\mathbb{P}$. The random variable \eqref{map_varphi_omega} is often referred to as the \emph{classical free field}. For $s$ as in Assumption \eqref{choice_of_s_assumption}, one has
\begin{equation}
\label{classical_free_field}
\mathbb{E}_{\omega} \Bigl[\|u^{\omega}\|_{H^{-s}(\T^d)}^2 \Bigr]
=\mathbb{E}_{\omega} \Biggl[\sum_{k \in \Z^d} \frac{|g_k(\omega)|^2}{\langle k \rangle^2}\,\langle k \rangle^{-2s}\Biggr] 
= \sum_{k \in \Z^d} \langle k \rangle^{-2-2s}<\infty\,.
\end{equation}
For a more detailed discussion about the above identification, we refer the reader to \cite[Remark 1.3]{FKSS17}. Throughout the sequel, we use the two different formulations for $\mu_{0}$ interchangeably, without further comment.

Recalling \eqref{Japanese_bracket}, we define the density
\begin{multline}
\label{sigma_{n,beta}}
\sigma_{n}:=\int_{H^{-s}} |u_n(x)|^2\,d \mu_{0}=\mathbb{E}_{\omega} \Biggl[\sum_{|k| \leq n} \frac{|g_k(\omega)|^2}{\langle k \rangle^2} \Biggr]=\sum_{|k| \leq n} \frac{1}{\langle k \rangle^2}
\\
\sim
\begin{cases}
\log n\quad  &\mbox{if } d=2
\\
n\quad  &\mbox{if } d=3\,.
\end{cases}
\end{multline}
Note that \eqref{sigma_{n,beta}} is independent of $x$ (since as $A$ is given by Assumption \ref{A_choice}).
With notation as in \eqref{u_n}--\eqref{sigma_{n,beta}}, for fixed $n$, the Wick-ordering of $|u_n(x)|^2$ is the random variable 
\begin{equation}
\label{:|u_n|^2:}
:|u_n(x)|^2: \;\equiv\; |u_n(x)|^2-\sigma_{n}=\Biggl|\sum_{|k| \leq n} \widehat{u}(k)e_k(x)\Biggr|^2-\sum_{|k| \leq n} \frac{1}{\langle k \rangle^2}\,.
\end{equation}
We first consider the random variable obtained by integrating \eqref{:|u_n|^2:}
\begin{equation}
\label{Integral_n}
:\|u_n\|_{L^2}^2: \;\equiv\; \int_{\T^d} :|u_n|^2: \, d x=\sum_{|k| \leq n} \biggl(|\widehat{u}(k)|^2-\frac{1}{\langle k \rangle^2}\biggr)\,.
\end{equation}
We can now give a rigorous definition of \eqref{renormalized_mass} when $d=2,3$.
\begin{lemma}
\label{Wick_ordered_mass_lemma}
For all $p \in [1,\infty)$, the sequence \eqref{Integral_n} is Cauchy in $L^p(\mu_{0})$. In particular, the random variable 
\begin{equation}
\label{Wick_ordered_mass}
:\|u\|_{L^2}^2: \;\equiv\; \lim_{n \rightarrow \infty} :\|u_n\|_{L^2}^2:\,,
\end{equation}
is defined as a limit in $L^p(\mu_{0})$. The limit is independent of $p$. 
\end{lemma}

\begin{proof}
 We consider $m,n \in \N$ with $m>n$ and observe that 
\begin{multline}
\label{Integral_n_Cauchy}
\int \Bigl|:\|u_n\|_{L^2}^2:-:\|u_m\|_{L^2}^2:\Bigr|^2\,d \mu_{0}
\\
=\int \Biggl|\sum_{n<|k|\leq m} \biggl(|\widehat{u}(k)|^2-\frac{1}{\langle k \rangle^2}\biggr)\Biggr|^2\,d \mu_{0}
\\
=\mathbb{E}_{\omega} \Biggl[\Biggl|\sum_{n<|k| \leq m} \frac{|g_k(\omega)|^2-1}{\langle k \rangle^2}\Biggr|^2\Biggr]
\sim \sum_{n<|k| \leq m} \frac{1}{\langle k \rangle^4} \rightarrow 0\, \quad \mbox{as } n \rightarrow \infty\,.
\end{multline}
In \eqref{Integral_n_Cauchy}, we used the fact that for all $k \in \N$, we have 
\begin{equation*}
\mathbb{E}_{\omega}\bigl[|g_k|^2-1\bigr]=0\,,\quad \mathbb{E}_{\omega}\bigl[\bigl||g_k|^2-1\bigr|^2\bigr]=O(1)\,.
\end{equation*}
An analogous argument using Wick's theorem shows that \eqref{Integral_n} is Cauchy in $L^m(\mu_0)$ for every $m \in \N^*$. The claim in $L^p(\mu_0)$ follows by H\"{o}lder's inequality.
\end{proof}

\subsection{Malliavin derivative of the renormalized mass}
\label{Proof_of_Lemma_2.10}
We now prove Lemma \ref{Malliavin_derivative_M_lemma}.
\begin{proof}[Proof of Lemma \ref{Malliavin_derivative_M_lemma}]
We prove the claim when $d=2,3$. The argument when $d=1$ is simpler and is obtained by a minor modification of the one given below. We omit the details. 

Let us fix $p \in [1,\infty)$. The fact that $:\|u\|_{L^2}^2:$ belongs to $L^p(\mu_0)$ follows immediately from Lemma \ref{Wick_ordered_mass_lemma}. Let us now show that
\begin{equation*}
\nabla (:\|u\|_{L^2}^2:) \in L^p(\mu_0;H^{-s}) 
\end{equation*}
and that the identity \eqref{Wick_ordered_mass_Malliavin_derivative} holds. To this end, we consider $\pi_n:H^{-s} \rightarrow \mathbb{R}^{2n}$ given by\footnote{This is a slight change of convention for earlier, in accordance with \eqref{eigenfunctions_torus}.}
\begin{equation}
\label{pi_n}
\pi_n(u):=\Bigl(\bigl(\langle u, e_k \rangle_{L^2,\mathbb{R}}\bigr)_{|k| \leq n}, \bigl(\langle u, f_k \rangle_{L^2,\mathbb{R}}\bigr)_{|k| \leq n}\Bigr)\,.
\end{equation}
Note that \eqref{pi_n} is a suitable modification of the projection map \eqref{projection_map}; we slightly abuse notation by using the same symbol.

By \eqref{eigenfunctions_torus}, we have that for all $k \in \Z^d$
\begin{equation}
\label{e_k_inner_product}
\langle u,e_k \rangle_{L^2,\mathbb{R}}
=\mathrm{Re}\, \widehat{u}(k)\,.
\end{equation}
Moreover,
\begin{equation}
\label{f_k_inner_product}
\langle u,f_k \rangle_{L^2,\mathbb{R}}
=\mathrm{Re}\, \biggl[-\mathrm{i}\,\int_{\mathbb{T}^d} u(x)\, \mathrm{e}^{-2\pi \mathrm{i} k \cdot x}\,dx\biggr]=\mathrm{Re}\,\Bigl[-\mathrm{i}\,\widehat{u}(k)\Bigr]=\mathrm{Im}\, \widehat{u}(k)\,.
\end{equation}
In particular, letting $a_k := \mathrm{Re}\,\widehat{u}(k)$ and $b_k:=\mathrm{Im}\,\widehat{u}(k)$, we can rewrite \eqref{pi_n} as  
\begin{equation}
\label{pi_n_rewritten}
\pi_n(u)=\Bigl(\bigl(\mathrm{Re}\,\widehat{u}(k)\bigr)_{|k| \leq n}, \bigl(\mathrm{Im}\,\widehat{u}(k)\bigr)_{|k| \leq n}\Bigr)
\equiv \bigl((a_k)_{|k| \leq n}, (b_k)_{|k| \leq n}\bigr)\,.
\end{equation}
Using \eqref{e_k_inner_product}--\eqref{f_k_inner_product}, we can write \eqref{Integral_n} as 
\begin{equation}
\label{Integral_n_rewritten}
:\|u_n\|_{L^2}^2:=\sum_{|k| \leq n} \biggl(\mathrm{Re}\, \widehat{u}(k)^2+\mathrm{Im}\, \widehat{u}(k)^2-\frac{1}{\langle k \rangle^2}\biggr)\equiv \sum_{|k| \leq n} \biggl(a_k^2+b_k^2-\frac{1}{\langle k \rangle^2}\biggr)\,.
\end{equation}
We now use an approximation argument and \eqref{Integral_n_rewritten} to obtain that $:\|u_n\|_{L^2}^2: \in \mathbb{D}^{1,p}(\mu_0)$ with
\begin{equation}
\label{Integral_n_Malliavin_derivative}
\nabla \bigl(:\|u_n\|_{L^2}^2:\bigr)=2\sum_{|k| \leq n} \bigl(\mathrm{Re}\, \widehat{u}(k)e_k+\mathrm{Im} \, \widehat{u}(k)\,\ii e_k\bigr)=2 \sum_{|k|\leq n} \widehat{u}(k)e_k=2u_n\,.
\end{equation}

Let us explain how to apply an approximation argument to obtain the first equality in \eqref{Integral_n_Malliavin_derivative} in more detail. We define $\varphi: \R^{|k| \leq n} \times \R^{|k| \leq n} \rightarrow \R$ by
\begin{equation}
\label{Appendix_varphi}
\varphi\bigl((a_k)_{|k| \leq n}, (b_k)_{|k| \leq n}\bigr):=\sum_{|k| \leq n} \biggl(a_k^2+b_k^2-\frac{1}{\langle k \rangle^2}\biggr)\,.
\end{equation}
In particular, from \eqref{pi_n_rewritten}--\eqref{Integral_n_rewritten} and \eqref{Appendix_varphi}, we can write
\begin{equation}
\label{Appendix_varphi_2}
:\|u_n\|_{L^2}^2:=\varphi \circ \pi_n\,.
\end{equation}
We need to justify the first equality in \eqref{Integral_n_Malliavin_derivative} as $\varphi$ is smooth, but not compactly-supported. To this end, we consider
\begin{equation}
\label{Phi_choice}
\Phi \in \mathscr{C}^{\infty}_{\mathrm{c}}(\R^{|k| \leq n} \times \R^{|k| \leq n})\,,\,\,\,\Phi=1\,\,\, \text{near the origin}\,,\,\,\,\Phi_M:=\Phi(\cdot/M)\,.
\end{equation}
For $M \in \N$, we let
\begin{equation}
\label{Appendix_varphi_3}
\varphi_M\bigl((a_k)_{|k| \leq n}, (b_k)_{|k| \leq n}\bigr):=\varphi \bigl((a_k)_{|k| \leq n}, (b_k)_{|k| \leq n}\bigr)\,\Phi_M\bigl((a_k)_{|k| \leq n}, (b_k)_{|k| \leq n}\bigr)\,.
\end{equation}
Note that then $\varphi_M \in \mathscr{C}^{\infty}_{\mathrm{c}}(\R^{|k| \leq n} \times \R^{|k| \leq n})$.
By using the fact that 
\begin{equation}
\label{finite_moments}
\mathbb{E}_{\mu_0} [|\widehat{u}(k)|^q]<\infty \quad  \forall k \in \N \quad \forall q \in [1,\infty)\,,
\end{equation}
\eqref{Appendix_varphi}, \eqref{Phi_choice}, \eqref{Appendix_varphi_3}, and the dominated convergence theorem, we obtain that
\begin{equation}
\label{Appendix_varphi_4}
\lim_M \varphi_M \circ \pi_n=\varphi \circ \pi_n \,\, \text{in} \,\, L^p(\mu_0)\,.
\end{equation}
Now, for fixed $M$, $\varphi_M \circ \pi_n$ is a cylindrical function as in Definition \ref{cylindrical_test_functions}, and by \eqref{gradient_cylindrical_F}, \eqref{Appendix_varphi}, \eqref{Appendix_varphi_3}, we have
\begin{multline}
\label{Appendix_varphi_5}
\nabla (\varphi_M \circ \pi_n)=2u_n \,\Phi_M(\pi_n(u))
\\
+:\|u_n\|_{L^2}^2:\,\sum_{|k| \leq n} \bigl(\partial_j^{(1)}\Phi_M(\pi_n(u))e_k+\, \partial_j^{(2)}\Phi_M(\pi_n(u)) \ii e_k\bigr)\,.
\end{multline}
Here, we also recalled \eqref{pi_n_rewritten} and argued as for the last equality in \eqref{Integral_n_Malliavin_derivative}.
By using the dominated convergence theorem and \eqref{Phi_choice}, we obtain that the right-hand side of \eqref{Appendix_varphi_5} converges to the right-hand side in \eqref{Integral_n_Malliavin_derivative} in $L^p(\mu_0;H^{-s})$ as $M \rightarrow \infty$. We hence justify the first equality in \eqref{Integral_n_Malliavin_derivative} and deduce that $:\|u_n\|_{L^2}^2: \in \mathbb{D}^{1,p}(\mu_0)$ .

The claim now follows from \eqref{Wick_ordered_mass}, \eqref{Integral_n_Malliavin_derivative}, provided that we show
\begin{equation}
\label{P_n(u)_convergence_a}
\lim_{n \rightarrow \infty} \|u_n-u\|_{L^p(\mu_0;H^{-s})}=0\,.
\end{equation}
By using hypercontractivity (see \cite[Lemma 5.2]{AS21} or \cite[Theorem I.22]{Sim74}), \eqref{P_n(u)_convergence_a} follows from
\begin{equation}
\label{P_n(u)_convergence}
\lim_{n \rightarrow \infty} \|u_n-u\|_{L^2(\mu_0;H^{-s})}=0\,.
\end{equation}
In order to see that \eqref{P_n(u)_convergence} holds, we argue analogously as for \eqref{classical_free_field}
to deduce that
\begin{align}
\notag
\|u_n-u\|_{L^2(\mu_0;H^{-s})}^2&=\mathbb{E}_{\omega} \Biggl[\sum_{|k|>n} \frac{|g_k(\omega)|^2}{\langle k \rangle^2}\,\langle k \rangle^{-2s}\Biggr] 
\\
\label{P_n(u)_convergence_proof}
&=\sum_{|k|>n} \langle k \rangle^{-2-2s} \rightarrow 0\, \quad \mbox{as } n \rightarrow \infty\,,
\end{align}
by Assumption \ref{choice_of_s_assumption}.
\end{proof}

\subsection{Proof of Proposition \ref{Hartree_equation_Malliavin_derivative} (i)}
\label{Proof_of_Proposition_2.13_(i)}

We now prove Proposition \ref{Hartree_equation_Malliavin_derivative} (i). 

\begin{proof}[Proof of \ref{Hartree_equation_Malliavin_derivative} (i)] 
Throughout, let us fix $p \in [1,\infty)$.
The case $d=1$ follows from the proof of \cite[Proposition 5.1 (i), (iii)]{AS21}. We note that the analysis in this section of \cite{AS21} is independent of the sign of the nonlinearity.

Let us consider now the case $d=2,3$. 
Recalling \eqref{:|u_n|^2:}, we define the truncated nonlinearity for $n \in \N$
\begin{equation}
\label{A_h^I}
h_{n}^I(u):=\frac{1}{4}\,\int_{\mathbb{T}^d} \,\int_{\mathbb{T}^d} \,:|u_n(x)|^2:\,V(x-y)\,:|u_n(y)|^2:\,d x\,dy\,.
\end{equation}
In \cite[Appendix B, Proof of (5.9)]{AS21}, it is shown that the sequence $(h^I_{n})_n$ in \eqref{A_h^I} is Cauchy in $L^p(\mu_{0})$. We rigorously define $h^{I}$ as the limit of this sequence (and note that the limit is independent of $p$). In the proof of  \cite[Proposition 5.1 (ii)]{AS21}, one considers $V$ such that $\widehat{V} \geq 0$, but the argument follows without this assumption in an analogous way. We refer the reader to the aforementioned work for details. In particular, we deduce that $h^{I} \in L^p(\mu_{0})$.

We now show that $\nabla h^{I} \in L^p(\mu_{0};H^{-s})$. 
Let us first show that for $n \in \N$, we have
\begin{equation}
\label{Hartree_equation_Malliavin_derivative_1}
h_{n}^I \in \mathbb{D}^{1,p}(\mu_0)\,,\qquad \nabla h_{n}^I (u)=P_n \Bigl[u_n\, \bigl(V*\,:|u_n|^2:\,\bigr)\Bigr]\,.
\end{equation}
In order to prove \eqref{Hartree_equation_Malliavin_derivative_1}, let us first note with notation as in \eqref{:|u_n|^2:} and \eqref{pi_n_rewritten}, we can write
\begin{multline}
\label{star}
:|u_n(x)|^2:=\Biggl|\sum_{|k| \leq n} \widehat{u} (k)\,\mathrm{e}^{2\pi \mathrm{i} k \cdot x}\Biggr|^2-\sigma_{n}
=\Biggl|\sum_{|k| \leq n} (a_k+\mathrm{i} b_k) \mathrm{e}^{2\pi \mathrm{i} k \cdot x}\Biggr|^2-\sigma_{n}
\\
=\sum_{|k_1|,|k_2| \leq n} (a_{k_1}+\mathrm{i} b_{k_1}) \,(a_{k_2}-\mathrm{i} b_{k_2})\,\mathrm{e}^{2\pi \mathrm{i} (k_1-k_2)\cdot x}-\sigma_{n} \equiv I_{n,x}\,. 
\end{multline}
From \eqref{star}, we can rewrite \eqref{A_h^I} as
\begin{multline}
\label{h_n^I_rewritten_1}
h_n^I(u)=\frac{1}{4} \,\int_{\mathbb{T}^d} \,\int_{\mathbb{T}^d} \Biggl(\sum_{|k_1|,|k_2| \leq n} (a_{k_1}+\mathrm{i} b_{k_1}) \,(a_{k_2}-\mathrm{i} b_{k_2})\,\mathrm{e}^{2\pi \mathrm{i} (k_1-k_2)\cdot x}-\sigma_{n} \Biggr)\,
\\
\times
V(x-y)\,
\Biggl(\sum_{|\ell_1|,|\ell_2| \leq n} (a_{\ell_1}+\mathrm{i} b_{\ell_1}) \,(a_{\ell_2}-\mathrm{i} b_{\ell_2})\,\mathrm{e}^{2\pi \mathrm{i} (\ell_1-\ell_2)\cdot y}-\sigma_{n} \Biggr)\, dx\, dy\,.
\end{multline}
By Parseval's identity, and assumption that $V$ is even, we can rewrite \eqref{h_n^I_rewritten_1} as
\begin{multline}
\label{h_n^I_rewritten_2}
h_n^I(u)=\frac{1}{4}\,\mathop{\sum_{|k_1|,|k_2|,|\ell_1|,|\ell_2| \leq n}}_{k_1-k_2+\ell_1-\ell_2=0} (a_{k_1}+\mathrm{i}b_{k_1})\,(a_{k_2}-\mathrm{i}b_{k_2}) \,(a_{\ell_1}+\mathrm{i}b_{\ell_1})\,(a_{\ell_2}-\mathrm{i} b_{\ell_2})\,\widehat{V}(k_1-k_2)
\\
+\frac{1}{2} \sigma_n \widehat{V}(0) \,\sum_{|k| \leq n} (a_k^2+b_k^2)-\frac{1}{4} \sigma_n^2 \widehat{V}(0)
\\
=:\psi\bigl((a_{k_1})_{|k_1| \leq n}, (a_{k_2})_{|k_2| \leq n},(a_{\ell_1})_{|\ell_1| \leq n},(a_{\ell_2})_{|\ell_2| \leq n}\bigr)\,.
\end{multline}
Analogously as in in \eqref{Appendix_varphi}, we note that $\psi: \R^{|k_1| \leq n} \times \R^{|k_2| \leq n} \times \R^{|\ell_1| \leq n} \times \R^{|\ell_2| \leq n} \rightarrow \R$ given by \eqref{h_n^I_rewritten_2} is smooth, but not compactly supported. In order to compute the Malliavin derivative, we use an approximation argument as for \eqref{Integral_n_Malliavin_derivative} above. Similarly as in \eqref{Phi_choice}, we take
\begin{multline}
\label{Psi_choice}
\Psi \in \mathscr{C}^{\infty}_{\mathrm{c}}(\R^{|k_1| \leq n} \times \R^{|k_2| \leq n} \times \R^{|\ell_1| \leq n} \times \R^{|\ell_2| \leq n})\,,\,\,\,
\Psi=1\,\,\, \text{near the origin}\,,
\\
\Psi_M:=\Psi(\cdot/M)\,.
\end{multline}
Similarly as in \eqref{Appendix_varphi_3}, we define 
\begin{equation}
\label{Appendix_psi_3}
\psi_M:=\psi \, \Psi_M \in \mathscr{C}^{\infty}_{\mathrm{c}}(\R^{|k_1| \leq n} \times \R^{|k_2| \leq n} \times \R^{|\ell_1| \leq n} \times \R^{|\ell_2| \leq n})\,,\quad M \in \N\,.
\end{equation}
By \eqref{h_n^I_rewritten_2} and \eqref{pi_n_rewritten}, we can write
\begin{equation}
\label{h_n^I_rewritten_3}
h_n^I=\psi \circ \pi_n\,.
\end{equation}

By using \eqref{finite_moments}, \eqref{h_n^I_rewritten_2}--\eqref{Appendix_psi_3}, the fact that $\widehat{V} \in \ell^{\infty}(\Z^d)$, the dominated convergence theorem, and arguing analogously\footnote{The same argument applies for any polynomial in the variables $a_m,b_m$. We only need this result in two cases, namely for the truncated Wick-ordered mass and for the truncated interaction.} as for \eqref{Integral_n_Malliavin_derivative}, we get that $h_n^I$ is the $L^p(\mu_0)$ limit of $\psi_M \circ \pi_n$ as $M \rightarrow \infty$ and $\nabla h_n^I \in L^p(\mu_0;H^{-s})$ exists and is obtained as the $L^p(\mu_0;H^{-s})$ limit of $\nabla (\psi_M \circ \pi_n)$ as $M \rightarrow \infty$. For fixed $M \in \N$, we can compute $\nabla (\psi_M \circ \pi_n)$ as in \eqref{gradient_cylindrical_F} by \eqref{Appendix_psi_3}. In summary, we obtain that the approximation arguments as for \eqref{Integral_n_Malliavin_derivative} allow us to compute the Malliavin derivative in \eqref{h_n^I_rewritten_1}--\eqref{h_n^I_rewritten_2} by differentiating with respect to the $a_m, b_m$ as if the expressions were smooth compactly supported functions of these variables.

Let us fix $x \in \mathbb{T}^d$. We recall \eqref{pi_n_rewritten} and \eqref{star} to compute for $z \in \mathbb{T}^d$ the quantity 
\begin{multline}
\label{nabla_I_{n,y}}
\sum_{|m| \leq n} \frac{\partial I_{n,x}}{\partial a_{m}} \,\mathrm{e}^{2\pi \mathrm{i} m \cdot z}+\mathrm{i} \sum_{|m| \leq n} \frac{\partial I_{n,x}}{\partial b_{m}} \,\mathrm{e}^{2\pi \mathrm{i} m \cdot z}
\\
=\sum_{|m| \leq n} \sum_{|k| \leq n} \Bigl[(a_k+\mathrm{i}b_k)\mathrm{e}^{2\pi \mathrm{i} (k-m) \cdot x} +(a_k-\mathrm{i} b_k)\, \mathrm{e}^{2\pi \mathrm{i}(m-k) \cdot x}\Bigr]\,\mathrm{e}^{2\pi \mathrm{i} m \cdot z}
\\
+
\mathrm{i}\,\sum_{|m| \leq n} \sum_{|k| \leq n} \Bigl[-\mathrm{i}(a_k+\mathrm{i}b_k)\mathrm{e}^{2\pi \mathrm{i} (k-m) \cdot x} +\mathrm{i}(a_k-\mathrm{i} b_k)\, \mathrm{e}^{2\pi \mathrm{i}(m-k) \cdot x}\Bigr]\,\mathrm{e}^{2\pi \mathrm{i} m \cdot z}
\\
=2 \sum_{|k|,|m| \leq n}(a_k + \mathrm{i} b_k)\, \mathrm{e}^{2\pi \mathrm{i} (k-m) \cdot x}\,\mathrm{e}^{2 \pi \mathrm{i} m \cdot z}=2 \sum_{|m| \leq n} u_n(x)\,\mathrm{e}^{2\pi \mathrm{i} m \cdot (z-x)}\,.
\end{multline} 
For the last equality in \eqref{nabla_I_{n,y}}, we summed in $|k| \leq n$ using 
\begin{equation*}
\sum_{|k| \leq n} (a_k+\mathrm{i}b_k)\,\mathrm{e}^{2\pi \mathrm{i} k \cdot x}=u_n(x)\,.
\end{equation*}
We now use \eqref{h_n^I_rewritten_1} and \eqref{nabla_I_{n,y}} to write for $z \in \T^d$
\begin{multline}
\label{integral_F}
\nabla h_n^I(u)(z)=\sum_{|\ell| \leq n}  \biggl\{\int_{\mathbb{T}^d} u_n(x)\,\bigl(V*\,:|u_n|^2:\,\bigr)(x)\,\mathrm{e}^{-2\pi \mathrm{i} \ell \cdot x}\, dx\biggr\}\,\mathrm{e}^{2 \pi \mathrm{i} \ell \cdot z}
\\
=P_n\Bigl[u_n \,\bigl(V*\,:|u_n|^2:\,\bigr)\Bigr](z)\,.
\end{multline}
We hence prove \eqref{Hartree_equation_Malliavin_derivative_1}.

Using \eqref{Hartree_equation_Malliavin_derivative_1} and Assumption \ref{Assumption_on_V}, we now use Wick's theorem, hypercontractivity, and argue analogously as in the proof of \cite[Proposition 5.1 (2)]{AS21} to deduce that $\nabla h^I \in L^p(\mu_{0})$. As before, the fact that $\widehat{V}$ is no longer assumed to be nonnegative does not change the argument. 
By the limiting procedures summarized above, we always obtain that $h^I:H^{-s} \rightarrow \R$ is Borel measurable.
This completes the proof of Proposition \ref{Hartree_equation_Malliavin_derivative} (i).

\end{proof}

\begin{remark}
\label{formula_for_Malliavin_derivative_Hartree}
The analysis of  \cite[Proposition 5.1 (2)]{AS21} and \cite[Section 3]{Soh19}, based on Wick's theorem shows us that the limit of \eqref{Hartree_equation_Malliavin_derivative_1} in $L^p(\mu_0;H^{-s})$ is explicitly equal to 
\begin{equation}
\label{nabla_h^I_formula}
\bigl(V*\,:|u|^2:\,\bigr)u \equiv \nabla h^I(u)\,.
\end{equation}
\end{remark}

\section{Proof of \eqref{Bourgain_large_deviation_estimate} and conclusion of the proof of Proposition \ref{Hartree_equation_Malliavin_derivative}}
\label{Proof of Bourgain's large deviation estimate}

In this section, our goal is to prove the large deviation estimate given by  \eqref{Bourgain_large_deviation_estimate}. This is the main step in the proof of Proposition \ref{Hartree_equation_Malliavin_derivative} (ii) above. 
We divide the proof into several steps.

\subsection{Rewriting the interacting term in the Hamiltonian} 
\label{Rewriting the interacting term in the Hamiltonian} 
As in \cite[Section 1 (4)]{Bou97}, we first rewrite \eqref{h^I} as 

\begin{equation}
\label{h^I_rewritten}
h^I(u)=\frac{1}{4} \,\sum_{k \in \Z^d \setminus \{0\}} \widehat{|u|^2}(k)\, \widehat{|u|^2}(-k)\,\widehat{V}(k)+\frac{1}{4} \,\widehat{V}(0)\,\Bigl(:\|u\|_{L^2}^2:\Bigr)^2\,.
\end{equation}
In order to show \eqref{h^I_rewritten}, we first show that for $n \in \N$, we can rewrite \eqref{A_h^I} as 
\begin{equation}
\label{h^I_rewritten_n}
h_{n}^I(u)=\frac{1}{4} \,\sum_{k \in \Z^d \setminus \{0\}} \widehat{|u_n|^2}(k)\, \widehat{|u_n|^2}(-k)\,\widehat{V}(k)+\frac{1}{4} \,\widehat{V}(0)\,\Bigl(:\|u_n\|_{L^2}^2:\Bigr)^2\,.
\end{equation}
Here, we recall the convention \eqref{u_n}.

Let us note that \eqref{h^I_rewritten_n} implies \eqref{h^I_rewritten} since 
we know that $h_{n}^I(u)$ and $:\|u_n\|_{L^2}^2:$ converge to $h^I(u)$ and $:\|u\|_{L^2}^2:$ in $L^p(\mu_{0})$ for all $p \in [1,\infty)$ as $n \rightarrow \infty$  by using Lemma \ref{Wick_ordered_mass_lemma} (i) and Proposition \ref{Hartree_equation_Malliavin_derivative} (i) respectively. In particular, we get that the sum
\begin{equation*}
\sum_{k \in \Z^d \setminus \{0\}} \widehat{|u|^2}(k)\, \widehat{|u|^2}(-k)\,\widehat{V}(k)
\end{equation*}
is the $L^p(\mu_{0})$ limit as $n \rightarrow \infty$ of the (finite) sums
\begin{equation*}
\sum_{k \in \Z^d \setminus \{0\}} \widehat{|u_n|^2}(k)\, \widehat{|u_n|^2}(-k)\,\widehat{V}(k)\,.
\end{equation*}
(Alternatively, one can check this claim directly by using the assumption \eqref{V_assumption}).

We now show \eqref{h^I_rewritten_n}. Using \eqref{:|u_n|^2:}, we can write \eqref{A_h^I} as 
\begin{multline}
\label{h^I_rewritten_n_proof}
\frac{1}{4} \int_{\T^d} \int_{\T^d} \bigl(|u_n(x)|^2-\sigma_{n} \bigr) \,V(x-y)\, \bigl(|u_n(y)|^2-\sigma_{n}\bigr)\, dx\, dy
\\
=\frac{1}{4} \int_{\T^d} \int_{\T^d} |u_n(x)|^2 \,V(x-y)\, |u_n(y)|^2\, dx\, dy
\\
+\frac{1}{4}\, \widehat{V}(0)\,\biggl(\sigma_{n}^2-2\sigma_{n}\int_{\T^d} |u_n(x)|^2\,dx \biggr)
\\
=\frac{1}{4}\, \sum_{k \in \Z^d}  \widehat{|u_n|^2}(k)\, \widehat{|u_n|^2}(-k)\,\widehat{V}(k)+\frac{1}{4} \,\widehat{V}(0)\biggl(\sigma_{n}^2-2\sigma_{n} \int_{\T^d} |u_n(x)|^2\,dx \biggr)
\\
=\frac{1}{4}\, \sum_{k \in \Z^d \setminus \{0\}}  \widehat{|u_n|^2}(k)\, \widehat{|u_n|^2}(-k)\,\widehat{V}(k)
\\
+\frac{1}{4}\, \widehat{V}(0)\biggl(\sigma_{n}^2-2\sigma_{n} \int_{\T^d} |u_n(x)|^2\,dx+ \biggl\{\int_{\T^d} |u_n(x)|^2\,dx\biggr\}^2 \biggr)
\\
=\frac{1}{4} \,\sum_{k \in \Z^d \setminus \{0\}} \widehat{|u_n|^2}(k)\, \widehat{|u_n|^2}(-k)\,\widehat{V}(k)+\frac{1}{4} \,\widehat{V}(0)\,\Bigl(:\|u_n\|_{L^2}^2:\Bigr)^2\,.
\end{multline}
Note that for the second equality in \eqref{h^I_rewritten_n_proof}, we used Parseval's theorem, and for the third equality, we considered the terms in the sum with $k \neq 0$ and the term with $k=0$ separately. Finally, for the fourth equality, we recalled \eqref{Integral_n}. We hence obtain \eqref{h^I_rewritten_n}, which in turn implies \eqref{h^I_rewritten}. 

Using \eqref{h^I_rewritten}, we note that \eqref{Bourgain_large_deviation_estimate} follows if we show that, given $C_{\mathrm{Big}}>0$ arbitrarily large, for all $\lambda>1$, we have\footnote{Alternatively, we just need to consider $\lambda>0$ bounded away from zero. We henceforth fix $\lambda>1$ in what follows.}
\begin{multline}
\label{Bourgain_large_deviation_estimate_rewritten}
\mathbb{P}_{\omega} \Biggl(\sum_{k \in \Z^d \setminus \{0\}} \widehat{|u^{\omega}|^2}(k)\, \widehat{|u^{\omega}|^2}(-k)\,\widehat{V}(k)>\lambda\,\,\bigcap \,\, 
\Biggl| \sum_{k \in \Z^d} \frac{|g_k(\omega)|^2-1}{\langle k \rangle^2}\Biggr| \leq \BB \Biggr)
\\
\lesssim_{\BB,C_{\mathrm{Big}}} \ee^{-C_{\mathrm{Big}} \lambda}\,.
\end{multline}
Here, we recall \eqref{map_varphi_omega}. The rest of the appendix is devoted to the proof of \eqref{Bourgain_large_deviation_estimate_rewritten}.

\subsection{Concentration inequalities}
\label{Concentration inequalities}

In this section, we summarize the concentration inequalities that we use in the sequel to prove \eqref{Bourgain_large_deviation_estimate_rewritten}. For a detailed explanation and proofs of the results in this section, we refer the reader to \cite{Ver12} and \cite{Rudelson_Vershynin}, as well as \cite{Boucheron_Lugosi_Massart}.

We say that a random variable $X$ is \emph{subgaussian} if
\begin{equation}
\label{Psi_2}
\|X\|_{\psi_2}:=\mathop{\sup}_{p \geq 1} p^{-\frac{1}{2}} \,\bigl(\mathbb{E} |X|^p\bigr)^{\frac{1}{p}}<\infty\,.
\end{equation}
In particular complex Gaussian random variables are subgaussian. The following estimate holds for subgaussian random variables.
\begin{lemma}[Hoeffding's Inequality]
\label{Hoeffding's_inequality}
Let $(X_{\ell})_{\ell \in \Z}$ be a sequence of independent centred subgaussian random variables such that there exists $M>0$ with the property that $\|X_{\ell}\|_{\psi_2} \leq M$ for all $\ell$. Let $(a_{\ell})_{\ell \in \Z}$ be a sequence of complex numbers. There exists a universal constant $c>0$ such that for all $\varrho \geq 0$, we have
\begin{equation*}
\mathbb{P} \Biggl(\Bigl|\sum_{\ell} a_{\ell} X_{\ell}\Bigr| > \varrho \Biggr) \lesssim \exp\Biggl\{-\frac{c \varrho^2}{M^2 \|a\|_{\ell^2}^2}\Biggr\}\,.
\end{equation*}
\end{lemma}

We say that a random variable $X$ is \emph{subexponential} if
\begin{equation}
\label{Psi_1}
\|X\|_{\psi_1}:=\mathop{\sup}_{p \geq 1} p^{-1} \,\bigl(\mathbb{E} |X|^p\bigr)^{\frac{1}{p}}<\infty\,.
\end{equation}
The following estimate holds for subexponential random variables.
\begin{lemma}[Bernstein's Inequality]
\label{Bernstein's_inequality}
Let $(X_{\ell})_{\ell \in \Z}$ be a sequence of independent centred subexponential random variables such that there exists $M>0$ with the property that $\|X_{\ell}\|_{\psi_1} \leq \Xi$ for all $\ell$. Let $(a_{\ell})_{\ell \in \Z}$ be a sequence of complex numbers. There exists a universal constant $c>0$ such that for all $\varrho \geq 0$, we have
\begin{equation*}
\mathbb{P} \Biggl(\Bigl|\sum_{\ell} a_{\ell} X_{\ell}\Bigr| > \varrho \Biggr) \lesssim \exp\Biggl\{-c \,\mathrm{min} \Biggl(\frac{\varrho^2}{\Xi^2 \|a\|_{\ell^2}^2},\frac{\varrho}{\Xi \|a\|_{\ell^{\infty}}}\Biggr)\Biggr\}\,.
\end{equation*}
\end{lemma}
We can deduce Lemmas \ref{Hoeffding's_inequality}--\ref{Bernstein's_inequality} from \cite[Proposition 5.10]{Ver12} and \cite[Proposition 5.16]{Ver12}. The aforementioned results are stated for finite sums, but can be extended to the form above by the same proofs. We also refer the reader to \cite[Sections 2.6 and 2.8]{Boucheron_Lugosi_Massart} for further details.

Let us note one further concentration inequality.

\begin{lemma}[The Hanson-Wright Inequality]
\label{Hanson_Wright_inequality}
Let $X=(X_1,\ldots,X_n)$ be a random vector whose entries are independent centred subgaussian random variables.
Denote by $\Xi:= \max_{1 \leq \ell \leq n} \|X_{\ell}\|_{\psi_2}$. Let $\mathcal{A}$ be an $n \times n$ matrix. There exists a uniform constant such that for all $\varrho \geq 0$, we have
\begin{equation*}
\mathbb{P} \biggl(\Bigl|\overline{X}^T\cdot \mathcal{A} \cdot X - \mathbb{E} \Bigl[\overline{X}^T\cdot \mathcal{A} \cdot X \Bigr]\Bigr|>\varrho\biggr) \lesssim  \exp\Biggl\{-c\,\min\ \Biggl(\frac{\varrho^2}{\Xi^4\|\mathcal{A}\|_{\mathrm{HS}}^2},\frac{\varrho}{\Xi\|\mathcal{A}\|}\Biggr)\Biggr\}\,.
\end{equation*}
Above,  $\|\mathcal{A}\|_{\mathrm{HS}}=\sqrt{\mathrm{Tr}(\mathcal{A}^*\mathcal{A})}$ denotes the Hilbert-Schmidt (or Frobenius) norm and $\|\mathcal{A}\|$ denotes the operator norm of $\mathcal{A}$.
\end{lemma}
Lemma \ref{Hanson_Wright_inequality} in the above form can be deduced from \cite[Theorem 1.1]{Rudelson_Vershynin}
and \cite[Section 3.1]{Rudelson_Vershynin}.

\subsection{Proof of \eqref{Bourgain_large_deviation_estimate_rewritten} when $d=3$}
In this section, we prove \eqref{Bourgain_large_deviation_estimate_rewritten} when $d=3$. In Section \ref{Appendix_2D_proof} below, we explain how to modify the arguments to prove \eqref{Bourgain_large_deviation_estimate_rewritten} when $d=2$.

Let us henceforth fix $d=3$. In all of the sums below, $k$ is an element of $\Z^3$. Note that there exists $\bb>0$ such that for all $K>0$, we have
\begin{equation}
\label{C_0_choice}
\sum_{|k| \leq K} \frac{1}{\langle k \rangle^2} \leq \bb K\,.
\end{equation}
With $\bb$ as in \eqref{C_0_choice}, we consider 
\begin{equation}
\label{C_1_choice}
\bbb \geq \bb+\BB\,.
\end{equation}

In Step 1, we show a concentration estimate with a truncated Wick-ordered mass. The reason for the precise choice of parameters will become apparent in Steps 5 and 6 below.
With $\bbb>0$ as in \eqref{C_1_choice}, we show that there exists $\bbbb \sim \bbb$ such that for all $\lambda>0$, we have
\begin{multline}
\label{33}
\mathbb{P}_{\omega} \Biggl(\sum_{|k| \leq M} \frac{|g_k(\omega)|^2}{\langle k \rangle^2} >2\bbb\, \mathrm{max}\{M,\lambda^{\frac{1}{3}}\}\,\,\bigcap \,\, 
\Biggl| \sum_{k \in \Z^3} \frac{|g_k(\omega)|^2-1}{\langle k \rangle^2}\Biggr| \leq \BB \Biggr)
\\
\lesssim_{\BB} \ee^{-\bbbb\,\mathrm{max}\{M^3,\lambda\}}\,.
\end{multline}
Estimates of the type \eqref{33} will be used in the proof of  \eqref{Bourgain_large_deviation_estimate_rewritten}. Moreover, from \eqref{33}, we note that we need to consider two different regimes depending on which of the terms $M^3$ and $\lambda$ is larger (which are studied in Steps 5 and 6 respectively).
\paragraph{\textbf{Step 1: Proof of \eqref{33}.}}
We now show \eqref{33}. When estimating this probability, we consider $\omega \in \Omega$ such that 
\begin{equation}
\label{29}
\Biggl|\sum_{m \in \Z^3} \frac{|g_m(\omega)|^2-1}{\langle m \rangle^2}\Biggr| \leq \BB\,, 
\end{equation}
Let $\tilde{g}_k \sim_d g_k$ be independent copies of the Gaussian random variables $g_k$. In particular, the $\tilde{g}_k$ are independent and no longer satisfy the constraint given by  \eqref{29}. For these $\tilde{g}_k$, we note that
\begin{equation}
\label{31}
\mathbb{P}_{\omega} \Biggl(\Biggl|\sum_{|k| > M} \frac{|\tilde{g}_k(\omega)|^2-1}{\langle k \rangle^2}\Biggr| \geq \bbb M \Biggr) \lesssim \ee^{-\bbbb M^3} \quad \text{for } \bbbb \sim \bbb\,.
\end{equation}
We obtain estimate \eqref{31} by using Lemma \ref{Bernstein's_inequality}. More precisely, we note that $|\tilde{g}_k|^2-1$ are independent, centred subexponential random variables, i.e. they satisfy
\begin{equation}
\label{Psi_1_bound}
\bigl\||\tilde{g}_k|^2-1\bigr\|_{\psi_1} \lesssim 1\,,
\end{equation}
uniformly in $k$, where we recall \eqref{Psi_1} (see also \cite[Lemma 5.14]{Ver12}).
We consider the sequence $(b_k)$ given by 
\begin{equation}
\label{Sequence_bk}
b_k:=\frac{1}{\langle k \rangle^2}\,\mathbbm{1}_{|k|>M}\,.
\end{equation}
From \eqref{Sequence_bk}, we deduce that 
\begin{equation}
\label{Sequence_bk_l2}
\|b\|_{\ell^2}^2=\sum_{|k|>M} \frac{1}{\langle k \rangle^4} \sim \int_{|x|>M} \frac{1}{|x|^4}\, dx \sim \frac{1}{M}
\end{equation}
and
\begin{equation}
\label{Sequence_bk_linfty}
\|b\|_{\ell^{\infty}}^2 \sim \frac{1}{M^2}\,.
\end{equation}
Using \eqref{Psi_1_bound} and \eqref{Sequence_bk_l2}--\eqref{Sequence_bk_linfty} in Lemma \ref{Bernstein's_inequality}, we deduce that for a universal constant $c>0$
\begin{equation}
\label{31b}
\mathbb{P}_{\omega} \Biggl(\Biggl|\sum_{|k| > M} \frac{|\tilde{g}_k|^2-1}{\langle k \rangle^2}\Biggr| \geq \bbb M \Biggr) \lesssim \ee^{-c\, \mathrm{min} \{(\bbb M)^2\cdot M, (\bbb M) \cdot M^2\}}=\ee^{-c \bbb M^3}\,,
\end{equation}
which implies \eqref{31}.
In \eqref{31b}, we assumed that $\bbb \geq 1$, which we can do from \eqref{C_1_choice}, by considering $\BB \geq 1$. We make this assumption in the sequel.

We now explain how we can deduce \eqref{33} from a variant of \eqref{31}. Let us consider the event that
\begin{equation}
\label{33_event}
\sum_{|k| \leq M} \frac{|g_k(\omega)|^2}{\langle k \rangle^2} >2\bbb\, \mathrm{max}\{M,\lambda^{\frac{1}{3}}\}\,\,\, \text{and } \,\, 
\Biggl| \sum_{k \in \Z^3} \frac{|g_k(\omega)|^2-1}{\langle k \rangle^2}\Biggr| \leq \BB\,.
\end{equation}
 By \eqref{C_0_choice} and \eqref{29}, followed by \eqref{C_1_choice} and the fact that $\lambda>1$, we have that for $\omega$ satisfying \eqref{33_event}, the following estimate holds.
\begin{multline}
\label{30}
\sum_{|k| \leq \mathrm{max}\{M,\lambda^{1/3}\}} \frac{|g_k(\omega)|^2}{\langle k \rangle^2}=\sum_{|k| \leq \mathrm{max}\{M,\lambda^{1/3}\}} \frac{1}{\langle k \rangle^2}+\sum_{|k| \leq \mathrm{max}\{M,\lambda^{1/3}\}} \frac{|g_k(\omega)|^2-1}{\langle k \rangle^2} 
\\
\leq \bb \,\mathrm{max}\{M,\lambda^{1/3}\} + \BB-\sum_{|k|>\mathrm{max}\{M,\lambda^{1/3}\}} \frac{|g_k(\omega)|^2-1}{\langle k \rangle^2}
\\
\leq \bbb \,\mathrm{max}\{M,\lambda^{1/3}\} -\sum_{|k|>\mathrm{max}\{M,\lambda^{1/3}\}} \frac{|g_k(\omega)|^2-1}{\langle k \rangle^2}
\,.
\end{multline}
Moreover for $\omega$ satisfying \eqref{33_event}, we have that 
\begin{equation}
\label{33_b} 
2\bbb \,\mathrm{max}\{M,\lambda^{\frac{1}{3}}\}<\sum_{|k| \leq \mathrm{max}\{M,\lambda^{1/3}\}} \frac{|g_k(\omega)|^2}{\langle k \rangle^2}\,.
\end{equation}
Combining \eqref{30} and \eqref{33_b}, it follows that for $\omega$ as in \eqref{33_event}, we have
\begin{equation}
\label{33_event_b}
\sum_{|k|>\mathrm{max}\{M,\lambda^{1/3}\}} \frac{|g_k(\omega)|^2-1}{\langle k \rangle^2}<-\bbb \,\mathrm{max}\{M,\lambda^{\frac{1}{3}}\}\,.
\end{equation}
From \eqref{31} and \eqref{33_event_b} (with suitably modified notation), we deduce that the left-hand side of \eqref{33} is
\begin{equation}
\label{33_proof}
\leq \mathbb{P}_{\omega} \Biggl(\Biggl|\sum_{|k|>\mathrm{max}\{M,\lambda^{1/3}\}} \frac{|g_k(\omega)|^2-1}{\langle k \rangle^2}\Biggr|>\bbb \,\mathrm{max}\{M,\lambda^{\frac{1}{3}}\}\Biggr) \lesssim \ee^{-\bbbb\,\mathrm{max}\{M^3,\lambda\}}
\end{equation}
for suitable $\bbbb \sim \bbb$. It is important to note that in \eqref{33_proof}, we can treat the $g_k$ as independent random variables. We hence deduce \eqref{33}.

\paragraph{\textbf{Step 2: A dyadic decomposition.}} We now perform a dyadic decomposition in \eqref{Bourgain_large_deviation_estimate_rewritten}. Once we do this, we estimate each dyadic piece separately, using suitable concentration inequalities. The precise estimates are done in Steps 5 and 6 below. The regimes studied in these cases are determined by inequality \eqref{33} from Step 1.

In what follows, we write $M=2^\ell, \ell \in \N_0$ for a dyadic integer. Moreover, for $k \in \Z^3$, we write $|k| \simeq M$ if 
\begin{equation}
\label{k_sim_M}
M/2 \leq |k| <M\,. 
\end{equation}
Finally, with $\lambda>1$ fixed as before, we consider a sequence $(\lambda_{1,M})_M$ of positive numbers indexed by dyadic integers $M$ of the form
\begin{equation}
\label{lambda_1M_choice}
\lambda_{1,M}= \frac{C\lambda}{(\log M)^2}\,, 
\end{equation}
where the constant $C>0$ in \eqref{lambda_1M_choice} is chosen such that
\begin{equation}
\label{lambda_1M_condition}
\sum_{M\, \mathrm{dyadic}} \lambda_{1,M} = \lambda\,.
\end{equation}
From \eqref{lambda_1M_condition} and a union bound, it follows that the left-hand side of \eqref{Bourgain_large_deviation_estimate_rewritten} is 
\begin{multline}
\label{35b}
\leq \sum_{M\,\mathrm{dyadic}} \mathbb{P}_{\omega} \Biggl(\sum_{|k| \simeq M} \widehat{|u^{\omega}|^2}(k)\, \widehat{|u^{\omega}|^2}(-k)\,\widehat{V}(k)>\lambda_{1,M}
\\
\bigcap \,\, 
\Biggl| \sum_{m \in \Z^3} \frac{|g_m(\omega)|^2-1}{\langle m \rangle^2}\Biggr| \leq \BB \Biggr)\,.
\end{multline}
For $M$ a dyadic integer, let us write 
\begin{equation}
\label{v_M_definition}
v_M:=\mathop{\mathrm{max}}_{|k| \simeq M} |\widehat{V}(k)|\,.
\end{equation}
By Assumption \ref{Assumption_on_V}, we have 
\begin{equation}
\label{36}
v_M \lesssim M^{-2-\varepsilon}\,.
\end{equation}
By \eqref{36} and the pigenhole principle, for a fixed dyadic integer $M$, the condition
\begin{equation}
\label{36b}
\sum_{|k| \simeq M} \widehat{|u^{\omega}|^2}(k)\, \widehat{|u^{\omega}|^2}(-k)\,\widehat{V}(k)>\lambda_{1,M}
\end{equation}
implies that there exists $k \in \Z^3$ with $|k| \simeq M$ such that 
\begin{equation}
\label{37}
\bigl|\widehat{|u^{\omega}|^2}(k)\bigr| \geq \biggl(\frac{\bbo \lambda_{1,M}}{M^3 v_M}\biggr)^{\frac{1}{2}}
\end{equation}
for some constant $\bbo>0$ (depending on the implied constant in \eqref{36}).
Therefore, 
\begin{multline}
\label{37b}
\mathbb{P}_{\omega} \Biggl(\sum_{|k| \simeq M} \widehat{|u^{\omega}|^2}(k)\, \widehat{|u^{\omega}|^2}(-k)\,\widehat{V}(k)>\lambda_{1,M}
\\
\bigcap \,\, 
\Biggl| \sum_{k \in \Z^3} \frac{|g_k(\omega)|^2-1}{\langle k \rangle^2}\Biggr| \leq \BB \Biggr)
\\
\leq \sum_{|k| \simeq M} \mathbb{P}_{\omega} \Biggl(\bigl|\widehat{|u^{\omega}|^2}(k)\bigr| \geq \biggl(\frac{\bbo \lambda_{1,M}}{M^3 v_M}\biggr)^{\frac{1}{2}}\,\,\bigcap \,\, 
\Biggl| \sum_{m \in \Z^3} \frac{|g_m(\omega)|^2-1}{\langle m \rangle^2}\Biggr| \leq \BB \Biggr)\,.
\end{multline}
From \eqref{35b} and \eqref{37b}, it follows that given $M$ a dyadic integer and $k \in \Z^3$ with $|k| \simeq M$, we need to estimate
\begin{equation}
\label{38}
\mathbb{P}_{\omega} \Biggl(\bigl|\widehat{|u^{\omega}|^2}(k)\bigr| \geq \biggl(\frac{\bbo \lambda_{1,M}}{M^3 v_M}\biggr)^{\frac{1}{2}}\,\,\bigcap \,\, 
\Biggl| \sum_{m \in \Z^3} \frac{|g_m(\omega)|^2-1}{\langle m \rangle^2}\Biggr| \leq \BB \Biggr)\,.
\end{equation}
Let us first recall \eqref{map_varphi_omega} and write
\begin{equation}
\label{38b}
\widehat{|u^{\omega}|^2}(k)=\sum_{j \in \Z^3} \frac{g_j(\omega)\,\overline{g_{j-k}(\omega)}}{\langle j \rangle \,\langle j-k \rangle}\,.
\end{equation}
We note that, in the sum \eqref{38b} we have by \eqref{k_sim_M} that
\begin{equation}
\label{38b2}
\mathrm{max} \{|j|,|j-k|\} \geq \frac{|k|}{2} \geq \frac{M}{4}\,.
\end{equation}
From \eqref{38b}--\eqref{38b2}, it follows that for $k \in \Z^3$ with $|k| \simeq M$, we have
\begin{equation}
\label{39}
\bigl|\widehat{|u^{\omega}|^2}(k)\bigr| \leq \mathop{\sum_{K \geq M/2}}_{K\,\mathrm{dyadic}} 
\Biggl|\sum_{j:\, \mathrm{max}\{|j|,|j-k|\} \simeq K} \frac{g_j(\omega)\,\overline{g_{j-k}(\omega)}}{\langle j \rangle\,\langle j-k \rangle}\Biggr|\,.
\end{equation}
Let us now fix $\lambda_2>0$ and a dyadic integer\footnote{This is a slight abuse of the notation given in earlier sections, but it is more consistent with \cite{Bou97} and the rest of the appendix. We emphasize that, throughout the appendix, $\lambda_2$ is a positive parameter and not an eigenvalue of $A$.}
\begin{equation}
\label{K_and_M_condition}
K \geq M/2\,.
\end{equation}

Our goal is to estimate
\begin{equation}
\label{39b}
\mathbb{P}_{\omega} \Biggl(\Biggl|\sum_{j:\, \mathrm{max}\{|j|,|j-k|\} \simeq K} \frac{g_j(\omega)\,\overline{g_{j-k}(\omega)}}{\langle j \rangle\,\langle j-k \rangle}\Biggr|>\lambda_2\,\,\bigcap \,\, 
\Biggl| \sum_{m \in \Z^3} \frac{|g_m(\omega)|^2-1}{\langle m \rangle^2}\Biggr| \leq \BB \Biggr)\,.
\end{equation}
We recall that throughout $|k| \simeq M$.
Let us show that 
\begin{equation}
\label{53}
\eqref{39b} \lesssim_{\bboo} \ee^{-\bboo K^3}+ \ee^{-\bbooo \lambda_2^2 K}\,,
\end{equation}
where $\bboo>0$ can be chosen arbitrarily large\footnote{with lower threshold depending on $\BB$; see \eqref{48b} and \eqref{48b2} below.} and 
\begin{equation}
\label{c_1_c_2}
\bbooo \sim \frac{1}{\bboo}\,.
\end{equation}
\paragraph{\textbf{Step 3: Proof of \eqref{53}}}

We now show \eqref{53}. Let us first note that by symmetry, it suffices to consider the contributions to \eqref{39b} where $|j| \geq |j-k|$ and $|j| \simeq K$. The contributions where $|j-k| \geq |j|$ and $|j-k| \simeq K$ are estimated analogously. We consider two cases depending on the relative sizes of $K$ and $M$ (recalling \eqref{K_and_M_condition}).
\paragraph{\textbf{\emph{Case 1:} $M \leq K/4$}}

In this case, we will show the required bound if we drop the condition \eqref{29}. In other words, we estimate 
\begin{equation}
\label{40}
\mathbb{P}_{\omega} \Biggl(\Biggl|\sum_{|j| \simeq K\,,|j| \geq |j-k|} \frac{g_j(\omega)\,\overline{g_{j-k}(\omega)}}{\langle j \rangle\,\langle j-k \rangle}\Biggr|>\lambda_2
\Biggr)\,.
\end{equation}
In \eqref{40}, we are considering 
\begin{equation*}
|k| \simeq M \leq 2K\,, \quad |j| \simeq K\,,
\end{equation*}
which implies that 
\begin{equation}
\label{40b}
|k|<M\,.
\end{equation}
Here, we recalled the convention \eqref{k_sim_M}.  Note that \eqref{40b} implies 
\begin{equation}
\label{40c}
4|k|<4M\leq K \leq 2|j|\,, 
\end{equation}
from where we deduce $|k|<\frac{1}{2}|j|$. Hence, by the triangle inequality we have
\begin{equation}
\label{41}
|j-k| > \frac{1}{2}|j| \geq \frac{1}{4}K\,.
\end{equation}
Before proceeding, we note a general fact about matrices. Given a matrix $\mathcal{A}=(\mathcal{A}_{ij})$ its operator norm $\|\mathcal{A}\|$ satisfies 
\begin{equation}
\label{matrix_bound}
\|\mathcal{A}\| \leq \sqrt{\mathcal{A}_{\mathrm{row}} \mathcal{A}_{\mathrm{col}}}\,,
\end{equation}
where 
\begin{equation}
\label{A_row_col}
\mathcal{A}_{\mathrm{row}}:=\mathop{\mathrm{sup}}_{i} \sum_j |\mathcal{A}_{ij}|\,,\quad \mathcal{A}_{\mathrm{col}}:=\mathop{\mathrm{sup}}_{j} \sum_i |\mathcal{A}_{ij}|\,.
\end{equation}
We note that \eqref{matrix_bound} follows from the Cauchy-Schwarz inequality. More precisely, we have
\begin{multline*}
\Biggl(\sum_i \bigl|\sum_{j} \mathcal{A}_{ij}v_j\bigr|^2\Biggr)^{1/2} \leq \Biggl(\sum_i \Biggl[\sum_j |\mathcal{A}_{ij}|\Biggr]\Biggl[\sum_j |\mathcal{A}_{ij}| |v_j|^2\Biggr]\Biggr)^{1/2}
\\
\leq \sqrt{\mathcal{A}_{\mathrm{row}}}\, \Biggl(\sum_j \sum_i |\mathcal{A}_{ij}|\,|v_j|^2 \Biggr)^{1/2} \leq 
\sqrt{\mathcal{A}_{\mathrm{row}}}\,\sqrt{\mathcal{A}_{\mathrm{col}}}\,\Biggl(\sum_{j} |v_j|^2\Biggr)^{1/2}\,,
\end{multline*}
which implies \eqref{matrix_bound}.

We now consider the matrix $\mathcal{A}$ given by
\begin{equation}
\label{43}
\mathcal{A}_{ij}:=\frac{1}{\langle j \rangle \langle j-k \rangle}\, \delta_{i,j-k}\,\mathbbm{1}_{|j| \simeq K}\, \mathbbm{1}_{|j| \geq |j-k|}\,.
\end{equation}
Here, $\delta_{\cdot,\cdot}$ denotes the Kronecker delta function. We can consider \eqref{43} as a $cK \times cK$ matrix for some (universal) constant $c>0$. Using $|j| \simeq K$ and \eqref{41} we have 
\begin{equation}
\label{43b}
|j| \sim |j-k| \sim K\,.
\end{equation}
Using \eqref{43b} and recalling \eqref{A_row_col}, we have 
\begin{equation}
\label{43c}
\mathcal{A}_{\mathrm{row}} \lesssim \frac{1}{K^2}\,,\quad \mathcal{A}_{\mathrm{col}} \lesssim \frac{1}{K^2}\,.
\end{equation}
Combining \eqref{matrix_bound} and \eqref{43c}, it follows that 
\begin{equation}
\label{44}
\|\mathcal{A}\| \lesssim \frac{1}{K^2}\,.
\end{equation}
Using \eqref{43b}, we have
\begin{equation}
\label{45}
\|\mathcal{A}\|_{\mathrm{HS}}^2 \lesssim \sum_{|j| \simeq K} \frac{1}{K^4} \lesssim \frac{1}{K}\,.
\end{equation}
Let us now write
\begin{equation}
\label{45b}
\sum_{|j| \simeq K\,,|j| \geq |j-k|} \frac{g_j(\omega)\,\overline{g_{j-k}(\omega)}}{\langle j \rangle\,\langle j-k \rangle}
=\overline{X}^T\cdot \mathcal{A} \cdot X\,,
\end{equation}
where $\mathcal{A}$ is the matrix given by \eqref{43} and $X$ is a random vector $X=(g_{\ell})$ for a suitably chosen finite set of indices $\ell$.
Since $k \neq 0$, we have that 
\begin{equation}
\label{45c}
\mathbb{E} \Bigl[\overline{X}^T\cdot \mathcal{A} \cdot X \Bigr]=0\,.
\end{equation}
From \eqref{45b}--\eqref{45c}, we have that
\begin{equation}
\label{45d}
\eqref{40} \leq \mathbb{P}_{\omega} \biggl(\Bigl|\overline{X}^T\cdot \mathcal{A} \cdot X - \mathbb{E} \Bigl[\overline{X}^T\cdot \mathcal{A} \cdot X \Bigr]\Bigr|>\lambda_2\biggr)\,.
\end{equation}
By Lemma \ref{Hanson_Wright_inequality}, \eqref{44}, and \eqref{45}, we deduce that 
\begin{equation}
\label{46}
\eqref{45d} \lesssim \mathrm{exp} \biggl[ - \tilde{c} \,\mathrm{min} \biggl\{\frac{\lambda_2^2}{\|A\|_{\mathrm{HS}}^2}, \frac{\lambda_2}{\|A\|}\biggr\}\biggr] \leq \mathrm{exp} \Bigl[ - c \,\mathrm{min} \{\lambda_2^2 K, \lambda_2 K^2\}\Bigr]\,,
\end{equation}
for some universal constants $\tilde{c}, c>0$.

We now rewrite the upper bound \eqref{46}. Let $C_{\mathrm{large}} \gg 1$ denote a large constant\footnote{not to be confused with $C_{\mathrm{Big}}$ from \eqref{Bourgain_large_deviation_estimate_rewritten}.}.

\begin{itemize}
\item If $\lambda_2 \geq C_{\mathrm{large}}\, K$, then $\mathrm{min} \{\lambda_2^2 K, \lambda_2 K^2\}=\lambda_2K^2 \geq C_{\mathrm{large}}\,K^3.$
\item If $\lambda_2<C_{\mathrm{large}}\,K$, then $\mathrm{min} \{\lambda_2^2 K, \lambda_2 K^2\}>\frac{\lambda_2^2K}{C_{\mathrm{large}}}$. The latter inequality holds automatically if $\lambda_2^2K  \leq \lambda_2K^2$. Otherwise, we note that $\lambda_2 K^2>\frac{\lambda_2^2 K}{C_{\mathrm{large}}}$.
\end{itemize}
In particular, from \eqref{46}, we obtain that
\begin{equation}
\label{47}
\mathbb{P}_{\omega} \Biggl(\Biggl|\sum_{|j| \simeq K\,,|j| \geq |j-k|} \frac{g_j(\omega)\,\overline{g_{j-k}(\omega)}}{\langle j \rangle\,\langle j-k \rangle}\Biggr|>\lambda_2 \Biggr) \lesssim_{\bboo} \ee^{-\bboo K^3}+\ee^{-\bbooo \lambda_2^2 K}\,,
\end{equation}
where $\bboo>0$ can be chosen arbitrarily large and $\bbooo \sim \frac{1}{\bboo}$.

\paragraph{\textbf{\emph{Case 2:} $M > K/4$}} 
Recalling that $K \geq M/2$, we that in this case we have $M \sim K$. Hence,
\begin{equation}
\label{41b}
|j-k| \leq |j|+|k| \lesssim K + M \lesssim K\,.
\end{equation}
In particular, in this case we no longer have \eqref{41}, which was used in the earlier calculations. 
We write
\begin{multline}
\label{48}
\mathbb{P}_{\omega} \Biggl(\Biggl|\sum_{|j| \simeq K\,,\,|j| \geq |j-k|} \frac{g_j(\omega)\,\overline{g_{j-k}(\omega)}}{\langle j \rangle\,\langle j-k \rangle}\Biggr|>\lambda_2\,\,\bigcap \,\, 
\Biggl| \sum_{m \in \Z^3} \frac{|g_m(\omega)|^2-1}{\langle m \rangle^2}\Biggr| \leq \BB 
\Biggr) \leq
\\
\mathbb{P}_{\omega} \Biggl(\Biggl|\sum_{|j| \simeq K\,,\,|j| \geq |j-k|>K/4} \frac{g_j(\omega)\,\overline{g_{j-k}(\omega)}}{\langle j \rangle\,\langle j-k \rangle}\Biggr|>\frac{\lambda_2}{2}\,\,\bigcap \,\, 
\Biggl| \sum_{m \in \Z^3} \frac{|g_m(\omega)|^2-1}{\langle m \rangle^2}\Biggr| \leq \BB
\Biggr)
\\
+\mathbb{P}_{\omega} \Biggl(\Biggl|\sum_{|j| \simeq K\,,\,|j| \geq |j-k|\,,\,|j-k| \leq K/4} \frac{g_j(\omega)\,\overline{g_{j-k}(\omega)}}{\langle j \rangle\,\langle j-k \rangle}\Biggr|>\frac{\lambda_2}{2}\,\,
\\
\bigcap \,\, \Biggl| \sum_{m \in \Z^3} \frac{|g_m(\omega)|^2-1}{\langle m \rangle^2}\Biggr| \leq \BB
\Biggr)=:I+II\,.
\end{multline}

The term $I$ in \eqref{48} is estimated analogously as for \eqref{47} (because $|j-k| \gtrsim K$). We now estimate the term $II$ in \eqref{48}. Now we have to use the truncation of the Wick-ordered mass condition \eqref{29}. 

Before proceeding, let us first note that
\begin{equation}
\label{50}
\mathbb{P}_{\omega} \Biggl(\sum_{|\ell| \leq K/4} \frac{|g_{\ell}(\omega)|^2}{\langle \ell \rangle^2} \geq \bbbbb K\, \,\,\bigcap \,\, 
\Biggl| \sum_{m \in \Z^3} \frac{|g_m(\omega)|^2-1}{\langle m \rangle^2}\Biggr| \leq \BB \Biggr)
\lesssim \ee^{-\bbbbbb K^3}
\end{equation}
with
\begin{equation}
\label{48b}
\bbbbb>2\bb+4\BB
\end{equation}
for $\bb$ as in \eqref{C_0_choice} above and
\begin{equation}
\label{C_4_choice}
\bbbbbb \sim \bbbbb\,.
\end{equation} 
In order to show \eqref{50}, suppose that 
\begin{equation}
\label{49}
\sum_{|\ell| \leq K/4} \frac{|g_{\ell}(\omega)|^2}{\langle \ell \rangle^2} \geq \bbbbb K \quad  \text{and }\quad \Biggl| \sum_{m \in \Z^3} \frac{|g_m(\omega)|^2-1}{\langle m \rangle^2}\Biggr| \leq \BB\,.
\end{equation}
Then, we have 
\begin{equation}
\label{49b}
\sum_{|\ell| \leq K/4} \frac{|g_{\ell}(\omega)|^2-1}{\langle \ell \rangle^2} \geq \bbbbb K-\bb K > \frac{1}{2}\bbbbb K\,.
\end{equation}
Here, we recalled \eqref{C_0_choice} and \eqref{48b}. Using \eqref{48b} and \eqref{49}--\eqref{49b}, we deduce that
\begin{equation}
\label{49c}
\Biggl|\sum_{|\ell| > K/4} \frac{|g_{\ell}(\omega)|^2-1}{\langle \ell \rangle^2}\Biggr|  \gtrsim \bbbbb K\,.
\end{equation}
From \eqref{49c}, it follows that 
\begin{multline*}
\mathbb{P}_{\omega} \Biggl(\sum_{|\ell| \leq K/4} \frac{|g_{\ell}(\omega)|^2}{\langle \ell \rangle^2} \geq \bbbbb K\, \,\,\bigcap \,\, 
\Biggl| \sum_{m \in \Z^3} \frac{|g_m(\omega)|^2-1}{\langle m \rangle^2}\Biggr| \leq \BB \Biggr)
\\
\leq \mathbb{P}_{\omega} \Biggl(\Biggl|\sum_{|\ell| > K/4} \frac{|g_{\ell}(\omega)|^2-1}{\langle \ell \rangle^2}\Biggr|  \gtrsim \bbbbb K\Biggr) \lesssim \ee^{-\bbbbbb K^3}
\end{multline*}
with $C_4$ as in \eqref{C_4_choice}.
For the last inequality, we used \eqref{31}. We hence obtain \eqref{50}. In the sequel, we apply the following more general result which follows from the same proof.
\begin{equation}
\label{51}
\mathbb{P}_{\omega} \Biggl(\sum_{\ell \in S} \frac{|g_{\ell}(\omega)|^2}{\langle \ell \rangle^2} \geq \bbbbb K\, \,\,\bigcap \,\, 
\Biggl| \sum_{m \in \Z^3} \frac{|g_m(\omega)|^2-1}{\langle m \rangle^2}\Biggr| \leq \BB \Biggr)
\lesssim \ee^{-\bbbbbb K^3}\,,
\end{equation}
whenever $S \subset \Z^3$ is a set contained in the ball $|\ell| \leq K/4$.

We now bound $II$ from \eqref{48} as follows.
\begin{multline}
\label{52}
II \leq \mathbb{P}_{\omega} \Biggl(\Biggl|\sum_{|j| \simeq K\,,\,|j| \geq |j-k|\,,\,|j-k| \leq K/4} \frac{g_j(\omega)\,\overline{g_{j-k}(\omega)}}{\langle j \rangle\,\langle j-k \rangle}\Biggr|>\frac{\lambda_2}{2}\,\, 
\bigcap\,\,
\\
\sum_{|j| \simeq K\,,\,|j| \geq |j-k|\,,\,|j-k| \leq K/4} \frac{|g_{j-k}(\omega)|^2}{\langle j-k\rangle^2} \geq \bbbbb K\,\, \bigcap \,\, \Biggl| \sum_{m \in \Z^3} \frac{|g_m(\omega)|^2-1}{\langle m \rangle^2}\Biggr| \leq \BB
\Biggr)
\\
+ \mathbb{P}_{\omega} \Biggl(\Biggl|\sum_{|j| \simeq K\,,\,|j| \geq |j-k|\,,\,|j-k| \leq K/4} \frac{g_j(\omega)\,\overline{g_{j-k}(\omega)}}{\langle j \rangle\,\langle j-k \rangle}\Biggr|>\frac{\lambda_2}{2}\,\,
\bigcap \,\,
\\
\sum_{|j| \simeq K\,,\,|j| \geq |j-k|\,,\,|j-k| \leq K/4} \frac{|g_{j-k}(\omega)|^2}{\langle j-k\rangle^2} < \bbbbb K\,\,
\bigcap \,\, \Biggl| \sum_{m \in \Z^3} \frac{|g_m(\omega)|^2-1}{\langle m \rangle^2}\Biggr| \leq \BB
\Biggr)
\\=:II_1+II_2\,.
\end{multline}
By \eqref{51}, we have that
\begin{multline}
\label{II_1_bound}
II_1 \leq  \mathbb{P}_{\omega} \Biggl(\sum_{|j| \simeq K\,,\,|j| \geq |j-k|\,,\,|j-k| \leq K/4} \frac{|g_{j-k}(\omega)|^2}{\langle j-k\rangle^2} \geq \bbbbb K
\\
\,\, \bigcap \,\, \Biggl| \sum_{m \in \Z^3} \frac{|g_m(\omega)|^2-1}{\langle m \rangle^2}\Biggr| \leq \BB
\Biggr) \lesssim \ee^{-\bbbbbb K^3}\,.
\end{multline}

In order to estimate $II_2$, we need to use a different argument. The first observation is that for all $|j_1|, |j_2| \simeq K$ we have $j_1 \neq j_2-k$. This is true by the triangle inequality because
\begin{equation*}
|j_1-j_2| \leq |j_1|+|j_2| < 2K \leq M/2 \leq |k|\,.
\end{equation*}
Here, we recalled the convention \eqref{k_sim_M}. In particular, it follows that all of the random variables
\begin{equation*}
(g_j)_{|j| \simeq K}\,,\quad (g_{j-k})_{|j| \simeq K}
\end{equation*}
are pairwise independent.
We note that
\begin{multline}
\label{II_2_bound}
II_2 \leq \mathbb{P}_{\omega} \Biggl(\Biggl|\sum_{|j| \simeq K\,,\,|j| \geq |j-k|\,,\,|j-k| \leq K/4} \frac{g_j(\omega)\,\overline{g_{j-k}(\omega)}}{\langle j \rangle\,\langle j-k \rangle}\Biggr|>\frac{\lambda_2}{2}\,\,
\bigcap \,\,
\\
\sum_{|j| \simeq K\,,\,|j| \geq |j-k|\,,\,|j-k| \leq K/4} \frac{|g_{j-k}(\omega)|^2}{\langle j-k\rangle^2} < \bbbbb K
\Biggr)
\\
=\mathbb{P}_{\omega} \Biggl(\Biggl|\sum_{j \in \Z^3}a_j X_j\Biggr|>\frac{\lambda_2}{2}\,\,\bigcap\,\, \sum_{j \in \Z^3} |a_j|^2 <\bbbbb K
\Biggr)\,,
\end{multline}
where 
\begin{equation}
\label{X_j_definition}
X_j:=\frac{g_j(\omega)}{\langle j \rangle}\,\mathbbm{1}_{|j| \simeq K}\,\mathbbm{1}_{|j| \geq |j-k|}\,\mathbbm{1}_{|j-k| \leq K/4}
\end{equation}
and
\begin{equation}
\label{a_j_definition}
a_j:=\frac{\overline{g_{j-k}(\omega)}}{\langle j-k \rangle}\,\mathbbm{1}_{|j| \simeq K}\,\mathbbm{1}_{|j| \geq |j-k|}\,\mathbbm{1}_{|j-k| \leq K/4}\,.
\end{equation}
We note that the random variables \eqref{X_j_definition} are centred, subgaussian, and satisfy
\begin{equation}
\label{X_j_bound}
\|X_j\|_{\psi_2} \leq \bigg\|\frac{g_j}{\langle j \rangle}\, \mathbbm{1}_{|j| \simeq K} \biggr\|_{\psi_2} \sim \frac{1}{K}\,.
\end{equation}
In \eqref{X_j_bound}, we recall \eqref{Psi_2}.
Using Lemma \ref{Hoeffding's_inequality}, \eqref{a_j_definition}, and \eqref{X_j_bound}, it follows that 
\begin{equation}
\label{II_2_bound_2}
\eqref{II_2_bound} \lesssim \ee^{-c \frac{\lambda_2^2 K}{\bbbbb}}
\end{equation}
for some universal constant $c>0$. The above application of Lemma \ref{Hoeffding's_inequality} is justified by the aforementioned independence\footnote{In \cite{Bou97} this is referred to as a decoupling argument.}. Here, we treat the $X_j$ in \eqref{X_j_definition} as the subgaussian random variables and the $a_j$ in \eqref{a_j_definition} as the coefficients.
Using \eqref{48b}--\eqref{C_4_choice}, \eqref{II_1_bound}, and \eqref{II_2_bound_2}, we deduce that
\begin{equation}
\label{Case2_bound}
\eqref{48}  \lesssim \ee^{-\bboo K^3}+\ee^{-\bbooo \lambda_2^2 K}\,,
\end{equation}
where 
\begin{equation}
\label{48b2}
\bboo \sim \bbbbb>0 
\end{equation}
can be chosen arbitrarily large and $\bbooo \sim \frac{1}{\bbbbbb} \sim \frac{1}{\bboo}$. We hence deduce \eqref{53} from \eqref{47} and \eqref{Case2_bound}. Note that $\bboo \sim \bbbbb$ is chosen large enough depending on $\BB$ by \eqref{48b}.
\paragraph{\textbf{Step 4: Choosing the parameters and putting the dyadic decomposition together}}
We now choose the parameters precisely and put the dyadic decomposition from Step 3 together.
Let us choose  $\lambda_{2,M}$ as the right-hand side of \eqref{37}, i.e.\
\begin{equation}
\label{55}
\lambda_{2,M}:=\biggl(\frac{\bbo \lambda_{1,M}}{M^3 v_M}\biggr)^{\frac{1}{2}}\,.
\end{equation}
Recalling \eqref{Bourgain_large_deviation_estimate_rewritten}, \eqref{35b}, \eqref{37}, \eqref{38}, and \eqref{39}, it follows that we need to bound
\begin{multline}
\label{dyadic_sum}
\sum_{M} \sum_{|k| \simeq M} \mathbb{P}_{\omega} \Biggl(\sum_{K \geq M/2} \Biggl|\sum_{j:\, \mathrm{max}\{|j|,|j-k|\} \simeq K} \frac{g_j(\omega)\,\overline{g_{j-k}(\omega)}}{\langle j \rangle\,\langle j-k \rangle}\Biggr|>\lambda_{2,M}\,\,
\\
\bigcap \,\, 
\Biggl| \sum_{m \in \Z^3} \frac{|g_m(\omega)|^2-1}{\langle m \rangle^2}\Biggr| \leq \BB \Biggr) 
\\
\lesssim \sum_{M} M^3 \mathop{\mathrm{max}}_{|k| \simeq M} \Biggl[ \mathbb{P}_{\omega} \Biggl(\sum_{K \geq M/2} \Biggl|\sum_{j:\, \mathrm{max}\{|j|,|j-k|\} \simeq K} \frac{g_j(\omega)\,\overline{g_{j-k}(\omega)}}{\langle j \rangle\,\langle j-k \rangle}\Biggr|>\lambda_{2,M}\,\,
\\
\bigcap \,\, 
\Biggl| \sum_{m \in \Z^3} \frac{|g_m(\omega)|^2-1}{\langle m \rangle^2}\Biggr| \leq \BB\Biggr)\Biggr]\,.
\end{multline}
In \eqref{dyadic_sum} and in the sequel, $M$ and $K$ are always summed over sets of dyadic integers.

For the remainder of the proof, we estimate \eqref{dyadic_sum}. In Step 5 below, we estimate the contribution to \eqref{dyadic_sum} when $M \geq \lambda^{1/3}$. For this contribution, we can directly use the estimate \eqref{53} proved in Step 3 with suitably chosen $\lambda_2$ depending on $M$ and $K$ and obtain the wanted bound in \eqref{Bourgain_large_deviation_estimate_rewritten}. In Step 6 below, we estimate the contribution to \eqref{dyadic_sum} when $M<\lambda^{1/3}$. Here, we also apply the estimate \eqref{33} proved in Step 1. Intuitively, the distinction between these two cases comes from the bound on the right-hand side of \eqref{33}. Alternatively, we can view this as being dictated by the first term on the right-hand side of \eqref{53}.

\paragraph{\textbf{Step 5: Estimating the contribution to \eqref{dyadic_sum} when $M \geq \lambda^{1/3}$}}

Given $\gamma>0$ small, let us define the sequence $(\alpha_K)_{K\,\mathrm{dyadic}}$ as
\begin{equation}
\label{56}
\alpha_K:=\frac{\mathcal{C}_{\gamma}}{K^{\gamma}}\,,
\end{equation}
where the constant $\mathcal{C}_{\gamma}>0$ in \eqref{56} is chosen such that 
\begin{equation}
\label{56b}
\sum_{K\,\mathrm{dyadic}} \alpha_K=1\,.
\end{equation}
From \eqref{56b}, by using a union bound, followed by \eqref{53}, we get that
\begin{multline}
\label{59}
\sum_{M \geq \lambda^{1/3}} M^3 \mathop{\mathrm{max}}_{|k| \simeq M} \Biggl[\mathbb{P}_{\omega} \Biggl(\sum_{K \geq M/2}\Biggl|\sum_{j:\, \mathrm{max}\{|j|,|j-k|\} \simeq K} \frac{g_j(\omega)\,\overline{g_{j-k}(\omega)}}{\langle j \rangle\,\langle j-k \rangle}\Biggr|>\lambda_{2,M}\,\,
\\
\bigcap \,\, 
\Biggl| \sum_{m \in \Z^3} \frac{|g_m(\omega)|^2-1}{\langle m \rangle^2}\Biggr| \leq \BB\Biggr)\Biggr]
\\
\leq 
\sum_{M \geq \lambda^{1/3}} M^3 \mathop{\mathrm{max}}_{|k| \simeq M} \Biggl[\sum_{K \geq M/2} \mathbb{P}_{\omega} \Biggl( \Biggl|\sum_{j:\, \mathrm{max}\{|j|,|j-k|\} \simeq K} \frac{g_j(\omega)\,\overline{g_{j-k}(\omega)}}{\langle j \rangle\,\langle j-k \rangle}\Biggr|>\alpha_K \lambda_{2,M}\,\,
\\
\bigcap \,\, 
\Biggl| \sum_{m \in \Z^3} \frac{|g_m(\omega)|^2-1}{\langle m \rangle^2}\Biggr| \leq \BB\Biggr)\Biggr]
\\
\lesssim \sum_{M \geq \lambda^{1/3}} M^3 \sum_{K \geq M/2} \ee^{-\bboo K^3}+\sum_{M \geq \lambda^{1/3}} M^3 \sum_{K \geq M/2} \ee^{-\bbooo \alpha_K^2 \lambda_{2,M}^2 K}\,.
\end{multline}
Let us recall that in \eqref{59} as in \eqref{53} we are taking $\bboo>0$ sufficiently large (depending on $\BB$) and $\bbooo \sim \frac{1}{\bboo}$ as in \eqref{c_1_c_2}.

We estimate each term in \eqref{59} separately. For the first term, we note that
\begin{multline}
\label{60}
\sum_{M \geq \lambda^{1/3}} M^3 \sum_{K \geq M/2} \ee^{-\bboo K^3} \lesssim_{\bboo}
\sum_{M \geq \lambda^{1/3}} M^3 \sum_{K \geq M/2} \ee^{-\frac{\bboo}{2} K^3}\,\frac{1}{K^4} 
\\
\lesssim \sum_{M \geq \lambda^{1/3}} \frac{1}{M}\,\ee^{-\frac{\bboo}{16} M^3} 
\leq \ee^{-\frac{\bboo}{16} \lambda} \sum_{M \geq \lambda^{1/3}} \frac{1}{M} \lesssim \ee^{-\frac{\bboo}{16} \lambda}\,.
\end{multline}

Let us now estimate the second term in \eqref{59}. By using \eqref{lambda_1M_choice}, \eqref{55}, and \eqref{56}, followed by \eqref{36}, we first note that
\begin{equation}
\label{57}
\alpha_K^2 \lambda_{2,M}^2 K=\alpha_K^2 \frac{\bbo \lambda_{1,M}}{M^3v_M} K \geq c \,\biggl(\frac{K}{M}\biggr)^{1-2\gamma}  \,\frac{M^{\varepsilon-2\gamma}}{(\log M)^2}\,\lambda\,.
\end{equation}
In \eqref{57}, $c>0$ is a constant depending on $\gamma$ (through \eqref{56}).

In the sequel, we assume that $\gamma>0$ is chosen sufficiently small such that
\begin{equation}
\label{58}
\gamma<\mathrm{min}\Bigl\{\frac{\varepsilon}{3},\frac{1}{2}\Bigr\}=\frac{\varepsilon}{3}\,.
\end{equation}
Moreover, given a fixed large constant $C_{\mathrm{large}}>0$, we can find a dyadic integer $M_0 \equiv M_0(C_{\mathrm{large}},\varepsilon,\gamma)$ such that 
\begin{equation}
\label{61}
\frac{M^{\varepsilon-3\gamma}}{(\log M)^2} \geq C_{\mathrm{large}}\,\,\,\mbox{whenever} \,\,\,M \geq M_0\,.
\end{equation}
The choice of $M_0$ as in \eqref{61} is possible by \eqref{58}. In order to estimate the second term in \eqref{59}, it suffices to separately estimate
\begin{equation}
\label{62}
J_1:=\sum_{\lambda^{1/3} \leq M < M_0} M^3 \sum_{K \geq M/2} \exp\biggl\{-\bbooo \,\biggl(\frac{K}{M}\biggr)^{1-2\gamma}  \,\frac{M^{\varepsilon-2\gamma}}{(\log M)^2}\,\lambda\biggr\}
\end{equation}
and
\begin{equation}
\label{63}
J_2:=\sum_{M \geq \max\{M_0,\lambda^{1/3}\}} M^3 \sum_{K \geq M/2} \exp\biggl\{-\bbooo \,\biggl(\frac{K}{M}\biggr)^{1-2\gamma} \,\frac{ M^{\varepsilon-2\gamma}}{(\log M)^2}\,\lambda\biggr\}\,.
\end{equation}
Note that in \eqref{59} and \eqref{62}--\eqref{63} the constants $\bbooo$ are not necessarily the same, but they differ by a multiplicative constant coming from the constant $c$ in \eqref{57}.
We write them in this way for simplicity of notation.

In order to estimate \eqref{62}, we use the elementary inequality 
\begin{equation}
\label{auxuliary_inequality}
\ee^{-x} \lesssim \frac{1}{x}
\end{equation} for $x>0$ 
and write
\begin{multline}
\label{62b}
J_1 \lesssim_{\bbooo} \sum_{\lambda^{1/3} \leq M < M_0} M^3 \sum_{K \geq M/2}  \exp\biggl\{-\frac{\bbooo}{2} \,\biggl(\frac{K}{M}\biggr)^{1-2\gamma} \, \frac{M^{\varepsilon-2\gamma}}{(\log M)^2}\,\lambda\biggr\}
\\
\times
\biggl(\frac{M}{K}\biggr)^{1-2\gamma}\, \frac{(\log M)^2}{M^{\varepsilon-2\gamma}}\, \frac{1}{\lambda}\,.
\end{multline}
For $M$ and $K$ as in the sum \eqref{62}, the expression in the exponent above satisfies
\begin{equation}
\label{62c}
\biggl(\frac{K}{M}\biggr)^{1-2\gamma}  \, \frac{M^{\varepsilon-2\gamma}}{(\log M)^2}\,\lambda \geq \biggl(\frac{1}{2}\biggr)^{1-2\gamma}\,\frac{\lambda^{1+\frac{\varepsilon-2\gamma}{3}}}{(\log M_0)^2}\,.
\end{equation}
Here, we also recalled \eqref{58}.
Combining \eqref{62b}--\eqref{62c}, and recalling that $\lambda>1$, it follows that for
\begin{equation}
\label{tilde_c_2}
\tilde{c}_2 =\frac{\bbooo}{2} \biggl(\frac{1}{2}\biggr)^{1-2\gamma} \sim \bbooo\,,
\end{equation}
we have 
\begin{multline}
\label{64}
J_1 \lesssim_{\bbooo} \Biggl(\sum_{\lambda^{1/3} \leq M < M_0} M^{3} \,\frac{(\log M)^2}{M^{\varepsilon-2\gamma}} \sum_{K \geq M/2} \biggl(\frac{M}{K}\biggr)^{1-2\gamma}\Biggr)
\\
\times \exp\biggl\{-\lambda^{1+\frac{\varepsilon-2\gamma}{3}}\,\frac{\tilde{c}_2}{(\log M_0)^2}\biggr\}
\\
\lesssim_{\varepsilon,\gamma,M_0} \exp\biggl\{-\lambda^{1+\frac{\varepsilon-2\gamma}{3}}\,\frac{\tilde{c}_2}{(\log M_0)^2}\biggr\}
\lesssim_{\bbooo,M_0,C_{\mathrm{large}}} \ee^{-\bbooo C_{\mathrm{large}} \lambda}\,.
\end{multline}
In order to deduce the last inequality in \eqref{64}, we recalled \eqref{58} and \eqref{tilde_c_2}.

Let us now estimate \eqref{63}. We use \eqref{61} and \eqref{auxuliary_inequality} to obtain that for $\tilde{c}_2$ as in \eqref{tilde_c_2}, we have
\begin{multline}
\label{65}
J_2 \leq \sum_{M \geq \max\{M_0,\lambda^{1/3}\}} M^3 \sum_{K \geq M/2} \exp\biggl\{-\bbooo \,\biggl(\frac{K}{M}\biggr)^{1-2\gamma} C_{\mathrm{large}} \lambda M^{\gamma} \biggr\}
\\
\lesssim_{\bbooo} 
\sum_{M \geq \max\{M_0,\lambda^{1/3}\}} M^3 \exp\Bigl\{-\tilde{c}_2 C_{\mathrm{large}} \lambda M^{\gamma}\Bigr\} \sum_{K \geq M/2}
\biggl(\frac{M}{K}\biggr)^{1-2\gamma}\frac{1}{C_{\mathrm{large}} \lambda M^{\gamma}}
\\
\lesssim_{\gamma}
\sum_{M \geq \max\{M_0,\lambda^{1/3}\}} M^3 \exp\Bigl\{-\tilde{c}_2 C_{\mathrm{large}} \lambda M^{\gamma}\Bigr\} 
\\
\lesssim_{\tilde{c}_2,C_{\mathrm{large}},\gamma}
\sum_{M \geq \max\{M_0,\lambda^{1/3}\}} M^3 \exp\biggl\{-\frac{\tilde{c}_2}{2} C_{\mathrm{large}} \lambda M^{\gamma}\biggr\} \,\biggl(\frac{1}{\tilde{c}_2 C_{\mathrm{large}}  \lambda M^{\gamma}}\biggr)^{4/\gamma}
\\
\lesssim_{\tilde{c}_2,C_{\mathrm{large}},\gamma} \exp\biggl\{-\frac{\tilde{c}_2}{2} C_{\mathrm{large}} \lambda \biggr\} \sum_{M \geq \max\{M_0,\lambda^{1/3}\}} \frac{1}{M} \lesssim \ee^{-\frac{\tilde{c}_2}{2} C_{\mathrm{large}} \lambda}\,.
\end{multline}
In order to obtain the penultimate line of \eqref{65}, we used the inequality $\ee^{-x} \lesssim_{\gamma} x^{-\frac{4}{\gamma}}$ for $x>0$.

We now combine \eqref{60}, \eqref{64}, and \eqref{65} to conclude that 
\begin{equation}
\label{Big_Case_1}
\eqref{59} \lesssim_{\bboo,C_{\mathrm{large}},\varepsilon,\gamma} \ee^{-\frac{\bboo}{16} \lambda} + \ee^{-\bbooo C_{\mathrm{large}} \lambda}+\ee^{-\frac{\tilde{c}_2}{2} C_{\mathrm{large}} \lambda}\,.
\end{equation}
Here, $\bboo$ can be chosen arbitrarily large and the other parameters are given by \eqref{c_1_c_2} and \eqref{tilde_c_2}.
Now, given $C_{\mathrm{Big}}>0$ large, we first choose 
\begin{equation}
\label{parameter_choice_1_Big_Case_1}
\bboo>16C_{\mathrm{Big}}\,.
\end{equation} 
We then take $C_{\mathrm{large}}>0$ sufficiently large such that 
\begin{equation}
\label{parameter_choice_2_Big_Case_1}
\bbooo C_{\mathrm{large}}>C_{\mathrm{Big}}\,,\quad \frac{\tilde{c}_2}{2} \,C_{\mathrm{large}}>C_{\mathrm{Big}}\,.
\end{equation}
Substituting this into \eqref{Big_Case_1}, we obtain
\begin{equation}
\label{Big_Case_1_bound}
\eqref{59} \lesssim_{C_{\mathrm{Big}}} \ee^{-C_{\mathrm{Big}} \lambda}\,.
\end{equation}
This is an acceptable upper bound.
\paragraph{\textbf{Step 6:  Estimating the contribution to \eqref{dyadic_sum} when $M < \lambda^{1/3}$}}
We are interested in estimating
\begin{multline}
\label{Big_Case_2}
\sum_{M<\lambda^{1/3}} M^3 \mathop{\mathrm{max}}_{|k| \simeq M} \Biggl[ \mathbb{P}_{\omega} \Biggl(\sum_{K \geq M/2} \Biggl|\sum_{j:\, \mathrm{max}\{|j|,|j-k|\} \simeq K} \frac{g_j(\omega)\,\overline{g_{j-k}(\omega)}}{\langle j \rangle\,\langle j-k \rangle}\Biggr|>\lambda_{2,M}\,\,
\\
\bigcap \,\, 
\Biggl| \sum_{m \in \Z^3} \frac{|g_m(\omega)|^2-1}{\langle m \rangle^2}\Biggr| \leq \BB\Biggr)\Biggr]
\\
\leq 
\sum_{M<\lambda^{1/3}} M^3 \mathop{\mathrm{max}}_{|k| \simeq M} \Biggl[ \mathbb{P}_{\omega} \Biggl(\sum_{\frac{M}{2} 
\leq K<\frac{\lambda^{1/3}}{2}}\, \sum_{j:\, \mathrm{max}\{|j|,|j-k|\} \simeq K} \Biggl|\frac{g_j(\omega)\,\overline{g_{j-k}(\omega)}}{\langle j \rangle\,\langle j-k \rangle}\Biggr|>\frac{\lambda_{2,M}}{2}\,\,
\\
\bigcap \,\, 
\Biggl| \sum_{m \in \Z^3} \frac{|g_m(\omega)|^2-1}{\langle m \rangle^2}\Biggr| \leq \BB\Biggr)\Biggr]
\\
+\sum_{M<\lambda^{1/3}} M^3 \mathop{\mathrm{max}}_{|k| \simeq M} \Biggl[\sum_{K \geq \frac{\lambda^{1/3}}{2}}  \mathbb{P}_{\omega} \Biggl(\Biggl|\sum_{j:\, \mathrm{max}\{|j|,|j-k|\} \simeq K} \frac{g_j(\omega)\,\overline{g_{j-k}(\omega)}}{\langle j \rangle\,\langle j-k \rangle}\Biggr|>\frac{\alpha_K \lambda_{2,M}}{2}\,\,
\\
\bigcap \,\, 
\Biggl| \sum_{m \in \Z^3} \frac{|g_m(\omega)|^2-1}{\langle m \rangle^2}\Biggr| \leq \BB\Biggr)\Biggr]=:L_1+L_2\,.
\end{multline}
For \eqref{Big_Case_2}, we recalled \eqref{56b} and used a union bound.

Before we start, let us note that in the regime $M<\lambda^{1/3}$, we have by  \eqref{55}  followed by \eqref{lambda_1M_choice} and \eqref{36} that
\begin{equation}
\label{68}
\lambda_{2,M}=\biggl(\frac{\bbo \lambda_{1,M}}{M^3 v_M}\biggr)^{\frac{1}{2}} \gtrsim \biggl(\frac{\lambda}{M^{1-\varepsilon}\, (\log M)^2} \biggr)^{\frac{1}{2}} \gtrsim_{\varepsilon} \biggl(\frac{\lambda}{M^{1-\frac{\varepsilon}{2}}}\biggr)^{\frac{1}{2}}>\lambda^{\frac{1}{3}+\frac{\varepsilon}{12}}\,.
\end{equation}
We now estimate $L_1$ given in \eqref{Big_Case_2}. By symmetry, it suffices to consider the contribution where $|j| \geq |j-k|$. The other contribution is estimated analogously. Let us note that
\begin{multline}
\label{68b}
\sum_{M/2 \leq K<\lambda^{1/3}/2} \mathop{\sum_{|j| \simeq K}}_{|j| \geq |j-k|}\biggl|\frac{g_j(\omega)\,\overline{g_{j-k}(\omega)}}{\langle j \rangle\,\langle j-k \rangle}\biggr| \leq \sum_{|j| \leq \lambda^{1/3}/2} 
\biggl|\frac{g_j(\omega)\,\overline{g_{j-k}(\omega)}}{\langle j \rangle\,\langle j-k \rangle}\biggr| 
\\
\leq \Biggl(\sum_{|j| \leq \lambda^{1/3}/2} \frac{|g_j(\omega)|^2}{\langle j \rangle^2} \Biggr)^{1/2} \, \Biggl(\sum_{|j| \leq \lambda^{1/3}/2}  \frac{|g_{j-k}(\omega)|^2}{\langle j-k \rangle^2} \Biggr)^{1/2} \leq \sum_{|\ell| \leq 2 \lambda^{1/3}} \frac{|g_{\ell}(\omega)|^2}{\langle \ell \rangle^2}\,.
\end{multline}
In the second line of \eqref{68b}, we first used the Cauchy-Schwarz inequality and then noted that for $|j| \leq \lambda^{1/3}/2$ and $|k| \simeq M \leq \lambda^{1/3}$ we have $|j-k| \leq 2\lambda^{1/3}$.
By a suitable modification of \eqref{33} proved in Step 1, it follows that we have
\begin{equation}
\label{68c}
\mathbb{P}_{\omega} \Biggl( \sum_{|\ell| \leq 2 \lambda^{1/3}} \frac{|g_{\ell}(\omega)|^2}{\langle \ell \rangle^2}>\tilde{C}_1 \lambda^{1/3}\,\,\bigcap \,\, 
\Biggl| \sum_{m \in \Z^3} \frac{|g_m(\omega)|^2-1}{\langle m \rangle^2}\Biggr| \leq \BB \Biggr)
\lesssim_{\BB} \ee^{-\tilde{C}_2\lambda}\,,
\end{equation}
where 
\begin{equation}
\label{68d}
\tilde{C}_1 \sim \tilde{C}_2 \sim \bbb\,,
\end{equation}
for $\bbb$ as in \eqref{C_1_choice}. From \eqref{68b}--\eqref{68c}, we deduce that
\begin{multline}
\label{33*}
\mathbb{P}_{\omega} \Biggl(\sum_{M/2 \leq K<\lambda^{1/3}/2}\, \sum_{|j| \simeq K:\,|j| \geq |j-k|}\biggl|\frac{g_j(\omega)\,\overline{g_{j-k}(\omega)}}{\langle j \rangle\,\langle j-k \rangle}\biggr| \geq \tilde{C}_1 \lambda^{1/3}\,\,
\\
\bigcap \,\, 
\Biggl| \sum_{m \in \Z^3} \frac{|g_m(\omega)|^2-1}{\langle m \rangle^2}\Biggr| \leq \BB \Biggr) \lesssim \ee^{-\tilde{C}_2 \lambda}\,,
\end{multline}
for constants as in \eqref{68d}.

We now recall \eqref{68} and use \eqref{33*} with $\tilde{C}_1 \sim \lambda^{\varepsilon/12}$ (which is admissible in \eqref{68d} for $\lambda$ large enough) to deduce that, for some universal constant $c>0$, we have
\begin{multline}
\label{69}
L_1 \lesssim \sum_{M<\lambda^{1/3}} M^3\, \ee^{-c \lambda^{1+\frac{\varepsilon}{12}}} \leq \ee^{-c \lambda^{1+\frac{\varepsilon}{12}}}\, \sum_{M <\lambda^{1/3}} M^3 \lesssim \ee^{-c \lambda^{1+\frac{\varepsilon}{12}}}
\\
\lesssim \ee^{-\frac{c}{2} \lambda^{1+\frac{\varepsilon}{12}}} \lesssim_{C_{\mathrm{Big}}} \ee^{-C_{\mathrm{Big}} \lambda}\,,
\end{multline}
for $C_{\mathrm{Big}}>0$ arbitrarily large. In \eqref{69}, we used 
\begin{equation}
\label{dyadic_sum_M_lambda^{1/3}}
\sum_{M <\lambda^{1/3}} M^3 \lesssim \lambda\,.
\end{equation}

We now estimate $L_2$ given in \eqref{Big_Case_2}. Arguing analogously as for \eqref{59}, it follows that we need to estimate
\begin{equation}
\label{70}
\sum_{M < \lambda^{1/3}} M^3 \sum_{K \geq \lambda^{1/3}/2} \ee^{-\bboo K^3}+\sum_{M < \lambda^{1/3}} M^3 \sum_{K \geq \lambda^{1/3}/2} \ee^{-\bbooo \alpha_K^2 \lambda_{2,M}^2 K}=:L_{2,1}+L_{2,2}\,.
\end{equation}
In \eqref{70}, the choice of parameters is the same as in \eqref{59}, i.e.\, we take $\bboo$ large and $\bbooo \sim \frac{1}{\bboo}$ as in \eqref{c_1_c_2} and recall \eqref{lambda_1M_choice}, \eqref{36}, \eqref{55}, and \eqref{56}.

Let us first estimate $L_{2,1}$ in \eqref{70}. We observe that
\begin{multline}
\label{71}
L_{2,1} \lesssim_{\bboo} \sum_{M < \lambda^{1/3}} M^3 \sum_{K \geq \lambda^{1/3}/2} \ee^{-\frac{\bboo}{2} K^3} \frac{1}{K} \leq \ee^{-\frac{\bboo}{16} \lambda}\, \sum_{M < \lambda^{1/3}} M^3\, \sum_{K \geq \lambda^{1/3}/2}  \frac{1}{K}
\\
\lesssim_{\bboo} \ee^{-\frac{\bboo}{16} \lambda}\,\lambda
\lesssim \ee^{-\frac{\bboo}{20}\lambda}\,,
\end{multline}
which is an acceptable upper bound since $\bboo$ can be chosen arbitrarily large.
In \eqref{71}, we again used \eqref{dyadic_sum_M_lambda^{1/3}}.

We now estimate $L_{2,2}$ in \eqref{70}. Recalling \eqref{57} and arguing analogously as for \eqref{62b}, we get that for $\bbooo \sim \frac{1}{\bboo}$
\begin{multline}
\label{72}
L_{2,2} \leq \sum_{M < \lambda^{1/3}} M^3 \sum_{K \geq \lambda^{1/3}/2} \exp\biggl\{-\bbooo \,\biggl(\frac{K}{M}\biggr)^{1-2\gamma} \frac{M^{\varepsilon-2\gamma}}{(\log M)^2}\,\lambda\biggr\}
\\
\lesssim 
\sum_{M < \lambda^{1/3}} M^3 \sum_{K \geq \lambda^{1/3}/2} \exp\biggl\{-\frac{\bbooo}{2} \,\biggl(\frac{K}{M}\biggr)^{1-2\gamma} \frac{M^{\varepsilon-2\gamma}}{(\log M)^2} \,\lambda \biggr\}
\\
\times \frac{1}{\bbooo}\,
\biggl(\frac{M}{K}\biggr)^{1-2\gamma}\,\frac{(\log M)^2}{M^{\varepsilon-2\gamma}}\,\frac{1}{\lambda}
\\
\lesssim_{\bbooo} 
\sum_{M < \lambda^{1/3}} M^{4-2\gamma}\, (\log M)^2\sum_{K \geq \lambda^{1/3}/2} \exp\biggl\{-\frac{\bbooo}{2} \,K^{1-2\gamma}
\,\frac{M^{-1+\varepsilon}}{(\log M)^2}\,\lambda\biggr\}
\\
\times K^{-1+2\gamma}\,.
\end{multline}
Let us now analyse the exponent in \eqref{72}. For $M<\lambda^{1/3}$ and $K \geq \lambda^{1/3}/2$, we have that 
\begin{equation}
\label{72b}
(\log M)^2 <  (\log \lambda^{1/3})^2 \lesssim_{\varepsilon, \gamma} \lambda^{\frac{\varepsilon-2\gamma}{6}}\,,\quad M^{-1+\varepsilon}>\lambda^{\frac{-1+\varepsilon}{3}}\,,\quad K^{1-2\gamma} \gtrsim_{\gamma} \lambda^{\frac{1-2\gamma}{3}}\,.
\end{equation}
Here, we recalled \eqref{58}.
Putting everything together, we get from \eqref{72b} that 
\begin{equation}
\label{72c}
K^{1-2\gamma}\,\frac{M^{-1+\varepsilon}}{(\log M)^2}\,\lambda \gtrsim_{\varepsilon,\gamma} \lambda^{1+\frac{\varepsilon-2\gamma}{6}}\,.
\end{equation}
Substituting \eqref{72c} into \eqref{72}, we get that
\begin{multline}
\label{73}
L_{2,2} \lesssim \ee^{-\tilde{c}_2 \lambda^{1+\frac{\varepsilon-2\gamma}{6}}} \,\sum_{M < \lambda^{1/3}} M^{4-2\gamma}\, (\log M)^2\sum_{K \geq \lambda^{1/3}/2} K^{-1+2\gamma}
\\
\lesssim \ee^{-\tilde{c}_2 \lambda^{1+\frac{\varepsilon-2\gamma}{6}}} \, \lambda^{\frac{4-2\gamma}{3}}\,(\log \lambda^{\frac{1}{3}})^3 \lesssim \ee^{-\frac{\tilde{c}_2}{2} \lambda^{1+\frac{\varepsilon-2\gamma}{6}}} \lesssim_{C_{\mathrm{Big}}} \ee^{-C_{\mathrm{Big}} \lambda}\,,
\end{multline}
for $C_{\mathrm{Big}}>0$ arbitrarily large. In the above calculation, we recalled \eqref{58}. 
In conclusion, the contribution to \eqref{dyadic_sum} when $M < \lambda^{1/3}$ satisfies the wanted upper bound by \eqref{69}, \eqref{71}, and \eqref{73}. This completes the proof of \eqref{Bourgain_large_deviation_estimate_rewritten} when $d=3$.

\subsection{Proof of \eqref{Bourgain_large_deviation_estimate_rewritten} when $d=2$}
\label{Appendix_2D_proof}

We now explain the proof of \eqref{Bourgain_large_deviation_estimate_rewritten} when $d=2$. The argument is similar to the one when $d=3$ given above. We outline the main differences. Now, instead of as in \eqref{C_0_choice}, the constant $\bb>0$ is taken such that  for all $K>0$, we have
\begin{equation}
\label{C_0_choice*}
\sum_{|k| \leq K} \frac{1}{\langle k \rangle^2} \leq \bb \log K\,.
\end{equation}
We consider $\bbb$ as in \eqref{C_1_choice}, for $\bb$ as in \eqref{C_0_choice*}.

Let us note that, when $d=2$, instead of \eqref{31}, we have 
\begin{equation}
\label{31*}
\mathbb{P}_{\omega} \Biggl(\Biggl|\sum_{|k| > M} \frac{|\tilde{g}_k(\omega)|^2-1}{\langle k \rangle^2}\Biggr| \geq \bbb \log M \Biggr) \lesssim \ee^{-\bbbb M^2 \log M} \quad \text{for } \bbbb \sim \bbb\,.
\end{equation}
In order to obtain \eqref{31*}, we use Lemma \ref{Bernstein's_inequality} with $(b_k)$ given as in \eqref{Sequence_bk} (but now in two dimensions). The only change is that instead of \eqref{Sequence_bk_l2}, we now have
\begin{equation}
\label{Sequence_bk_l2*}
\|b\|_{\ell^2}^2 \sim \frac{1}{M^2}\,.
\end{equation}
Using \eqref{Psi_1_bound}, \eqref{Sequence_bk_linfty} (which carry over to $d=2$), and \eqref{31*} in Lemma \ref{Bernstein's_inequality}, we deduce that for a universal constant $c>0$
\begin{multline*}
\mathbb{P}_{\omega} \Biggl(\Biggl|\sum_{|k| > M} \frac{|\tilde{g}_k|^2-1}{\langle k \rangle^2}\Biggr| \geq \bbb \log M \Biggr) 
\\
\lesssim \ee^{-c\, \mathrm{min} \{(\bbb \log M)^2\cdot M^2, (\bbb \log M) \cdot M^2\}}=\ee^{-c \bbb M^2 \log M}\,,
\end{multline*}
which implies \eqref{31*}.

\paragraph{\emph{\textbf{Modifications to Step 3}}}
We now explain the modifications needed for the analogue of Step 3 above. Let us consider $\lambda_2>0$, dyadic integers $K,M$ with $K \geq M/2$ (as in \eqref{K_and_M_condition}) and $k \in \Z^2$ with $|k| \simeq M$. 
We show the following modification of \eqref{53}.
\begin{multline}
\label{(2.35)*-(2.36)*}
\mathbb{P}_{\omega} \Biggl(\Biggl|\sum_{j:\, \mathrm{max}\{|j|,|j-k|\} \simeq K} \frac{g_j(\omega)\,\overline{g_{j-k}(\omega)}}{\langle j \rangle\,\langle j-k \rangle}\Biggr|>\lambda_2\,\,\bigcap \,\, 
\Biggl| \sum_{m \in \Z^2} \frac{|g_m(\omega)|^2-1}{\langle m \rangle^2}\Biggr| \leq \BB\Biggr)
\\
\lesssim_{\bboo} \ee^{-\bboo K^2}+ \ee^{-\frac{\bbooo \lambda_2^2 K^2}{\log K}}\,,
\end{multline}
where $\bboo>0$ can be chosen arbitrarily large and $\bbooo \sim \frac{1}{\bboo}$ (as in \eqref{c_1_c_2}).
As in Step 3 above, we consider two different cases, depending on the relative sizes of $M$ and $K$.

\paragraph{\textbf{\emph{Case 1:} $M \leq K/4$}}
We argue analogously as for $d=3$. The matrix $\mathcal{A}$ defined by \eqref{43} still satisfies \eqref{44}. The estimate \eqref{45} is replaced by 
\begin{equation*}
\|\mathcal{A}\|_{\mathrm{HS}}^2 \lesssim \sum_{|j| \simeq K} \frac{1}{K^4} \lesssim \frac{1}{K^2}\,.
\end{equation*}
Using Lemma \ref{Hanson_Wright_inequality} as for \eqref{45d}--\eqref{46}, we obtain that
\begin{equation}
\label{(2.52)*}
\mathbb{P}_{\omega} \Biggl(\Biggl|\sum_{|j| \simeq K\,,|j| \geq |j-k|} \frac{g_j(\omega)\,\overline{g_{j-k}(\omega)}}{\langle j \rangle\,\langle j-k \rangle}\Biggr|>\lambda_2\Biggr) \lesssim \mathrm{exp} \Bigl[ - c \,\mathrm{min} \{\lambda_2^2 K^2, \lambda_2 K^2\}\Bigr]\,,
\end{equation}
for some universal constant $c>0$. This left hand side of \eqref{(2.52)*} is the expression that we need to estimate by \eqref{40}.

We now rewrite \eqref{(2.52)*}. Let $C_{\mathrm{large}} \gg 1$ be a large constant.

\begin{itemize}
\item If $\lambda_2 \geq C_{\mathrm{large}}$, then $\mathrm{min} \{\lambda_2^2 K^2, \lambda_2 K^2\}=\lambda_2K^2 \geq C_{\mathrm{large}}\,K^2.$
\item If $\lambda_2<C_{\mathrm{large}}$, then $\mathrm{min} \{\lambda_2^2 K^2, \lambda_2 K^2\}>\frac{\lambda_2^2K^2}{C_{\mathrm{large}}}$. 
In order to see the latter inequality, we note that $\lambda_2^2 K^2>\frac{\lambda_2^2 K^2}{C_{\mathrm{large}}}$. Furthermore, 
since $\lambda_2<C_{\mathrm{large}}$, we have 
$\lambda_2 K^2=\frac{\lambda_2^2 K^2}{\lambda_2}>\frac{\lambda_2^2 K^2}{C_{\mathrm{large}}}$.

\end{itemize}
In particular, from \eqref{(2.52)*}, we obtain that
\begin{equation}
\label{47*}
\mathbb{P}_{\omega} \Biggl(\Biggl|\sum_{|j| \simeq K\,,|j| \geq |j-k|} \frac{g_j(\omega)\,\overline{g_{j-k}(\omega)}}{\langle j \rangle\,\langle j-k \rangle}\Biggr|>\lambda_2 \Biggr) \lesssim_{\bboo} \ee^{-\bboo K^2}+\ee^{-\bbooo \lambda_2^2 K^2}\,,
\end{equation}
where $\bboo>0$ can be chosen arbitrarily large and $\bbooo \sim \frac{1}{\bboo}$.

\paragraph{\textbf{\emph{Case 2:} $M > K/4$}}
In this case, we do the splitting $I^*+II^*$, with $I^*$ and $II^*$ defined analogously as in \eqref{48} above (except that now the summations are over subsets of $\Z^2$). As in the analysis when $d=3$, the term $I^*$ is estimated analogously as \eqref{47*}. We hence need to estimate 
\begin{multline}
\label{48*}
II^* \equiv \mathbb{P}_{\omega} \Biggl(\Biggl|\sum_{|j| \simeq K\,,\,|j| \geq |j-k|\,,\,|j-k| \leq K/4} \frac{g_j(\omega)\,\overline{g_{j-k}(\omega)}}{\langle j \rangle\,\langle j-k \rangle}\Biggr|>\frac{\lambda_2}{2}\,\,
\\
\bigcap \,\, \Biggl| \sum_{m \in \Z^2} \frac{|g_m(\omega)|^2-1}{\langle m \rangle^2}\Biggr| \leq \BB
\Biggr)\,.
\end{multline}
We modify \eqref{50} and note that for $d=2$, we have
\begin{multline}
\label{50*}
\mathbb{P}_{\omega} \Biggl(\sum_{|\ell| \leq K/4} \frac{|g_{\ell}(\omega)|^2}{\langle \ell \rangle^2} \geq \bbbbb \log K\, \,\,\bigcap \,\, 
\Biggl| \sum_{m \in \Z^2} \frac{|g_m(\omega)|^2-1}{\langle m \rangle^2}\Biggr| \leq \BB \Biggr)
\\
\lesssim \ee^{-\bbbbbb K^2\log K}\,,
\end{multline}
with choice of parameters as in \eqref{C_0_choice*}, \eqref{48b}--\eqref{C_4_choice}.

To prove \eqref{50*}, suppose that 
\begin{equation}
\label{49*}
\sum_{|\ell| \leq K/4} \frac{|g_{\ell}(\omega)|^2}{\langle \ell \rangle^2} \geq \bbbbb \log K \quad  \text{and }\quad \Biggl| \sum_{m \in \Z^2} \frac{|g_m(\omega)|^2-1}{\langle m \rangle^2}\Biggr| \leq \BB\,.
\end{equation}
Arguing analogously as for \eqref{49c}, we obtain from \eqref{49*} that
\begin{equation}
\label{49c*}
\Biggl|\sum_{|\ell| > K/4} \frac{|g_{\ell}(\omega)|^2-1}{\langle \ell \rangle^2}\Biggr|  \gtrsim \bbbbb \log K\,.
\end{equation}
From \eqref{49c*}, we hence have
\begin{multline*}
\mathbb{P}_{\omega} \Biggl(\sum_{|\ell| \leq K/4} \frac{|g_{\ell}(\omega)|^2}{\langle \ell \rangle^2} \geq \bbbbb \log K\, \,\,\bigcap \,\, 
\Biggl| \sum_{m \in \Z^2} \frac{|g_m(\omega)|^2-1}{\langle m \rangle^2}\Biggr| \leq \BB \Biggr)
\\
\leq \mathbb{P}_{\omega} \Biggl(\Biggl|\sum_{|\ell| > K/4} \frac{|g_{\ell}(\omega)|^2-1}{\langle \ell \rangle^2}\Biggr|  \gtrsim \bbbbb \log K\Biggr) \lesssim \ee^{-\bbbbbb K^2 \log K}
\end{multline*}
with $C_4$ as in \eqref{C_4_choice}.
For the last inequality, we used \eqref{31*}. We hence obtain \eqref{50*}. The generalization analogous to \eqref{51} also holds, i.e.\
\begin{equation}
\label{51*}
\mathbb{P}_{\omega} \Biggl(\sum_{\ell \in S} \frac{|g_{\ell}(\omega)|^2}{\langle \ell \rangle^2} \geq \bbbbb \log K\, \,\,\bigcap \,\, 
\Biggl| \sum_{m \in \Z^2} \frac{|g_m(\omega)|^2-1}{\langle m \rangle^2}\Biggr| \leq \BB \Biggr)
\lesssim \ee^{-\bbbbbb K^2 \log K}\,,
\end{equation}
whenever $S \subset \Z^2$ is a set contained in the ball $|\ell| \leq K/4$.
Analogously as in \eqref{52}, we estimate \eqref{48*} as
\begin{multline}
\label{52*}
II^* \leq \mathbb{P}_{\omega} \Biggl(\Biggl|\sum_{|j| \simeq K\,,\,|j| \geq |j-k|\,,\,|j-k| \leq K/4} \frac{g_j(\omega)\,\overline{g_{j-k}(\omega)}}{\langle j \rangle\,\langle j-k \rangle}\Biggr|>\frac{\lambda_2}{2}\,\, 
\bigcap\,\,
\\
\sum_{|j| \simeq K\,,\,|j| \geq |j-k|\,,\,|j-k| \leq K/4} \frac{|g_{j-k}(\omega)|^2}{\langle j-k\rangle^2} \geq \bbbbb \log K\,\, \bigcap \,\, \Biggl| \sum_{m \in \Z^2} \frac{|g_m(\omega)|^2-1}{\langle m \rangle^2}\Biggr| \leq \BB
\Biggr)
\\
+ \mathbb{P}_{\omega} \Biggl(\Biggl|\sum_{|j| \simeq K\,,\,|j| \geq |j-k|\,,\,|j-k| \leq K/4} \frac{g_j(\omega)\,\overline{g_{j-k}(\omega)}}{\langle j \rangle\,\langle j-k \rangle}\Biggr|>\frac{\lambda_2}{2}\,\,
\bigcap \,\,
\\
\sum_{|j| \simeq K\,,\,|j| \geq |j-k|\,,\,|j-k| \leq K/4} \frac{|g_{j-k}(\omega)|^2}{\langle j-k\rangle^2} < \bbbbb \log K\,\,
\bigcap \,\, \Biggl| \sum_{m \in \Z^2} \frac{|g_m(\omega)|^2-1}{\langle m \rangle^2}\Biggr| \leq \BB
\Biggr)
\\=:II_1^*+II_2^*\,.
\end{multline}
From \eqref{51*}, we have 
\begin{multline}
\label{II_1_bound*}
II_1^* \leq  \mathbb{P}_{\omega} \Biggl(\sum_{|j| \simeq K\,,\,|j| \geq |j-k|\,,\,|j-k| \leq K/4} \frac{|g_{j-k}(\omega)|^2}{\langle j-k\rangle^2} \geq \bbbbb \log K
\\
\,\, \bigcap \,\, \Biggl| \sum_{m \in \Z^2} \frac{|g_m(\omega)|^2-1}{\langle m \rangle^2}\Biggr| \leq \BB
\Biggr) \lesssim \ee^{-\bbbbbb K^2 \log K}\,.
\end{multline}
Recalling \eqref{X_j_definition}--\eqref{a_j_definition}, and arguing analogously as for \eqref{II_2_bound}, we get that 
\begin{equation}
\label{II_2_bound*}
II_2^* \leq \mathbb{P}_{\omega} \Biggl(\Biggl|\sum_{j \in \Z^2}a_j X_j\Biggr|>\frac{\lambda_2}{2}\,\,\bigcap\,\, \sum_{j \in \Z^2} |a_j|^2 <\bbbbb \log K
\Biggr)\,.
\end{equation}
Using Lemma \ref{Hoeffding's_inequality}, \eqref{X_j_bound}, and arguing as for \eqref{II_2_bound_2}, we obtain 
\begin{equation}
\label{II_2_bound_2*}
\eqref{II_2_bound*} \lesssim \ee^{-c \frac{\lambda_2^2 K^2}{\bbbbb \log K}}
\end{equation}
for some universal constant $c>0$.

Using \eqref{48b}--\eqref{C_4_choice}, \eqref{II_1_bound*}, and \eqref{II_2_bound_2*}, we deduce that
\begin{equation}
\label{Case2_bound*}
\eqref{48*}  \lesssim \ee^{-\bboo K^2 \log K}+\ee^{-\frac{\bbooo \lambda_2^2 K^2}{\log K}}\,,
\end{equation}
where $\bboo \sim \bbbbb>0$ can be chosen arbitrarily large and $\bbooo \sim \frac{1}{\bbbbbb} \sim \frac{1}{\bboo}$.
We hence deduce \eqref{(2.35)*-(2.36)*} from \eqref{47*} and \eqref{Case2_bound*}.

\paragraph{\emph{\textbf{Modifications to Step 4}}}
We define $\lambda_{1,M}$ as in \eqref{lambda_1M_choice}--\eqref{lambda_1M_condition} and $v_M$ as in \eqref{v_M_definition}. By Assumption \ref{Assumption_on_V}, instead of \eqref{36}, we now have 
\begin{equation}
\label{36*}
v_M \lesssim M^{-\varepsilon}\,.
\end{equation}
Let $\lambda_{1,M}$ be defined as in \eqref{lambda_1M_choice}--\eqref{lambda_1M_condition}.
Instead of \eqref{55}, we define 
\begin{equation}
\label{55*}
\lambda_{2,M}:=\biggl(\frac{\bbo \lambda_{1,M}}{M^2 v_M}\biggr)^{\frac{1}{2}}\,.
\end{equation}

Using the above choice of parameters, and arguing as for \eqref{dyadic_sum}, our goal is to estimate

\begin{multline}
\label{dyadic_sum*}
\sum_{M} M^2 \mathop{\mathrm{max}}_{|k| \simeq M} \Biggl[ \mathbb{P}_{\omega} \Biggl(\sum_{K \geq M/2} \Biggl|\sum_{j:\, \mathrm{max}\{|j|,|j-k|\} \simeq K} \frac{g_j(\omega)\,\overline{g_{j-k}(\omega)}}{\langle j \rangle\,\langle j-k \rangle}\Biggr|>\lambda_{2,M}\,\,
\\
\bigcap \,\, 
\Biggl| \sum_{m \in \Z^3} \frac{|g_m(\omega)|^2-1}{\langle m \rangle^2}\Biggr| \leq \BB\Biggr)\Biggr]\,.
\end{multline}
We now estimate \eqref{dyadic_sum*}. When $d=2$, we split \eqref{dyadic_sum*} into contributions when $M 
\geq \lambda^{1/2}$ and $M<\lambda^{1/2}$. This allows us to modify the arguments from Steps 5 and 6 when $d=3$. Intuitively, the splitting is dictated by the scaling of the bound in \eqref{(2.35)*-(2.36)*} (in particular by the first term).

\paragraph{\emph{\textbf{Modifications to Step 5}}} Recall that we are considering $M \geq \lambda^{1/2}$.
We take $\alpha_K$ as in \eqref{56}--\eqref{56b}. We apply \eqref{56b} followed by \eqref{(2.35)*-(2.36)*} and argue as for \eqref{59} to deduce that 
\begin{multline}
\label{59*}
\sum_{M \geq \lambda^{1/2}} M^2 \mathop{\mathrm{max}}_{|k| \simeq M} \Biggl[\mathbb{P}_{\omega} \Biggl(\sum_{K \geq M/2}\Biggl|\sum_{j:\, \mathrm{max}\{|j|,|j-k|\} \simeq K} \frac{g_j(\omega)\,\overline{g_{j-k}(\omega)}}{\langle j \rangle\,\langle j-k \rangle}\Biggr|>\lambda_{2,M}\,\,
\\
\bigcap \,\, 
\Biggl| \sum_{m \in \Z^2} \frac{|g_m(\omega)|^2-1}{\langle m \rangle^2}\Biggr| \leq \BB\Biggr)\Biggr]
\\
\lesssim \sum_{M \geq \lambda^{1/2}} M^2 \sum_{K \geq M/2} \ee^{-\bboo K^2}+\sum_{M \geq \lambda^{1/2}} M^2 \sum_{K \geq M/2} \ee^{-\frac{\bbooo \alpha_K^2 \lambda_{2,M}^2 K^2}{\log K}}\,,
\end{multline}
for $\bboo>0$ sufficiently large and $\bbooo \sim \frac{1}{\bboo}$.

By arguing analogously as for \eqref{60}, we estimate the first term in \eqref{59*} as
\begin{equation}
\label{60*}
\sum_{M \geq \lambda^{1/2}} M^2 \sum_{K \geq M/2} \ee^{-\bboo K^2} \lesssim_{\bboo} \ee^{-\frac{\bboo}{8} \lambda}\,.
\end{equation}
In order to estimate the second term in \eqref{59*}, we note that by \eqref{lambda_1M_choice}, \eqref{55*}, and \eqref{56}, followed by \eqref{36*}, we have that
\begin{equation}
\label{57*}
\frac{\alpha_K^2 \lambda_{2,M}^2 K^2}{\log K} \geq  \frac{\tilde{c}}{K^{2\gamma} \log K}\, \frac{K^2}{M^{2-\varepsilon} (\log M)^2}\,\lambda \geq c  \,\biggl(\frac{K}{M}\biggr)^{2-3\gamma} \frac{M^{\varepsilon-3\gamma}}{(\log M)^2} \,\lambda,
\end{equation}
for some constants $c,\tilde{c}>0$ depending on $\gamma$. In the sequel, instead of \eqref{58}, we assume that 
\begin{equation}
\label{58*}
\gamma<\mathrm{min}\Bigl\{\frac{\varepsilon}{4},\frac{2}{3}\Bigr\}=\frac{\varepsilon}{4}\,.
\end{equation}
We modify \eqref{61} when $d=2$. Namely, given $C_{\mathrm{large}}>0$, by \eqref{58*} we can find a dyadic integer 
$M_0 \equiv M_0(C_{\mathrm{large}},\varepsilon,\gamma)$ such that 
\begin{equation}
\label{61*}
\frac{M^{\varepsilon-4\gamma}}{(\log M)^2} \geq C_{\mathrm{large}}\,\,\,\mbox{whenever} \,\,\,M \geq M_0\,.
\end{equation}
Arguing as in \eqref{62}--\eqref{63}, it suffices to estimate
\begin{equation}
\label{62*}
J_1^*:=\sum_{\lambda^{1/2} \leq M < M_0} M^2 \sum_{K \geq M/2} \exp\biggl\{-\bbooo \,\biggl(\frac{K}{M}\biggr)^{2-3\gamma} \frac{M^{\varepsilon-3\gamma}}{(\log M)^2} \,\lambda\biggr\}
\end{equation}
and
\begin{equation}
\label{63*}
J_2^*:=\sum_{M \geq \max\{M_0,\lambda^{1/2}\}} M^2 \sum_{K \geq M/2}  \exp\biggl\{-\bbooo \,\biggl(\frac{K}{M}\biggr)^{2-3\gamma} \frac{M^{\varepsilon-3\gamma}}{(\log M)^2} \,\lambda\biggr\}\,.
\end{equation}
Arguing analogously as for \eqref{64}, and using \eqref{58*}, we estimate \eqref{62*} as

\begin{equation}
\label{64*}
J_1^*
\lesssim_{\bbooo,\varepsilon,\gamma,M_0} \exp\biggl\{-\lambda^{1+\frac{\varepsilon-3\gamma}{2}}\,\frac{\tilde{c}_2}{(\log M_0)^2}\biggr\}
\lesssim_{\bbooo,M_0,C_{\mathrm{large}}} \ee^{-\bbooo C_{\mathrm{large}} \lambda}\,,
\end{equation}
where in \eqref{64*}, instead of \eqref{tilde_c_2}, we take
\begin{equation}
\label{tilde_c_2*}
\tilde{c}_2 =\frac{\bbooo}{2} \biggl(\frac{1}{2}\biggr)^{2-3\gamma} \sim \bbooo\,.
\end{equation}
Using \eqref{61*} and arguing similarly as for \eqref{65}, we estimate \eqref{63*} as
\begin{multline}
\label{65*}
J_2^* \leq \sum_{M \geq \max\{M_0,\lambda^{1/2}\}} M^2 \sum_{K \geq M/2} \exp\biggl\{-\bbooo \,\biggl(\frac{K}{M}\biggr)^{2-3\gamma} C_{\mathrm{large}} \lambda M^{\gamma} \biggr\}
\\
\lesssim \ee^{-\frac{\tilde{c}_2}{2} C_{\mathrm{large}} \lambda}\,,
\end{multline}
with $\tilde{c}_2$ as in \eqref{tilde_c_2*}. Combining \eqref{60*}, \eqref{64*}, and \eqref{65*}, we conclude that 
\begin{equation}
\label{Big_Case_1*}
\eqref{59*} \lesssim_{\bboo,C_{\mathrm{large}},\varepsilon,\gamma} \ee^{-\frac{\bboo}{16} \lambda} + \ee^{-\bbooo C_{\mathrm{large}} \lambda}+\ee^{-\frac{\tilde{c}_2}{2} C_{\mathrm{large}} \lambda}\,.
\end{equation}
By choosing the parameters analogously to \eqref{parameter_choice_1_Big_Case_1}--\eqref{parameter_choice_2_Big_Case_1}, we obtain that \eqref{Big_Case_1*} implies
\begin{equation}
\label{Big_Case_1_bound*}
\eqref{59*} \lesssim_{C_{\mathrm{Big}}} \ee^{-C_{\mathrm{Big}} \lambda}\,.
\end{equation}

\paragraph{\emph{\textbf{Modifications to Step 6}}} Recall that we are considering $M < \lambda^{1/2}$.
Similarly to \eqref{Big_Case_2}, we are interested in estimating
\begin{multline}
\label{Big_Case_2*}
\sum_{M<\lambda^{1/2}} M^2 \mathop{\mathrm{max}}_{|k| \simeq M} \Biggl[ \mathbb{P}_{\omega} \Biggl(\sum_{\frac{M}{2}
\leq K<\frac{\lambda^{1/2}}{2}}\, \sum_{j:\, \mathrm{max}\{|j|,|j-k|\} \simeq K} \Biggl|\frac{g_j(\omega)\,\overline{g_{j-k}(\omega)}}{\langle j \rangle\,\langle j-k \rangle}\Biggr|>\frac{\lambda_{2,M}}{2}\,\,
\\
\bigcap \,\, 
\Biggl| \sum_{m \in \Z^2} \frac{|g_m(\omega)|^2-1}{\langle m \rangle^2}\Biggr| \leq \BB\Biggr)\Biggr]
\\
+\sum_{M<\lambda^{1/2}} M^2 \mathop{\mathrm{max}}_{|k| \simeq M} \Biggl[\sum_{K \geq \frac{\lambda^{1/2}}{2}}  \mathbb{P}_{\omega} \Biggl(\Biggl|\sum_{j:\, \mathrm{max}\{|j|,|j-k|\} \simeq K} \frac{g_j(\omega)\,\overline{g_{j-k}(\omega)}}{\langle j \rangle\,\langle j-k \rangle}\Biggr|>\frac{\alpha_K \lambda_{2,M}}{2}\,\,
\\
\bigcap \,\, 
\Biggl| \sum_{m \in \Z^2} \frac{|g_m(\omega)|^2-1}{\langle m \rangle^2}\Biggr| \leq B\Biggr)\Biggr]=:L_1^*+L_2^*\,.
\end{multline}
When $d=2$, instead of \eqref{68}, we have
\begin{equation}
\label{68*}
\lambda_{2,M}=\biggl(\frac{\bbo \lambda_{1,M}}{M^2 v_M}\biggr)^{\frac{1}{2}} \gtrsim \biggl(\frac{\lambda}{M^{2-\varepsilon}\, (\log M)^2} \biggr)^{\frac{1}{2}} \gtrsim_{\varepsilon} \biggl(\frac{\lambda}{M^{2-\frac{\varepsilon}{2}}}\biggr)^{\frac{1}{2}}>\lambda^{\frac{\varepsilon}{8}}\,.
\end{equation}
Here, we used  \eqref{lambda_1M_choice}, \eqref{36*}, \eqref{55*}, and the assumption that $M<\lambda^{1/2}$.

By arguing as for \eqref{68b} and using \eqref{68*}, we obtain that
\begin{equation}
\label{68b*}
L_1^* \lesssim \lambda\, \mathbb{P}_{\omega} \Biggl( \sum_{|\ell| \leq 2\lambda^{1/2}} \frac{|g_{\ell}(\omega)|^2}{\langle \ell \rangle^2} \gtrsim \lambda^{\varepsilon/8}\,\,
\bigcap \,\, 
\Biggl| \sum_{m \in \Z^2} \frac{|g_m(\omega)|^2-1}{\langle m \rangle^2}\Biggr| \leq \BB\Biggr)\,.
\end{equation}
For \eqref{68b*}, we also used $\sum_{M<\lambda^{1/2}} M^2 \lesssim \lambda$. Let us now estimate the probability in \eqref{68b*}.
Suppose that $\omega$ satisfies
\begin{equation*}
\sum_{|\ell| \leq 2\lambda^{1/2}} \frac{|g_{\ell}(\omega)|^2}{\langle \ell \rangle^2} \gtrsim \lambda^{\varepsilon/8}\,,\qquad
\Biggl| \sum_{m \in \Z^2} \frac{|g_m(\omega)|^2-1}{\langle m \rangle^2}\Biggr|\leq \BB\,.
\end{equation*}
We then obtain 
\begin{equation}
\label{68c*}
\lambda^{\varepsilon/8} \lesssim \sum_{|\ell| \leq 2\lambda^{1/2}} \frac{1}{\langle \ell \rangle^2} + \sum_{|\ell| \leq 2\lambda^{1/2}} \frac{|g_{\ell}(\omega)|^2-1}{\langle \ell \rangle^2} \lesssim \log \lambda + \BB - \sum_{|\ell| > 2\lambda^{1/2}} \frac{|g_{\ell}(\omega)|^2-1}{\langle \ell \rangle^2}\,.
\end{equation}
From \eqref{68c*}, we obtain that
\begin{multline}
\label{68d*}
\mathbb{P}_{\omega} \Biggl( \sum_{|\ell| \leq 2\lambda^{1/2}} \frac{|g_{\ell}(\omega)|^2}{\langle \ell \rangle^2} \gtrsim \lambda^{\varepsilon/8}\,\,
\bigcap \,\, 
\Biggl| \sum_{m \in \Z^2} \frac{|g_m(\omega)|^2-1}{\langle m \rangle^2}\Biggr| \leq \BB\Biggr)
\\
\leq \mathbb{P}_{\omega} \Biggl( \Biggl|\sum_{|\ell| > 2\lambda^{1/2}} \frac{|g_{\ell}(\omega)|^2-1}{\langle \ell \rangle^2} \Biggr|\gtrsim_\BB \lambda^{\varepsilon/8}
\Biggr)\,.
\end{multline}
We note that 
\begin{equation}
\label{68e*}
\eqref{68d*} \lesssim \ee^{-c \lambda^{1+\frac{\varepsilon}{8}}}\,,
\end{equation}
for some constant $c>0$ depending on $\BB$.
In order to obtain \eqref{68e*}, we use Lemma \ref{Bernstein's_inequality} with 
\begin{equation*}
X_{\ell}=(|g_{\ell}(\omega)|^2-1)\,\mathbbm{1}_{|\ell|>2\lambda^{1/2}}\,,\quad b_{\ell}=\frac{1}{\langle \ell \rangle^2}\, \mathbbm{1}_{|\ell|>2\lambda^{1/2}}\,.
\end{equation*}
We note that then $\|b\|_{\ell^{\infty}} \sim \|b\|_{\ell^2}^2 \sim \frac{1}{\lambda}$, and \eqref{68e*} follows.
Combining \eqref{68b*}, \eqref{68d*}, and \eqref{68e*}, we deduce that 
\begin{equation}
\label{69*}
L_1^* \lesssim \ee^{-\frac{c}{2} \lambda^{1+\frac{\varepsilon}{8}}}  \lesssim_{C_{\mathrm{Big}}} \ee^{-C_{\mathrm{Big}} \lambda}\,,
\end{equation}
for $C_{\mathrm{Big}}>0$ arbitrarily large.

In order to estimate $L_2^*$ in \eqref{Big_Case_2*}, by arguing as for \eqref{70}, it suffices to estimate
\begin{multline}
\label{70*}
\sum_{M < \lambda^{1/2}} M^2 \sum_{K \geq \lambda^{1/2}/2} \ee^{-\bboo K^2}+\sum_{M < \lambda^{1/2}} M^2 \sum_{K \geq \lambda^{1/2}/2} \ee^{-\bbooo \alpha_K^2 \lambda_{2,M}^2 K^2/ \log K}
\\
=:L_{2,1}^*+L_{2,2}^*\,.
\end{multline}
with $\bboo>0$ large and $\bbooo \sim \frac{1}{\bboo}$. When $d=2$, we use \eqref{(2.35)*-(2.36)*} instead of \eqref{53}, and thus arrive at the bound on the right-hand side of \eqref{70*}.
By arguing analogously as in \eqref{71}, we obtain that 
\begin{equation}
\label{71*}
L_{2,1}^* \lesssim  \ee^{-\frac{\bboo}{10}\lambda}\,,
\end{equation}
which is an acceptable upper bound.

We now estimate $L_{2,2}^*$ in \eqref{70*}. We argue as for \eqref{72} but use \eqref{57*} instead of \eqref{57} and write
\begin{multline}
\label{72*}
L_{2,2}^* \leq \sum_{M < \lambda^{1/2}} M^2 \sum_{K \geq \lambda^{1/2}/2} \exp\biggl\{-\bbooo \,\biggl(\frac{K}{M}\biggr)^{2-3\gamma} \frac{M^{\varepsilon-3\gamma}}{(\log M)^2} \,\lambda\biggr\}
\\
\lesssim_{\bbooo} 
\sum_{M < \lambda^{1/2}} M^{4-3\gamma}\, (\log M)^2\sum_{K \geq \lambda^{1/2}/2} \exp\biggl\{-\frac{\bbooo}{2} \,\biggl(\frac{K}{M}\biggr)^{2-3\gamma} \frac{M^{\varepsilon-3\gamma}}{(\log M)^2} \,\lambda\biggr\}
\\
\times K^{-2+3\gamma}\,.
\end{multline}
For $M<\lambda^{1/2}$ and $K \geq \lambda^{1/2}/2$, we have that
\begin{equation}
\label{72b*}
(\log M)^2 \lesssim_{\varepsilon, \gamma} \lambda^{\frac{\varepsilon-3\gamma}{4}}\,,\quad M^{-2+\varepsilon}>\lambda^{\frac{-2+\varepsilon}{2}}\,,\quad K^{2-3\gamma} \gtrsim_{\gamma} \lambda^{\frac{2-3\gamma}{2}}\,.
\end{equation}
Above, we recall \eqref{58*}.
Substituting \eqref{72b*} into \eqref{72*}, and recalling \eqref{tilde_c_2*}, we obtain that
\begin{multline}
\label{73*}
L_{2,2}^* \lesssim  \ee^{-\tilde{c}_2 \lambda^{1+\frac{\varepsilon-3\gamma}{4}}} \,\sum_{M < \lambda^{1/3}} M^{4-3\gamma}\, (\log M)^2\sum_{K \geq \lambda^{1/2}/2} K^{-2+3\gamma}
\\
\lesssim \ee^{-\frac{\tilde{c}_2}{2} \lambda^{1+\frac{\varepsilon-3\gamma}{4}}} \lesssim_{C_{\mathrm{Big}}} \ee^{-C_{\mathrm{Big}} \lambda}\,,
\end{multline}
for $C_{\mathrm{Big}}>0$ arbitrarily large. This is possible by \eqref{58*}. The contribution to \eqref{dyadic_sum*} when $M < \lambda^{1/2}$ satisfies the wanted upper bound by \eqref{69*}, \eqref{71*}, and \eqref{73*}. This completes the proof of \eqref{Bourgain_large_deviation_estimate_rewritten} when $d=2$.

\subsection{Positivity of the partition function when $d=2,3$}
\label{Appendix_B_Partition_function}

In this section, we prove that $z^{(\delta)} \equiv z^{(\delta)}_{\BB}$ defined in \eqref{z_R_definition} is positive when $d=2,3$, as was claimed in Proposition \ref{Hartree_equation_Malliavin_derivative} (iii). We can deduce this by recalling Definition \ref{cut-off_chi} and by using the following result. 
\begin{lemma}
    \label{app.lem.partition_funct}
Let $d=2,3$. For all $R>0$, we have
\begin{equation*}
\mathbb{P}_{\omega} \Biggl(\Biggl| \sum_{\ell \in \Z^d} \frac{|g_{\ell}(\omega)|^2-1}{\langle \ell \rangle^2}\Biggr| \leq R\Biggr)>0
\end{equation*}
\end{lemma}

\begin{proof}
Let us first consider the case $d=3$. 
Let $M \in \N^{+}$ be given. 
We apply Lemma \ref{Bernstein's_inequality} $X_\ell=|g_{\ell}|^2-1 \in \psi_1$ and $a_{\ell}=\frac{1}{\langle \ell \rangle^2}\,\mathbf{1}_{|\ell|>M}$. Then, we have 
\begin{equation}
\label{d=3_a_bounds}
\|a\|_{\ell^2}^2 \sim \frac{1}{M}\,,\qquad \|a\|_{\ell^{\infty}} \sim \frac{1}{M^2}\,.
\end{equation}
Using Lemma \ref{Bernstein's_inequality} and \eqref{d=3_a_bounds}, we get that
\begin{equation}
\label{d=3_Bernstein_inequality}
\mathbb{P}_{\omega} \Biggl(\Biggl| \sum_{|\ell| >M} \frac{|g_{\ell}(\omega)|^2-1}{\langle \ell \rangle^2}\Biggr| > \frac{R}{2}\Biggr)\lesssim \exp\Bigl\{-C \,\mathrm{min} \bigl(R^2M,RM^2\bigr)\Bigr\}\,.
\end{equation}
Note that in \eqref{d=3_Bernstein_inequality}, the constant $C$ and the implied constant are independent of $M,R$.
From \eqref{d=3_Bernstein_inequality}, we obtain that
\begin{equation}
\label{d=3_Bernstein_inequality_2}
\mathbb{P}_{\omega} \Biggl(\Biggl| \sum_{|\ell| >M} \frac{|g_{\ell}(\omega)|^2-1}{\langle \ell \rangle^2}\Biggr| > \frac{R}{2}\Biggr)<1
\end{equation}
for $M$ sufficiently large depending on $R$. Let us henceforth fix $M$ satisfying \eqref{d=3_Bernstein_inequality_2}.
In particular, for such $M$, we have 
\begin{equation}
\label{d=3_Bernstein_inequality_3}
\mathbb{P}_{\omega} \Biggl(\Biggl| \sum_{|\ell| >M} \frac{|g_{\ell}(\omega)|^2-1}{\langle \ell \rangle^2}\Biggr| \leq \frac{R}{2}\Biggr)>0\,.
\end{equation}
We note that 
\begin{equation}
\label{d=3_Bernstein_inequality_4}
\mathbb{P}_{\omega} \Biggl(\Biggl| \sum_{|\ell| \leq M} \frac{|g_{\ell}(\omega)|^2-1}{\langle \ell \rangle^2}\Biggr| \leq \frac{R}{2}\Biggr)>0\,,
\end{equation}
as we are considering a finite sum of the random variables\footnote{More precisely, we know that for every $|\ell| \leq M$ we have 
\begin{equation*}
\mathbb{P}_{\omega} \Biggl(\Biggl| \frac{|g_{\ell}(\omega)|^2-1}{\langle \ell \rangle^2}\Biggr| \leq \frac{R}{2 (2M+1)^3}\Biggr)>0\
\end{equation*}
and we use the triangle inequality as in the step that follows.}.
Using the triangle inequality, the independence of the $g_{\ell}$, and \eqref{d=3_Bernstein_inequality_3}--\eqref{d=3_Bernstein_inequality_4}, we deduce that
\begin{multline*}
\mathbb{P}_{\omega} \Biggl(\Biggl| \sum_{\ell \in \Z^3} \frac{|g_{\ell}(\omega)|^2-1}{\langle \ell \rangle^2}\Biggr| \leq R\Biggr) 
\\
\geq \mathbb{P}_{\omega} \Biggl(\Biggl| \sum_{|\ell|  \leq M} \frac{|g_{\ell}(\omega)|^2-1}{\langle \ell \rangle^2}\Biggr| \leq \frac{R}{2} \bigcap \sum_{|\ell| >M} \frac{|g_{\ell}(\omega)|^2-1}{\langle \ell \rangle^2}\Biggr| \leq \frac{R}{2}\Biggr)
\\
=\mathbb{P}_{\omega} \Biggl(\Biggl| \sum_{|\ell|  \leq M} \frac{|g_{\ell}(\omega)|^2-1}{\langle \ell \rangle^2}\Biggr| \leq \frac{R}{2}\Biggr) \, \mathbb{P}_{\omega} \Biggl(\Biggl| \sum_{|\ell| >M} \frac{|g_{\ell}(\omega)|^2-1}{\langle \ell \rangle^2}\Biggr| \leq \frac{R}{2}\Biggr)>0\,.
\end{multline*}
This concludes the proof when $d=3$. The proof when $d=2$ is analogous. Instead of \eqref{d=3_a_bounds}, we have 
\begin{equation*}
\|a\|_{\ell^2}^2 \sim \frac{1}{M^2}\,,\qquad \|a\|_{\ell^{\infty}} \sim \frac{1}{M^2}\,.
\end{equation*}
Hence, instead of \eqref{d=3_Bernstein_inequality}, we have
\begin{equation*}
\mathbb{P}_{\omega} \Biggl(\Biggl| \sum_{|\ell| >M} \frac{|g_{\ell}(\omega)|^2-1}{\langle \ell \rangle^2}\Biggr| > \frac{R}{2}\Biggr)\lesssim \exp\Bigl\{-C \,\mathrm{min} \bigl(R^2M^2,RM^2\bigr)\Bigr\}\,.
\end{equation*}
The claim now follows as before.
\end{proof}

\end{document}